\documentclass{article}
\usepackage{amsmath} 
\usepackage{amsthm} 
\usepackage{mathtools} 
\usepackage{amssymb} 
\usepackage[left=1.5in,right=1.5in,top=1in,bottom=1in]{geometry}
\usepackage{hyperref} 
\usepackage{cleveref} 
\usepackage{mathrsfs} 
\usepackage{tikz-cd} 
\usepackage{tikz} 
\usepackage{import} 
\usepackage{enumitem} 
\usepackage[mathscr]{eucal}

\DeclareMathOperator{\out}{Out}

\DeclareMathOperator{\st}{st}
\DeclareMathOperator{\stab}{Stab}
\DeclareMathOperator{\pf}{PF}
\DeclareMathOperator{\per}{Per}
\DeclareMathOperator{\fix}{Fix}
\DeclareMathOperator{\rank}{rank}
\newcommand{\pfmin}{\pf_{\operatorname{min}}}
\newcommand{\rneg}{\mathcal{R}_{\operatorname{NEG}}}

\newtheorem{THM}{Theorem}
\newtheorem{thm}{Theorem}
\numberwithin{thm}{section}
\newtheorem{lem}[thm]{Lemma}
\newtheorem{claim}[thm]{Claim}
\newtheorem{prop}[thm]{Proposition}
\newtheorem{cor}[thm]{Corollary}
\theoremstyle{definition}
\newtheorem{ex}[thm]{Example}

\usetikzlibrary{decorations.markings, shapes, quotes}
\tikzset{->-/.style = {
		very thick,
		decoration = {
			markings,
			mark = at position .5 with {\arrow{>}}
		},
		postaction = {decorate}
	},
	pt/.style = {
		circle,
		fill = black,
		scale = 0.5,
	},
}

\begin{document}
\title{CTs for free products}
\author{Rylee Alanza Lyman}
\maketitle
\begin{abstract}
    The fundamental group of a finite graph of groups
    with trivial edge groups is a free product.
    We are interested in those outer automorphisms
    of such a free product that permute
    the conjugacy classes of the vertex groups.
    We show that in particular cases of interest, such as 
    where vertex groups are themselves finite free products of finite and cyclic groups,
    given such an outer automorphism,
    after passing to a positive power,
    the outer automorphism is represented
    by a particularly nice kind of relative train track map
    called a CT.
    CTs were first introduced by Feighn and Handel
    for outer automorphisms of free groups.
    We develop the theory of attracting laminations for
    and principal automorphisms of free products.
    We prove that outer automorphisms of free products
    satisfy an index inequality reminiscent of
    a result of Gaboriau, Jaeger, Levitt and Lustig
    and sharpening a result of Martino.
    Finally, we prove a result reminiscent of a theorem of Culler
    on the fixed subgroup of an automorphism of a free product
    whose outer class has finite order.
\end{abstract}

A homotopy equivalence $f\colon G \to G$
of a connected graph $G$ is a \emph{train track map}
when it maps vertices to vertices
and $f^k$ restricts,
on each edge of $G$,
to an immersion for all $k \ge 1$.
Train track maps were introduced
by Bestvina--Handel in \cite{BestvinaHandel},
where they showed that an appropriate generalization,
called a \emph{relative train track map,}
exists representing each $\varphi \in \out(F_n)$,
the outer automorphism group of a free group
of finite rank $n$.

There are by now many proofs of the existence
of relative train track maps for free products:
the original proof by Collins--Turner
\cite{CollinsTurner}
used graphs of spaces 
\`a la Scott and Wall \cite{ScottWall};
Francaviglia and Martino \cite{FrancavigliaMartino}
used the Lipschitz metric on Outer Space
and viewed relative train track maps
as twisted equivariant maps of trees; and
the present author \cite{Myself} worked directly in the graph of groups.

In the intervening time,
the theory of relative train track maps for $\out(F_n)$
has progressed,
first with \emph{improved relative train track maps}
in \cite{BestvinaFeighnHandel},
and more recently with \emph{CTs} in \cite{FeighnHandel}.
Free products have lagged behind.
The purpose of this paper is to rectify this situation.

Let $F$ be a group that decomposes as a free product
of the form
\[ F = A_1*\dotsb*A_n * F_k, \]
where the $A_i$ are countable groups and $F_k$
is a free group of rank $k$.
Let $[[A_i]]$ denote the conjugacy class of $A_i$,
and write $\mathscr{A} = \{[[A_1]],\ldots,[[A_n]]\}$.
Write $\out(F,\mathscr{A})$ for the subgroup of
$\out(F)$ consisting of outer automorphisms
that permute the conjugacy classes in $\mathscr{A}$.

\begin{THM}
    \label{CTtheorem}
    Each rotationless $\varphi \in \out(F,\mathscr{A})$
    is represented by a CT.
\end{THM}

See \Cref{CTsection} for the definition of a CT,
and see \Cref{principalsection} for the definition
of rotationless.
The question of whether a given $\varphi\in \out(F,\mathscr{A})$
has a rotationless power seems to require some restrictions on groups in $\mathscr{A}$:
our sufficient condition in \Cref{rotationlesscriterion}
requires finite-order elements of $A \in \mathscr{A}$ to have bounded order
and requires each automorphism $\Phi\colon A \to A$ to have a bound on the period
of $\Phi$-periodic elements.
It would be interesting to know to what extent this sufficient condition is necessary.
We prove in \Cref{freeproductsoffiniteandycycliccor} 
together with \Cref{rotationlessisrotationless}
that if $F$ is a finite free product of finite and cyclic groups
(so that each group in $\mathscr{A}$ is again of this form)
then each $\varphi \in \out(F,\mathscr{A})$ has a rotationless power.

We present one application of \Cref{CTtheorem}.
The original application of relative train track maps in \cite{BestvinaHandel}
was to prove the \emph{Scott conjecture,}
which predicts that for an automorphism
$\Phi\colon F_n \to F_n$,
the rank of the fixed subgroup $\fix(\Phi)$ satisfies
\[  \operatorname{rank}(\fix(\Phi)) \le n. \]
Collins and Turner \cite{CollinsTurner} generalized this result to free  products;
here the \emph{(Kurosh) rank} of a free product $F = A_1*\cdots*A_n*F_k$
is the quantity $n + k$.
A subgroup $H$ of $F$
inherits a free product decomposition from $\mathscr{A}$
and thus there is a notion of the \emph{(Kurosh subgroup) rank} of $H$,
which a priori may be infinite.
Their results imply that if $\Phi\colon (F,\mathscr{A}) \to (F,\mathscr{A})$
is an automorphism of $F$ preserving $\mathscr{A}$, then
\[  \operatorname{rank}(\fix(\Phi)) \le n + k. \]
In actual fact,
they state their results only for the Grushko decomposition
of a finitely generated group,
but their proofs do not use this assumption in any essential way.

Back in the setting of free groups,
Gaboriau, Jaeger, Levitt and Lustig \cite{GaboriauJaegerLevittLustig}
define an \emph{index} $i(\varphi)$ for an outer automorphism
$\varphi \in \out(F_n)$;
it is a sum over automorphisms $\Phi \colon F_n \to F_n$
representing $\varphi$  up to an equivalence relation called \emph{isogredience}
(see \Cref{principalsection} for a definition)
of the quantity
\[  \max \left\{0, 
\operatorname{rank}(\fix(\Phi)) + \frac{1}{2}\alpha(\Phi) - 1 \right\}, \]
where $\alpha(\Phi)$ is the number of $\fix(\Phi)$-orbits
of \emph{attracting fixed points} (see \Cref{laminationsection})
for the action of $\Phi$ on the Gromov boundary $\partial F_n$.
They prove that
\[  i(\varphi) \le n - 1. \]
Martino \cite{Martino} established the analogous result for free products.

Most recently, Feighn and Handel define a new index $j(\varphi)$
that gives weight $1$ rather than $\frac{1}{2}$ to certain classes
of attracting fixed points called \emph{eigenrays}, see \Cref{indexsection} for details.
Using CTs in \cite{FeighnHandelAlg},
they show that their index satisfies
\[  j(\varphi) \le n-1. \]

Following \cite{GuirardelHorbez}, we say
a \emph{Grushko $(F,\mathscr{A})$-tree} is a simplicial tree $T$
equipped with an $F$-action with trivial edge stabilizers
such that $\mathscr{A}$ is the set of conjugacy classes of nontrivial vertex stabilizers.
Rather than consider the Gromov boundary of some (and hence any) Grushko $(F,\mathscr{A})$-tree $T$,
we, like Guirardel--Horbez, consider the \emph{Bowditch boundary,}
which has the advantage of being compact (when the groups in $\mathscr{A}$ are countable)
and which is obtained from the Gromov boundary by adding
the vertices in $T$ of infinite valence, suitably topologized.
An automorphism may fix one of these points of infinite valence
without fixing any elements of its stabilizer group.
In our version of the index theorem 
(see \Cref{preciseindexthm} and \Cref{actuallynoneedtotreatC2differently})
we count the number of $\fix(\Phi)$-orbits of such infinite valence points as well,
and we obtain the following theorem.

\begin{THM}
    \label{indexthm}
    Suppose that each group in $\mathscr{A}$ is finitely generated.
    If $\varphi \in \out(F,\mathscr{A})$ has a rotationless power,
    the index $j(\varphi)$ satisfies
    \[  j(\varphi) \le n + k - 1. \]
\end{THM}

Because our proof uses CTs, the rotationless power assumption is necessary.
A key step in the proof uses ``Nielsen realization'' for free products,
due to Hensel--Kielak \cite{HenselKielak},
who require groups in $\mathscr{A}$ to be finitely generated.
It would be interesting to know whether
a Nielsen realization theorem holds without the finite generation assumption,
perhaps following Hensel--Kielak's proof
using Grushko $(F,\mathscr{A})$-trees rather than particular Cayley graphs for $F$.

Along the way to proving \Cref{indexthm},
we prove a version of a theorem of Culler \cite[Theorem 3.1]{Culler},
which may be of independent interest.
See \Cref{Cullerprop} for a precise statement.
The realization statement in the theorem relies on \cite{HenselKielak}.
A subgroup of $F$ is non-peripheral if it is not conjugate into any $A \in \mathscr{A}$.
\begin{THM}
    \label{CullerTHM}
    Suppose $\varphi \in \out(F,\mathscr{A})$ has finite order,
    that each $A \in \mathscr{A}$ is finitely generated,
    and that $\Phi$ represents $\varphi$.
    There exists an automorphism of graphs of groups $f\colon \mathcal{G} \to \mathcal{G}$
    representing $\varphi$.
    Either $\fix(\Phi)$ is cyclic and non-peripheral
    or $\fix(\Phi)$ is conjugate to the fundamental group of a graph of groups $\mathcal{H}$
    that has an injective-on-edges immersion into a component of $\fix(f)$.
\end{THM}
Culler's theorem says that if $\Phi$ represents $\varphi \in \out(F_n)$
and $\fix(\Phi)$ is not cyclic, then it is a free factor
that is a component of $\fix(f)$.
In our theorem, the injective-on-edges assertion is meant to replace the free factor conclusion,
which is too strong in general.
Likewise, because of the presence of almost fixed but not fixed edges,
one should not expect an entire component of $\fix(f)$ to appear in the statement.

\paragraph{}
The broad-strokes strategy of this  paper is to find the correct equivariant
perspective so that arguments scattered across
\cite{BestvinaHandel}, \cite{BestvinaFeighnHandelSolvable}, 
\cite{BestvinaFeighnHandelSolvable}, \cite{FeighnHandel},
\cite{FeighnHandelAlg}, and other papers
can be adapted without too much extra effort.
We have made some effort to synthesize these results,
providing a self-contained exposition
with the exception of the proof of the basic existence of relative train track maps,
for which the reader is referred to \cite{CollinsTurner},
\cite{FrancavigliaMartino} or \cite{Myself}.

Despite the similarities with the case of free groups,
there are also some interesting differences.
We describe a few for the expert reader.
Perhaps the first difference is the presence of the word \emph{almost}
in constructions like \emph{almost periodic, almost Nielsen path,} and \emph{almost linear.}
Here \emph{almost} refers to the presence of vertex group elements.

A reader familiar with relative train track theory for $\out(F_n)$
might like to imagine beginning with a relative train track map of a graph $f\colon G \to G$,
and then collapsing each component of some filtration element $G_{r-1}$
to obtain a map of graphs of groups $f'\colon \mathcal{G} \to \mathcal{G}$.
We show in \Cref{basicssection} that $f'$ is again a relative train track map,
but to define the map of the collapsed graph of groups $f'\colon \mathcal{G} \to \mathcal{G}$,
one needs to make some choices (see \Cref{basicssection} for details).
The failure of $f$ to preserve these choices
results in the presence of vertex group elements at the ends of images of edges
with endpoints in $G_{r-1}$.

Thus a non-exponentially growing edge $E$ of $H_r$ in $G$
becomes an \emph{almost periodic edge} of $H'_r$ in $\mathcal{G}$;
here ``almost'' because in general we have $f'(E) = Eg$ 
for some vertex group element $g$.
By making different choices,
we may transform for instance almost fixed edges into genuinely fixed edges,
but it is typically not possible to do this simultaneously 
for each edge incident to a vertex with nontrivial vertex group.
An \emph{almost Nielsen path} differs from a Nielsen path
in that $f(\sigma)$ is homotopic to $g\sigma h$
for vertex group elements $g$ and $h$.
An \emph{almost linear edge} has a suffix which is an almost Nielsen path.

A second difference, observed by Martino in his thesis \cite{MartinoThesis},
has to do with almost linear edges that are not genuinely linear.
Suppose $f\colon G \to G$ is a relative train track map
and that $E$ is an edge in an exponentially growing stratum $H_r$
such that $f(E) = Eu$ for some path $u$.
If $\tilde E$, $\tilde u$ and $\tilde f \colon \Gamma \to \Gamma$ are lifts to the universal cover
such that the path $\tilde f(\tilde E) = \tilde E\tilde u$,
then
\[
    \tilde E \subset \tilde f_\sharp(\tilde E) \subset \tilde f^2_\sharp(\tilde E) \subset \cdots
\]
is an increasing sequence of tight paths in $\Gamma$ whose union is a ray $\tilde R_{\tilde E}$
converging to a fixed point $P \in \partial\Gamma \cong \partial F_n$.
The number of $H_r$-edges of the path $f^k_\sharp(u)$ grows exponentially
so \cite[Proposition 1.1]{GaboriauJaegerLevittLustig} implies that $P$ is an attracting fixed point.
In fact, one does not need exponential growth of $u$
to show that $P$ is an (``algebraic'') attracting fixed point in this way,
merely that it grows at all.
If it does not grow, then since there are finitely many paths of a given length in a graph,
some $f^k_\sharp(u)$ is a periodic Nielsen path,
and $P$ is instead in the boundary of the fixed subgroup 
of the automorphism corresponding to the lift $\tilde f$.
For free products, it is no longer true in general that there are finitely many graph-of-groups 
paths of a given length,
and in fact if $f^k_\sharp(u)$ is a concatenation of almost Nielsen paths but not a Nielsen path,
then one cannot expect $P$ to be in the boundary of the fixed subgroup.
Fortunately, Martino shows that in this situation
although $P$ is not an ``algebraic'' attracting fixed point,
it is still an attracting fixed point;
see \Cref{GJLLprop} and \Cref{CT3}.

Since we work in the Bowditch boundary $\partial(F,\mathscr{A})$, rather than the Gromov boundary,
there are also fixed points coming from vertices of infinite valence.
It appears possible for these points to be neither in the boundary of the fixed subgroup
nor attractors.
Fortunately the vertices of infinite valence determine a well-defined subspace of $\partial(F,\mathscr{A})$,
so the distinction is less important.
As in \cite[Section 3]{FeighnHandel}, we discuss \emph{principal automorphisms,}
\emph{principal periodic points} and \emph{principal lifts.}
A first guess might be that since vertices with nontrivial vertex group are permuted
by homotopy equivalences of graphs of groups,
these periodic points might be principal.
In fact this is true with one exception, the case of the cyclic group of order 2.
Vertices with $C_2$ vertex group need not be principal
from the perspective of the Bowditch boundary,
and in fact often are not principal, see \Cref{CT1}.

An extreme example of this behavior is the case of \emph{dihedral pairs;}
subgraphs of groups of $\mathcal{G}$ isomorphic to the quotient of $\mathbb{R}$ 
by the standard action of $C_2*C_2$
which are $f$-periodic
with no outward-pointing periodic directions.
Although dihedral pairs are preserved by rotationless $f$,
since $C_2$ points are not principal,
there is no guarantee that $f$ induces the trivial outer automorphism of $C_2*C_2$,
and indeed there are examples where $f$ does not, see \Cref{CT2}.
In Feighn--Handel's construction of CTs,
one key step is to get rid of periodic circles with no outward-pointing periodic directions
by ``untwisting'' them.
Imagine this taking place on a graph equipped with a $C_2$ action,
so that the quotient is a graph of groups with $C_2$ vertex groups.
The problem is that this ``untwisting'' cannot be done $C_2$-equivariantly.
Rather than get rid of dihedral pairs,
we are forced to build them into the theory.
Thus a useful key feature of CTs for free groups:
namely that every filtration element contains a principal point,
fails for free products in general.

\paragraph{}
Here is the organization of the paper.
We give a quick review of Bass--Serre theory and relative train track maps
in \Cref{basicssection} by way of setting the scene.
We expect many readers to be interested in the case of $(F,\mathscr{A})$ where $F$ is a free group,
and we describe collapsing relative train track maps from graphs to graphs of groups.
\Cref{improvingsection} is given to proving the analogue of
\cite[Theorem 2.19]{FeighnHandel},
which provides better relative train track maps for all outer automorphisms
$\varphi \in \out(F,\mathscr{A})$. 
In \Cref{boundarysection},
we study the action of automorphisms of $(F,\mathscr{A})$ on the Bowditch boundary.
We give an account of a result from Martino's thesis,
an analogue of \cite[Proposition 1.1]{GaboriauJaegerLevittLustig} for free products.
We develop the theory of \emph{attracting laminations} for outer automorphisms of free products
in \Cref{laminationsection}.
Attracting laminations are very useful in the study of $\out(F_n)$,
and we have provided a bit more of the theory than we actually have need of in the paper,
for example in \Cref{PFhomomorphism} the existence of a homomorphism
from the stabilizer of an attracting lamination to $\mathbb{Z}$.
In \Cref{principalsection}, we study principal automorphisms of $(F,\mathscr{A})$
and define rotationless outer automorphisms and relative train track maps
and prove their equivalence under the technical assumption
of the existence of a rotationless iterate of a relative train track map.
The construction of CTs for rotationless outer automorphisms
of free products is accomplished in \Cref{CTsection}.
\emph{Geometric strata} play an important role in the definiton of an
improved relative train track map in \cite{BestvinaFeighnHandel}
but do not figure into the definition of a CT.
We have not attempted to develop a theory of geometric strata for free products,
but such a project would surely be interesting.
\Cref{indexthm} and \Cref{CullerTHM} are proven in \Cref{indexsection}.
One useful consequence of Feighn--Handel's index inequality
is the following.
They prove that if $\varphi \in \out(F_n)$ is in the kernel of the action of $\out(F_n)$
on the homology of $F_n$ with $\mathbb{Z}/3\mathbb{Z}$-coefficients
and induces the trivial permutation on a certain finite set whose size is bounded by their index theorem,
then $\varphi$ is rotationless.
Such a result for $\out(F,\mathscr{A})$, say for $F$ a free product of finite and cyclic groups,
would certainly be interesting.
Certainly such $\out(F,\mathscr{A})$ are virtually torsion free,
and our index theorem provides a bound on the relevant finite set,
but we are unaware of a torsion-free subgroup of finite index
with the special properties $\operatorname{IA}_n(\mathbb{Z}/3\mathbb{Z})$ enjoys.

Part of this work originally appeared in the author's thesis \cite{MyThesis}.
The author would like to thank Jean-Pierre Mutanguha and Sami Douba for helpful conversations,
Lee Mosher for many helpful conversations, encouragement and comments on early drafts of this work
and the anonymous referee for many helpful comments which helped
improve the exposition of this article.

\section{Relative train track maps on graphs of groups}
\label{basicssection}

The purpose of this section is to set the scene
for the later sections.
We define graphs of groups with trivial edge groups,
as well as
maps and homotopies thereof.
We recall the definition of a relative train track map
and prove a proposition about
relative train track maps realizing a given
nested sequence of $\varphi$-invariant free factor systems.
We describe how to collapse relative train track maps on graphs
to obtain relative train track maps on graphs of groups
and conversely, how to blow up relative train track maps
on graphs of free groups to relative train track maps on graphs
given relative train track maps representing the maps on vertex groups.

\paragraph{Graphs of groups with trivial edge groups.}
A \emph{graph of groups with trivial edge groups} $\mathcal{G}$
is a graph $G$ together with an assignment of groups
$\mathcal{G}_v$ to the vertices of $G$.
Throughout the paper, unless otherwise specified,
our graphs of groups will be assumed to be connected and finite.

Associated to a graph of groups $\mathcal{G}$ with trivial edge groups
with underlying graph $G$,
we can build a \emph{graph of spaces} $X_{\mathcal{G}}$.
See \cite{ScottWall} for more details.
For each vertex $v$ of $G$,
take a connected CW complex $X_v$
that is a $K(\mathcal{G}_v,1)$
with one vertex $\star_v$ and fix an identification
$\pi_1(X_v,\star_v) = \mathcal{G}_v$.
If $V$ is the set of vertices of $G$
and $E$ is the set of edges 
(which we think of as coming with a choice of orientation),
the graph of spaces $X_{\mathcal{G}}$
is the quotient of the disjoint union
\[  \coprod_{v\in V} X_v \amalg E\times [0,1] \]
by the equivalence relation identifying the point $(e,1)$ with $\star_{\tau(e)}$
and $(e,0)$ with $\star_{\iota(e)}$,
where $\tau(e)$ denotes the terminal vertex of the edge $e$ 
and $\iota(e)$ denotes the initial vertex.
There is a natural ``retraction'' $X_\mathcal{G} \to G$
that collapses each $X_v$ to the point $\star_v$.

The \emph{fundamental group of the graph of groups $\pi_1(\mathcal{G})$}
is the fundamental group of the graph of spaces $X_{\mathcal{G}}$.
For convenience, choose a basepoint $p \in X_{\mathcal{G}}$
in the image of the retraction $X_{\mathcal{G}} \to G$.
Each loop in $\pi_1(X_{\mathcal{G}},p)$ may be represented by an
\emph{edge path} $\gamma$ of the form
\[  \gamma = e'_1g_1e_2g_2\ldots e_kg_ke'_{k+1}. \]
Here $e_2,\ldots,e_k$ are edges of $G$,
$e'_1$ and $e'_{k+1}$ are terminal and initial segments of edges,
respectively,
and the concatenation $e'_1e_2\ldots e_ke'_{k+1}$
is a path in $G$.
The $g_i$ for $1 \le i \le k$ are elements of
$\pi_1(X_{v_i},\star_{v_i}) = \mathcal{G}_{v_i}$,
where $v_i = \tau(e_i) = \iota(e_{i+1})$.
We allow $\gamma$ to begin or end at vertices,
in which case the initial or terminal segments should be dropped
from the notation.
A path is \emph{nontrivial} if it contains a (segment of) an edge $e_i$.

Notice that the notion of an edge path makes sense
without reference to the space $X_\mathcal{G}$.
Homotopy rel endpoints of paths in $X_\mathcal{G}$
yields a corresponding notion of homotopy for edge paths in $\mathcal{G}$;
it is generated by adding or removing segments of the form $e\bar e$
for $e$ an oriented edge of $\mathcal{G}$ and $\bar e$
the edge $e$ in the opposite orientation.
An edge path is \emph{tight} if it cannot be shortened by a homotopy.

\paragraph{The Bass--Serre tree.}
The fundamental group of a finite graph of groups with trivial edge groups
is a free product of the form
\[  F = A_1*\cdots*A_n*F_k, \]
where the $A_i$ are the nontrivial vertex groups of $\mathcal{G}$ and $F_k$,
a free group of rank $k$,
is the ordinary fundamental group of $G$.
Recall from the introduction our notation $\mathscr{A}$
for the set of conjugacy classes of the $A_i$.
We will call such a graph of groups
a \emph{Grushko $(F,\mathscr{A})$-graph of groups.}
We underline that although we use the term Grushko,
this splitting is \emph{not} assumed to be
the Grushko splitting of a finitely generated group.
For example $F$ may be a free group,
and the $A_i$ may be free factors of $F$.

A \emph{Grushko $(F,\mathscr{A})$-tree}
is a simplicial tree $T$ equipped 
with a minimal $F$-action with trivial edge stabilizers
and vertex stabilizers trivial or conjugate
to some $A_i$.
The fundamental theorem of Bass--Serre theory asserts that
associated to each Grushko $(F,\mathscr{A})$-tree,
there is a quotient Grushko $(F,\mathscr{A})$-graph of groups
(which is finite if the tree is \emph{minimal,}
i.e.~there is no proper $F$-invariant subtree)
and conversely associated to any (not necessarily finite)
Grushko $(F,\mathscr{A})$-graph of groups $\mathcal{G}$,
there is a Grushko $(F,\mathscr{A})$-tree $\Gamma$ called the \emph{Bass--Serre tree}
for $\mathcal{G}$.
For the constructions of the quotient graph of groups
and Bass--Serre tree, the reader is referred to
\cite{Trees}, \cite{Bass}, or \cite[Section 1]{Myself}.
Let us remark that to identify a Grushko $(F,\mathscr{A})$-tree $\Gamma$
as the Bass--Serre tree of a Grushko $(F,\mathscr{A})$-graph of groups $\mathcal{G}$,
there is a choice not only of a basepoint $p$ in $\mathcal{G}$
and a lift $\tilde p$ to $\Gamma$,
but also a choice of fundamental domain for the action of $F$ on $\Gamma$
containing $\tilde p$.
However, beginning with the graph of groups and constructing the Bass--Serre tree
requires merely the choice of a basepoint, as in the case of topological spaces.

Given a vertex $v$ of $G$, write $[p,v]$ for the set of homotopy classes rel endpoints
of paths in $\mathcal{G}$ from $p$ to $v$.
There is a natural right action of $\mathcal{G}_v$ on $[p,v]$:
an element $g$ sends the homotopy class of a path $\gamma$ to the homotopy class
of the composite path $\gamma g$.
The vertex set of $\Gamma$ is the set
\[  \coprod [p,v]/\mathcal{G}_v \]
where the disjoint union is over the set of vertices of $G$.
One can extend the above definition,
defining points of $\Gamma$ to be (equivalence classes of) homotopy classes of paths
beginning at $p$.
Alternatively one may say that
two vertices $[\gamma]\mathcal{G}_v$ and $[\gamma']\mathcal{G}_w$
are connected by an edge if $\bar\gamma'\gamma$ is homotopic to a path
of length one, i.e.~of the form $geg'$ for vertex group elements $g$ and $g'$ and $e$ an edge of $G$.

\paragraph{Maps of graphs of groups.}
A \emph{map} of graphs of groups (with trivial edge groups)
$f\colon \mathcal{G} \to \mathcal{G}'$
is a pair of maps $f\colon G \to G'$ and $f_X \colon X_\mathcal{G} \to X_{\mathcal{G}'}$
such that the following diagram commutes
\[  \begin{tikzcd}
    X_\mathcal{G} \ar[r, "f_X"] \ar[d, "r"] & X_{\mathcal{G}'} \ar[d, "r'"] \\
    G \ar[r, "f"] & G',
\end{tikzcd}\]
where $r$ and $r'$ are the retractions.
A \emph{homotopy} between maps of graphs of groups
is a commutative diagram of homotopies.
A map $f\colon \mathcal{G} \to \mathcal{G}'$
is a \emph{homotopy equivalence} if (as one might expect)
there is a map $g\colon \mathcal{G}' \to \mathcal{G}$
such that $gf$ and $fg$ are homotopic to the respective identity maps.
It is not too hard to see that each map
$f\colon \mathcal{G} \to \mathcal{G}'$ is homotopic to a map
that sends the vertices $\star_v$ of $X_\mathcal{G}$ to vertices
$\star_{v'}$ of $X_{\mathcal{G}'}$
and thus sends edges of $G \subset X_\mathcal{G}$
to (possibly trivial) edge paths in $\mathcal{G}'$.
Throughout the rest of the paper,
we will understand a \emph{map} of graphs of groups
to be a map of this kind,
following the convention established in \cite[Section 1]{Myself}.

Like edge paths, a map of graphs of groups $f\colon \mathcal{G} \to \mathcal{G}$
that is a homotopy equivalence
has a combinatorial shadow in $\mathcal{G}$ that does not depend on $X_\mathcal{G}$:
The data of $f$ is
a continuous map $f\colon G \to G$ taking vertices to vertices
along with,
for each edge $e$ of $G$,
an edge path for which we write $f(e)$ in $\mathcal{G}$
(such that the underlying path in $G$ agrees with the image of $e$
under the continuous map $f\colon G \to G$),
and for each vertex $v$ of $G$ with nontrivial vertex group,
an isomorphism $f_v\colon \mathcal{G}_v \to \mathcal{G}_{f(v)}$.
Although such data is sufficient to determine a map $f\colon \mathcal{G} \to \mathcal{G}$,
let us remark that
not every such map is a homotopy equivalence;
one can check whether one has a homotopy equivalence by performing \emph{folds}
as in \cite{BestvinaFeighnGTrees} or \cite{Dunwoody}.
Homotopy for $f$
has a combinatorial shadow in $\mathcal{G}$
generated by the following operations.
\begin{enumerate}
    \item For each edge $e$ of $G$, alter the edge path $f(e)$
        (thought of as an edge path in $\mathcal{G}$)
        by a homotopy.
    \item For a vertex $v$ of $G$ with trivial vertex group
        and an edge path $\gamma$ with endpoints at vertices
        and initial vertex $f(v)$,
        replace $f(v)$ with the terminal vertex of $\gamma$
        and append $\gamma$ to the edge path $f(e)$
        for each oriented edge $e$ with $\tau(e) = v$.
    \item For a vertex $v$ of $G$ with nontrivial vertex group
        and a group element $g \in \mathcal{G}_{f(v)}$,
        replace $f_v \colon \mathcal{G}_v \to \mathcal{G}_{f(v)}$
        with the map $x \mapsto gf_v(x) g^{-1}$,
        and append $g^{-1}$ to the edge path $f(e)$
        for each oriented edge $e$ with $\tau(e) = v$.
\end{enumerate}
A map $f\colon \mathcal{G} \to \mathcal{G}$ acts on edge paths
by the rule that $f$ sends $e$ to the edge path $f(e)$ in $\mathcal{G}$
and sends the vertex group element $g \in \mathcal{G}_v$
to $f_v(g) \in \mathcal{G}_{f(v)}$.
For initial and terminal segments of edges,
there is a small amount of indeterminacy:
suppose we divide the edge $e$ into $e'e''$
at the preimage of a vertex $v$ with nontrivial vertex group
such that the image $f(e)$ contains the vertex group element $g \in \mathcal{G}_v$.
Then we may freely factor $g = g'g''$
and say that $f(e')$ ends with $g'$ and $f(e'')$ begins with $g''$.
We will be careful about choosing factorizations but note that
in practice, there is usually a convenient choice of factorization,
and different choices yield the same result.

\paragraph{Lifting maps to the Bass--Serre tree.}
Given a path $\gamma \in \mathcal{G}$,
let $[\gamma]$ denote its homotopy class rel endpoints.
Given a basepoint $p \in G$
and a path $\sigma$ from $p$ to $f(p)$,
the homotopy equivalence $f$ induces an automorphism
$\Phi \colon \pi_1(\mathcal{G},p) \to \pi_1(\mathcal{G},p)$
defined as
\[  \Phi([\gamma]) = [\sigma f(\gamma) \bar\sigma]. \]
Differing choices of path $\sigma$ change $\Phi$ within its outer class 
$\varphi \in \out(\pi_1(\mathcal{G}))$.
We say $f$ \emph{represents} $\varphi$.
If in addition to sending vertices to vertices,
the map $f$ sends edges to nontrivial tight edge paths,
we say that $f$ is a \emph{topological representative} of $\varphi$.

Let $\Gamma$ be the Bass--Serre tree of $\mathcal{G}$ with basepoint $\tilde p$
lifting $p$.
The choice of path $\sigma$ also defines a lift
$\tilde f \colon \Gamma \to \Gamma$ defined as follows.
If $\tilde v$ is the vertex $[\gamma]\mathcal{G}_v$ of $\Gamma$,
The point $\tilde f(\tilde x)$ is the vertex $[\sigma f(\gamma)]\mathcal{G}_{f(v)}$.
Similarly if $\tilde x$ is a point $[\gamma]$ in the interior of an edge,
then $\tilde f(\tilde x)$ is the point $[\sigma f(\gamma)]$.
One notes that this is well-defined.
The lifted map $\tilde f$ is \emph{$\Phi$-twisted equivariant} in the sense that
\[  \tilde f(g.\tilde x) = \Phi(g).\tilde f(\tilde x) \]
for each point $\tilde x$ in $\Gamma$ and each element $g$ of $\pi_1(\mathcal{G},p)$.
We say that the lifted map $\tilde f$ 
\emph{corresponds to} or \emph{is determined by} $\Phi$
and vice versa.

\paragraph{Directions, the map $Df$.}
Let $v$ be a vertex of $G$,
and let $\st(v)$ denote the set of oriented edges $e$ of $G$
with initial vertex $\iota(e) = v$.
A \emph{direction} at $v$ is an element of the set
\[  \coprod_{e \in \st(v)} \mathcal{G}_{v} \times \{e\}. \]
An identification of $\Gamma$ with the Bass--Serre tree of $\mathcal{G}$
determines, in particular, a bijection between the set of directions at $v$
and those of any lift $\tilde v$ in $\Gamma$.
To wit, if $\tilde e$ is an edge of $\Gamma$ with initial vertex $\tilde v = [\gamma]\mathcal{G}_v$,
where we suppose that $\gamma$ ends with a trivial vertex group element,
then this bijection sends the edge $\tilde e$ with initial vertex $\tilde v$
and terminal vertex $\tilde w$
to the pair $(g,e)$ if $\tilde w = [\gamma ge]\mathcal{G}_w$.
The map $g \mapsto [\gamma g\bar \gamma]$ defines an isomorphism from $\mathcal{G}_v$
to the stabilizer of $\tilde v$ in $\Gamma$
making the bijection above $\mathcal{G}_v$-equivariant.

A topological representative $f\colon \mathcal{G} \to \mathcal{G}$
determines a map $Df$ from the set of directions based at $v$
to the set of directions based at $f(v)$ in the following way.
Let $e$ be an edge in $\st(v)$,
and suppose the edge path $f(e)$ begins with $g_0e_1$.
We define
\[  Df(g,e) = (f_v(g)g_0,e_1). \]
Note that $f(ge)$ begins $f_v(g)g_0e_1$,
so another way of describing this map is as the first vertex group element and edge
of the edge path $f(ge)$.
A direction $(g,e)$ is \emph{almost periodic} if $Df^k(g,e) = (g',e)$ for some $k \ge 1$,
and \emph{almost fixed} if $k =1$.
If we can (possibly by increasing $k$) take $g' = g$,
then the direction is \emph{periodic} or \emph{fixed,} respectively.

A lift $\tilde f$ of $f$ is determined by the image
of the point $\tilde p$
together with (if $\tilde p$ is a vertex with nontrivial vertex group)
the direction $D\tilde f(\tilde e)$ 
for an edge $\tilde e$ with $\iota(\tilde e) = \tilde p$.
Alternatively, if $\tilde p$ is a vertex with nontrivial vertex group
and $\tilde x \ne \tilde p$ is a second point,
then $\tilde f$ is determined by the image of $\tilde p$ and the image of $\tilde x$.
It follows that the correspondence between lifts and automorphisms
defined in the previous titled paragraph is a bijection (for topological representatives,
or more generally, maps that do not collapse edges to vertices).

Suppose $v$ is a vertex with nontrivial vertex group $\mathcal{G}_v$,
and that $\tilde v$ is the point $[\gamma]\mathcal{G}_v$ of $\Gamma$.
Suppose further that $f\colon \mathcal{G} \to \mathcal{G}$ fixes $v$.
Then if $\tilde f$ is the lift of $f$ determined by the path $\sigma$ from $p$ to $f(p)$,
we see that $\tilde f(\tilde v) = \tilde v$ if and only if $\sigma = \gamma h f(\bar\gamma)$
for some element $h \in \mathcal{G}_v$.
In this situation,
if $\tilde e$ is the edge connecting the vertex $\tilde v$ to the vertex 
$\tilde w = [\gamma ge]\mathcal{G}_w$,
we have that
\[  \tilde f(\tilde w) = [\sigma f(\gamma) f_v(g)f(e)]\mathcal{G}_{f(w)}
= [\gamma hf_v(g)f(e)]\mathcal{G}_{f(w)}, \]
and we see that in the notation of the previous paragraph and 
under the $\mathcal{G}_v$-equivariant bijection of directions at $\tilde v$
with directions at $v$,
where $\tilde e$ corresponds to $(g,e)$,
we have that
$D\tilde f(\tilde e)$ corresponds to $(hf_v(g)g_0,e_1)$.

\paragraph{Markings, filtrations, transition matrices.}
Given a group $F$ that splits as a free product of the form
\[  F = A_1*\cdots*A_n*F_k, \]
where as usual the $A_i$ are groups and $F_k$ is a free group of rank $k$,
consider the following Grushko $(F,\mathscr{A})$-graph of groups $\mathbb{G}$,
the \emph{thistle with $n$ prickles and $k$ petals.}
There are $n+1$ vertices; one of which, $\star$, has trivial vertex group,
and the others each have vertex group isomorphic to some $A_i$.
There are $n+k$ edges,
the first $n$ of which connect a vertex with nontrivial vertex group to $\star$,
and the remaining $k$ of which form loops based at $\star$.
A \emph{marked graph of groups} is 
a Grushko $(F,\mathscr{A})$-groups $\mathcal{G}$
together with a homotopy equivalence, the \emph{marking},
$m\colon \mathbb{G} \to \mathcal{G}$.
(Note that in this paper, a marked graph of groups always has trivial edge groups.)
It is mildly convenient to allow the marking $m$
to \emph{not} map the vertex $\star$ to a vertex of $\mathcal{G}$.
Fix once and for all an identification $F = \pi_1(\mathbb{G},\star)$.
Every automorphism $\Phi\colon (F,\mathscr{A}) \to (F,\mathscr{A})$
permuting the conjugacy classes in $\mathscr{A}$
admits a (pointed) topological representative $f\colon(\mathbb{G},\star) \to (\mathbb{G},\star)$.
The marking on a marked graph of groups $(\mathcal{G},m)$ provides an identification
of $\pi_1(\mathcal{G},m(\star))$ with $F$.
Thus it makes sense to say that a topological representative 
$f\colon \mathcal{G} \to \mathcal{G}$
represents $\varphi \in \out(F,\mathscr{A})$.

A \emph{filtration} on a marked graph of groups $\mathcal{G}$
with respect to a topological representative $f\colon \mathcal{G} \to \mathcal{G}$
is a (strictly) increasing sequence
\[  \varnothing = G_0 \subset G_1 \subset \cdots \subset G_m = G \]
of $f$-invariant subgraphs.
We assume that $G_1$ is \emph{nontrivial} in the sense that it contains an edge of $G$.
(Some authors would begin with the vertices of $G$
with nontrivial vertex group as $G_0$.
This makes no difference for our purposes.)
The subgraphs in the filtration are not required to be connected.

The \emph{$r$th stratum} of $\mathcal{G}$ is the subgraph $H_r$
containing those edges of $G_r$ not contained in $G_{r-1}$.
When we consider $G_r$ and $H_r$ as graphs of groups in their own right,
the vertex groups are equal to what they are in $\mathcal{G}$;
in the language of \cite{Bass},
we work primarily with subgraphs of groups, not subgraphs of subgroups.
An edge path has \emph{height $r$} if it is contained in $G_r$
and meets the interior of $H_r$.

A \emph{turn} based at a vertex $v$ of $G$
is an unordered pair of directions in $\mathcal{G}$ based at $v$.
The map of directions $Df$ determines a self-map of the set of turns in $\mathcal{G}$
also called $Df$.
A turn  is \emph{degenerate} if it is of the form
$\{(g,e),(g,e)\}$,
and \emph{illegal} if some iterate of $Df$ maps it to a degenerate turn.
A turn is \emph{legal} if it is not illegal.
A path $\gamma = \ldots e_kg_ke_{k+1}\ldots$ is said to \emph{cross} or \emph{contain}
the turn $\{(1,\bar e_k),(g_k,e_{k+1})\}$.
There is a diagonal action of $\mathcal{G}_v$ on the set of turns based at $v$;
what is primarily important is the $\mathcal{G}_v$-orbit of the turn.
If a path has height $r$ and contains no illegal turns in $H_r$,
then we say that it is \emph{$r$-legal.}

Given a stratum $H_r$, 
the \emph{transition (sub)matrix} $M_r$
is the square matrix whose columns are the edges of $H_r$
and whose $(i,j)$-entry is the number of times the $f$-image of the $j$th edge
crosses the $i$th edge in either direction.
A filtration is \emph{maximal} when each $M_r$ is either
irreducible or the zero matrix.
If $M_r$ is irreducible,
we say that $H_r$ is an \emph{irreducible stratum} and a \emph{zero stratum} otherwise.
All filtrations in this paper will be assumed to be maximal unless otherwise stated;
we will think of a maximal filtration as part of the data
of a topological representative $f\colon \mathcal{G} \to \mathcal{G}$.

Associated to each irreducible stratum $H_r$,
the matrix $M_r$ has a \emph{Perron--Frobenius eigenvalue} $\lambda_r \ge 1$.
If $\lambda_r > 1$, we say that $H_r$ is an \emph{exponentially growing stratum.}
Otherwise we have $\lambda_r = 1$,
$M_r$ is a transitive permutation matrix
and we say that $H_r$ is a \emph{non-exponentially growing stratum.}
In the literature, ``exponentially growing'' and ``non-exponentially growing'' are often
shortened to ``EG'' and ``NEG,'' respectively.

\paragraph{Relative train track maps.}
Let $f\colon \mathcal{G} \to \mathcal{G}$ be a topological representative
with associated filtration $\varnothing = G_0 \subset\cdots\subset G_m = G$.
Given a path $\sigma$ in $\mathcal{G}$,
let $f_\sharp(\sigma)$ denote the (unique) tight path homotopic rel endpoints
to $f(\sigma)$.
The map $f$ is a \emph{relative train track map}
if for every exponentially growing stratum $H_r$, we have
\begin{enumerate}
    \item[\hypertarget{EG-i}{(EG-i)}] Directions in $H_r$ are mapped
        to directions in $H_r$ by $Df$;
        every turn with one edge in $H_r$ and the other in $G_{r-1}$ is legal.
    \item[\hypertarget{EG-ii}{(EG-ii)}] If $\sigma \subset G_{r-1}$
        is a homotopically nontrivial path with endpoints in $H_r \cap G_{r-1}$,
        then $f_\sharp(\sigma)$ is as well.
    \item[\hypertarget{EG-iii}{(EG-iii)}] If $\sigma \subset G_r$
        is a tight $r$-legal path,
        then $f(\sigma)$ is an $r$-legal path.
\end{enumerate}

We have the following theorem,
for which the reader is referred to
\cite{CollinsTurner}, \cite{FrancavigliaMartino} or \cite[Theorem 3.2]{Myself}.
\begin{thm}
    \label{rtttheorem}
    Every $\varphi \in \out(F,\mathscr{A})$ is represented by a relative train track map
    $f\colon \mathcal{G} \to \mathcal{G}$ on a marked graph of groups $\mathcal{G}$.
\end{thm}

We extend this theorem by bringing in free factor systems.

\paragraph{Free factor systems.}
A \emph{$(F,\mathscr{A})$-free splitting} is a simplicial tree $T$
equipped with an action of $F$
with trivial edge stabilizers in which subgroups in $\mathscr{A}$ are elliptic
(i.e.~fix points of $T$).
We assume that some element of $F$ is hyperbolic (i.e.~not elliptic) in $T$.
A \emph{proper free factor of $F$ relative to $\mathscr{A}$}
is a vertex stabilizer in some $(F,\mathscr{A})$-free splitting.
It has  \emph{positive complexity} in the sense of \cite[Section 2]{Myself}
if it is not contained in $\mathscr{A}$.
If $F^i$ is a free factor of $F$ relative to $\mathscr{A}$,
let $[[F^i]]$ denote its conjugacy class.
If $F^1,\ldots, F^m$ are free factors of $F$ relative to $\mathscr{A}$
with positive complexity and $F^1*\cdots*F^m$ is a free factor of $F$
(in the sense that there is a subgroup $F^{m+1}$ 
such that $F = (F^1*\cdots*F^m)*F^{m+1}$),
the collection $\{[[F^{1}]],\ldots,[[F^m]]\}$ is a \emph{free factor system.}
We stress that we only  allow free  factors of positive complexity
to be part of a free factor system.

\begin{ex}
    Suppose $f\colon \mathcal{G} \to \mathcal{G}$ is a topological representative
    and $\mathcal{G}$ is a marked graph of groups
    and $G' \subset \mathcal{G}$ is an $f$-invariant subgraph
    with noncontractible (hence nontrivial)
    connected components $C_1,\ldots,C_k$.
    Here we say $C_1$ is \emph{noncontractible}
    if it contains at least two vertices with nontrivial vertex group
    or has nontrivial ordinary fundamental group;
    otherwise we say $C$ is \emph{contractible}.
    The conjugacy classes  $[[\pi_1(\mathcal{G|_{C_i}})]]$
    of the fundamental groups of the $C_i$ are well-defined.
    We define
    \[  \mathcal{F}(G') = \{[[\pi_1(\mathcal{G}|_{C_1})]],\ldots,
    [[\pi_1(\mathcal{G}|_{C_k})]]\}.\]
    Notice that each $\pi_1(\mathcal{G}|_{C_i})$ has positive complexity.
    We say that $G'$ \emph{realizes} $\mathcal{F}(G')$.
\end{ex}

The group $\out(F,\mathscr{A})$ acts on the set of conjugacy classes
of free factors of $F$ relative to $\mathscr{A}$.
If $\varphi \in \out(F,\mathscr{A})$ is an outer automorphism
and $F'$ a free factor of $F$ 
then we say $[[F']]$ is \emph{$\varphi$-invariant}
if $\varphi([[F']]) = [[F']]$.
In this case there is some automorphism $\Phi\colon (F,\mathscr{A}) \to (F,\mathscr{A})$
representing $\varphi$ such that $\Phi(F') = F'$.
The restriction $\Phi|_{F'}$ is well-defined up to an inner automorphism of $F'$,
so it induces an outer automorphism $\varphi|_{F'} \in \out(F')$,
which we will call the \emph{restriction} of $\varphi$ to $F'$.
More generally one can restrict $\varphi$ to any subgroup of $F$ which is its own normalizer.
A free factor system is $\varphi$-invariant if each conjugacy class
of each free  factor contained in it is $\varphi$-invariant.

There is a partial order $\sqsubset$ on free factor systems:
we say that $[[F^1]] \sqsubset [[F^2]]$ if $F^1$ is conjugate to a subgroup of $F^2$,
and we say that $\mathcal{F}_1 \sqsubset \mathcal{F}_2$ for free factor systems
$\mathcal{F}_1$ and $\mathcal{F}_2$ if for each $[[F^i]] \in \mathcal{F}_1$,
there exists $[[F^j]] \in \mathcal{F}_2$ such that $[[F^i]]\sqsubset [[F^j]]$.

\begin{prop}
    \label{rttprop}
    Let $\varphi \in \out(F,\mathscr{A})$ be an outer automorphism.
    If $\mathcal{F}_1 \sqsubset \dotsb \sqsubset \mathcal{F}_d$
    is a nested sequence of $\varphi$-invariant free factor systems
    relative to $\mathscr{A}$,
    then there is a relative train track map
    $f\colon \mathcal{G} \to \mathcal{G}$ representing $\varphi$
    such that each free factor system
    is realized by some element of the filtration
    $\varnothing = G_0 \subset \cdots \subset G_m = G$.
\end{prop}

\begin{proof}
    The first step is to construct a topological representative
    $f\colon \mathcal{G} \to \mathcal{G}$ and an associated filtration
    such that each $\mathcal{F}_i$ is realized by a filtration element.
    We proceed by induction on $d$,
    the length of the nested sequence of $\varphi$-invariant free factor systems.
    The case $d=0$ is vacuous.
    Let $\mathcal{F}_d = \{[[F^1]],\ldots,[[F^\ell]]\}$,
    and choose automorphisms $\Phi_i \colon (F,\mathscr{A}) \to (F,\mathscr{A})$
    representing $\varphi$ such that $\Phi_i(F^i) = F^i$.
    For each $j$ with $1  \le j \le d-1$
    and each $i$ satisfying $1 \le i \le \ell$,
    write $\mathcal{F}^i_j$  for the set of conjugacy classes
    of free factors in $\mathcal{F}_j$ that are conjugate into $F^i$.
    Then for each $i$,
    \[  \mathcal{F}^i_1 \sqsubset \cdots \sqsubset \mathcal{F}^i_{d-1}\]
    is a nested sequence of $\varphi$-invariant free factor systems.
    By induction, there are graphs of groups $\mathcal{G}^i$
    and topological representatives $f_i\colon \mathcal{G}^i \to \mathcal{G}^i$
    representing the restriction 
    of $\varphi$ to $F^i$
    together with associated (not necessarily maximal!) filtrations
    $G^i_1 \subset \cdots \subset G^i_{d-1}$ 
    such that $G^i_j$ realizes $\mathcal{F}^i_j$.
    Inductively, we may assume that $f_i$ fixes some vertex $v_i$ of $\mathcal{G}^i$,
    that $\mathcal{G}^i$ has no valence-one or valence-two vertices
    with trivial vertex group,
    and that the marking on $\mathcal{G}^i$ identifies 
    $F^i$ with $\pi_1(\mathcal{G}^i,v_i)$
    and $\Phi_i$ with 
    $(f_i)_\sharp\colon \pi_1(\mathcal{G}^i,v_i) \to \pi_1(\mathcal{G}^i,v_i)$.

    Take a complementary free factor $F^{\ell+1}$ so that 
    $F = F^1 * \cdots F^\ell*F^{\ell+1}$
    and an associated thistle $\mathcal{G}^{\ell+1}$.
    (If $F^{\ell+1}$ is trivial, $\mathcal{G}^{\ell+1}$ is a vertex $\star$.)
    The graph of groups $\mathcal{G}$
    is constructed as follows:
    begin with the disjoint union of the $\mathcal{G}^i$.
    Glue $\mathcal{G}^{\ell+1}$ to $\mathcal{G}^1$
    by identifying the vertex $\star$ of $\mathcal{G}^{\ell+1}$ 
    with trivial vertex group
    with  the fixed point $v_1$.
    Then attach an edge $E_i$ connecting $v_1$ to $v_i$ for $2 \le i \le \ell$.
    The resulting graph of groups has no valence-one or valence-two vertices
    with trivial vertex group.
    Collapsing each $E_i$ to $v_1$ gives a homotopy equivalence of $\mathcal{G}$
    onto a graph whose fundamental group is naturally identified with
    $F^1*\cdots F^\ell*F^{\ell+1} = F$.
    This provides the (essential data of) a marking on  $\mathcal{G}$.

    Define $f\colon \mathcal{G} \to \mathcal{G}$ from the $f_i$ as follows.
    For $2 \le i \le \ell$, there is $c_i \in F$ such that
    $\Phi_1(x) = c_i\Phi_i(x)c_i^{-1}$ for all $x \in F$.
    Let $\gamma_i$ be loops based at $v_1$ 
    that are identified under the marking with $c_i$.
    On each $\mathcal{G}^i$ for $2 \le i \le \ell$, set $f = f_i$.
    Define $f(E_i) = \gamma_i E_i$,
    and define $f$ on $\mathcal{G}^{\ell+1}$ according to $\Phi_1$.
    Then $f_\sharp\colon \pi_1(\mathcal{G},v_1) \to \pi_1(\mathcal{G},v_1)$
    induces $\Phi_1 \colon (F,\mathscr{A}) \to (F,\mathscr{A})$.
    For $1 \le j \le d-1$, define $G_j = \bigcup_{j=1}^\ell G^i_j$,
    and define $G_d = \bigcup_{i=1}^\ell \mathcal{G}^i$.
    Then \[\varnothing = G_0 \subset \cdots \subset G_d \subset G_{d+1} = G\]
    is an $f$-invariant filtration,
    and each $\mathcal{F}_i$ for $i$ satisfying $1 \le i \le d$
    is realized by $G_i$.
    Complete this filtration to a maximal filtration
    (by adding in intermediate $f$-invariant subgraphs).
    This completes the first step.

    The next step is to promote our topological representative 
    to a relative train track map.
    Note (cf.~the proof of \cite[Lemma 2.6.7]{BestvinaFeighnHandel})
    that the moves described in \cite[Section 2 and 3]{Myself}
    all preserve the property of realizing free factors.
    More precisely, suppose $C_1$ and $C_2$
    are disjoint noncontractible components of some filtration element $G_r$ 
    of $\mathcal{G}$
    and that $f'\colon \mathcal{G}' \to \mathcal{G}'$
    is obtained from $f\colon \mathcal{G} \to \mathcal{G}$
    by collapsing a pretrivial forest, folding, subdivision,
    invariant core subdivision, valence-one homotopy or valence-two homotopy.
    If $p\colon \mathcal{G} \to \mathcal{G}'$ is the identifying
    homotopy equivalence,
    then $p(C_1)$  and $p(C_2)$  
    are disjoint, noncontractible subgraphs of $\mathcal{G}'$.

    To see this in the case of collapsing a pretrivial forest $X$,
    note that because $C_1$ and $C_2$ are noncontractible,
    there exists $k \ge 1$ such that $f^k(C_1) \subset C_1$
    and $f^k(C_2) \subset C_2$.
    If some component of $X$ intersected $C_1$ and $C_2$,
    then the same must be true of $f^k(X)$,
    thus since $C_1$ and $C_2$ are disjoint, this contradicts the fact that $X$
    is pretrivial.
    
    Thus we may work  freely and apply the proof of \cite[Theorem 3.2]{Myself}
    to produce a relative train track map.
\end{proof}

The proof of \Cref{rtttheorem} in \cite{Myself}
has the following useful corollary.
Let $f\colon \mathcal{G} \to \mathcal{G}$ be a topological representative.
Let $\pf(f)$ be the set of Perron--Frobenius eigenvalues
of the exponentially growing strata of $f$ in nonincreasing order.
The set of $\pf(f)$ for $f$ a topological representative 
of $\varphi \in \out(F,\mathscr{A})$
satisfying a certain condition called \emph{boundedness}
ordered lexicographically has a minimum, $\pfmin$
(see \cite[Section 3]{Myself} 
and note that boundedness depends solely on $\pf(f)$).
The proof of \Cref{rtttheorem} shows that if $f$ is a bounded topological representative
satisfying \hyperlink{EG-i}{(EG-i)} but not \hyperlink{EG-iii}{(EG-iii)},
then $\pf(f)$ can be decreased.
Therefore we have the following corollary.

\begin{cor}[Corollary  3.7 of \cite{Myself}]
    \label{pfcorollary}
    If $f\colon \mathcal{G} \to \mathcal{G}$ is a topological representative
    satisfying \hyperlink{EG-i}{(EG-i)} and with
    $\pf(f)=  \pfmin$,
    then the exponentially growing strata of $f$ satisfy
    \hyperlink{EG-iii}{(EG-iii)}.
\end{cor}

We also have the following useful characterization of \hyperlink{EG-ii}{(EG-ii)}.

\begin{lem}\label{trivialedgegroupsEG2}
    Let $f\colon \mathcal{G} \to \mathcal{G}$ be a topological representative
    and $H_r$ an exponentially growing stratum that satisfies \hyperlink{EG-i}{(EG-i)}.
    Property \hyperlink{EG-ii}{(EG-ii)} for $H_r$
    is a finite property for any contractible component $C$ of $G_{r-1}$
    which does not contain a vertex with nontrivial vertex group.
    In the contrary case there exists $k \ge 1$ such that $f^{k}(C) \subset C$
    and \hyperlink{EG-ii}{(EG-ii)} for connecting paths in $C$
    is equivalent to the condition that each vertex in $H_r \cap C$ is periodic;
    i.e.~$f^k$ permutes the elements of the set $H_r \cap C$.
\end{lem}

\begin{proof}
    The first statement, 
    where $C$ is contractible and contains no vertex with nontrivial vertex group,
    is clear,
    since there are only finitely many tight paths with endpoints at vertices in $C$.

    So suppose that there exists $k\ge 1$ such that $f^k(C) \subset C$.
    Then since $H_r$ satisfies \hyperlink{EG-i}{(EG-i)},
    $f^k$ maps the finite set of vertices $H_r \cap C$ into itself.
    If $f^k$ does not permute the elements of $H_r \cap C$,
    then there are a pair of vertices $v$ and $w$ whose $f^k$-image is equal, say to $u$.
    Then there is a path $\sigma$ connecting $v$ to $w$ whose image under $f^k_\sharp$
    is trivial (consider what a homotopy inverse to $f^{k}$ does to $u$)
    and \hyperlink{EG-ii}{(EG-ii)} fails.

    So suppose that $f^k$ permutes the elements of $H_r \cap C$.
    If $\alpha$ is a tight connecting path in $C$ with distinct endpoints in
    $H_r \cap C$, then $f^k(\alpha)$ also has distinct endpoints
    and therefore $f_\sharp(\alpha)$ is nontrivial.
    If instead $\alpha$ is a tight connecting path in $C$ with identical endpoints, say $v$,
    then it determines a nontrivial loop in $\pi_1(\mathcal{G},v)$.
    If $f_\sharp(\alpha)$ fails to be nontrivial
    then $\alpha$ must have the form $\sigma g \bar\sigma$ for $g \in \mathcal{G}_w$
    an element of a vertex group
    and $f(v) = f(w)$.
    Since $C$ is a component of $G_{r-1}$,
    vertices with nontrivial vertex group are permuted,
    and $f^k(w) = f^k(v) \in H_r \cap C$,
    we conclude that in fact $w \in H_r \cap C$
    and therefore $v = w$.
    Since $\alpha$ was tight, we conclude $\sigma$ is a nontrivial loop in $\pi_1(\mathcal{G},v)$,
    and $f_\sharp(\sigma)$ (and therefore $f_\sharp(\alpha)$)
    is nontrivial provided $\sigma$ is not itself a path of the form
    $\rho g' \bar\rho$ for $g' \in \mathcal{G}_v$.
    On the other hand, if it \emph{is} of that form,
    then we can ask the question of $\rho$ and so on
    until at last we come to a path $\mu$ which is too short to be of the above form.
    We have that $f_\sharp(\mu)$ is nontrivial.
    Recall that nondegenerate turns of the form $\{(g_1,e),(g_2,e)\}$ 
    where $g_1$ and $g_2$ are distinct are legal,
    therefore $f_\sharp(\mu g''\bar \mu) = f_\sharp(\mu)f_v(g'')f_\sharp(\bar\mu)$
    is a tight concatenation of tight paths,
    and we conclude that $f_\sharp(\alpha)$ is nontrivial.
    Therefore \hyperlink{EG-ii}{(EG-ii)} holds for connecting paths in $C$ in this case.
\end{proof}

\paragraph{Collapsing relative train track maps.}
Suppose $f\colon G \to G$
and $\varnothing = G_0  \subset G_1 \subset \cdots \subset G_m = G$
are a relative train track map and filtration in the sense of \cite{BestvinaHandel}
(i.e.~in our sense with all vertex groups trivial)
and fix $r$ satisfying $2 \le r \le m$.
To aid the reader interested in the case $F = F_n$,
we describe how to collapse the filtration element $G_{r-1}$ of $f$
to obtain a new relative train track map
$f'\colon \mathcal{G}' \to \mathcal{G}'$ on a Grushko $(F,\mathcal{F}(G_{r-1}))$-graph of groups
and a collapse map $p\colon G  \to \mathcal{G}'$.

Begin by defining a subgraph $H$ of $G$ inductively as follows:
begin with $H = G_{r-1}$
and then repeatedly add to $H$ all those edges whose $f$-image is entirely contained in $H$.
Each edge of $H \setminus G_{r-1}$ belongs to a zero stratum.
The subgraph $H$ is $f$-invariant;
let $G'$ be the graph obtained by collapsing each component of $H$ to a point,
and let $p\colon G \to G'$ be the collapse map.
For each component $C$ of $H$,
let $T_C$ be a maximal tree in $C$, and $c \in C$ a vertex.
For $v$ a vertex of $C$, let $\eta_v$ be the unique tight path in $T_C$ from $c$ to $v$.

Each component $C$ of $H$ corresponds to a vertex $v_C$ of $G'$;
set $\mathcal{G}'_{v_C} = \pi_1(C,c)$.
Each vertex $v$ of $G$ not contained in $H$ corresponds to a vertex, also called $v$, of $G'$;
let $\mathcal{G}'_v$ be the trivial group.
As a map of graphs of groups,
$p$ is the ``identity'' on edges not contained in $H$
and sends an edge path $\gamma$ in $H$ with initial vertex $v$ and terminal vertex $w$
to the vertex group element $[\eta_v\gamma\bar\eta_w]$.

The map $f\colon G \to G$
descends to a map $f'\colon G' \to G'$.
We would like the following diagram to commute for paths in $G$ with endpoints
either in the complement of $H$ or at the vertices $c \in C$
\[  \begin{tikzcd}
    G \ar[r, "f"] \ar[d, "p"] & G \ar[d, "p"] \\
    \mathcal{G}' \ar[r, "f'"] & \mathcal{G}'.
\end{tikzcd}    \]
Suppose that $f'(v_C) = v_{C'}$.
Choose a path $\sigma_C$ from $c'$ to $f(c)$.
The rule $[\gamma] \mapsto [\sigma_Cf(\gamma)\bar\sigma_C]$ defines a homomorphism
$f'_{v_C}\colon \mathcal{G}_{v_C} \to \mathcal{G}_{f'(v_C)}$.
One checks that if $\mathcal{G}_{v_C}$ is nontrivial, then this map is an isomorphism.
A na\"ive definition of the map $f'\colon \mathcal{G}' \to \mathcal{G}'$
on edges
would be to define $f'(E)$ as $pf(E)$.
In fact this is correct except at the endpoints,
where we need to multiply by some vertex group element.
To see why,
consider a pair of edges $E$ and $E'$ with initial endpoints $v$ and $v'$,
respectively, in a component $C$ of $H$.
Let $\gamma$ be the unique tight path in $T_C$ 
connecting $v$ to $v'$.
Then the path $\bar E\gamma E'$ in $\mathcal{G}$ projects to the path $\bar EE'$ in $\mathcal{G}'$.
The equation $f'p = pf$ implies that $f'(\bar EE') = pf(\bar E\gamma E')$:
it may happen that $pf(\gamma)$ is contained in $T_{f(C)}$,
but in general this will not be the case.
Define the first vertex group element of $f(E)$
to be $p(\sigma_C f(\eta_v)) = [\sigma_C f(\eta_v)\bar\eta_{f(v)}]$.
Then check that the vertex group element $pf(\gamma)$ satisfies
\[
    pf(\gamma) = pf(\bar\eta_v\eta_{v'}) = p(f(\bar\eta_v)\bar\sigma_C\sigma_C f(\eta_{v'}))
    = [\eta_{f(v)}f(\bar\eta_v)\bar\sigma_C\sigma_C f(\eta_{v'})\bar\sigma_{f(v')}]
\]
as required.

\begin{lem}
    The map $f'\colon \mathcal{G}' \to \mathcal{G}'$ is a relative train track map.
\end{lem}

\begin{proof}
    First we show that $f'$ is a topological representative if $f$ was.
    Indeed, if $f(E)$ is a tight path, then $pf(E)$ is a tight path,
    for if $f(E)$ contains a subpath of the form $E'b\bar E'$,
    where $b$ is in $H$ and $E'$ is not,
    then $b$ determines a nontrivial loop,
    so $p(b)$ is a nontrivial vertex group element.
    Put another way, $f'_\sharp(E)$ is obtained from $f_\sharp(E)$
    by possibly adding vertex group elements at the ends
    and replacing maximal subpaths in $H$ with the corresponding vertex group elements.
    Thus we see that if $H_i$ is an irreducible stratum for $i > r-1$,
    then no edge of $H_i$ belongs to $H$,
    and the edges of $H_i$ determine an irreducible stratum $H'_j$ in $\mathcal{G}'$
    of the same kind.
    Likewise if $H_i$ is a zero stratum but not every edge of $H_i$ belongs to $H$,
    then $H_i$ determines a zero stratum $H'_j$ in $\mathcal{G}'$.

    Property \hyperlink{EG-i}{(EG-i)} is clearly satisfied for each exponentially growing stratum $H'_j$.
    If $H'_j$ is exponentially growing and $D$ is a component of $G_{i-1}$
    such that $f^k(D) \subset D$ for  some $k \ge  1$,
    then each vertex $v'$ of $H'_{j} \cap D$ corresponds 
    either to a vertex $v$ of $G$ or a component $C$ of $H$.
    In either case, the preimage of $v'$ contains an $f$-periodic vertex, 
    so the vertex $v'$ is $f'$-periodic.
    If $D$ is not eventually mapped into itself,
    then $D$ is contractible and contains no vertices with nontrivial vertex group.
    Each homotopically nontrivial path $\gamma'$ in $D$ with endpoints in $H'_j\cap D$
    has lifts to a homotopically nontrivial path $\gamma$ in $H_i\cap p^{-1}(D)$.
    The path $\gamma'$ may have many lifts, but each lift satisfies that $f_\sharp(\gamma)$
    is homotopically nontrivial and not entirely contained in $H$,
    so $f'_\sharp(\gamma')$ is homotopically nontrivial;
    this proves \hyperlink{EG-ii}{(EG-ii)} by \Cref{trivialedgegroupsEG2}.
    Property \hyperlink{EG-iii}{(EG-iii)} for $H'_j$ follows from \hyperlink{EG-ii}{(EG-ii)}
    and \hyperlink{EG-iii}{(EG-iii)} for $H_i$.
    Namely, suppose $E$ and $E'$ are edges of $H_i$ with terminal endpoint $v$
    and initial endpoint $w$ respectively,
    that $b$ is a maximal subpath of the component $C$ of $H$,
    and that $EbE'$ is a subpath of the $f^k$-image of some edge.
    Then the vertex group element at the center of $f'(Ep(b)E')$
    is 
    \[ [\eta_{f(v)}f(\bar\eta_v)\bar\sigma_C\sigma_Cf(\eta_v)f(b)
    f(\bar\eta_w)\bar\sigma_C\sigma_Cf(\eta_w)\bar\eta_{f(w)}] = pf(b). \]
    If $f'(\bar E)$ and $f'(E)$ begin with the same edge of $H'_j$,
    then since $f(b)$ is homotopically nontrivial, 
    it must form a loop and therefore the turn crossed by $p(EbE')$
    is not mapped to a degenerate turn by $Df$.
    Since the subpath and $k$ were arbitrary, we conclude that \hyperlink{EG-iii}{(EG-iii)} holds.
\end{proof}

\begin{ex}
    Consider the graph $G$ in \Cref{collapsing1fig},
    and the relative train track map
    \[  f \begin{dcases}
        \alpha_1 \mapsto \alpha_2 \mapsto \alpha_3 \mapsto \alpha_4 \mapsto \beta_1\gamma_1\alpha_1 \\
        \beta_1 \mapsto  \beta_2 \mapsto  \beta_3 \mapsto \beta_4 \mapsto \alpha_1\beta_1 \\
        \gamma_1 \mapsto \gamma_2 \mapsto \gamma_3 \mapsto \gamma_4 \mapsto \gamma_1\alpha_1 \\
        x_1 \mapsto x_2 \mapsto x_3 \mapsto x_4 \mapsto x_1\beta_1 \\
        A \mapsto BE\bar Ax_1\alpha_1\bar x_1 A \\
        B \mapsto C \\
        C \mapsto D \bar E \\
        D \mapsto A \\
        E \mapsto \bar Bx_2\beta_2\gamma_2\bar x_2 B E.
    \end{dcases} \]
    There are three strata: $H_1$ consists of the edges with Greek letters
    and is exponentially growing,
    $H_2  = \{x_1,x_2,x_3,x_4\}$ is non-exponentially growing,
    and $H_3 = \{A,B,C,D,E\}$ is exponentially growing.
    The illegal turns in the interior of exponentially growing strata are drawn sharp in the figure.
    \begin{figure}[ht!]
		\centering
	    \def\svgwidth{\columnwidth}
	        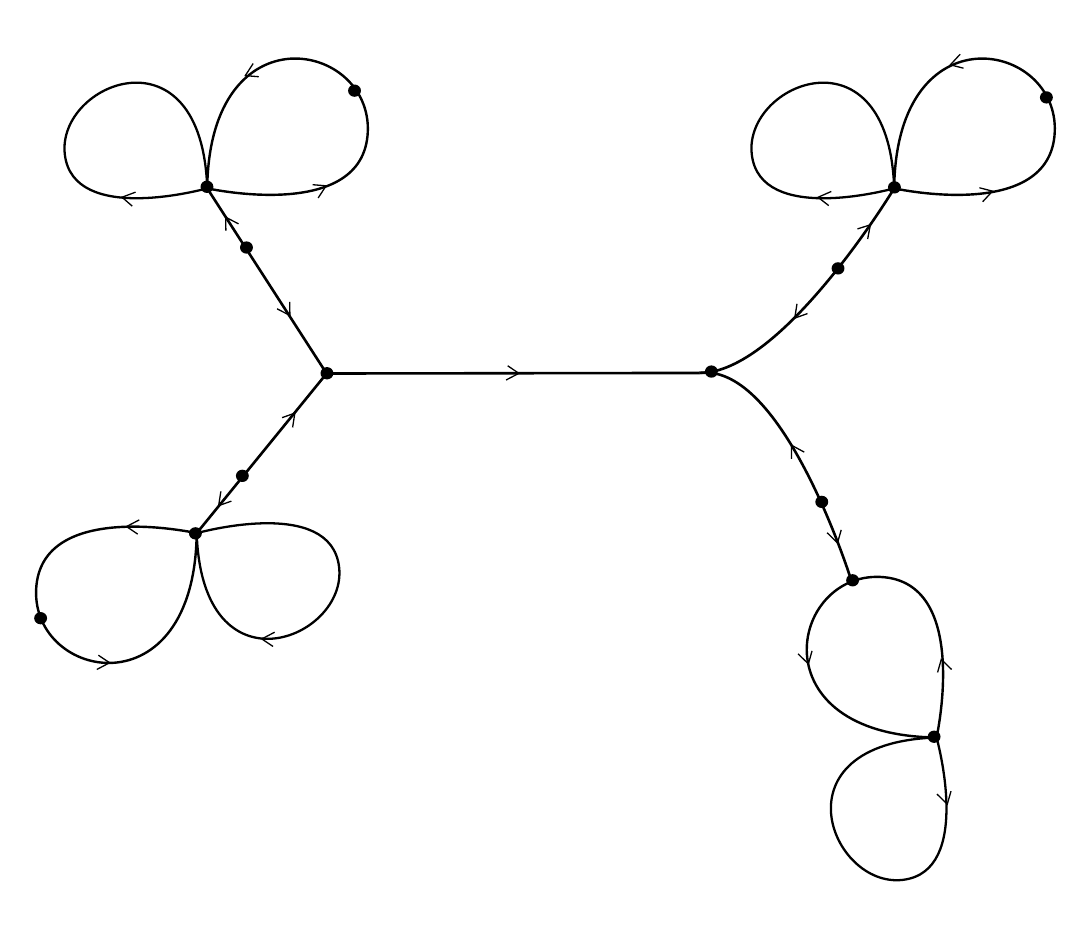
	
		\caption{The relative train track map $f\colon G \to G$.}\label{collapsing1fig}
    \end{figure}

    We will collapse $G_2$. Since $H_3$ is exponentially growing,
    no edge in $H_3$ maps entirely into $G_2$, so the subgraph $H$ is $G_2$.
    There are four components, $C_1$, $C_2$, $C_3$ and $C_4$
    containing those edges with matching subscripts.
    Choose the maximal trees consisting of the edges $x_i$ and $\gamma_i$,
    and let the basepoints $c_i$ be the terminal endpoints of the $x_i$.
    Take the paths $\sigma_{C_i}$ to be the trivial paths at $c_i$.
    The vertex groups are all isomorphic to $F_2$,
    with generators $a_i$ and $b_i$ corresponding to the loops $\alpha_i$
    and $\beta_i\gamma_i$, respectively.
    The isomorphisms of vertex groups
    are generated by the rule
    \[  \begin{dcases}
        a_1 \mapsto a_2 \mapsto a_3 \mapsto a_4 \mapsto b_1a_1 \\
        b_1 \mapsto b_2 \mapsto b_3 \mapsto b_4 \mapsto a_1b_1a_1.
    \end{dcases}  \]
    The relative train track map $f'\colon \mathcal{G}' \to \mathcal{G}'$ is defined on edges as
    \[  f' \begin{dcases}
        A \mapsto BE\bar Aa_1A \\
        B \mapsto C  \\
        C \mapsto D\bar E \\
        D \mapsto a_1b_1A \\
        E \mapsto \bar B b_2BE.
    \end{dcases}\]
    See \Cref{collapsingfig2}.
    \begin{figure}[ht!]
        \centering
	    \def\svgwidth{\columnwidth}
	        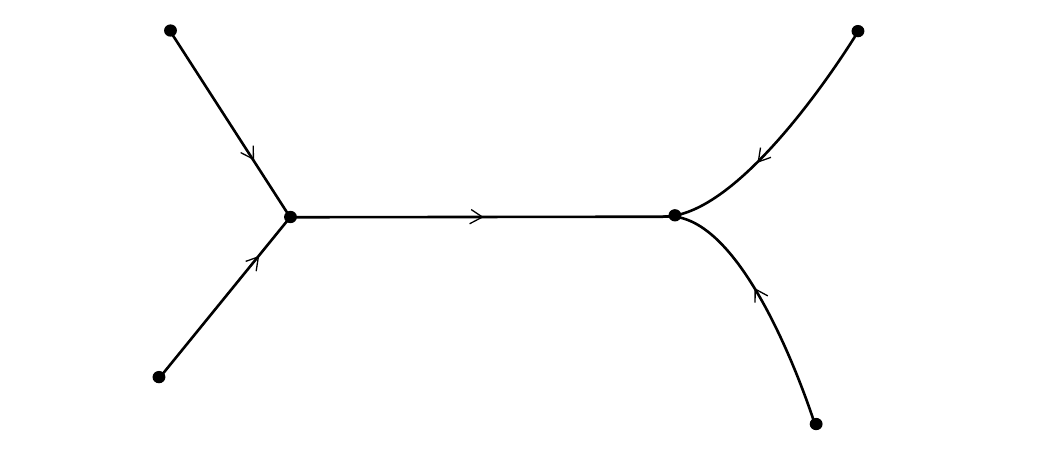
	
        \caption{The relative train track map $f'\colon \mathcal{G}' \to \mathcal{G}'$.}
        \label{collapsingfig2}
    \end{figure}
\end{ex}

There is a dual process of ``blowing up'' a relative train track map
representing $\varphi \in \out(F,\mathscr{A})$ to a ``finer'' free product decomposition
$(F,\mathscr{B})$ where $\mathscr{B} \sqsubset \mathscr{A}$.
In general it is not quite true that relative train track maps blow up to relative train track maps:
one needs to restore properties \hyperlink{EG-i}{(EG-i)} and \hyperlink{EG-ii}{(EG-ii)}
by performing the ``(invariant) core subdivision'' and ``collapsing inessential connecting paths''
moves of \cite{BestvinaHandel} or \cite{Myself}.
Blowing up plays no role in what follows, so we do not pursue the idea further.

\paragraph{Bounded cancellation.}
To finish out this section,
we prove the following bounded cancellation lemma for free products
(cf. \cite{Cooper} or \cite[Lemma 2.3.1]{BestvinaFeighnHandel} for free groups),
which is surely well-known to experts,
but for which we were unable to find a published complete, general proof.

\begin{lem}
    \label{boundedcancellation}
    Let $\mathcal{G}$ and $\mathcal{G}'$ be marked graphs of groups
    with no valence-one vertices with trivial vertex group,
    and let $h\colon \mathcal{G} \to \mathcal{G}'$
    be a homotopy equivalence
    taking vertices to vertices.
    There is a constant $C$ with the following properties.
    \begin{enumerate}
        \item If $\rho = \alpha\beta$ 
            is a tight concatenation of tight paths in $\mathcal{G}$,
            then $h_\sharp(\rho)$ is  obtained from $h_\sharp(\alpha)$
            and $h_\sharp(\beta)$
            by concatenating and cancelling $c \le C$ edges (and vertex group elements)
            from the terminal end of $h_\sharp(\alpha)$ with $c$ edges
            (and vertex group elements) from the initial end of $h_\sharp(\beta)$.
        \item If $\tilde h\colon \Gamma \to \Gamma'$
            is a lift of $h$ to the Bass--Serre trees,
            $\tilde\alpha$ is a line in $\Gamma$ 
            (a proper, linear embedding $\mathbb{R} \to \Gamma$)
            and $\tilde x \in \tilde\alpha$,
            then $\tilde h(\tilde x)$ can be connected to the line
            $\tilde h_\sharp(\tilde\alpha)$ by a path with $c \le C$ edges.
        \item Suppose that $\tilde h\colon \Gamma \to \Gamma'$ 
            is a lift to the Bass--Serre trees and that $\tilde\alpha \subset \Gamma$
            is a (finite) tight path.
            Define $\tilde\beta \subset \Gamma'$ by removing $C$
            initial and $C$ terminal edges from  $\tilde h_\sharp(\tilde\alpha)$.
            If $\tilde\gamma$ is a line in $\Gamma$ that contains $\tilde\alpha$
            as a subpath,
            then $\tilde h_\sharp(\tilde\gamma)$ contains $\tilde\beta$
            as a subpath.
    \end{enumerate}
\end{lem}

\begin{proof}
    Note that if $h_1$ and $h_2$ are homotopy equivalences that additionally
    map edges to edges or collapse edges to vertices
    and satisfy the conclusions of this lemma
    with constants $C_1$ and $C_2$,
    then their composition $h_2h_1$ satisfies the conclusions of this lemma
    with constant $C_1 + C_2$.
    This is because $(h_2h_1)_\sharp = (h_2)_\sharp(h_1)_\sharp$.

    If $H \subset \mathcal{G}$ is a (contractible) forest,
    then collapsing each component of $H$ to a point yields a homotopy equivalence
    $h\colon \mathcal{G} \to \mathcal{G}'$.
    It is not hard to see that $h$ satisfies the conclusions of the lemma with $C = 0$.

    Since the trivial group is, in particular, finitely generated,
    we may use \cite[Theorem 2.1]{Dunwoody},
    which says that after collapsing a contractible forest, 
    we may decompose $h$ 
    (or more properly speaking,
    its subdivision into a map sending edges to edges)
    into a finite product of folds,
    each of which is a homotopy equivalence.

    We claim that a homotopy equivalence which is a single fold
    $f\colon \mathcal{G} \to \mathcal{G}'$ of a pair of edges
    $g_1e_1$ and $g_2e_2$ with $\iota(e_1) = \iota(e_2)$
    satisfies the conclusions of the lemma  with constant $C = 1$,
    from which the lemma follows.
    Note that because edge groups of marked graphs of groups are trivial,
    we may assume that $e_1$ and $e_2$ are distinct edges of $\mathcal{G}$
    (in order for $f$ to be a fold factor of $h$),
    and thus we have $\tau(e_1) \ne \tau(e_2)$
    for otherwise the fold would fail to be a homotopy equivalence.
    For the same reason,
    at most one of $\tau(e_1)$ and $\tau(e_2)$ has nontrivial vertex group.

    We prove that the first conclusion holds;
    the arguments for the others are similar.
    Suppose first that $\alpha = \alpha'\bar e_1g_1^{-1}$ is a tight concatenation,
    where $\alpha'$ is some tight, possibly trivial path in $\mathcal{G}$
    and that similarly $\beta = g_2e_2\beta'$ for some  tight path $\beta'$.
    By assumption, the initial edge (if any) of $\bar\alpha'$ 
    is distinct from the initial edge of $\beta'$
    and at least one of these edges (if any) is not $e_1$ or $e_2$.
    Therefore the concatenation $f_\sharp(\alpha')f_\sharp(\beta')$ is tight
    and equal to $f_\sharp(\alpha\beta)$.
    In all other cases, $f_\sharp(\alpha\beta) = f_\sharp(\alpha)f_\sharp(\beta)$,
    so we see that $f$ satisfies the first conclusion of this lemma
    with constant $C=1$.
\end{proof}

\section{Improving Relative Train Track Maps}\label{improvingsection}

As a step towards the existence of CTs,
Feighn and Handel construct in~\cite[Theorem 2.19]{FeighnHandel} 
relative train track maps that satisfy a number of extra properties.
The goal of this section is to prove the existence of such relative train track maps
for outer automorphisms of free products.
To state the theorem we require some more terminology.
We begin by introducing \emph{splittings} of edge paths
in order to introduce a family of paths that cannot be split,
the \emph{indivisible periodic almost Nielsen paths.}

\paragraph{Splittings}
A decomposition of a path, line or circuit $\sigma$ 
into subpaths $\sigma = \ldots \sigma_1\sigma_2\ldots$
is a \emph{splitting,} written with raised dots as
$\sigma = \ldots\cdot\sigma_1\cdot\sigma_2\cdot\ldots$
if $f^k_\sharp(\sigma) = \ldots f^k_\sharp(\sigma_1)f^k_\sharp(\sigma_2)\ldots$.
That is, $f^k_\sharp(\sigma)$ is obtained from $f^k(\sigma)$ 
by tightening each $f^k(\sigma_i)$ and then concatenating.
For our purposes, merely performing multiplication in a vertex group does not count as tightening.

\paragraph{}
The main application of the relative train track properties is the following lemma.
Note that the definition of a relative train track map makes sense
in the event that $f\colon \mathcal{G} \to \mathcal{G}$
is merely a map of graphs of groups that sends vertices to vertices
and edges to nontrivial edge paths.
We work mainly with relative train track maps that are topological representatives
(so $f$ maps edges to \emph{tight} edge paths),
but this more expansive definition allows for an iterate of such a relative train track map
to be a relative train track map.

\begin{lem}[cf. Lemma 2.9(2) of \cite{FeighnHandel}]\label{rttlemma}
    Suppose $f\colon \mathcal{G} \to \mathcal{G}$ is a relative train track map
    with exponentially growing stratum $H_r$,
    and that $\sigma$ is a path
    with endpoints  at vertices of $H_r$
    which is $r$-legal.
    Then the decomposition of $\sigma$ into single edges of $H_r$
    and maximal subpaths in $G_{r-1}$ is a splitting.
    
    In particular if $f$ is also a topological representative
    then if $\sigma = \alpha_1\beta_1 \ldots \alpha_n\beta_n$
    is a decomposition into subpaths where $\alpha_i \subset H_r$
    and $\beta_i \subset G_{r-1}$
    (allow $\alpha_1$ and $\beta_n$ to be trivial, but assume all others are nontrivial),
    then 
    \[  f_\sharp(\sigma) = f(\alpha_1)f_\sharp(\beta_1)\ldots f(\alpha_n)f_\sharp(\beta_n) \]
    is a tight concatenation of tight paths.
\end{lem}

Note that strictly speaking the filtration for $f$ ought to be enlarged to create the filtration
for $f^k$,
but that \hyperlink{EG-i}{(EG-i)}, \hyperlink{EG-ii}{(EG-ii)} and \hyperlink{EG-iii}{(EG-iii)}
are still satisfied by $f^k$ for any stratum $H_r$ which is an exponentially growing stratum for $f$.

\begin{proof}
    Let $\sigma = \ldots \alpha_1\alpha_2\ldots$
    be a decomposition of the $r$-legal path $\sigma$ into single edges of $H_r$
    and maximal subpaths in $G_{r-1}$, and fix $k \ge 1$.
    Each $f^k_\sharp(\sigma_i)$ is nontrivial, 
    either because $\sigma_i$ is an edge of $H_r$ 
    and thus $f^k_\sharp(\sigma_i)$ is $r$-legal by \hyperlink{EG-iii}{(EG-iii)},
    or by \hyperlink{EG-ii}{(EG-ii)}.
    Either by assumption or by \hyperlink{EG-i}{(EG-i)},
    the turn at the common endpoint of $\bar\sigma_i$ and $\sigma_{i+1}$ is legal,
    so no cancellation occurs at the common endpoint of 
    $f^k_\sharp(\sigma_i)$ and $f^k_\sharp(\bar\sigma_{i+1})$.

    In particular, if $f$ is a topological representative
    with decomposition $\sigma = \alpha_1\beta_1\ldots \alpha_n\beta_n$ as in the statement,
    then $f(\alpha_1)$ is already a tight path by \hyperlink{EG-iii}{(EG-iii)}
    because $\alpha_i$ is legal.
    The path $f_\sharp(\beta_i)$ is nontrivial (provided $\beta_i$ is) by \hyperlink{EG-ii}{(EG-ii)},
    and the turn at the common endpoint of $f(\alpha_i)$ and $f_\sharp(\bar\beta_{i})$
    is nondegenerate by \hyperlink{EG-i}{(EG-i)},
    so \[  f_\sharp(\sigma) = f(\alpha_1)f_\sharp(\beta_1)\ldots f(\alpha_n)f_\sharp(\beta_n) \]
    is a tight concatenation of tight paths.
\end{proof}

An application of the above lemma is the following result.
The proof is identical to the original, using only \Cref{rttlemma}
(which is~\cite[Lemma 5.8]{BestvinaHandel})
and the topology of the underlying graph of $\mathcal{G}$, so we omit it.

\begin{lem}[Lemma 5.10 of \cite{BestvinaHandel}]
    \label{lengthfunction}
    Suppose that $f\colon\mathcal{G} \to \mathcal{G}$ is a relative train track map
    and that $H_r$ is an exponentially growing stratum.
    There is a length function $L_r(\sigma)$ for paths $\sigma$ in $G_r$
    with the property that $L_r(f(\sigma)) = L_r(f_\sharp(\sigma)) = \lambda_rL_r(\sigma)$
    for any $r$-legal path $\sigma$ in $G_r$.
    If we assume that $f$ acts linearly with respect to some metric
    on the underlying graph of $\mathcal{G}$,
    then if $\sigma$ contains an initial or terminal segment of an edge in $H_r$,
    then $L_r(\sigma) > 0$.
\end{lem}

We have an application to splittings of paths that are not necessarily legal.
Suppose that $f\colon\mathcal{G} \to \mathcal{G}$ is a relative train track map,
that $H_r$ is an exponentially growing stratum,
that $\sigma$ is a tight path in $\mathcal{G}$,
and that $\alpha$ is a subpath of $\sigma$ in $G_r$ with endpoints at vertices.
If there are $k$ edges of $H_r$ to the left and to the right of $\alpha$ in $\sigma$,
define $W_k(\alpha)$ to be the subpath of $\sigma$ that begins
with the $k$th edge of $H_r$ to the left of $\alpha$ and ends
with the $k$th edge of $H_r$ to the right of $\alpha$.
We say that $\alpha$ is \emph{$k$-protected in $\sigma$}
if its first and last edges are in $H_r$,
if $W_k(\alpha)$ is in $G_r$ and if $W_k(\alpha)$ is $r$-legal.

\begin{lem}[cf. Lemma 4.2.2 of~\cite{BestvinaFeighnHandel}]
    \label{kprotectedsplitting}
    Suppose that $f\colon \mathcal{G} \to \mathcal{G}$ is a relative train track map
    and that $H_r$ is an exponentially growing stratum.
    There is a constant $K$ such that if
    $\sigma$ is a tight path in $\mathcal{G}$
    and if $\alpha$ in $G_r$ is a $K$-protected subpath of $\sigma$,
    then $\sigma$ can be split at the endpoints of $\alpha$.
\end{lem}

\begin{proof}
    We follow the proof in~\cite[Lemma 4.2.2]{BestvinaFeighnHandel}.
    Choose $\ell$ so that the $f^\ell$-image of an edge in $H_r$
    contains at least two edges in $H_r$.
    Let $K = 2\ell C$, where $C$ is a bounded cancellation constant for $f$
    as in \Cref{boundedcancellation}.

    We show that if $\sigma = \sigma_1\alpha\sigma_2$,
    then $f^i_\sharp(\sigma) = f^i_\sharp(\sigma_1)f^i_\sharp(\alpha)f^i_\sharp(\sigma_2)$
    for $1 \le i \le \ell$
    and that $f^\ell_\sharp(\alpha)$ is $K$-protected in $f^\ell_\sharp(\sigma)$,
    so the result follows by induction.

    So take $i$ satisfying $1 \le i \le \ell$.
    Let $W_K(\alpha) = \tau_1\alpha\tau_2$.
    Since $\alpha$ begins and ends with edges of $H_r$
    and $W_K(\alpha)$ is $r$-legal,
    \Cref{rttlemma} implies that 
    $f^i_\sharp(W_K(\alpha)) = f^i_\sharp(\tau_1)f^i_\sharp(\alpha)f^i_\sharp(\tau_2)$.
    Write $\sigma_1 = \beta_1\tau_1$ and $\sigma_2 = \tau_2\beta_2$.
    By \Cref{boundedcancellation} (applied iteratively),
    when $f^i_\sharp(\beta_1)f^i_\sharp(\tau_1)$
    and $f^i_\sharp(\tau_2)f^i_\sharp(\beta_2)$ are tightened to compute
    $f^i_\sharp(\sigma_1)$ and $f^i_\sharp(\sigma_2)$,
    at most $iC$ edges are canceled.
    But $\tau_i$ contains at least $2\ell C$ edges,
    so we conclude that $f^i_\sharp(\sigma_1)f^i_\sharp(\alpha)f^i_\sharp(\sigma_2)$
    is a tight concatenation of tight paths.
    In fact, the same argument applies with $\alpha$ replaced by $W_{\ell C}(\alpha)$.
    By the choice of $\ell$,
    we have that $f^\ell_\sharp(W_{\ell C}(\alpha))$ 
    contains $W_{K}(f^\ell_\sharp(\alpha))$,
    so $f^\ell_\sharp(\alpha)$ is $K$-protected in $f^\ell_\sharp(\tau)$.
\end{proof}

\paragraph{Almost Nielsen paths.}
A path $\sigma$ is a \emph{periodic almost Nielsen path} 
with respect to a topological representative
$f\colon \mathcal{G} \to \mathcal{G}$ if $\sigma$ is nontrivial
and $f^k_\sharp(\sigma) = g\sigma g'$ for some $k\ge 1$
and vertex group elements $g$ and $g'$.
The minimal such $k$ is the \emph{period} of $\sigma$,
and $\sigma$ is an \emph{almost Nielsen path} if it has period $1$.
If $g$ and $g'$ are trivial,
then we have \emph{periodic Nielsen paths} and \emph{Nielsen paths.}
A periodic almost Nielsen path is \emph{indivisible} if it cannot be written 
as a concatenation of nontrivial periodic almost Nielsen paths.
(Note that we do not allow interior vertex group elements to change.)
We consider an equivalence relation on periodic almost Nielsen paths,
where $\sigma$ is is equivalent to $\rho$ if $\rho = g\sigma g'$,
for vertex group elements $g$ and $g'$

\begin{lem}[cf. Lemma 5.11 \cite{BestvinaHandel}]
    \label{finitelymanynielsenpaths}
    Suppose $f\colon \mathcal{G} \to \mathcal{G}$ is a relative train track map
    and that $H_r$ is an exponentially growing stratum.
    There are only finitely many equivalence classes 
    of indivisible periodic almost Nielsen paths
    $\sigma$ in $G_r$ that meet the interior of $H_r$.
    Each such $\sigma$ has exactly one illegal turn in $H_r$,
    thus $\sigma = \alpha\beta$, where $\alpha$ and $\beta$ are $r$-legal paths
    and the turn in $H_r$ at the common vertex of $\bar\alpha$ and $\beta$ is illegal.

    An indivisible periodic almost Nielsen path of height $r$
    has period 1 if and only if the directions determined by $\alpha$ and $\bar\beta$
    are almost fixed.
\end{lem}

\begin{proof}
    The proof follows~\cite[Lemma 5.11]{BestvinaHandel}
    and~\cite[Lemma 4.2.5]{BestvinaFeighnHandel}.
    Let $\sigma$ be an indivisible periodic almost Nielsen path
    that meets the interior of $H_r$.
    By \Cref{lengthfunction}, $\sigma$ cannot be $r$-legal.
    So let $\sigma = \alpha\beta\ldots$ be a decomposition of $\sigma$
    into maximal $r$-legal subpaths.
    Let $k$ be the period of $\sigma$.
    For $m \ge 1$, let $a_m$ be the smallest initial segment of $\alpha$
    that satisfies $f^{km}_{\sharp}(a_m) = g''\alpha$
    for some vertex group element $g''$.
    Let $b_m$ be the complementary segment of $\alpha$,
    so $\alpha = a_m b_m$.
    There is a largest initial segment $c_m$ of $\beta$
    such that $f^{km}_\sharp(b_m c_m)$ is trivial.
    Let $d_m$ be the complementary segment of $\beta$,
    so $\beta = c_m d_m$.
    It is not hard to see that we must have $f^k_\sharp(a_{m+1}) = ga_m$
    and $f^k_\sharp(b_{m+1}c_{m+1}) = b_m c_m$.
    Therefore $f^k_\sharp(\bigcap_{m=1}^\infty a_m) = \bigcap_{m=1}^\infty a_m$
    and $f^k_\sharp(\bigcup_{m=1}^\infty b_m c_m) = \bigcup_{m=1}^\infty b_m c_m$.
    Since $\sigma$ is indivisible,
    we must have that $\bigcap_{m=1}^\infty a_m$ is a point
    and that $\bigcup_{m=1}^\infty b_m c_m$ is the interior of $\sigma$.

    We have shown that $\sigma$ satisfies the following conditions:
    \begin{enumerate}
        \item The path $\sigma = \alpha\beta$ has exactly one illegal turn in $H_r$.
        \item The initial and terminal (partial) edges of $\sigma$ are in $H_r$.
        \item The number of $H_r$-edges in $f^\ell_\sharp(\sigma)$ is bounded independently of $\ell$.
    \end{enumerate}
    We will show that there are only finitely many paths $\sigma$
    satisfying these three conditions
    up to the equivalence relation above.

    Let $\sigma$ be a path satisfying these conditions,
    and as above write $\sigma = \alpha\beta$,
    where $\alpha$ and $\beta$ are $r$-legal.
    When $f^\ell(\alpha)$ and $f^\ell(\beta)$ are tightened to
    $f^\ell_\sharp(\alpha)$ and $f^\ell_\sharp(\beta)$,
    no edges of $H_r$ are canceled.
    When $f^\ell_\sharp(\alpha)f^\ell_\sharp(\beta)$ 
    is tightened to $f^\ell_\sharp(\sigma)$,
    an initial segment of $f^\ell_\sharp(\bar\alpha)$ 
    is canceled with an initial segment of $f^\ell_\sharp(\beta)$.
    Since $f^\ell_\sharp(\sigma)$ has an illegal turn in $H_r$,
    (it must have at least one by \Cref{lengthfunction}
    and cannot have more because $\alpha$ and $\beta$ are $r$-legal)
    the first edges which are not canceled are contained in $H_r$.
    Since nondegenerate turns with the same underlying edge are legal,
    the first edges which are not canceled are distinct edges.
    Therefore the cancellation 
    between $f^\ell_\sharp(\alpha)$ and $f^\ell_\sharp(\bar\beta)$
    is determined by the ordered list of oriented edges
    of $f^\ell_\sharp(\alpha)$ and $f^\ell_\sharp(\bar\beta)$ in $H_r$:
    the two paths cancel until the first distinct edges of $H_r$ are reached.
    Call these lists $f^\ell_\sharp(\alpha) \cap H_r$
    and $f^\ell_\sharp(\bar\beta) \cap H_r$,
    and note that they are determined by
    $\alpha\cap H_r$ and $\bar\beta\cap H_r$, respectively.

    We claim that $\sigma \cap H_r$ takes on only finitely many values
    as $\sigma$ varies over paths satisfying the three conditions above.
    Note that if $\sigma$ has a splitting,
    then one of the pieces of the splitting is $r$-legal and meets the interior of $H_r$.
    Then the number of $H_r$ edges of $\sigma$ would grow without bound
    by \Cref{rttlemma} in contradiction to our assumption.
    But \Cref{kprotectedsplitting} then implies that the number of $H_r$ edges of $\sigma$
    is bounded independently of $\sigma$.
    Therefore it follows that $\sigma \cap H_r$ takes on only finitely many values
    except for the possibility that differing amounts of the first and last edges 
    $\alpha_0$ and $\beta_0$ may occur.
    That is, suppose $\sigma'$ is another path satisfying conditions 1 through 3 above,
    and that $\sigma \cap H$ and $\sigma' \cap H_r$ are identical
    except that the lengths of $\alpha_0$ and $\alpha'_0$
    and the lengths of $\beta_0$ and $\beta'_0$ differ.
    For concreteness, suppose $\alpha'_0$ is a proper subset of $\alpha_0$,
    and that $A$ is what is left over.
    Since $A$ is $r$-legal, the number of $H_r$-edges in $f^\ell_\sharp(A)$
    grows without bound.
    This implies that all of $f^\ell_\sharp(\bar\alpha') \cap H_r$ is canceled with 
    a proper initial segment $X$ of $f^\ell_\sharp(\beta) \cap H_r$
    for sufficiently large $\ell$.
    But $X$, like all initial segments of $f^\ell_\sharp(\beta)$,
    either contains or is contained in $f^\ell_\sharp(\beta')\cap H_r$.
    If ``contains,'' then all of $f^\ell_\sharp(\bar\beta') \cap H_r$
    is canceled with part of $f^\ell_\sharp(\alpha')\cap H_r$.
    If ``contained in,'' then all of $f^\ell_\sharp(\alpha')\cap H_r$
    is canceled with part of $f^\ell_\sharp(\bar\beta') \cap H_r$.
    In either case $f^\ell_\sharp(\sigma')$ is legal, a contradiction.

    Property 2, the fact that the number of $H_r$ edges of $\alpha$ and $\beta$
    is bounded independently of $\sigma$,
    and the fact that $\alpha_0$ and $\beta_0$ take on only finitely many values
    (up to multiplication by elements of $\mathcal{G}_{\iota(\alpha_0)}$
    and $\mathcal{G}_{\tau(\beta_0)}$ if $\alpha_0$ and $\beta_0$ are whole edges)
    implies that there exists $\ell > 0$ independent of $\sigma$
    such that $f^\ell_\sharp(\sigma)$ is obtained from
    $f^\ell_\sharp(\alpha_0)$ and $f^\ell_\sharp(\beta_0)$ by 
    concatenating and canceling at the juncture.
    This implies that $f^\ell_\sharp(\sigma)$ itself takes on only finitely many values
    up to multiplication by vertex group elements at the ends,
    and therefore so does $\sigma$.

    We now turn to the second paragraph of the lemma.
    Suppose $\sigma$ is an indivisible periodic almost Nielsen path of height $r$
    with period $p$, and write $\sigma = \alpha\beta$.
    The proof follows \cite[Lemma 2.11]{FeighnHandel}.
    After subdividing at the endpoints of $f^k_\sharp(\sigma)$ for $0 \le k \le p-1$,
    we may assume the endpoints of $\sigma$ are vertices.
    Since $\alpha$ and $\beta$ are $r$-legal,
    the relative train track property implies that $Df$ maps the initial directions
    of $\alpha$ and $\bar\beta$ to the initial directions of $f_\sharp(\alpha)$ and $f_\sharp(\beta)$
    respectively.
    If $\sigma$ has period 1,
    then we see that these directions are almost fixed,
    since $f_\sharp(\sigma) = g\sigma g'$ is obtained from $f_\sharp(\alpha)$ and $f_\sharp(\bar\beta)$
    by cancelling their maximal common terminal segment.

    Now suppose that the initial directions of $\alpha$ and $\bar\beta$ are almost fixed
    and write $f_\sharp(\sigma) = \alpha_1\beta_1$.
    The first edge $E$ of $\alpha$ is also the first edge of $\alpha_1$,
    and the argument above shows that $\alpha$ and $\alpha_1$ are initial segments
    of $f_\sharp^{Np}(E)$ for large $N$.
    Thus either $\alpha$ is an initial segment of $g\alpha_1$ 
    for some vertex group element $g$ or vice versa.
    Assume $\alpha$ is an initial segment of $g\alpha_1$.

    Suppose toward a contradiction that $\alpha_1 = g\alpha\gamma$ for some nontrivial path $\gamma$.
    The path $\alpha_2 = \bar\beta\gamma$ is a subpath of $f_\sharp^{Np}(\bar\beta)$
    for large $N$ and is thus $r$-legal.
    The path $\alpha_2\beta_1 \simeq \beta\bar\alpha\alpha_1\beta_1$
    is a nontrivial periodic almost Nielsen path
    with exactly one illegal turn in $H_r$.
    It is therefore indivisible.
    Notice that $\bar\beta$ and $g'^{-1}\bar\beta_1$ 
    share a common initial subpath in $H_r$ 
    for some vertex group element $g'$
    by assumption,
    therefore the same is true of $\alpha_2$ and $g'^{-1}\bar\beta_1$.
    In fact, the argument above shows that $\alpha_2$ and $g'^{-1}\bar\beta_1$
    are initial subpaths of a common path $\delta$.
    They cannot be equal, so one is a proper initial subpath of the other.
    But note that the difference between the number of $H_r$ edges
    in $f^{Np}_\sharp(\alpha_2)$ and $f^{Np}_\sharp(\bar\beta_1)$ must grow exponentially in $N$,
    contradicting the fact that $\alpha_2\beta_1$ is a periodic almost Nielsen path.
    This contradiction shows that $\alpha_1 = g\alpha$,
    and a symmetric argument shows that $\beta_1 = \beta g'$,
    so $p = 1$.
\end{proof}

Let $P_r$ be the set of equivalence classes of paths
satisfying items 1 through 3 in
the proof of \Cref{finitelymanynielsenpaths};
this is a finite set.
The following lemma will be used in \Cref{CTsection}.

\begin{lem}[cf.~Lemma 4.2.6 of \cite{BestvinaFeighnHandel}]
    \label{egsplitting}
    Suppose that $f\colon \mathcal{G} \to \mathcal{G}$ is a relative train track map
    and that $H_r$ is an exponentially growing stratum.
    If $\sigma$ is a path in $G_r$ with the property that each path
    $f^k_\sharp(\sigma)$ has the same number of illegal turns in $H_r$,
    then $\sigma$ can be split into subpaths that are either $r$-legal or elements of $P_r$.
\end{lem}

\begin{proof}
    The proof is identical to \cite[Lemma 4.2.6]{BestvinaFeighnHandel}.
    We prove the statement by induction on $m$, the number of illegal turns in $H_r$ that $\sigma$ has.
    If $m = 0$, then we are done.
    Suppose that $m = 1$ and that $\sigma$ cannot be split;
    we will show that $\sigma$ belongs to $P_r$.
    The hypothesis that $\sigma$ has exactly one illegal turn in $H_r$ is satisfied.
    If the number of $H_r$-edges in $f^\ell_\sharp(\sigma)$ is not bounded independent of $\ell$,
    then \Cref{kprotectedsplitting} implies that $\sigma$ may be split.
    Similarly if the first and last (possibly partial) edges of $\sigma$ are not contained in $H_r$,
    then \Cref{rttlemma} implies that $\sigma$ 
    may be split at the initial vertex of the first edge of $H_r$ in $\sigma$.
    Therefore if $\sigma$ cannot be split, then $\sigma$ belongs to $P_r$.

    So suppose $m > 1$.
    Choose lifts $\tilde f\colon \Gamma \to \Gamma$ and
    $\tilde\sigma$ in $\Gamma$ and
    decompose $\tilde\sigma = \tilde\sigma_1\ldots \tilde\sigma_{m+1}$
    so that each juncture projects to an illegal turn in $H_r$ and each $\tilde\sigma_i$
    projects to an $r$-legal path.
    By hypothesis, each $\tilde f^k_\sharp(\tilde\sigma)$ has a decomposition
    $\tilde f^k_\sharp(\tilde\sigma) = \tilde\tau^k_1\ldots\tilde\tau^k_{m+1}$ 
    into maximal $r$-legal subpaths.
    The set $\tilde S^2_k = \{\tilde x \in \tilde\sigma_2 
    : \tilde f^k(\tilde x) \in \tilde f^k_\sharp(\tilde\sigma)\}$ is closed,
    and one can argue by induction that $f^N$ maps $\bigcap_{k=1}^N(\tilde S^2_N)$ onto $\tilde\tau^N_2$
    for all $N \ge 1$.
    Therefore $\tilde S^2 = \bigcap_{k=1}^\infty\tilde S^2_k$ is nonempty.
    It is not hard to see that $\sigma$ can be split at a point $x$
    if and only if $\tilde f^k(\tilde x)$ belongs to $\tilde f^k_\sharp(\tilde\sigma)$
    for all $k \ge 0$.
    Therefore $\sigma$ can be split at (the projection of) any point in $\tilde S^2$.
    This splits $\tilde\sigma$ into subpaths that have fewer than $m$ illegal turns in $H_r$,
    and thus induction completes the proof.
\end{proof}

\paragraph{Non-exponentially growing strata.}
A non-exponentially growing stratum $H_r$ is \emph{almost periodic}
if the edges $E_1,\dotsc,E_k$ of $H_r$
satisfy $f(E_i) = g_iE_{i+1}h_i$ 
for vertex group elements $g_i$ and $h_i$ and
with indices taken mod $k$.
If $k = 1$, we say $H_r$ is an \emph{almost fixed} stratum and $E_1$ is an \emph{almost fixed} edge.

If $H_r$ is a non-exponentially growing but not periodic stratum
for a topological representative $f\colon \mathcal{G} \to \mathcal{G}$,
each edge $e$ has a subinterval which is eventually mapped back over $e$,
so the subinterval contains a periodic point.
After declaring all of these periodic points to be vertices,
reordering, reorienting, and possibly replacing $H_r$
with two non-exponentially growing strata,
we may assume the edges $E_1,\dotsc,E_k$ of $H_r$
satisfy $f(E_i) = g_iE_{i+1}u_i$,
where indices are taken mod $k$, $g_i$ is a vertex group element and $u_i$ is a path in $G_{r-1}$.
Henceforth we always adopt this convention.

\paragraph{Dihedral Pairs.}
A \emph{dihedral pair} is a pair of (oriented) edges $E_i$ and $E_j$
with the following properties.
\begin{enumerate}
    \item $E_i$ and $E_j$ have a common initial vertex $v$
        with trivial vertex group.
        We refer to $v$ as the \emph{center vertex of the dihedral pair $E_i$ and $E_j$.}
    \item $E_i$ and $E_j$ have terminal vertices 
        each with valence one in $G$ and $C_2$ vertex group.
    \item The edges $E_i$ and $E_j$ are either fixed or swapped by $f$.
        (Up to homotopy, we may assume that $f(E_i) = E_j$
        and $f(E_j) = E_i$,
        rather than for example $E_j g$, where $g$ is the nontrivial element of $C_2$.)
\end{enumerate}
Dihedral pairs will play an exceptional role in \Cref{CTsection}.
If the edges $E_i$ and $E_j$ are fixed by $f$,
they determine fixed strata $H_i$ and $H_j$ that are forests.
By rearranging strata, we may assume that $H_j = H_{i+1}$.
In this situation, we say that $H_i$ is the 
\emph{bottom half of the dihedral pair.}

\paragraph{}
For a topological representative $f\colon \mathcal{G} \to \mathcal{G}$,
let $\per(f)$ denote the set of $f$-periodic points in $G$
(as a map of topological spaces).
The subset of points with period $1$ is $\fix(f)$.
A subgraph $C \subset G$ is \emph{wandering} if $f^k(C) \subset \overline{G\setminus C}$
for all $k \ge 1$ and is \emph{non-wandering} otherwise.

The \emph{core} of a subgraph $C \subset G$ is the minimal subgraph (of groups) $K$ of $C$
such that the inclusion is a homotopy equivalence.

\paragraph{Enveloped zero strata.}
Suppose that $f\colon\mathcal{G} \to \mathcal{G}$ is a topological representative,
that $u < r$ and that the following hold.
\begin{enumerate}
	\item The stratum $H_u$ is irreducible.
	\item The stratum $H_r$ is exponentially growing. 
        Each component of $G_r$ is noncontractible.
	\item For each $i$ satisfying $u < i < r$, the stratum $H_i$ is a zero stratum
		that is a component of $G_{r-1}$,
		and each vertex of $H_i$ has valence at least two in $G_r$.
\end{enumerate}
Then we say that each $H_i$ is \emph{enveloped by $H_r$,}
and write $H^z_r = \bigcup_{k=u+1}^r H_k$.

The following theorem is the main result of this section;
its proof occupies the remainder of the section.
\begin{thm}[\cite{FeighnHandel} Theorem 2.19]
	\label{improvedrelativetraintrack}
    Given an outer automorphism $\varphi \in \out(F,\mathscr{A})$
	there is a relative train track map $f\colon\mathcal{G} \to \mathcal{G}$
    on a marked Grushko $(F,\mathscr{A})$-graph of groups
	and filtration $\varnothing = G_0 \subset G_1 \subset \dotsb \subset G_m = G$
	representing $\varphi$ satisfying the following properties:
	\begin{enumerate}
		\item[\hypertarget{V}{(V)}] The endpoints of all 
            indivisible periodic almost Nielsen paths are vertices.
		\item[\hypertarget{P}{(P)}] If a periodic or almost periodic stratum $H_m$ 
            is a forest,
            then either 
			there exists a filtration element $G_j$ 
			such that $\mathcal{F}(G_j) \ne \mathcal{F}(G_\ell\cup H_m)$
			for any filtration element $G_\ell$
            or $H_m$ is the bottom half of a dihedral pair.
		\item[\hypertarget{Z}{(Z)}] 
			Each zero stratum $H_i$ is enveloped by an exponentially growing stratum $H_r$.
			Each vertex of $H_i$ is contained in $H_r$ and meets only edges in $H_i\cup H_r$.
		\item[\hypertarget{NEG}{(NEG)}] The terminal endpoint of an edge in a 
			non-exponentially growing stratum $H_i$ is periodic,
            and if the stratum is not almost periodic,
            the terminal endpoint
			is contained in a filtration element $G_j$ with $j < i$ that is its own core.
		\item[\hypertarget{F}{(F)}] The core of a filtration element $G_r$ is a filtration element,
            unless $H_r$ is the bottom half of a dihedral pair,
            in which case $G_{r-1}$ and $G_{r+1}$ are their own core.
	\end{enumerate}

	Moreover, if $\mathcal{F}_1 \sqsubset \dotsb \sqsubset \mathcal{F}_d$ is a nested sequence
	of $\varphi$-invariant free factor systems, we may assume that each free factor system
	is realized by some filtration element.
\end{thm}

Before we turn to the proof,
we prove a lemma and a simple consequence of \Cref{improvedrelativetraintrack}
that we will need later.

\begin{lem}[\cite{FeighnHandel} Lemma 2.10]
    \label{egvalence}
    Let $H_r$ be an exponentially growing stratum of a relative train track map 
    $f\colon\mathcal{G} \to \mathcal{G}$ and $v$ a vertex of $H_r$.
    Then there is a legal turn in $G_r$ based at $v$.
\end{lem}

\begin{proof}
    The proof is identical to \cite[Lemma 2.10]{FeighnHandel}.
    There is a point $w$ in the interior of an edge $E$ of $H_r$
    and $j > 0$ such that $f^j(w) = v$.
    By \hyperlink{EG-iii}{(EG-iii)},
    the point $v$ is in the interior of an $r$-legal path.
    By \hyperlink{EG-i}{(EG-i)} and \hyperlink{EG-iii}{(EG-iii)},
    the turn based at $v$ crossed by this path is legal.
\end{proof}

In particular, the lemma implies that if $v$ has valence one in $G_r$,
then $v$ has nontrivial vertex group.

\begin{lem}
    \label{legalturn}
    Suppose $f\colon \mathcal{G} \to \mathcal{G}$ is a relative train track map
    satisfying properties \hyperlink{Z}{(Z)} and \hyperlink{NEG}{(NEG)}.
    If $v$ is a vertex of a filtration element $G_k$ that is its own core,
    then there is a legal turn based at $v$ in $G_k$.
\end{lem}

\begin{proof}
    The proof is by induction on $k$;
    the statement is vacuously true for $k = 0$.
    If $v$ has nontrivial vertex group,
    then the lemma holds,
    since nondegenerate turns of the form $\{(g_1,e),(g_2,e)\}$ 
    with the same underlying edge $e$ are legal.
    So suppose $v$ has trivial vertex group.
    Since $G_k$ is its own core,
    $v$ has valence at least two in $G_k$,
    and we may assume that there is an illegal turn based at $v$.
    By \Cref{egvalence} and property \hyperlink{Z}{(Z)}
    we may assume that $v$ is not incident to any edges 
    of exponentially growing or zero strata.
    The directions determined by periodic edges
    and the initial endpoints of non-periodic non-exponentially growing edges are periodic,
    so since we assume there is an illegal turn,
    we conclude that $v$ is the terminal endpoint of a non-periodic non-exponentially growing edge,
    and thus by \hyperlink{NEG}{(NEG)} 
    $v$ is contained in a lower stratum that is its own core;
    we conclude that there is a legal turn based at $v$ by induction.
\end{proof}

\begin{proof}[Proof of \Cref{improvedrelativetraintrack}]
	We adapt the proof of~\cite[Theorem 2.19, pp.~56--62]{FeighnHandel}.
	Begin with a relative train track map $f\colon \mathcal{G} \to \mathcal{G}$
	which has no valence-one vertices with trivial vertex group.
    By \Cref{rttprop}, we may assume that the filtration for $f$ realizes our
    nested sequence of $\varphi$-invariant free factor systems.
    There is no loss in assuming $f$ to be a topological representative,
    so we shall.

	\paragraph{Property (V).} To prove \hyperlink{V}{(V)},
    we need to collect more information about almost Nielsen paths.
    We saw in \Cref{finitelymanynielsenpaths} that if
    $f\colon \mathcal{G} \to \mathcal{G}$ is a relative train track map
    and $H_r$ is an exponentially growing stratum,
    then there are only finitely many equivalence classes of
    indivisible periodic almost Nielsen paths of height $r$.

    If $H_r$ is a zero stratum, 
    then there are no indivisible periodic almost Nielsen paths of height $r$.
    If $H_r$ is a periodic or an almost periodic stratum,
    then there are no indivisible periodic almost Nielsen paths of height $r$.
    (Note that if $E$ is an almost periodic edge,
    then it is a periodic almost Nielsen path,
    but it is not indivisible.)

    If $H_r$ is a non-exponentially growing stratum which is not almost periodic,
    then assuming that our relative train track map $f$
    acts linearly with respect to some metric on $G$,
    there are no periodic points in the interior of each edge in $H_r$.
    Thus the endpoints of periodic almost Nielsen paths of height $r$,
    if there are any, are vertices.

    All together, this implies that \hyperlink{V}{(V)} can be accomplished
    by declaring a finite number
    of periodic points in the interior of exponentially growing strata to be vertices.
    The resulting map is still a relative train track map.
	
    Let us remark that for the remainder of our construction,
    if $f\colon \mathcal{G} \to \mathcal{G}$ is our relative train track map
    satisfying \hyperlink{V}{(V)}
    and $f'\colon \mathcal{G}' \to \mathcal{G}'$ is the final relative train track map,
    then there is a bijection $H_r \to H'_s$
    from the exponentially growing strata for $f$ to those for $f'$ such that
    \begin{enumerate}
        \item $H_r$ and $H'_s$ have the same number of edges, and
        \item The number of equivalence classes of indivisible almost Nielsen paths 
            for $f$ of height $r$
            (which is finite by \Cref{finitelymanynielsenpaths})
            is equal to the number of equivalence classes of indivisible almost Nielsen paths
            for $f'$ of height $s$.
    \end{enumerate}
    
    For all of the moves we will use, the first item is obvious.
    For some of the moves we will use, 
    namely valence-two homotopies not involving exponentially growing strata
    and reordering strata, both properties are obvious.
    For \emph{sliding,} which is defined below, the second item is part of \Cref{sliding lemma}.
    For the remaining moves, namely \emph{tree replacements,} which are defined below,
    and collapsing forests away from exponentially growing strata
    we have the following lemma, which demonstrates that the second item holds.

    \begin{lem}[cf.~Lemma 2.16 of \cite{FeighnHandel}]
        Suppose that the topological representatives
        $f\colon \mathcal{G} \to \mathcal{G}$ and $f'\colon \mathcal{G}' \to \mathcal{G}'$
        are relative train track maps with exponentially growing strata $H_r$ and $H'_s$
        respectively
        such that all indivisible periodic almost Nielsen paths of height $r$ or $s$
        respectively
        have endpoints at vertices.
        Suppose further that $p\colon \mathcal{G} \to \mathcal{G}'$ is a homotopy equivalence such that
        the following hold.
        \begin{enumerate}
            \item The map $p$ is such that $p(G_r) = G'_s$, $p(G_{r-1}) = G'_{s-1}$
                and $p$ induces a bijection 
                between the set of edges of $H_r$ and the set of edges of $H'_s$.
            \item We have $p_\sharp f_\sharp(\sigma) = f'_\sharp p_\sharp(\sigma)$
                for all tight paths $\sigma$ in $G_r$ with endpoints at vertices.
        \end{enumerate}
        Then $p_\sharp$ induces a period-preserving bijection 
        between the indivisible periodic almost Nielsen paths of height $r$ for $f$
        and those of height $s$ for $f'$.
    \end{lem}

    \begin{proof}
        The proof is identical to \cite[Lemma 2.16]{FeighnHandel}.
        Let $\sigma$ be a tight path of height $r$ with endpoints at vertices,
        and let $\sigma' = p_\sharp(\sigma)$. It has height $s$.
        
        Observe that no edges in $H'_s$ are cancelled when $p(\sigma)$ is tightened to $p_\sharp(\sigma)$,
        for if there were, then $\sigma$ has a subpath of the form $\sigma_0 = E\tau \bar E$,
        where $E$ is an edge of $H_r$ and $p_\sharp(\sigma_0)$ is trivial.
        But the closed path $\sigma_0$ determines 
        a nontrivial element of the fundamental group of $\mathcal{G}$ by assumption,
        so this is a contradiction,
        from which it follows by item 1 that
        \begin{enumerate}
            \item[3.] The number of edges of $H_r$ in $\sigma$ is equal to the number of $H'_s$
                edges of $\sigma' = p_\sharp(\sigma)$.
        \end{enumerate}

        We claim that $\sigma$ is $r$-legal if and only if $\sigma'$ is $s$-legal.
        If $E$ is an edge of $H_r$, then by item 2 
        and the fact that $f$ and $f'$ are topological representatives,
        we have $f'(p(E)) = f'_\sharp(p_\sharp(E)) = p_\sharp(f_\sharp(E)) = p_\sharp(f(E))$.
        By item 3, this path has the same number of edges in $H'_s$ as $f(E)$ has edges in $H_r$.
        It follows that before tightening,
        $f(\sigma)$ has as many edges in $H_r$ as $f'(\sigma')$ has in $H'_s$.
        After tightening, item 3 implies that $p_\sharp f_\sharp(\sigma) = f'_\sharp(\sigma')$
        has as many edges in $H'_s$ as $f_\sharp(\sigma)$ has in $H_r$,
        so we see that $\sigma$ is $r$-legal if and only if $\sigma'$ is $r$-legal.
        What's more, the number of illegal turns of $\sigma$ in $H_r$ 
        equals the number of illegal turns of $\sigma'$ in $H'_s$.

        Assume that $\sigma$ is an indivisible periodic almost Nielsen path with height $r$
        and period $k$. By item 2, we have
        \[  (f')^k_\sharp(\sigma') = (f')^k_\sharp(p_\sharp(\sigma)) = p_\sharp f^k_\sharp(\sigma) = 
        p_\sharp(g \sigma g') = g'' \sigma' g''' \]
        where $g$, $g'$, $g''$ and $g'''$ are vertex group elements.
        Therefore $\sigma'$ is a periodic almost Nielsen path.
        If $\sigma$ is indivisible, then the first and last edges of $\sigma$ belong to $H_r$
        and $\sigma$ has exactly one illegal turn in $H_r$.
        By the argument above, $\sigma'$ has these properties for $H'_s$ and so is indivisible.
        Let $E_i$ and $E_j$ be the first and last edges of $\sigma$.
        Then by \Cref{finitelymanynielsenpaths}, 
        $k$ is the minimum positive integer such that $E_i$ and $E_j$
        are the first and last edges of $f^k_\sharp(\sigma)$.
        By item 2, it follows that the period of $\sigma'$ equals the period of $\sigma$.
        To prove that if $\sigma'$ is an indivisible periodic almost Nielsen path of period $k$
        then $\sigma$ is as well,
        observe that if $\sigma'$ is a tight path in $G'_s$ whose first edge is $p(E_i)$
        and whose last edge is $p(E_j)$,
        then there is a unique tight path $\sigma$ whose first edge is $E_i$, whose last edge is $E_j$
        and satisfies $p_\sharp(\sigma) = \sigma'$.
        If $\sigma'$ is a periodic almost Nielsen path of period $k$,
        then the uniqueness of $\sigma$ implies that $\sigma$ is a periodic almost Nielsen path
        of period $k$.
        As above, if $\sigma'$ is indivisible, then it has exactly one illegal turn in $H'_s$,
        and $\sigma$ has exactly one illegal turn in $H_r$, so $\sigma$ is indivisible.
    \end{proof}
	\paragraph{Sliding.}
	The move \emph{sliding} was introduced in~\cite[Section 5.4, p.~579]{BestvinaFeighnHandel}.
	Suppose $H_i$ is a non-periodic, 
    non-exponentially growing stratum that satisfies our convention:
	the edges $E_1,\dotsc,E_k$ of $H_i$ satisfy $f(E_j) = g_jE_{j+1}u_j$ where indices are taken mod $k$,
    $g_j$ is a vertex group element
	and $u_j$ is a path in $G_{i-1}$. 
	We will call the edge of $H_i$ we focus on $E_1$.
	Let $\alpha$ be a path in $G_{i-1}$ from the terminal vertex of $E_1$ to some vertex of $G_{i-1}$.
	Define a new graph of groups $\mathcal{G}'$ by removing $E_1$ from $\mathcal{G}$
	and gluing in a new edge $E'_1$ with initial vertex equal to the initial vertex of $E_1$
	and terminal vertex the terminal vertex of the path $\alpha$.
	See \Cref{slidingfig}.
	Define homotopy equivalences $p\colon \mathcal{G} \to \mathcal{G}'$ 
	and $p'\colon\mathcal{G}' \to \mathcal{G}$
	by sending each edge other than $E_1$ and $E'_1$ to itself,
	and defining $p(E_1) = E'_1\bar\alpha$ and $p'(E'_1) = E_1\alpha$.
	Define $f'\colon\mathcal{G}' \to \mathcal{G}'$ by tightening
	$pfp'\colon \mathcal{G}' \to \mathcal{G}'$.
	If $G_r$ is a filtration element of $\mathcal{G}$,
	define $G'_r = p(G_r)$.
	The $G'_r$ form the filtration for $f'\colon\mathcal{G}' \to \mathcal{G}'$.
	
	\begin{figure}[!ht]
		\centering
	    \def\svgwidth{\columnwidth}
	        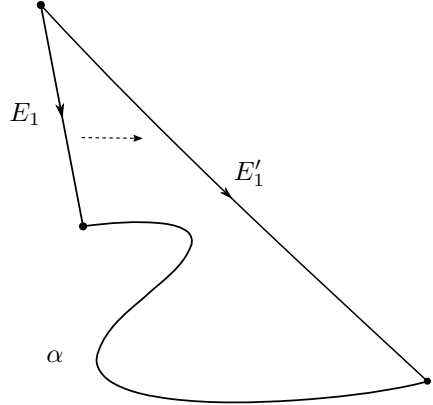
	
		\caption{Sliding $E_1$ along $\alpha$}\label{slidingfig}
	\end{figure}

	\begin{lem}[\cite{FeighnHandel} Lemma 2.17]
        \label{sliding lemma}
		Suppose $f' \colon\mathcal{G}' \to \mathcal{G}'$
		is obtained from a topological representative
        $f\colon\mathcal{G} \to \mathcal{G}$ 
		by \emph{sliding $E_1$ along $\alpha$} as described above.
		Let $H_i$ be the non-exponentially growing stratum of $\mathcal{G}$
		containing $E_1$, and let $k$ be the number of edges of $H_i$.
		\begin{enumerate}
			\item $f'\colon\mathcal{G}' \to \mathcal{G}'$
				is a relative train track map if $f\colon\mathcal{G} \to \mathcal{G}$ was.
			\item $f'|_{G'_{i-1}} = f|_{G_{i-1}}$.
			\item If $k = 1$, then $f'(E'_1) = g_1E'_1[\bar\alpha u_1f(\alpha)]$,
				where $[\gamma]$ denotes the path obtained from $\gamma$ by tightening.
			\item If $k\ne 1$, then $f'(E_k) = g_kE'_1[\bar\alpha u_k]$,
				$f'(E'_1) = g_1E_2[u_1f(\alpha)]$ and
				$f'(E_j) = g_jE_{j+1}u_j$
				for $2 \le j \le k-1$.
			\item For each exponentially growing stratum $H_r$,
                the map $p_\sharp$ defines a bijection between the set
				of indivisible periodic Nielsen paths in $\mathcal{G}$ of height $r$
				and the indivisible Nielsen paths in $\mathcal{G}'$ of height $r$.
		\end{enumerate}
	\end{lem}

    \begin{proof}
        The proof is a straightforward adaptation 
        of \cite[Lemma 5.4.1]{BestvinaFeighnHandel}.
        We include a proof for completeness.
        Suppose first that $k =1$.
        Then
        \[
            f'(E'_1) = (pfp')_\sharp(E'_1) = (pf)_\sharp(E_1\alpha)
            = p_\sharp(g_1E_1f(\alpha)) = g_1E'_1[\bar\alpha u_1f(\alpha)].
        \]
        Next suppose that $k \ne 1$.
        Then
        \[
            f'(E_k) = (pfp')_\sharp(E_k) = p_\sharp(g_kE_1[u_k])
            = g_kE'_1[\bar\alpha u_k],
        \]
        and
        \[
            f'(E'_1) = (pfp')_\sharp(E'_1) = (pf)_\sharp(E_1\alpha)
            = g_1E_2[u_1f(\alpha)].
        \]
        This proves items 3 and 4.

        Item 2 is clear from the definition of sliding.
        To prove item 1, it suffices to consider each  stratum $H_j$
        with $j > i$.
        So suppose $f\colon \mathcal{G} \to \mathcal{G}$ was a relative train track map.
        If $H_j$ is a zero stratum and $E$ is an edge of $H_j$,
        then $f(E) \subset G_{j-1}$, so $f'(E) = (pf)_\sharp(E) \subset G'_{j-1}$,
        and $H'_j$ is still a zero stratum.
        Similarly if $H_j$ is a non-exponentially growing stratum,
        then $H'_j$ is still non-exponentially growing.
        For any nontrivial tight paths $\beta \subset \mathcal{G}$ 
        and $\gamma' \subset \mathcal{G}'$
        with endpoints at vertices,
        we have $(p'p)_\sharp(\beta) = \beta$
        and $(pp')_\sharp(\gamma') = \gamma'$,
        so $p_\sharp(\beta)$ and $p'_\sharp(\gamma')$ are nontrivial.

        Suppose $H_r$ is an exponentially growing stratum
        and $E$ is an edge of $H_r$.
        Property \hyperlink{EG-i}{(EG-i)} implies that
        \[f(E) = a_1b_1a_2b_2\ldots a_\ell b_\ell a_{\ell + 1},\]
        where $a_i \subset H_r$ and $b_i \subset G_{r-1}$ are nontrivial tight paths.
        Then 
        \[  f'(E) = (pf)_\sharp(E) = a_1 p_\sharp(b_1) a_2 p_\sharp(b_2)\ldots
        a_\ell p_\sharp(b_\ell) a_{\ell + 1}. \]
        Thus $H'_r$ is exponentially growing, its transition matrix
        is equal to the transition matrix of $H_r$, 
        and \hyperlink{EG-i}{(EG-i)} is satisfied.
        If $\gamma'$ in $G'_{r-1}$ is a nontrivial path with endpoints in
        $H'_r \cap  G'_{r-1}$,
        then $p'_\sharp(\gamma')$ is nontrivial,
        and so are $(fp')_\sharp(\gamma')$ and $(pfp')_\sharp(\gamma') = f'_\sharp(\gamma')$.
        Therefore $H'_r$ satisfies \hyperlink{EG-ii}{(EG-ii)}.
        In fact, \hyperlink{EG-iii}{(EG-iii)} is also satisfied:
        the map $Df$ of turns in $H_r$ is unchanged,
        so the legality or illegality of turns in $H'_r$ is identical to the situation in $H_r$.
        Thus $f'\colon \mathcal{G}' \to \mathcal{G}'$ is a relative train track map.
        
        If $\sigma'$ in $\mathcal{G}'$ satisfies $f'^k_\sharp(\sigma') = g \sigma' g'$
        for vertex group elements $g$ and $g'$,
        then $p'_\sharp(\sigma')$ satisfies
        \[  f^k_\sharp (p'_\sharp(\sigma')) = (p'p)_\sharp f^k_\sharp p'_\sharp (\sigma')
        = p'_\sharp f'^k_\sharp(\sigma') = g'' p'_\sharp(\sigma')g''', \]
        where for example $g'' = p'_\sharp(g)$ is a vertex group element.
        Similarly, if $\sigma \in \mathcal{G}$ satisfies 
        $f^k_\sharp(\sigma) = g\sigma g'$,
        then $p_\sharp(\sigma)$  satisfies
        \[  f'^k_\sharp(p_\sharp(\sigma)) = p_\sharp f^k_\sharp (p'p)_\sharp(\sigma)
        = p_\sharp f^k_\sharp(\sigma) = g'' p_\sharp(\sigma) g''', \]
        where for example $g'' = p_\sharp(\sigma)$ is a vertex group element.
        If $\sigma' \subset G'_r$, then $p'_\sharp(\sigma') \subset G_r$,
        and if $\sigma \subset G_r$, then $p_\sharp(\sigma) \subset G'_r$.
        Thus $p_\sharp$ induces a period-preserving, height-preserving bijection
        from the set of periodic almost Nielsen paths 
        for $f\colon \mathcal{G}  \to \mathcal{G}$ 
        to the set of periodic almost Nielsen paths 
        for $f'\colon \mathcal{G}' \to \mathcal{G}'$.
    \end{proof}

	\paragraph{Property (NEG) part one.}
	We will first show that the terminal vertex of an edge in a
	non-exponentially growing stratum $H_i$ is either periodic or has valence at least three.
    This is automatic if $H_i$ is almost periodic,
    so assume that $H_i$ is not almost periodic.
	Let $E_1,\dotsc,E_k$ be the edges of $H_i$.
	As usual, assume $f(E_j) = g_jE_{j+1}u_j$, where indices are taken mod $k$,
    $g_k$ is a vertex group element
	and $u_j$ is a path in $G_{i-1}$.
	Suppose the terminal vertex $v_1$ of $E_1$ is not periodic and has valence two.
	Then $v_1$ has trivial vertex group.
	If $E$ is the other edge incident to $v_1$, 
    then $E$ does not belong to an exponentially growing stratum by \Cref{egvalence}.

	We perform a valence-two homotopy of $\bar E_1$ over $E$
    as  in \cite[Lemma 2.5]{Myself} or \cite[Lemma 1.13]{BestvinaHandel}.
	If $v$ is a vertex of an exponentially growing stratum, $f(v) \ne v_1$,
	so before collapsing the pretrivial forest, properties \hyperlink{EG-i}{(EG-i)} through
	\hyperlink{EG-iii}{(EG-iii)} are preserved.
	The pretrivial forest is inductively constructed as follows:
	any edge which was mapped to $E$ is added,
	then any edge which is mapped into the pretrivial forest is added.
	Thus the argument above shows that no vertex of an exponentially growing stratum
	is incident to any edge in the pretrivial forest,
	so the property of being a relative train track map is preserved.
	After repeating this process finitely many times,
	the terminal vertex of $E_1$ is either periodic or has valence at least three in $\mathcal{G}$.

	Finally, we arrange that $v_1$ is periodic:
	the component of $G_{i-1}$ containing $v_1$ is non-wandering
	(because $f^{k-1}(u_1)$ is contained in it),
	so contains a periodic vertex $w_1$.
	Choose a path $\alpha$ in $G_{i-1}$ from $v_1$ to $w_1$
	and slide $E_1$ along $\alpha$.
	No valence-one vertices are created
	because $v_1$ was assumed to have valence at least three.
	Repeating this process for each edge in
	a non-almost periodic, non-exponentially growing stratum,
	we establish the first part of \hyperlink{NEG}{(NEG)},
	namely the following.

	\begin{enumerate}
		\item[\hypertarget{NEG*}{(NEG*)}]
			The terminal vertex of each edge in a non-almost periodic,
			non-exponentially growing stratum is periodic.
	\end{enumerate}

	\paragraph{Property (Z) part one.}
	Property \hyperlink{Z}{(Z)} has several parts.
	Let $H_i$ be a zero stratum.
	For $H_i$ to be enveloped, let $H_u$ the first irreducible stratum below $H_i$
	and $H_r$ the first irreducible stratum above.
	One condition we need is that no component of $G_r$ is contractible.
	Following \cite[p. 58]{FeighnHandel},
	we postpone that step and merely show here that each component is non-wandering.

	First we arrange that if a filtration element $G_i$ has a wandering component,
	then $H_i$ is a wandering component.
	Suppose that $G_i$ has wandering components.
	Call their union $W$ and their complement $N$.
	If $N$ is not precisely equal to a union of strata,
	the difference is that $N$ contains part but not all of a zero stratum,
	so we may divide this zero stratum to arrange so that $N$ is a union of strata.
	Thus $W$ is a union of zero strata.
	Since $N$ is $f$-invariant, we may push all strata in $W$ higher than all strata in $N$.
	We define a new filtration.
	Strata in $N$ and higher than $G_i$ remain unchanged.
	The strata that make up $W$ will be the components of $W$.
	If $C$ and $C'$ are such components, 
	$C'$ will be higher than $C$
	if $C' \cap f^k(C) = \varnothing$ for all $k \ge 0$.
	We complete this to an ordering on the components of $W$,
	yielding the new filtration.

	Now we work toward showing that zero strata are enveloped by exponentially growing strata.
	Suppose that $K$ is a component of the union of all zero strata in $\mathcal{G}$,
	that $H_i$ is the highest stratum that contains an edge of $K$
	and that $H_u$ is the highest irreducible stratum below $H_i$.
	We aim to show that $K \cap G_u = \varnothing$.
	So assume $K \cap G_u \ne \varnothing$.
	By the previous paragraph, because $H_u$ is irreducible,
	each component of $G_u$ is non-wandering,
	so $K$ meets $G_u$ in a unique component $C$ of $G_u$.
	If each vertex of $K$ has valence at least two in $C \cup K$,
	then each edge of $K$ belongs to a tight path in $K$ with endpoints in $C$,
	and we may close this path up in $C$ to form a tight loop in $K \cup C$.
	But some iterate of $f\colon \mathcal{G} \to \mathcal{G}$
	maps $K\cup C$ into $C$, so this is a contradiction.
	Therefore some vertex $v \in K$ has valence one in $K \cup C$.
    Because some iterate of $f\colon \mathcal{G} \to \mathcal{G}$
    maps $K \cup C$ into $C$,
    if we can show this vertex has valence one in $G$, it must have trivial vertex group.

	This valence-one vertex $v$ is not periodic, 
	so by \hyperlink{NEG*}{(NEG*)},
    $v$ is not an endpoint of an edge in a non-exponentially growing stratum.
    We saw in \Cref{trivialedgegroupsEG2} that \hyperlink{EG-ii}{(EG-ii)}
	for an exponentially growing stratum $H_r$
	is equivalent to the condition that vertices of $H_r \cap G_{r-1}$ 
	contained in non-wandering components of $G_{r-1}$ are periodic.
    Since $K \cup C$ is non-wandering,
	$v$ is not the endpoint of an edge in an exponentially growing stratum above $H_i$.
	By construction, $v$ is also not the endpoint of an edge of another zero stratum.
	Thus $v$ has valence one in $\mathcal{G}$, 
    but we have not produced any valence-one vertices with trivial 
    vertex group so far.
	This contradiction shows that $G_u \cap K = \varnothing$.

	The lowest edges in $K$ are mapped either to another zero stratum or into $G_u$.
	In any case, by connectivity, $K$ is wandering,
	so we can reorganize zero strata so that $K = H_i$.
	Repeating this for each component of the union of all zero strata,
	we have arranged that if $H_i$ is a zero stratum
	and $H_r$ is the first irreducible stratum above $H_i$,
	then $H_i$ is a component of $G_{r-1}$.

	Let $H_r$ be the first irreducible stratum above $H_i$.
	Because $H_r$ is irreducible, no component of $G_r$ is wandering,
	so the component that contains $H_i$ intersects $H_r$.
	No vertex of $H_i$ is periodic,
	so \hyperlink{NEG*}{(NEG*)} implies $H_r$ is exponentially growing,
	and the argument above shows that vertices of $H_i$
	meet only edges of $H_i$ and $H_r$
	and every vertex of $H_i$ has valence at least two in $G_r$.
	This satisfies every part of the definition of the zero stratum $H_i$
	being enveloped by $H_r$,
	except that we have not shown that all components of $G_r$ are non-contractible,
	only that they are non-wandering.

	Note also that $H_i$ is contained in the core of $G_r$:
	one obtains the core of $G_r$ by repeatedly removing from $G_r$
	any edge incident to a valence-one vertex with trivial vertex group.
	Each vertex of $H_i$ has valence at least two in $G_r$,
	and \Cref{egvalence} says that each valence-one vertex of $H_r$ has nontrivial vertex group.

	\paragraph{Tree replacement.}
	The final part of property \hyperlink{Z}{(Z)} we will establish now
	is that every vertex of $H_i$ is contained in $H_r$.
	We do so by Feighn and Handel's method of \emph{tree replacement.}
	Replace $H_i$ with a tree $H'_i$ whose vertex set is exactly $H_i\cap H_r$.
	We may do so for every zero stratum at once,
	(with \emph{a priori} different exponentially growing strata $H_r$, of course)
	and call the resulting graph of groups $\mathcal{G}'$.
	Let $X$ denote the union of all irreducible strata.
	There is a homotopy equivalence $p'\colon \mathcal{G}' \to \mathcal{G}$
	that is the identity on edges in $X$
	and sends each edge in a zero stratum $H'_i$
	to the unique path in $H_i$ with the same endpoints.
	Choose a homotopy inverse $p\colon \mathcal{G} \to \mathcal{G}'$
	that also restricts to the identity on edges in $X$
	and maps each zero stratum $H_i$ to the corresponding tree $H'_i$.
	Define $f'\colon \mathcal{G}' \to \mathcal{G}'$
	by tightening $pfp'\colon \mathcal{G}' \to \mathcal{G}'$.
	The map $f'$ still satisfies \hyperlink{EG-i}{(EG-i)}.
	Because $p_\sharp$ and $p'_\sharp$ 
	send nontrivial paths with endpoints in $X$ to nontrivial paths with endpoints in $X$,
	\hyperlink{EG-ii}{(EG-ii)} is preserved as well.
	Because $\pf(f) = \pf(f')$, 
    \Cref{pfcorollary} implies $f'$ satisfies \hyperlink{EG-iii}{(EG-iii)} as well.
	Nothing we have done so far changes the realization of free factor systems by filtration elements as well.
	We replace $f\colon \mathcal{G} \to \mathcal{G}$ by $f'\colon \mathcal{G}' \to \mathcal{G}'$ in what follows.

	\paragraph{Property (P).}
	Let $\mathcal{F}_1 \sqsubset \dotsb \sqsubset \mathcal{F}_d$ 
	be our chosen nested sequence of $\varphi$-invariant free factor systems.
	If none was specified, instead define this sequence to be 
	the sequence determined by $\mathcal{F}(G_r)$
	as $G_r$ varies among the filtration elements of $\mathcal{G}$.
	We will show that if $H_m$ is an almost periodic forest
    that is not the bottom half of a dihedral pair,
	then there is some $\mathcal{F}_i$ that is not realized by $H_m \cup G_\ell$
	for any filtration element $G_\ell$.
	Assume that this does not hold for some almost periodic forest $H_m$.
	Then for all $i$ satisfying $1 \le i \le d$, there is some filtration element $G_\ell$
	such that $\mathcal{F}(H_m \cup G_\ell) = \mathcal{F}_i$.
	In this case we will collapse an invariant forest containing $H_m$,
	reducing the number of non-exponentially growing strata.
	Iterating this process establishes \hyperlink{P}{(P)}.

	Let $Y$ be the set of all edges in $G\setminus H_m$ eventually mapped into $H_m$
	by some iterate of $f$.
	Each edge of $Y$ is thus contained in a zero stratum.
	We want to arrange that if $\alpha$ is a tight path in a zero stratum with endpoints at vertices
	that is not contained in $Y$,
	then $f_\sharp(\alpha)$ is not contained in $Y \cup H_m$.
    If there \emph{is} such a path $\alpha$
    such that $f_\sharp(\alpha)$ is contained in $Y \cup H_m$,
    let $E_i$ be an edge crossed by $\alpha$ and not contained in $Y$.
	Perform a tree replacement as above,
	removing $E_i$ and adding in an edge with endpoints at the endpoints of $\alpha$.
	By our preliminary form of property \hyperlink{Z}{(Z)},
	if a vertex incident to $E_i$ has valence two, then it is an endpoint of $\alpha$,
	so this process does not create valence-one vertices with trivial vertex group.
	The image of the new edge is contained in $Y\cup H_m$,
	so we add it to $Y$.
	Because there are only finitely many paths in zero strata with endpoints at vertices,
	we need only repeat this process finitely many times if necessary.

	Let $\mathcal{G}'$ be the graph of groups obtained by collapsing each component of $H_m \cup Y$
	to a point, and let $p\colon \mathcal{G} \to \mathcal{G}'$ be the quotient map.
    For each component $C$ of $H_m \cup Y$, choose a vertex $c \in C$.
    If $C$ contains a vertex with nontrivial vertex group, choose that vertex as $c$.
    If an oriented edge $E$ has initial vertex in $C$, let $\gamma_C$
    be the unique tight path containing no vertex group elements from $c$ to $\iota(E)$;
    otherwise let $\gamma_E$ be the trivial path.
	Identify the edges of $G \setminus (Y \cup H_m)$ with the edges of $G'$
	and define $f\colon\mathcal{G}' \to \mathcal{G}'$ on each edge $E$ of the complement
    by tightening $pf(\gamma_{E}E\gamma_{\bar E})$.
	By construction, 
    $f'\colon \mathcal{G}' \to \mathcal{G}'$ is a topological representative of $\varphi$,
	and $f'(E)$ is obtained from $f(E)$ by removing all occurrences of edges in $Y \cup H_m$
    and possibly multiplying by vertex group elements at the ends.
	If $p(H_r)$ is not collapsed to a point, the stratum $H_r$
	and $p(H_r)$ are thus of the same type,
	and we see that $f'$ has one less non-exponentially growing stratum
	and possibly fewer zero strata.
	The previous properties, \hyperlink{NEG*}{(NEG*)} and our preliminary form
	of \hyperlink{Z}{(Z)} are still satisfied.

	Let $H_r$ be an exponentially growing stratum.
    We will verify \hyperlink{EG-ii}{(EG-ii)}.
    Suppose first that $C$ is a wandering component of $p(G_{r-1})$.
    Then there are only finitely many tight paths with endpoints in $C \cap p(H_r)$.
    For each such path $\alpha'$, there is a unique tight path $\alpha$ in $G_{r-1}$
    such that $p(\alpha) = \alpha'$.
    By assumption, $\alpha$ is not contained in $Y \cup H_m$,
    so $f_\sharp(\alpha)$ is not contained in $Y \cup H_m$ by construction,
    and thus we see that $f'_\sharp(\alpha') = (pf)_\sharp(\alpha)$ is nontrivial.
    By \Cref{trivialedgegroupsEG2}, for each non-wandering component $C$ of $p(G_{r-1})$,
    \hyperlink{EG-ii}{(EG-ii)}
	is equivalent to checking that each vertex of $p(H_r)\cap C$
	is periodic.
    Let $v'$ be a vertex of $p(H_r) \cap C$.
	By assumption, there is a vertex $v \in H_r$ such that $p(v) = v'$.
	If $v$ is periodic, we are done.
    Thus we may assume that $v$ is not in $H_m$ (which is almost periodic).
    If $v \in Y$, then the component of $G_{r-1}$ containing $v$ is a zero stratum (and thus wandering),
    contradicting our assumption that $C$ is non-wandering.
    Therefore $v$ is the only preimage of $v'$.
    By the same reasoning, the component of $G_{r-1}$ containing $v$ must be non-wandering,
    so we conclude that $v$ and thus $v'$ is periodic.
	This verifies \hyperlink{EG-ii}{(EG-ii)}.
	It is easy to see that \hyperlink{EG-i}{(EG-i)}
	is still satisfied, and that $\pf(f) = \pf(f')$,
	so \Cref{pfcorollary} implies \hyperlink{EG-iii}{(EG-iii)} is still satisfied,
    so $f'\colon \mathcal{G}' \to \mathcal{G}'$ is a relative train track map.

	It remains to check that our family of free factor systems is still realized.
	Let $\mathcal{F}_j$ be such a free factor system.
	By assumption on $H_m$, there is $G_\ell$ such that
	$\mathcal{F}(G_\ell\cup H_m) = \mathcal{F}_j$.
	Each non-contractible component of $G_\ell \cup H_m$
	is mapped into itself by some iterate of $f$.
	Since $Y$ is eventually mapped into $H_m$,
	some iterate of $f$ induces a bijection between the non-contractible components
	of $G_\ell \cup H_m \cup Y$ and those of $G_\ell \cup H_m$,
	so $p(G_\ell)$ realizes $\mathcal{F}_j$.
	Repeating this process decreases the number of non-exponentially growing strata,
	so eventually property \hyperlink{P}{(P)} is established.
    
    Here is the main consequence of \hyperlink{P}{(P)} that we use.

    \begin{lem}
        \label{propertyPconsequence}
        Suppose that $f\colon \mathcal{G} \to \mathcal{G}$ satisfies \hyperlink{P}{(P)},
        that $H_m$ is an almost periodic forest
        that is not the bottom half of a dihedral pair,
        and that a vertex $v$ has trivial vertex group and valence one in $H_m$.
        Then $v \in G_j$ for some $j < m$.
    \end{lem}

    \begin{proof}
        Suppose $f\colon \mathcal{G} \to \mathcal{G}$ and $H_m$ are as in the statement.
        The restriction of $f$ to $H_m$ acts transitively on the components of $H_m$,
        and either each component is a single edge,
        or $f$ transitively permutes the valence-one vertices of $H_m$
        (necessarily with trivial vertex group in this contrary case, since $H_m$ is a forest).
        Property \hyperlink{P}{(P)} says that there exists a filtration element $G_j$
        such that $\mathcal{F}(G_j) \ne \mathcal{F}(G_\ell \cup H_m)$ for any filtration element $G_\ell$.
        In particular, $\mathcal{F}(G_j) \ne \mathcal{F}(G_j \cup H_m)$,
        so we see that $j < m$.
        Since $H_m$ is a forest, if it is disjoint from $G_j$,
        it contributes nothing to $\mathcal{F}(G_j \cup H_m)$
        (recall that we define free factor systems to be made up of free factors
        with \emph{positive complexity,} ruling out the fundamental group of each component of $H_m$).
        Therefore we conclude that some and hence every edge of $H_m$ is incident to a vertex of $G_j$.
        In fact, some and hence every valence one vertex of $H_m$ with trivial vertex group
        and valence one in $H_m$ must be a vertex of $G_j$.
    \end{proof}

    \paragraph{Property (Z).}
    To complete the proof of property \hyperlink{Z}{(Z)},
    we must show that if $H_r$ is an exponentially growing stratum,
    then all non-wandering components of $G_r$ are in fact non-contractible.
    If $C$ is a non-wandering component,
    the lowest stratum $H_i$ containing an edge of $C$ is either exponentially growing
    or almost periodic.
    If $H_i$ is exponentially growing, \Cref{egvalence}
    shows that each vertex of $H_i$ has valence at least two in $C$ (and thus $H_i$)
    or has nontrivial vertex group, showing that $C$ is noncontractible.
    If $H_i$ is the bottom half of a dihedral pair,
    then $r > i+1$ and $C$ contains the full dihedral pair and is thus non-contractible.
    If $H_i$ is almost periodic but not the bottom half of a dihedral pair,
    then \Cref{propertyPconsequence} shows that $H_i$ is not a forest so $C$ is non-contractible.
    Thus property \hyperlink{Z}{(Z)} follows from the form of \hyperlink{Z}{(Z)} 
    we have already established.

	\paragraph{Property (NEG).}
	Let $E$ be an edge in a non-exponentially growing stratum $H_i$
    which is not almost periodic.
	Let $C$ be the component of $G_{i-1}$ containing the terminal vertex $v$ of $E$;
	it is non-wandering by our work proving \hyperlink{NEG*}{(NEG*)}.
	By the argument in the previous paragraph, if $H_\ell$ is the lowest stratum containing an edge of $C$,
	then $H_\ell$ is either almost periodic or exponentially growing,
	and the argument in the previous paragraph shows that $H_\ell$ is either non-contractible
    or the bottom half of a dihedral pair.
	In fact, in the former case $H_\ell$ is its own core,
    while in the latter $H_\ell \cup H_{\ell+1}$ is its own core.
	In the exponentially growing case this follows
    since \Cref{egvalence} shows that every valence-one vertex of $H_\ell$
	has nontrivial vertex group.
	To see this in the almost periodic case, 
	observe that if some edge of $H_\ell$ is incident to a vertex with trivial vertex group
	and valence one in $H_\ell$, then every edge has this property and $H_\ell$ is a forest,
    in contradiction to \hyperlink{P}{(P)} and \Cref{propertyPconsequence}.
    Write $D$ for $H_\ell$ if $H_\ell$ is its own core,
    or for the dihedral pair $H_\ell \cup H_{\ell+1}$ otherwise.

	Choose a periodic vertex $w$ in $D$ and a path $\gamma$ from $v$ to $w$.
    If  $D$ is a dihedral pair, let $w$ be the center vertex of the dihedral pair.
	Slide $E$ along $\gamma$.
	The result is a relative train track map
	which still realizes our sequence of free factor systems
	and still satisfies \hyperlink{Z}{(Z)}.
	Working up through the filtration repeating this process establishes \hyperlink{NEG}{(NEG)}.
	This time sliding may have introduced valence-one vertices 
    with trivial vertex  group,
	but \hyperlink{NEG}{(NEG)}, \hyperlink{Z}{(Z)} and \Cref{egvalence}
	imply that only valence-one vertices are mapped to the valence-one vertices created.
	We perform valence-one homotopies to remove each of these vertices.
	If property \hyperlink{P}{(P)} is not satisfied, 
	restore it using the process above. Since the number of non-exponentially growing strata
	decreases, this process terminates.

	\paragraph{Property (F).}
	We want to show that the core of each filtration element is a filtration element.
	If $H_\ell$ is a zero stratum, then $\mathcal{F}(G_\ell) = \mathcal{F}(G_{\ell-1})$,
	so assume that $H_\ell$ is irreducible, and thus $G_\ell$ has no contractible components
    unless $H_\ell$ is the bottom half of a dihedral pair.
    So assume $H_\ell$ is not the bottom half of a dihedral pair.
	If a vertex $v$ with trivial vertex group has valence one in $G_\ell$,
	then \Cref{egvalence} and \hyperlink{Z}{(Z)} imply the incident edge $E$ 
	belongs to a non-exponentially growing stratum $H_i$.
	If $H_i$ were almost periodic, it would be a forest, because every edge would be incident
	to a vertex of valence one in $G_\ell$. 
    This contradicts \Cref{propertyPconsequence}.
	This exhausts the possibilities: $v$ must be the initial endpoint of a non-exponentially growing edge
    which is not almost periodic.
	All edges in such a stratum have initial vertex a valence-one vertex of $G_\ell$,
	and no vertex of valence at least two in $G_\ell$ maps to them.
	Thus we may push all such non-exponentially growing strata $H_i$ above $G_\ell \setminus H_i$.
	After repeating this process finitely many times,
	$\mathcal{F}(G_\ell)$ is realized by a filtration element that is its own core.
	Working upwards through the strata, we have arranged that each $\mathcal{F}(G_\ell)$
    is realized by a filtration element that is its own core.
    To complete the proof, rearrange strata by pushing the dihedral pair down the filtration 
    so that if $H_\ell$ is the bottom half of a dihedral pair,
    then $G_{\ell-1}$ is its own core. Then $G_{\ell+1}$ is also its own core and
    \hyperlink{F}{(F)} is satisfied.
\end{proof}

\section{Fixed points at infinity}
\label{boundarysection}
The purpose of this section is to study the action of
automorphisms of $(F,\mathscr{A})$ on the \emph{Bowditch boundary}
(to be defined presently)
of some and hence any Bass--Serre tree associated to a marked graph of groups.
This boundary was previously considered by Guirardel and Horbez in \cite{GuirardelHorbez}.
We define \emph{attractors} and \emph{repellers} for the action of an automorphism
$\Phi$ on the boundary $\partial(F,\mathscr{A})$.
The main result is a proposition from Martino's Thesis \cite{MartinoThesis},
which is an analogue of \cite[Proposition 1.1]{GaboriauJaegerLevittLustig} for free products.

\paragraph{The boundary $\partial(F,\mathscr{A})$.}
We begin this section by collecting some basic information
about our boundary of the free product $F$ 
and the action of $F$ and automorphisms of $F$ on it.

Let $\mathbb{G}$ be the thistle with $n$ prickles and $k$ petals
associated to our standing free product decomposition
\[  F = A_1*\cdots*A_n*F_k, \]
and let $\star$ be the vertex of $\mathbb{G}$ with trivial vertex group.
Let $T$ be the Bass--Serre tree of $\mathbb{G}$,
equipped with a basepoint $\tilde\star$ and fundamental domain,
thus defining an action of $F$ on $T$.

Recall that the Gromov boundary $\partial_\infty T$ of the tree $T$
may be identified with the set of singly infinite tight paths beginning at $\tilde x$.
Given a point $\tilde x \in T$,
a \emph{half-tree} based at $\tilde x$ is a component of $T \setminus \{\tilde x \}$
together with those boundary points $\xi \in \partial_\infty T$
such that the intersection of the tight path $\xi$ with the component
contains infinitely many edges.

The \emph{observer's topology} on $T \cup \partial_\infty T$
is generated by the set of half-trees, which form a sub-basis for the topology.
For more on the observer's topology, see \cite{CoulboisHillionLustig,GuirardelHorbez,Knopf}.
In the case where the groups $A_i$ are countable,
the observer's topology on $T \cup \partial_\infty T$ is second countable and compact.
It is easy to see that the observer's topology is Hausdorff.
A sequence of points $\{\tilde x_n\}$ in $T \cup \partial_\infty T$
converges to a point $\tilde x$ in the observer's topology
if for every $\tilde y \ne \tilde x$ in $T$,
the points $\tilde x_n$ eventually belong to the same half-tree based at $\tilde y$ as $\tilde x$.

The \emph{Bowditch boundary} $\partial T$ of $T$ is,
as a set,
the union of the Gromov boundary $\partial_\infty T$ of $T$
together with the set $V_\infty(T)$ of vertices of $T$ with infinite stabilizer.
If $\tilde x$ is a vertex with finite valence or a point in the interior of an edge of $T$,
it is easy to find a basic open neighborhood of $\tilde x$
(a finite intersection of half-trees)
containing only vertices of finite valence or points in the interior of an edge.
Thus $\partial T$ is closed in $T \cup \partial_\infty T$, so it is compact.
It is not hard to argue that it is a Cantor set.
The subspace $\partial_\infty T$ inherits the usual (visual) topology;
if we fix the basepoint $\tilde\star$,
a sequence $\{\tilde x_n\}$ in $T \cup \partial_\infty T$ converges to $\xi \in \partial_\infty T$
if and only if for every finite subpath $\tilde\gamma$
of the ray $\tilde R_{\tilde\star,\xi}$ determining $\xi$,
the tight path or ray $\tilde R_{\tilde\star,\tilde x_n}$ contains $\tilde\gamma$
for all large $n$.

Let $(\mathcal{G},m)$ be a marked graph of groups,
let $p = m(\star)$,
and let $\Gamma$ be Bass--Serre tree for $\mathcal{G}$.
We always assume a basepoint $\tilde p \in \Gamma$ has  been chosen,
thus defining an action of $F$ on $\Gamma$.
As in \Cref{basicssection} or \cite[Proposition 1.2]{Myself},
there is a lift $\tilde m\colon (T,\tilde\star) \to (\Gamma,\tilde p)$
to the Bass--Serre trees.
This lift is an $F$-equivariant quasi-isometry,
and Guirardel and Horbez prove in \cite[Lemma 2.2]{GuirardelHorbez}
that it has a unique continuous extension to a homeomorphism $\hat m\colon \partial T \to \partial\Gamma$.
They prove that the homeomorphism $\hat m$
does not depend on the $F$-equivariant map $\tilde m$
and furthermore that the map $\hat m$ sends $\partial_\infty T$ to $\partial_\infty\Gamma$
and $V_\infty(T)$ to $V_\infty(\Gamma)$.
This allows us to identify the Bowditch boundaries of any marked graph of groups with $\partial T$.
We shall let $\partial T$ play the role for free products 
that the Gromov boundary of the free group plays for free groups,
denoting it $\partial(F,\mathscr{A})$.
Let us denote the subspace $\partial_\infty T$ as $\partial_\infty(F,\mathscr{A})$
and the subspace $V_\infty(T)$ as $V_\infty(F,\mathscr{A})$.

There is also a convenient algebraic interpretation of $\partial(F,\mathscr{A})$
analogous to that in \cite{Cooper} and \cite{Martino}.
To wit, the orbit map $g \mapsto g.\tilde\star$
yields a well-defined compact topology on $F \cup \partial(F,\mathscr{A})$
that we now describe.
A choice of fundamental domain for $T$ containing $\tilde\star$
corresponds to a choice of orientation for each edge in $\mathbb{G}$
that forms a loop based at $\star$.
Choose the $A_i$ in their conjugacy classes so that elements of each $A_i$
fix a vertex in our fundamental domain,
and choose the free basis $S$ for $F_k$ whose elements correspond to loops
in $\mathbb{G}$ that traverse a single loop edge once in the positive orientation.
A \emph{letter} is an element of the set
$\bigcup_{i=1}^n(A_i\setminus\{1\})\cup S \cup S^{-1}$.
A \emph{word} is a finite or infinite sequence of letters of the form
\[  w = x_1x_2x_3\cdots. \]
A word is \emph{reduced} if adjacent letters neither cancel
nor coalesce into a single letter.
Each element $g$ of $F$ can be written uniquely as a reduced word $w$.
Corresponding to each finite reduced word $w$ there is a tight path in $T$
from $\tilde\star$ to another lift of $\star$;
the path projects to the element of $F = \pi_1(\mathbb{G},\star)$
corresponding to $w$.
(This correspondence is perfect because $\star$ has trivial vertex group.
In general there would be more group elements than paths.)
Points in $\partial_\infty(F,\mathscr{A})$ correspond to infinite reduced words,
while points in $V_\infty(F,\mathscr{A})$ correspond to cosets $gA$,
where $A \in \{A_1,\ldots,A_n\}$ is infinite.
There is a unique reduced word $w$ representing a representative of the coset $gA$
such that the last letter of $w$ does not belong to $A$.

A \emph{prefix} of a reduced word $\xi$ is a finite word $w$
such that there exists $\xi'$ so that $\xi = w\xi'$ is reduced as written.
A sequence of points in $F \cup \partial(F,\mathscr{A})$
converges to a point $\xi \in \partial_\infty(F,\mathscr{A})$
if and only if for each prefix of $\xi$ 
the associated sequence of words (which may be finite or infinite)
eventually shares that prefix.

A sequence of points in $F \cup \partial(F,\mathscr{A})$
converges to a point $wA \in V_\infty(F,\mathscr{A})$
if and only if each associated word eventually has a prefix of the form $wg$ for $g$ a letter of $A$,
but for each such $g$, the prefix $wg$ appears only finitely many times.

\paragraph{}
Each $c \in F$ acts on $\Gamma$ as an automorphism $T_c$
of the natural projection $\Gamma \to \mathcal{G}$
and induces a homeomorphism 
$\hat T_c\colon \partial(F,\mathscr{A}) \to \partial(F,\mathscr{A})$.
If $c$ is \emph{peripheral} (i.e.~conjugate into some $A_i$)
but nontrivial,
then $\hat T_c$ acts with a single fixed point on $\partial(F,\mathscr{A})$
if it is conjugate into an infinite $A_i$
and without fixed point if it is conjugate into a finite $A_i$.
If $c$ is nonperipheral,
then $\hat T_c$ has two fixed points in $\partial(F,\mathscr{A})$,
a sink $\hat T_c^+$ and a source $\hat T_c^-$.
The line (proper, linear embedding of $\mathbb{R}$)
in $\Gamma$ whose endpoints in $\partial_\infty(F,\mathscr{A})$
are $\hat T_c^+$ and $\hat T_c^-$ is the \emph{axis} of $T_c$;
denote it by $A_c$.
The image of $A_c$ in $\mathcal{G}$
is a \emph{circuit} (a tight path in $\mathcal{G}$ that forms a loop
and remains tight however it is cut into a path)
that corresponds under the marking to the conjugacy class of $c$
(or of a root of $c$).

Recall from \Cref{basicssection} that given a map
$f\colon \mathcal{G} \to \mathcal{G}$
representing an outer automorphism $\varphi \in \out(F,\mathscr{A})$,
choosing a path $\sigma$ in $\mathcal{G}$ from the basepoint $p$ to $f(p)$
defines both an automorphism $\Phi\colon (F,\mathscr{A}) \to (F,\mathscr{A})$
and a lift $\tilde f$ of $f$ to the Bass--Serre tree $\Gamma$
which is $\Phi$-twisted equivariant.
This defines a bijection between the set of lifts of 
$f\colon \mathcal{G} \to \mathcal{G}$
to $\Gamma$ and the set of automorphisms 
$\Phi\colon (F,\mathscr{A}) \to (F,\mathscr{A})$
representing $\varphi$.
Following Feighn--Handel \cite{FeighnHandel},
we say that $\tilde f$ \emph{corresponds to} or \emph{is determined by} $\Phi$
and vice versa.

Under the identification of $\partial\Gamma$ with $\partial(F,\mathscr{A})$,
a lift $\tilde f$ also determines a homeomorphism $\hat f$ of $\partial(F,\mathscr{A})$.
An automorphism $\Phi\colon (F,\mathscr{A}) \to (F,\mathscr{A})$
determines a bijection of the set of infinite words and cosets of infinite $A_i$;
it is not hard to see that this bijection is open,
and thus $\Phi$ determines a homeomorphism 
$\hat\Phi\colon \partial(F,\mathscr{A}) \to \partial(F,\mathscr{A})$.

\begin{lem}
    Assume notation as in the previous paragraph.
    We have $\hat f = \hat \Phi$ if and only if $\tilde f$ corresponds to $\Phi$.
\end{lem}

\begin{proof}
    Let $m'\colon \mathcal{G} \to \mathbb{G}$
    be a homotopy inverse for $m$,
    and let $\tilde m'\colon \Gamma \to T$ be the lift such that
    $\tilde m' \tilde m$ is equivariantly homotopic to the identity;
    thus $\hat m'\hat m$ is the identity of $\partial(F,\mathscr{A})$.
    The action of $\hat f$ on $\partial(F,\mathscr{A})$
    is the extension of $\tilde m'\tilde f\tilde m\colon T \to T$
    to the Bowditch boundary of $T$.
    Up to equivariant homotopy,
    we may assume that $\tilde m'\tilde f\tilde m$ fixes $\tilde \star$.
    If $\tilde f$ corresponds to $\Phi$,
    then $\tilde m'\tilde f\tilde m(g.\tilde \star) = \Phi(g).\tilde\star$,
    so it follows that $\hat f(\xi) = \hat\Phi(\xi)$ for all $\xi \in \partial(F,\mathscr{A})$.
    The same argument shows that if instead $\tilde f$ corresponds to $\Phi' \ne \Phi$,
    then $\hat f \ne \hat\Phi$.
\end{proof}

\begin{lem}
    \label{boundarybasics}
    Assume that $\tilde f\colon \Gamma \to \Gamma$ corresponds to
    $\Phi\colon (F,\mathscr{A}) \to (F,\mathscr{A})$.
    The following are equivalent.
    \begin{enumerate}
        \item $c \in \fix(\Phi)$.
        \item $T_c$ commutes with $\tilde f$.
        \item $\hat T_c$ commutes with $\hat f$.
    \end{enumerate}
    The above also imply the following for all $c \in F$
    and any automorphism $\Phi\colon (F,\mathscr{A}) \to (F,\mathscr{A})$.
    If $c$ is not peripheral and $\fix(\hat f)$ is nonempty,
    then the following are also equivalent to  the  above.
    \begin{enumerate}[resume]
        \item $\fix(\hat T_c) \subset \fix(\hat f)$.
        \item $\fix(\hat f)$ is $\hat T_c$-invariant.
    \end{enumerate}
\end{lem}

\begin{proof}
    Assume item 1. 
    For all $\tilde x \in \Gamma$,
    by $\Phi$-twisted equivariance, we have
    \[  \tilde fT_c(\tilde x) = T_{\Phi(c)}\tilde f(\tilde x) = T_c\tilde f(\tilde x),\]
    which proves item 2.
    Item 3 follows from item 2 by extending the equal  maps $\tilde fT_c$ 
    and $T_c\tilde f$ to $\partial(F,\mathscr{A})$.
    Assume item 3.
    We have that $\hat f\hat T_c$ and $\hat T_c\hat f$ are the extensions of 
    $\tilde fT_c$ and $T_c\tilde  f$ to $\partial(F,\mathscr{A})$.
    We have that $\Phi$-twisted equivariance and the effectiveness of the action of $F$
    on $\partial(F,\mathscr{A})$ imply $\Phi(c) = c$, proving item 1.
    
    Assume item 3, and let $\xi \in \fix(\hat f)$.
    We have 
    \[  \hat f(\hat T_c(\xi))  = \hat T_c\hat f(\xi) = \hat T_c(\xi), \]
    which  proves item 5.
    Item 4 is trivially true if $c$ is peripheral and conjugate into a finite $A_i$.
    If $c$ is conjugate into an infinite $A_i$,
    let $\xi$ be the fixed point of $\hat T_c$.
    We have
    \[  \hat f(\xi) = \hat f(\hat T_c(\xi)) = \hat T_c(\hat f(\xi)), \]
    so by the uniqueness of the fixed point $\xi$, we conclude $f(\xi) = \xi$.

    Assume item 3 and suppose $c$ is nonperipheral.
    We have
    \[  \hat T_c\hat f(\hat T_c^+) = \hat f\hat T_c(\hat T_c^+) = \hat f(\hat T_c^+) \]
    and similarly for $\hat T_c^-$.
    Moreover, the above equation implies that $\hat f(\hat T_c^+)$
    is a sink for $\hat T_c$,
    so  we conclude that $\fix(\hat T_c) \subset \fix(\hat f)$.
    Assuming  item 5, note that any closed nonempty $\hat T_c$-invariant set
    contains both of the fixed points of $\hat T_c$.
    If  we assume $\fix(\hat  f)$ is nonempty, item 5 implies item 4.
    So assume item 4.
    Ghen we have
    \[  \hat T_{\Phi(c)}(\hat T_c^+) = \hat T_{\Phi(c)} \hat f(\hat T_c^+)
    = \hat f(\hat T_c(\hat T_c^+))  = \hat f(\hat T_c^+) = \hat T_c^+\]
    and similarly $\hat T_{\Phi(c)}(\hat T_c^-) = \hat T_c^-$,
    so $\Phi(c)$  and $c$ share an  axis and thus a common root.
    That is, $c = a^k$ and $\Phi(c) = a^\ell$ for positive $k$ and $\ell$
    and some nonperipheral, root-free $a \in F$.
    Then $\Phi(c) = \Phi(a)^k = a^\ell$.
    Since nonperipheral elements have unique roots,
    this implies that $\Phi(a) = a^j$ for some positive $j$.
    But then $\Phi^{-1}(a)^j = a$,
    so since $a$ is root-free, we conclude $\Phi(a) = a$
    and thus $k = \ell$ and $c \in \fix(\Phi)$, so item 4 implies item 1.
\end{proof}

Given a subgroup $H$ of $F$,
the action of $H$ on $T$ induces a decomposition of $H$ as a free product.
We say that $H$ has \emph{finite (Kurosh subgroup) rank}
if this free product decomposition is of the form
$H = B_1*\cdots*B_p*F_{\ell}$,
or equivalently if the action of $H$ on its minimal subtree $T_H$ in $T$ is cocompact.
In this situation we define the \emph{(Kurosh subgroup) rank} of $H$ to be $p+\ell$.

\begin{lem}
    The inclusion $T_H \to T$ defines a closed embedding of $\partial T_H$ into $\partial T$
    that is well-defined independent of $T$,
    so we denote $\partial T_H$ as $\partial(H,\mathscr{A}|_{H})$.
\end{lem}

\begin{proof}
    The $H$-minimal subtree $T_H$ of $T$ is convex in $T$,
    so there is a well-defined inclusion of Gromov boundaries
    $\partial_\infty T_H \to \partial_\infty T$.
    If a vertex of $T_H$ has infinite stabilizer in $H$,
    it has infinite stabilizer in $F$,
    so there is also a well-defined inclusion $V_\infty(T_H) \to V_\infty(T)$.
    The argument in the proof of \cite[Lemma 2.1]{GuirardelHorbez} and \cite[Lemma 2.2]{GuirardelHorbez}
    implies that the injective set map $\partial T_H \to \partial T$ is continuous.
    Since $\partial T_H$ is compact, the map is an embedding.
    To show it is closed, it therefore suffices to show that $\partial T_H$ is closed in $\partial T$.

    Suppose that $\{\xi_n\}$ is a sequence in $\partial T$
    converging to some point $\xi \notin T_H$.
    There is a unique closest point $\tilde x$
    of $T_H$ closest to $\xi$ in the sense that $\tilde x$ is the only point of $T_H$
    on the ray $\tilde R_{\tilde x,\xi}$.
    Assume first that $\tilde x \ne \xi$.
    Then because $\xi_n$ is eventually in the same half-tree based at $\tilde x$ as $\xi$,
    we see that $\xi_n \notin \partial T_H$ for $n$ large.
    If $\tilde x = \xi$, then $\tilde x$ has finite valence in $T_H$ but infinite valence in $T$.
    Let $\tilde  e_1,\ldots,\tilde e_m$ be the edges $T_H$ with initial vertex $\tilde x$
    and let $\tilde v_1,\ldots,\tilde v_m$  be the corresponding terminal vertices.
    Again we see that because $\xi_n$ must be in the same half-tree based at $\tilde v_i$ 
    as $\xi$ for $n$ large,
    we see that $\xi_n \notin \partial T_H$ for $n$ large.

    Bounded cancellation implies that if $(\mathcal{G},m)$ is a marked graph of groups,
    then $\tilde m(T_H)$ is in a bounded neighborhood of $\Gamma_H$,
    the $H$-minimal subtree of $\Gamma$.
    Thus the inclusion of Gromov boundaries is well-defined independent of $T$;
    it is clear that the same is true for the inclusion of infinite-stabilizer points,
    so the embedding is independent of $T$.
\end{proof}

A point $\xi \in \partial_\infty(F,\mathscr{A})$ is an \emph{attractor}
for $\hat\Phi$ if the set $U$ of all points $\zeta$
such that the sequence $\{\Phi^n(\zeta)\}$ converges to $\xi$
contains an open neighborhood of $\xi$.
If $\zeta$ is an attractor for $\hat\Phi^{-1}$,
then we say that $\zeta$ is a \emph{repeller} for $\hat\Phi$.

\begin{prop}[cf.~Lemma 2.3 of \cite{FeighnHandel},
    Proposition 1.1 of~\cite{GaboriauJaegerLevittLustig},
    Proposition 5.1.14 of~\cite{MartinoThesis}]\label{GJLLprop}
    Assume that $\tilde f\colon \Gamma \to \Gamma$ corresponds to
    $\Phi\colon (F,\mathscr{A}) \to (F,\mathscr{A})$
    and that $\fix(\hat\Phi)$ contains at least three points 
    in $\partial_\infty(F,\mathscr{A})$.
    Denote $\fix(\Phi) = \mathbb{F}$ and $\mathbb{T} = \{T_c : c \in \mathbb{F}\}$.
    Then
    \begin{enumerate}
        \item If $\partial(\mathbb{F},\mathscr{A}|_\mathbb{F}) 
            \cap \partial_\infty(F,\mathscr{A})$
            is nonempty, it
            is naturally identified with the closure 
            of \[\{\hat T^{\pm}_c : T_c \in \mathbb{T}\}\]
            in $\partial(F,\mathscr{A})$.
            None of these points is isolated in $\fix(\hat\Phi)$.
        \item Every point in 
            $(\fix(\hat\Phi)\setminus\partial(\mathbb{F},\mathscr{A}|_\mathbb{F}))
            \cap \partial_\infty(F,\mathscr{A})$
            is isolated and is either an attractor or a repeller for $\hat \Phi$.
    \end{enumerate}
\end{prop}

We note that in fact item 2 in the proposition holds
without the assumption on the number of fixed points of $\hat\Phi$.

The  proof of this proposition is admittedly more complicated
than the corresponding statement for automorphisms of free groups.
The reason is that for automorphisms of free groups,
if a boundary point $\xi$ is an attractor for an automorphism $\Phi$,
then it is \emph{superlinearly} attracting,
in the sense that if $\xi_i$ is the $i$th prefix of $\xi$
and $k(i)$ is the part of the length of $\Phi(\xi_i)$ which is common with $\xi$,
then $k(i) - i$ goes to infinity with $i$.
Superlinear growth dominates bounded cancellation,
which allows one to more easily argue that one has an attractor.
For automorphisms of free products, there is no such guarantee of superlinear  growth,
and in fact there exist attractors for which $k(i) - i$ remains bounded.
Arguing that one has these kinds of attractors or repellers
when the given  fixed point is not in the boundary of the fixed subgroup 
is more delicate.

Here is an example, due to Martino:
let $A$ be a countably infinite group and $\Theta\colon A \to A$ an automorphism of infinite order.
Assume further that $\Theta$ does not act periodically on any nontrivial element of $A$.
Let $h\colon A \to A'$ be an isomorphism
and $F = A*A'*F_1$; let $t$ be a generator for $F_1$.
Consider the automorphism $\Phi\colon F \to F$
defined as $\Phi(a) = h(\Theta(a))$ for $a \in A$,
$\Phi(a') = h^{-1}(a')$ for $a' \in A'$,
and finally as $\Phi(t) = tg$ for some fixed nontrivial element $g \in A$.
Then $\fix(\Phi)$ is trivial and
\[  \xi = tgh(\Theta(g))\Theta(g)h(\Theta^2(g))\Theta^2(g)\ldots \]
is a linear attractor, in the sense that $k(i) = i+1$.

\begin{proof}
    If $\partial(\mathbb{F},\mathscr{A}|_{\mathbb{F}})
    \cap\partial_\infty(F,\mathscr{A})$ 
    contains at least three points
    then no point of $\partial(\mathbb{F},\mathscr{A}|_\mathbb{F})$
    is isolated in itself, so it is not isolated in $\fix(\hat\Phi)$.
    If instead $\partial(\mathbb{F},\mathscr{A}|_\mathbb{F})
    \cap\partial_\infty(F,\mathscr{A})$
    consists of two points (i.e.~$\mathbb{F}$ is $F_1$ or $C_2*C_2$),
    then $\partial(\mathbb{F},\mathscr{A}|_\mathbb{F}) = \{T_c^{\pm}\}$
    for some $c \in \fix(\Phi)$.
    There is a point $\xi \in \fix(\hat\Phi) = \fix(\hat f)$ 
    contained in $\partial_\infty(F,\mathscr{A})$
    different from $\hat T_c^{\pm}$.
    Since $\hat T_c$ acts with source--sink dynamics on $\partial(F,\mathscr{A})$
    and $\fix(\hat f)$ is $\hat T_c$-invariant by \Cref{boundarybasics},
    we see that $\lim_{n\to\infty}\hat T_c^n(\xi) = \hat T_c^+$
    and similarly $\lim_{n\to\infty} \hat T_c^{-n}(\xi) = \hat T_c^{-}$,
    so neither of these points are isolated in $\fix(\hat\Phi)$.

    Now consider any fixed point $\xi\in\partial_\infty(F,\mathscr{A})$.
    We follow the outline of \cite[Proposition 1.1]{GaboriauJaegerLevittLustig}
    and \cite[Proposition 5.1.14]{MartinoThesis}.
    Think of $\xi$ as an infinite word
    \[  \xi = x_1x_2x_3\cdots\]
    and write $\xi_i = x_1\cdots x_i$.
    Consider the words $w_i = \xi_i^{-1}\Phi(\xi_i)$.

    A few observations are in order.
    First, since $\Phi(\xi) = \xi$, we have $\lim_{i\to\infty}\Phi(\xi_i) = \xi$,
    so we can write each $\Phi(\xi_i)$ as $\xi_{k(i)}z_i$,
    where  $k(i) \to \infty$ with $i$.
    By bounded cancellation,
    which for words says that if $uv$ is reduced as written,
    then at most $2B$ letters of $\Phi(u)\Phi(v)$ cancel to form  the reduced word
    for $\Phi(uv)$, the length of the word $z_i$ is bounded by $B$ independently of $i$.

    \paragraph{Superlinear attractors and repellers.}
    Suppose that the length of the words $w_i$ grows without bound.
    Then the equation $\Phi(\xi_i) = \xi_{k(i)}z_i$ and our bound on the length
    of $z_i$ implies that $|k(i) - i|$ goes to infinity
    and in fact either $k(i) - i$ goes to infinity or $i - k(i)$ goes to infinity.

    Suppose first that $k(i) - i$ goes  to infinity.
    Then there exists $j$ such that if $i \le j$,
    then $k(i) > i + B$.
    Consider the basic open neighborhood $U_i$ of $\xi$ given by
    \[  U_i = \{\zeta : \zeta = \xi_i\zeta' \text{ is reduced as written}\}. \]
    (We do include boundary points of the form $wA$ in $U_i$.)
    Since $k(j) > j  + B$, bounded cancellation implies that 
    $\Phi(U_j) \subset U_{k(j) - B} \subset U_j$.
    Since $k(k(j))  > k(j) + B$,
    we have some sequence $k_n$ such that $k_n \to \infty$ with $n$ such that
    \[  \bigcap_{n=1}^\infty\Phi^n(U_j) \subset \lim_{n\to\infty} U_{k_n} = \{\xi\}. \]
    This shows that $\xi$ is  an attractor for the action $\Phi$.

    Suppose instead that $k(i) - i$ goes to $-\infty$.
    Since $\xi$ is fixed by $\Phi$,
    it is fixed by $\Phi^{-1}$ and as above we may write 
    $\Phi^{-1}(\xi_i) = \xi_{\bar k(i)}\bar z_i$.
    The equation $\xi_i = \Phi^{-1}(\xi_{k(i)})\Phi^{-1}(z_i)$
    implies that $\bar k(k(i)) - i$ is bounded independent of $i$.
    We have
    \[  \bar k(k(i)) - k(i) = (\bar k(k(i)) - i) - (k(i) - i) \]
    and the latter goes to infinity.
    Since $k(i)$ goes to infinity and $k(i+1)  - k(i)$ and $\bar k(i+1) - \bar k(i)$
    are both bounded,
    this implies $\bar k(i) - i$ goes to infinity.
    The argument in the previous paragraph applies to show that in this case
    $\xi$ is a repeller for the action of $\Phi$.

    \paragraph{In the boundary of the fixed subgroup.}
    So suppose that the length of $w_i$ is uniformly bounded,
    and observe that if $w_i = w_p$,
    then $\xi_p\xi_i^{-1}$ is fixed by $\Phi$.
    If we write $v_i = \xi_i^{-1}\Phi^{-1}(\xi_i)$,
    note that the same argument applies:
    if $v_i = v_p$, then $\xi_p\xi_i^{-1}$ is fixed by $\Phi$.
    If infinitely many of the words $w_i$ or infinitely many of the words $v_i$
    are drawn from a set of only finitely many letters,
    then $v_i$ or $w_i$ takes on some value infinitely often.
    If this happens, note that for fixed $i$,
    we have
    \[  \lim_{p\to\infty}\xi_p \xi_i^{-1} = \xi,\]
    so $\xi \in \partial(\mathbb{F},\mathscr{A}|_\mathbb{F})$.

    We claim that this latter event,
    infinitely many words drawn from finitely many letters,
    actually occurs if thre exists $j$ such that for all $i \ge j$,
    we have $k(i) < i$ and $\bar k(i) < i$.
    Since $k(i)$ goes to infinity, we may increase $j$ so that if $i \ge j$,
    then additionally $\bar k(k(i)) < k(i)$.

    \begin{claim}[cf. Lemma 5.1.13 of \cite{MartinoThesis}]
        \label{finitelettersclaim}
        There is a finite set of letters $L$ with the property that
        if $k(i-1) < k(i) < i$ and $\bar k(k(i)) < k(i)$,
        then either every letter of $v_{k(i)}$ belongs to $L$
        or every letter of $w_i$ belongs to $L$.
    \end{claim}

    \begin{proof}[Proof of \Cref{finitelettersclaim}.]
        Since $\Phi \colon (F,\mathscr{A}) \to (F,\mathscr{A})$
        permutes the conjugacy classes in $\mathscr{A}$, we have that
        $\Phi$ sends each $A_i$ to a conjugate of some $A_j$,
        say $g_iA_j g_i^{-1}$.
        Let $L'$ be the union of the set of letters 
        occuring in the normal form for each $g_i$
        and occuring in each $\Phi(s_i^{\pm1})$ 
        for $s_i$ in our fixed free basis $S$ for $F_k$
        together with their inverses.
        (It follows that each $s_i^{\pm 1}$ occurs in $L'$.)

        We have \[v_{k(i)} = \xi^{-1}_{k(i)}\Phi^{-1}(\xi_{k(i)}) 
            = \xi^{-1}_{k(i)}\xi_{\bar k(k(i))}\bar z_{k(i)}
        = x_{k(i)}^{-1}\dotsb x^{-1}_{\bar k(k(i)) +1}\bar z_{k(i)} \]
        since $\bar k(k(i)) < k(i)$.
        Since $\xi_{k(i)}v_{k(i)} = \Phi^{-1}(\xi_{k(i)})$,
        we must have that the first letter of $v_{k(i)}$
        belongs to the same free factor as $x_{k(i)}$,
        therefore $x_i^{-1}\dotsb x_{k(i)+ i}^{-1}v_{k(i)}$
        is reduced as written.
        Note that
        \[  \xi_i = \Phi^{-1}(\Phi(\xi_i)) = \Phi^{-1}(\xi_{k(i)})\Phi^{-1}(z_i) \]
        so \[\Phi^{-1}(z_i^{-1}) = \xi_i^{-1}\Phi^{-1}(\xi_{k(i)})
        = x_i^{-1}\dotsb x_{k(i) + 1}^{-1}v_{k(i)}.\]

        Notice that $\Phi(x_i) = z_{i-1}^{-1}x_{k(i-1)+ 1}\dotsb x_{k(i)}z_i$.
        Since we assume that $k(i-1) < k(i)$,
        and since by assumption the first letter of $z^{-1}_{i-1}$
        cannot be $x_{k(i-1)+1}$,
        we see that $z_i$ is a terminal subword of $\Phi(x_i)$.
        Since $x_i$ is a single letter,
        at most one letter of $z_i$ and thus $z_i^{-1}$ does not belong to $L'$.
        
        Let $L''$ be the union of the set of letters
        constructed for $\Phi^{-1}$ the way $L'$ was constructed for $\Phi$.
        Let $L'''$ be the union of $L'$ and $L''$.
        Thus every letter of $\Phi^{-1}(z_i^{-1})$ 
        is a product of at most three letters of $L'''$ except at most one
        (consider what happens when a product of elements of $L'''$ is reduced).
        Let $L^{(4)}$ be the set of nonidentity products of at most three elements of $L'''$.
        Therefore either every letter of $v_{k(i)}$ belongs to $L^{(4)}$
        or each of $x_{k(i)+1},\dotsc,x_i$ belongs to $L^{(4)}$.
        Now
        \[w_i = \xi_i^{-1}\Phi(\xi_i) = \xi^{-1}_i\xi_{k(i)}z_i 
        = x_i^{-1}\dotsb x_{k(i)+1}^{-1} z_i \]
        and we see that each letter of $w_i$ is a product
        of at most two elements
        of $L^{(4)}$, so enlarging once more to a set $L$ proves the claim
        (notice that the construction of $L$ is independent of $i$).
    \end{proof}

    Given \Cref{finitelettersclaim},
    notice that since $k(i)$ goes to infinity with $i$,
    we have that $k(i-1) < k(i)$ occurs infinitely often,
    so either infinitely many $v_{k(i)}$
    or infinitely many $w_i$
    are written with a finite set of letters $L$.
    In either case, 
    we see that $\xi \in \partial(\mathbb{F},\mathscr{A}|_{\mathbb{F}})$.

    \paragraph{Linear attractors}
    So we shall suppose to the contrary:
    the word lengths of $w_i$ and $v_i$ remain uniformly bounded,
    but given any finite set of letters $A$,
    there exists $j_0$ so large that if $i > j_0$,
    then $w_i$ and $v_i$ each contain a letter not in $A$.
    \Cref{finitelettersclaim} shows that we have that either
    $k(i) \ge i$ or $\bar k(i) \ge i$ occurs infinitely often.

    We will show that if $k(i) \ge i$ occurs infinitely often,
    then $\xi$ is an attractor for $\Phi$.
    If instead $\bar k(i) \ge i$ occurs infinitely often,
    then the same argument will show that $\xi$ is a repeller for $\Phi$.
    (Of course only one can occur.)
    We do this by considering a subsequence $i_r$ where $k(i_r)$ is strictly increasing
    and showing that if $k(i_r) \ge i_r$ for some $r$, then $k(i_{r+1}) \ge i_{r+1}$.
    We show that $i_{r+1} - i_r$ is bounded by some constant $C$,
    and in fact that $\Phi^{2C+2}(U_{i_r}) \subset U_{i_{r+1}}$.
    It follows that the open neighborhood $U^r$ of $\xi$ given by
    \[  U^r = U_{i_r} \cup \Phi(U_{i_r}) \cup \cdots \cup \Phi^{2C+1}(U_{i_r})\]
    satisfies $\Phi(U^r) \subset U^r$.
    A boundary point
    $\zeta$ is in $\bigcap_{n=1}^\infty \Phi^n(U^r)$ if for all $n$,
    there exists a nonnegative $c_n \le 2C+1$
    such that $\zeta \in \Phi^{n + c_n}(U_{i_r})$.
    Write $n + c_n = 2mC + c'_n$ for some nonnegative $c'_n \le 2C+1$.
    Then $\zeta \in \Phi^{c'_n}(U_{i_{r+m}})$.
    Thus
    \[  \bigcap_{n=1}^\infty \Phi^n(U^r) = \lim_{m\to\infty} U^{r+m} = \{\xi\}\]
    and we have shown that $\xi$ is an attractor for $\Phi$.
    
    Here is the definition of the subsequence $i_r$.
    Define $i_1$ to be the first $i$ such that $k(i) > 0$
    and choose $i_r$ inductively such that $i_r$ is the least integer satisfying
    $k(i_r) > k(i_{r-1})$.
    
    We claim that $i_{r+1} - i_r$ is bounded by a constant $C$ chosen as follows.
    Let $C$ be such that
    if the length of a reduced word $w$ is at least $C$,
    then the length of $\Phi(w)$ is greater than $3B$,
    where $B$ is the bounded cancellation constant.
    To see that $C$ exists, note that the length of $\Phi^{-1}(w)$
    is bounded by some constant multiple of the length of $w$,
    namely the maximum of the length of the image of a single letter under $\Phi^{-1}$.
    
    Suppose $i_{r+1} - i_r > C$.
    Then $\Phi(\xi_{i_{r+1} -1}) = \Phi(\xi_{i_r})\Phi(x_{i_r+1}\cdots x_{i_{r+1}-1})$
    has length greater than $k(i_r) - 2B + 3B = k(i_r) + B$ by bounded cancellation.
    But by assumption $k(i_{r+1} - 1) < k(i_r)$,
    and hence the length of $z_{i_{r+1}-1}$ is greater than $B$, a contradiction.

    Therefore to complete the proof, we show that if
    $k(i_r) \ge i_r$ for sufficiently large $r$, then $k(i_{r+1}) \ge i_{r+1}$
    and $\Phi^{2C+2}(U_{i_r}) \subset U_{i_{r+1}}$.
    To do this, we need to choose a finite set of letters to avoid.
    We now turn to defining that finite set of letters.

    We begin by recalling a result of Collins and Turner
    restated (allowing infinite words) by Martino.
    \begin{lem}[Proposition 2.4 of \cite{CollinsTurnerFixed},
        Lemma 5.1.3 of \cite{MartinoThesis}]
        \label{CollinsTurnerLemma}
        There exists a constant $\ell$ with the following property.
        Let $L'$ be the finite set of letters from the beginning
        of \Cref{finitelettersclaim}.
        Given one of our free factors $A_i$,
        we may write $\Phi(A_i) = g_iA_jg_i^{-1}$.
        Given a possibly infinite word $Y$,
        suppose $\Phi(Y) = g_ixY'$ is reduced as written,
        where $x \in A_j$.
        Then either
        \begin{enumerate}[label={(\alph*)}]
            \item $Y$ begins with a letter from $A_i$, or
            \item $x$ is the product of at most $\ell$ letters from $L'$.
        \end{enumerate}
    \end{lem}
    
    Before continuing with the construction of our finite set,
    we need a little notation.
    If $x$  is a letter in some $A_i$, we may write $\Phi(x)$ uniquely as
    $\Phi(x) = \mu_xx^\Phi \mu_x^{-1}$,
    where $\mu_x$ is a reduced word, $x^\Phi$ is a letter,
    and this product is reduced as written.
    The map $x \mapsto x^\Phi$ is a bijection of the set of letters
    that belong to $\bigcup_{i=1}^n (A_i\setminus \{1\})$.
    Given a finite set of letters $S$ and a positive integer $m$,
    we write $S^m$ for the finite set of letters
    that may be written as the product of at most $m$ letters in $S$.
    Thus $S \subset S^m$ and if $S$ is closed under taking inverses,
    so is $S^m$.

    We say that a letter $x$ has \emph{Martino's property P}
    if it satisfies the following three conditions:
    \begin{enumerate}
        \item $x \notin L'$.
            (Recall that this implies that $x$ is in some $A_i$.)
        \item $x^\Phi$ is not the  product of at most $3\ell$ letters in $L'$,
            i.e.~$x^\Phi \notin (L')^{3\ell}$.
        \item If $y$ and $y' \in (L')^\ell$ are in the same factor $A_j$
            as $x^\Phi$,
            then $(yx^\Phi y')^\Phi \notin (L')^{3\ell}$.
    \end{enumerate}

    It is clear that there exists a finite set of letters $M$
    such that if $x \notin M$, then $x$ has Martino's property P.
    Since $x \mapsto x^\Phi$ is a bijection,
    there is also a finite set of letters $M'$
    such that if $x^\Phi \notin M'$, then $x$  has Martino's property P.
    We may  suppose this set $M$ is closed under taking inverses.
    As it happens,
    we need the following stronger notion.

    Say that a letter $z$ is a \emph{descendent} of a letter $x$
    if there exist $y$ and $y'$ in $(L')^\ell \cup \{1\}$ such that $z = yx^\Phi y'$.
    More generally, say that $z$ is a \emph{$k$-descendent} of $x$
    for some positive integer $k$
    if there are letters $x = z_0,z_1,\dotsc, z_k = z$
    such that
    $z_i$ is a descendent of $z_{i-1}$
    for $i$ satisfying $1 \le i \le k$.

    \begin{claim}[Corollary 5.1.10 of \cite{MartinoThesis}]
        \label{MartinosPropertyP}
        There is a finite set of letters $N$
        such that if $x^\Phi \notin N$, then $x$ and 
        all $k$-descendents of $x$ have Martino's property P
        for $k$ satisfying $1  \le k \le C$.
    \end{claim}

    \begin{proof}[Proof of \Cref{MartinosPropertyP}.]
        The proof is a corrected version of the discussion
        preceding \cite[Corollary 5.1.10]{MartinoThesis}.
        Suppose first that some descendent of $x$ fails to have Martino's property P.
        Then there exist $y$ and $y'$ in $(L')^\ell \cup \{1\}$ such that
        $yx^\Phi y' \in M$.
        Since $x \mapsto x^\Phi$ is a bijection and $y$ and $y'$
        are drawn from finitely many possibilities,
        it is clear that there is a finite set $N_1$
        which we may assume to be closed under taking inverses
        such that if $x \notin N_1$, then all descendents of $x$ have Martino's property P.

        Now suppose that there is a finite set $N_i$ of letters
        such that if $x \notin N_i$, 
        then all $i$-descendents of $x$ have Martino's property P.
        If some $(i+1)$-descendent of $x$ fails to have property $P$,
        then $x$ has a descendent $yx^\Phi y' \in N_i$.
        The argument above shows that $x$ belongs to a finite set $N_{i+1}$.
        Therefore the finite set $N =  M' \cup \bigcup_{k=1}^{C} N_i$
        satisfies the conclusion of the claim.
        We assume that $N$ is closed under taking inverses.
    \end{proof}

    By assumption, there is $j_0$ so large that if $i \ge j_0$,
    then $w_i$ contains a letter not in $N^{C+1}$.
    The final step is the following claim.
    \begin{claim}
        \label{finalclaim}
        Suppose $i \ge j_0$, that $k(i) \ge i$,
        and that $k(i) > k(i-1)$.
        Then there exist integers $i = m_0 < m_1 < \dotsb < m_{C+1}$
        such that
        \[  \Phi(U_{m_k}) \subset U_{m_{k+1}-1} \]
        and
        \[  \Phi^2(U_{m_k}) \subset U_{m_{k+1}} \]
        for $k$ satisfying $0 \le k \le C$.
    \end{claim}

    \begin{proof}[Proof of \Cref{finalclaim}.]
        Consider $w_i = \xi_i^{-1}\Phi(\xi_i) = x_{i+1}\dotsb x_{k(i)}z_i$.
        This is reduced as written,
        but there may be no $x$ terms if $i = k(i)$.
        Since $i \ge j_0$, some letter of $w_i$ does not belong to $N^{C+1}$.
        We claim that in fact there is some $j \le i$
        such that the following conditions are satisfied.
        \begin{enumerate}
            \item $x_j^\Phi \notin N$.
            \item We may write $\Phi(\xi_{j-1})\mu_{x_j} = \xi_{i}z$
                for some possibly trivial word $z$
                such that this product is reduced as written if $z$ is nontrivial.
            \item We may write $\Phi(\xi_{j-1})\mu_{x_j}x_j^\Phi = \xi_{i}z'$
                for some possibly trivial word $z'$
                such that this product is reduced as written if $z'$ is nontrivial.
        \end{enumerate}
        Recall that since $k(i) > k(i-1)$, we have
        \[  \Phi(x_i) = z_{i-1}^{-1}x_{k(i-1)+1}\dotsb x_{k(i)}z_i. \]
        This product is not necessarily reduced as written,
        but $x_{k(i-1)+ 1}$ is not entirely canceled.
        First note that if the letter of $w_i$ not belonging to $N^{C+1}$ is in $z_i$,
        then in fact that letter is $x_i^\Phi \notin N$
        and items 2 and 3 follow.

        So suppose the letter is $x_k$ for $i+1 \le k \le k(i)$.
        We follow and correct \cite[Lemma 5.1.7]{MartinoThesis}.
        Let $s$ be the least integer such that $k(i_s) \ge k$.
        Note that $i_s \le i$.
        We have $k(i_{s} - 1) \le k(i_{s-1}) < k \le k(i_s)$.
        If
        \[  \Phi(x_{i_s}) = z_{i_s-1}^{-1}x_{k(i_s-1)+1}\dotsb x_{k(i_s)}z_{i_s} \]
        is reduced as written or if $k > k(i_s - 1)+ 1$,
        then $x_k$ occurs in the image of $x_{i_s}$,
        and in fact we have $x_k = x_{i_s}^\Phi$ and items 2 and 3 follow
        with $j = i_s$.
        If neither is the case,
        then we may  write $z_{i_s -1} = xz'$
        reduced as written where $x$ is in the same factor as $x_k = x_{k(i_s - 1)+1}$
        but not equal to it.
        We have that $x_{i_s}^\Phi = x^{-1}x_k$.
        Now, either
        \begin{enumerate}[label={(\alph*)}]
            \item $x^{-1}x_k \notin N^{C}$, or
            \item $x^{-1}x_k \in N^C$, in which case $x^{-1} \notin N^C$,
                since $x_k \notin N^{C+1}$.
        \end{enumerate}
        If item (a) holds, then items 1 through 3 hold with $j = i_s$.
        To see this, note that
        \[  \Phi(\xi_{i_s-1})\mu_{x_{i_s}} 
        = \xi_{k(i_s-1)}z_{i_s-1}z'^{-1} = \xi_{k-1}x \]
        and
        \[  \Phi(\xi_{i_s-1})\mu_{x_{i_s}}x_{i_s}^\Phi
        =  \xi_{k(i_s-1)}z_{i_s-1}z'^{-1}x^{-1}x_k =  \xi_{k}. \]
        (Recall that $k \ge i+1$.)

        Suppose item (b) holds.
        Notice  that if $k-1 = k(i_s-1) > k(i_s-2)$,
        then \[\Phi(x_{i_s -1}) 
        = z_{i_s-2}^{-1}x_{k(i_s-2)+1}\dotsb x_{k-1}xz',\]
        $x_{k(i_s-2)+1}$ is not entirely canceled,
        and we see that $x$ occurs in the image of $x_{i_s-1}$,
        so we conclude that $x = x_{i_s-1}^\Phi$
        and items 1 through 3 hold with $j = i_s-1$.

        So suppose $k(i_s-1) \le k(i_s -  2)$.
        Because $k(i_{s-1}) > k(i_{s-1}-1)$,
        we conclude that $i_s-1 > i_{s-1}$.
        Notice that \[k-1 = k(i_s - 1) \le k(i_s- 2) \le k(i_{s-1}) \le k-1,\]
        so we conclude $k(i_s -1) = k(i_s - 2) = k(i_{s-1}) = k-1$.

        The argument proceeds as above.
        If $x$ occurs in the image of $x_{i_s-1}$, we are done.
        If not, then $z_{i_s-2}= x'z''$, where this product is reduced as written
        and $x$ and $x'$ are in the same factor.
        If $x = x'$, then $x' \notin N^{C-1}$, and we may proceed considering $x_{i_s-2}$.
        If $x \ne x'$, then $x'^{-1}x$ occurs in the image of $x_{i_s-1}$
        since $\Phi(x_{i_s-1}) = z''^{-1}x'^{-1}xz'$.
        If $x'^{-1}x \notin N^{C-1}$, then we conclude with $j = i_s - 1$.
        If not, then $x'^{-1} \notin N^{C-1}$.

        We may repeat this argument,
        reducing our candidate for $j$ by one and our index $q$ of $N^q$.
        We always reach a positive conclusion if $k(j) > k(j-1)$,
        but we know that $k(i_{s-1}) > k(i_{s-1}-1)$,
        so we do reach  a positive conclusion in at most $C$ steps.

        \paragraph{}
        We now follow \cite[Lemma 5.1.8]{MartinoThesis}.
        Notice that since $x_j^\Phi \notin N$, the letter $x_j$
        has Martino's property P.
        Suppose we have $\zeta \in U_j$,  so $\zeta = \xi_j\zeta'$ is reduced as written.
        Then
        \[  \Phi(\zeta) = \Phi(\xi_{j-1})\mu_{x_j}x_j^\Phi\mu_{x_j}^{-1}\Phi(\zeta') \]
        The Collins--Turner result \Cref{CollinsTurnerLemma}
        quoted above implies
        that  since $x_j^\Phi \notin (L')^{3\ell}$ by Martino's property  P,
        the letter $x_j^\Phi$ is not entirely canceled in the above product,
        so there exist $y$ and $y' \in (L')^\ell$ (possibly trivial)
        so that
        \[  \Phi(\xi_j\zeta') = wyx_j^\Phi y'\zeta'' \]
        for some words $w$ and $\zeta''$ such that this product
        is reduced as written if we count $yx_j^\Phi y'$ as a single letter.
        In fact, since $x_j^\Phi$ does not entirely cancel
        we may write $w = \xi_{i}w'$ for some possibly trivial word $w'$
        such that the product is reduced as written if $w'$ is nontrivial.
        In fact, since $j \le i$, we may write $ w = \xi_{j}w''$
        for some possibly trivial word $w''$
        such that the product is reduced as written if $w''$  is nontrivial.
        To sum up, we have
        \[  \Phi(\xi_{j-1})\mu_{x_j}x_j^\Phi = \xi_{j}w'' yx_j^\Phi, \]
        where this product is reduced as written.
        Now if we apply $\Phi$ again, we have
        \begin{gather*}  
            \Phi^2(\xi_j) 
            = \Phi(\xi_{j}w''yx_j^\Phi \mu_{x_j}^{-1})
            = \Phi(\xi_{j-1})\mu_{x_j}x_j^\Phi\mu_{x_j}^{-1}
            \Phi(w'')\Phi(yx_j^\Phi)\Phi(\mu_{x_j}^{-1}) \\
            = \xi_j w''(yx_j^\Phi) \mu_{x_j}^{-1}\Phi(w'')
            \Phi(y x_j^\Phi )\Phi(\mu_{x_j}^{-1})
        \end{gather*}
        Since $w''$ does not begin with a letter in the same factor as $x_j$,
        the Collins--Turner result \Cref{CollinsTurnerLemma}
        implies we may find $y'' \in (L')^\ell \cup\{1\}$
        such that
        \[  \xi_jw''(yx_j^\Phi)\mu_{x_j}^{-1}\Phi(w'') 
        = \xi_j w''(yx_j^\Phi y'')v \]
        for some possibly trivial word $v$
        such that this product is reduced as written
        if we count $(yx_j^\Phi y'')$ as a single letter 
        and if we discard $v$ if it is trivial.
        Thus we have
        \[  \Phi^{2}(\xi_j) = \xi_j w''(y x_j^\Phi y'')v\Phi(y x_j^\Phi)
        \Phi(\mu_{x_j}^{-1})\]
        Notice that if $v = \mu_{yx_j^\Phi}^{-1}$
        and the letters $(yx_j^\Phi y'')$
        and $(y x_j^\Phi)^\Phi$ are in the same factor,
        then \Cref{CollinsTurnerLemma} implies $y x_j^\Phi y'' \in (L')^\ell$,
        which contradicts the second item in Martino's property P.
        Therefore applying \Cref{CollinsTurnerLemma} once more
        to $\Phi(\mu_{x_j}^{-1})$, we may write
        \[  \Phi^2(\xi_j) = \xi_j w''(y x_j^\Phi y''')v' 
        [z(y x_j^\Phi)^\Phi z'] v'' \]
        where either $v' \ne 1$
        or $x_j^\Phi$ and $(x_j^{\Phi})^\Phi$ are in different factors,
        where $z' \in (L')^\ell \cup \{1\}$,
        and where $v''$ is a possibly trivial word.

        Since $\Phi(\xi_i\zeta') = \xi_jw'' (yx_j^\Phi y')\zeta''$
        and $\mu_{y x_j^\Phi y'} = \mu_{y x_j^\Phi}$,
        we have
        \[  \Phi^2(\xi_j\zeta') = \xi_j w''(y x_j^\Phi y''')v'
        [z(yx_j^\Phi y')^\Phi z'']\zeta'''\]
        where $z'' \in (L')^\ell \cup \{1\}$ and thus this product 
        is reduced upon counting $(y x_j^\Phi y''')$
        and $[z(y x_j^\Phi y')^\Phi z'']$ as single letters.
        By Martino's property P, both of these letters are nontrivial.

        We claim that $(yx_j^\Phi y''') = x_{m_1}$ for some $m_1 > i$.
        Since $\xi_j w''(y x_j^\Phi y''')$ is always a subword of
        $\Phi(\xi_j\zeta')$, it is a subword of $\Phi^2(\xi) = \xi$,
        so we do have that $y x_j^\Phi y''' = x_{m_1}$
        for some $m_1$.
        Since we saw that $\xi_j w'' = \xi_i w'$,
        we have $m_1 > i$.

        Note that setting $i = m_0$, this proves that
        \[  \Phi(U_{m_0}) \subset U_{m_1-1}\]
        and
        \[  \Phi^2(U_{m_0})\subset U_{m_1}. \]

        We would like to repeat this argument
        with $m_1$ (and later $m_s$) playing the role of both $j$ and $i$
        to find $m_s$ for $2 \le s \le C+1$.
        To do this, we need to know that $x_{m_s}$ has property P
        provided that $x_{m_{s-1}}$ did
        and that we may write
        \begin{enumerate}
            \item $\Phi(\xi_{m_s - 1})\mu_{x_{m_s}} = \xi_{m_s} z$ and
            \item $\Phi(\xi_{m_s - 1})\mu_{x_{m_s}}x_{m_s}^\Phi = \xi_{m_s}z'$
        \end{enumerate}
        for possibly trivial words $z$ and $z'$ such that the respective
        products are reduced as written.
        First, observe that $x_{m_s}$ is an $s$-descendent of $x_j$,
        so it has Martino's property P if $s \le C$.
        The other property follows inductively once we note that
        \[  \Phi(\xi_{m_1-1})\mu_{x_{m_1}} = \Phi(\xi_jw'')\mu_{x_{m_1}}
        = \xi_j w'' x_{m_1} v' z \]
        where either $v' \ne 1$ or $x_{m_1}$ and $z$ are in different factors
        and is otherwise reduced as written and
        \[  \Phi(\xi_{m_1-1})\mu_{x_{m_1}}x_{m_1}^\Phi 
        = \xi_j w''x_{m_1} v'(zx_{m_1}^\Phi).\qedhere \]
    \end{proof}
    Note that if $i_r \ge j_0$ and $k(i_r) \ge i_r$,
    then $i_r$ satisfies the hypotheses of \Cref{finalclaim}.
    Since $i_{r+1} - i_r \le C$, there is $s$ with $m_s \le i_{r+1} < m_{s+1}$,
    so $k(i_{r+1}) \ge m_{s+1} - 1 \ge i_{r+1}$,
    and we see that $\Phi^{2C+2}(U_{i_r}) \subset U_{m_{C+1}} \subset U_{i_{r+1}}$.
\end{proof}

The following result is implied by the main result of \cite{Martino}.

\begin{lem}
    There are only finitely many $\mathbb{T}$-orbits of points
    in $(\fix(\hat\Phi)\cap\partial_\infty(F,\mathscr{A}))
    \setminus \partial(\mathbb{F},\mathscr{A}|_\mathbb{F})$.
\end{lem}

We will give an estimate for the number of $\mathbb{T}$-orbits of points
in $\fix(\hat\Phi)\setminus \partial(\mathbb{F},\mathscr{A}|_{\mathbb{F}})$
in \Cref{indexsection}.

\section{Attracting Laminations}
\label{laminationsection}
We now turn to attracting laminations for free products,
prove their existence and develop some of their properties.
We end up with a little more of the theory
than is actually needed for this paper,
(For example, the existence of a homomorphism $\pf_{\Lambda^+}\colon \stab(\Lambda^+) \to \mathbb{Z}$
from the stabilizer of an attracting lamination $\Lambda^+$ to $\mathbb{Z}$)
but we hope that it will prove useful for future work.

\paragraph{The space $\tilde{\mathcal{B}}$.}
In \cite{GuirardelHorbez}, Guirardel and Horbez consider
\emph{algebraic laminations} for free products.
To wit, let $\tilde{\mathcal{B}}$ be the space
\[  \tilde{\mathcal{B}} = (\partial(F,\mathscr{A})\times \partial(F,\mathscr{A})\setminus\Delta)
/ \mathbb{Z}/2\mathbb{Z}\]
where $\Delta$ is the diagonal and $\mathbb{Z}/2\mathbb{Z}$ acts by permuting the factors.
The diagonal action of $F$ on $\partial(F,\mathscr{A}) \times \partial(F,\mathscr{A})$
descends to an action on $\tilde{\mathcal{B}}$.
An element of $\tilde{\mathcal{B}}$ is an \emph{algebraic line,}
an unordered pair $(\alpha,\omega)$ of distinct elements of $\partial(F,\mathscr{A})$.
Let $\mathcal{B}$ be the quotient of $\tilde{\mathcal{B}}$ by the action of $F$.
An \emph{algebraic lamination} is a closed, $F$-invariant subset of $\tilde{\mathcal{B}}$
or equivalently a closed set in $\mathcal{B}$,
and the lines it comprises are called \emph{leaves.}

\paragraph{Lines in $\Gamma$ and in $\mathcal{G}$.}
Let $\mathcal{G}$ be a marked graph of groups with Bass--Serre tree $\Gamma$.
Between any two boundary points $\xi$ and $\zeta$,
there is a unique tight path in $\Gamma$ called the \emph{line} $\tilde L_{\xi,\zeta}$
from $\xi$ to $\zeta$.
Thus to a point $(\alpha,\omega)$ in $\tilde{\mathcal{B}}$
we may associate the unoriented image of the line $\tilde L_{\alpha,\omega}$,
thus yielding a \emph{space of lines in $\Gamma$,}
which we will denote $\tilde{\mathcal{B}}(\Gamma)$.
A line may be finite, singly infinite, or bi-infinite,
according to whether its endpoints lie in $V_\infty(F,\mathscr{A})$ or $\partial_\infty(F,\mathscr{A})$.
In \cite{BestvinaFeighnHandel},
a \emph{space of lines in $\Gamma$,} denoted $\tilde{\mathcal{B}}(\Gamma)$ is considered.
It is equipped with a ``compact--open'' topology:
if $\tilde\gamma \subset \Gamma$ is a finite tight edge path,
(even with endpoints at vertices, say),
define $N(\tilde\gamma)  \subset \tilde{\mathcal{B}}(\Gamma)$ 
to be the set of lines that contain $\tilde\gamma$ as a subpath.
In the free group setting, the sets $N(\tilde\gamma)$ form a basis 
for the topology on $\tilde{\mathcal{B}}(\Gamma)$.
For our purposes, the sets $N(\tilde\gamma)$
will only form a basis for the \emph{bi-infinite} lines.

To make the map $\tilde{\mathcal{B}} \to \tilde{\mathcal{B}}(\Gamma)$ a homeomorphism,
a basic open neighborhood of a line $\tilde\lambda$ in $\Gamma$
must be given by a pair of disjoint basic open neighborhoods of the endpoints
$\alpha$ and $\omega$ of $\tilde\lambda$.
In other words,
a sequence of lines $\{\lambda_n\}$ in $\Gamma$
converges to $\lambda$
if for each $\tilde y \in \Gamma$ distinct from the endpoints $(\alpha,\omega)$ of $\lambda$,
there is a choice of order for the endpoints $(\alpha_n,\omega_n)$ of $\lambda_n$
such that $\alpha_n$ belongs to the same half-tree based at $\tilde y$ as $\alpha$
and $\omega_n$ belongs to the same half-tree based at $\tilde y$ as $\omega$
for $n$ large.

Each line in $\Gamma$ may be projected to a \emph{line in $\mathcal{G}$.}
See \cite[Section 1]{Myself} for more details on projecting from $\Gamma$ to $\mathcal{G}$.
The \emph{space of lines in $\mathcal{G}$} is denoted $\mathcal{B}(\mathcal{G})$.
Projection defines a natural projection map from $\tilde{\mathcal{B}}(\Gamma)$
to $\mathcal{B}(\mathcal{G})$,
and we give $\mathcal{B}(\mathcal{G})$ the quotient topology.
Given a tight  edge path $\gamma$ in $\mathcal{G}$,
define $N(\gamma)$ to be the set of those lines in $\mathcal{G}$
that contain $\gamma$ as a subpath.
The sets $N(\gamma)$ form a basis for any \emph{bi-infinite} line in $\mathcal{B}(\mathcal{G})$.

If $f\colon \mathcal{G} \to \mathcal{G}'$ is a homotopy equivalence,
the homeomorphism $\hat f\colon \partial\Gamma \to \partial\Gamma'$
yields a homeomorphism 
$\tilde f_\sharp\colon \tilde{\mathcal{B}}(\Gamma) \to \tilde{\mathcal{B}}(\Gamma')$
and a homeomorphism
$f_\sharp \colon \mathcal{B}(\mathcal{G}) \to \mathcal{B}(\mathcal{G}')$.
If $\beta \in \mathcal{B}$ corresponds to $\lambda$ in $\mathcal{B}(\mathcal{G})$,
we say that $\lambda$ \emph{realizes} $\beta$ in $\mathcal{G}$.

\paragraph{Laminations.}
Recall that
a \emph{lamination} is a closed set of lines in $\mathcal{G}$
or a closed, $F$-invariant set of lines in $\Gamma$.
The lines it comprises are the \emph{leaves} of the lamination.
We are interested in \emph{attracting laminations,}
for which we need a little more terminology.

Recall that an element of $F$ is \emph{peripheral} if it is conjugate into
a vertex group of some (and hence any) marked graph of groups.
The conjugacy  class of a nonperipheral element of $F$ determines
a periodic bi-infinite line in $\mathcal{B}(\mathcal{G})$
that runs over the tight circuit determined by the conjugacy class.
A line $\beta$ in $\mathcal{B}$ is \emph{carried} by the conjugacy class
of a free factor $[[F^i]]$ if it is in the closure of the periodic lines
in $\mathcal{B}$ determined by the conjugacy classes of nonperipheral elements of $F^i$.
If $\mathcal{G}$ is a marked graph of groups and $K \subset \mathcal{G}$
is a connected subgraph such that $[[\pi_1(\mathcal{G}|_K)]] = [[F^i]]$,
then $\beta$ is carried by $[[F^i]]$ if and only if the realization of $\beta$
in $\mathcal{G}$ is contained in $K$.

Given $\varphi \in \out(F,\mathscr{A})$,
we say that $\beta' \in \mathcal{B}$ is \emph{weakly attracted 
to $\beta \in \mathcal{B}$ under the action of $\varphi$}
if $\varphi^k_\sharp(\beta') \to \beta$ (note that $\mathcal{B}$ is not Hausdorff).
A subset $U \subset \mathcal{B}$ is an 
\emph{attracting neighborhood of $\beta \in \mathcal{B}$
for the action of $\varphi$} if $\varphi_\sharp(U) \subset U$
and if $\{\varphi^k_\sharp(U) : k \ge 0\}$
is a neighborhood basis for $\beta \in \mathcal{B}$.
Finally, a line $\sigma$ in $\mathcal{G}$ 
is \emph{birecurrent} if it is bi-infinite and every subpath of $\sigma$
occurs infinitely often as a subpath of each end of $\sigma$.
Our first lemma on attracting laminations says that birecurrence
is a property of the abstract line.

\begin{lem}[\cite{BestvinaFeighnHandel} Lemma 3.1.4]
    \label{birecurrent}
    If some realization of $\beta \in \mathcal{B}$ in a marked graph of groups
    is birecurrent, then every realization of $\beta$ in a marked graph of groups
    is birecurrent.
    If $\beta$ is birecurrent, then $\varphi_\sharp(\beta)$ is birecurrent
    for every $\varphi\in \out(F,\mathscr{A})$.
\end{lem}

\begin{proof}
    The proof is identical to \cite[Lemma 3.1.4]{BestvinaFeighnHandel}.
    Suppose that $\sigma$ and $\sigma'$ are realizations of $\beta$
    in marked graphs of groups $\mathcal{G}$ and $\mathcal{G}'$ respectively.
    Since the subspaces $V_\infty(F,\mathscr{A})$ and $\partial_\infty(F,\mathscr{A})$
    are well-defined, $\sigma$ is bi-infinite if and only if $\sigma'$ is.
    Suppose that $\sigma$ is birecurrent,
    and let $h\colon \mathcal{G} \to \mathcal{G}$ be a homotopy equivalence
    that respects the markings.
    Let $C$ be the bounded cancellation constant for $h$.
    Choose lifts $\tilde\sigma \subset \Gamma$,
    $\tilde\sigma' \subset \Gamma'$
    and $\tilde h\colon \Gamma \to \Gamma'$ such that
    $\tilde h_\sharp(\tilde\sigma) = \tilde\sigma'$.
    Let $\tilde\sigma'_0$ be a finite subpath of $\tilde\sigma'$.
    Extend it to a finite subpath $\tilde\tau'$ of $\tilde\sigma'$
    by adding $C$ initial and terminal edges.
    Choose a finite subpath $\tilde\tau$ of $\tilde\sigma$ such that
    $\tilde h_\sharp(\tilde\tau)$ contains $\tilde\tau'$.
    Since $\sigma$ is birecurrent, each end of $\tilde\sigma$ contains
    infinitely many copies $\tilde\tau_i$ of $\tilde\tau$.
    Define $\tilde\mu_i'$ by removing $C$ initial and terminal edges of
    $\tilde h_\sharp(\tilde\tau_i)$.
    Bounded cancellation implies that $\tilde\mu_i'$ is a subpath of $\tilde\sigma'$.
    By construction, each $\tilde\mu'_i$ contains a copy of $\tilde\sigma'_0$,
    so we have shown that $\sigma'$ is birecurrent.

    As in the original, replacing  $h$ with a topological representative
    $f\colon\mathcal{G} \to \mathcal{G}$ of $\varphi \in \out(F,\mathscr{A})$
    in the argument above shows that $\varphi_\sharp(\beta)$
    is birecurrent for all $\varphi \in \out(F,\mathscr{A})$
    provided that $\beta$ is birecurrent.
\end{proof}

A closed subset $\Lambda^+ \subset \mathcal{B}$ is an
\emph{attracting lamination for $\varphi$}
if it is the closure of a single line $\beta$ such that:
\begin{enumerate}
    \item The line $\beta$ is birecurrent.
    \item The line $\beta$ has an attracting neighborhood for the action
        of some iterate of $\varphi$.
    \item The line $\beta$ is not carried by any $\varphi$-periodic
        $F_1$ or $C_2*C_2$ free factor.
\end{enumerate}
The line $\beta$ is said to be \emph{generic} for $\Lambda^+$.
The set of attracting laminations for $\varphi$ is denoted $\mathcal{L}(\varphi)$.

\begin{lem}[\cite{BestvinaFeighnHandel} Lemma 3.1.6]
    $\mathcal{L}(\varphi)$ is $\varphi$-invariant.
\end{lem}

\begin{proof}
    Again the proof is essentially identical 
    to \cite[Lemma 3.1.6]{BestvinaFeighnHandel}.
    Suppose that $\beta$ is generic  for $\Lambda^+ \in \mathcal{L}(\varphi)$.
    \Cref{birecurrent} shows that $\varphi_\sharp(\beta)$ is birecurrent.
    If $V$ is an attracting neighborhood for $\beta$ under the action of
    $\varphi^s$ for some $s \ge 1$,
    then $U = \varphi_\sharp(V)$ is an attracting neighborhood 
    for $\varphi_\sharp(\beta)$
    under the action of $\varphi^{s}$.
    If $[[B]]$ is a $\varphi$-periodic free factor of the form $F_1$ or $C_2*C_2$
    that carries $\varphi_\sharp(\beta)$,
    and if $\Phi^{-1}\colon (F,\mathscr{A}) \to (F,\mathscr{A})$
    represents $\varphi^{-1}$,
    then $[[\Phi^{-1}(B)]]$ is a $\varphi$-periodic
    free factor of the same form that carries $\beta$.
    Therefore we have shown that $\varphi_\sharp(\beta)$ is generic with respect
    to some lamination $\varphi_\sharp(\Lambda^+) \in \mathcal{L}(\varphi)$.
\end{proof}

A nonnegative integral matrix $M$ is \emph{aperiodic}
if there is some positive power $k$ such that $M^k$ has all positive entries.
Aperiodic matrices are irreducible.
Any irreducible matrix with a nonzero diagonal entry is aperiodic.
Suppose $f\colon \mathcal{G} \to \mathcal{G}$
is a relative train track map with filtration
$\varnothing = G_0 \subset G_1 \subset \cdots \subset G_m = G$
and that $H_r$ is an exponentially growing stratum.
We say that $H_r$ is \emph{aperiodic} if its transition matrix is aperiodic,
and that $f\colon \mathcal{G} \to \mathcal{G}$ is \emph{eg-aperiodic}
if each exponentially growing stratum is aperiodic.

\begin{lem}[\cite{BestvinaFeighnHandel} Lemma 3.1.9]
    \label{laminationconstruction}
    Suppose that $f\colon \mathcal{G} \to \mathcal{G}$ is a relative train track map
    and that $H_r$ is an aperiodic exponentially growing stratum.
    There is an attracting lamination $\Lambda^+$
    with generic leaf $\beta$ such that $H_r$ is the highest stratum
    crossed by the realization $\lambda$ of $\beta$ in $\mathcal{G}$.
\end{lem}

    For $f\colon \mathcal{G} \to \mathcal{G}$ and $H_r$ and $k \ge 0$,
    define a \emph{$k$-tile} to be a path of the form
    $f^k_\sharp(E)$, for $E$ an edge of $H_r$, in either orientation.
    A path in $\mathcal{G}$ is a \emph{tile} if it is a $k$-tile for some $k$.
    A \emph{$k$-tiling} of a path in $G_r$
    is a decomposition of the path into subpaths that are either $k$-tiles,
    vertex group elements,
    or contained in $G_{r-1}$.
    A bi-infinite path $\lambda$ has an \emph{exhaustion by tiles}
    if each of its subpaths occurs as a subpath of a tile in $\lambda$.
    If $\lambda$ has an exhaustion by tiles,
    then \hyperlink{EG-iii}{(EG-iii)} implies that $\lambda$ is $r$-legal.

\begin{proof}
    We follow \cite[Lemma 3.1.9]{BestvinaFeighnHandel}.
    Choose an edge $E$ of $H_r$ and $m > 0$ such that
    $f_\sharp^m(E) = \alpha E \beta$ for some nontrivial tight paths
    $\alpha$ and $\beta$ in $G_r$.
    Write $h = f^m$ and choose lifts $\tilde E$, $\tilde \alpha$, $\tilde\beta$
    and $\tilde h\colon \Gamma \to \Gamma$ so that
    $\tilde h(\tilde E) = \tilde \alpha\tilde  E \tilde\beta$.
    Define $\tilde\tau_j = \tilde h^j_\sharp(\tilde E)$;
    it is the lift of a $jm$-tile.
    We have $\tilde\tau_0 = \tilde E$,
    $\tilde\tau_1 = \tilde\alpha\tilde\tau_0\tilde\beta$
    and more generally $\tilde\tau_{j+1}=\tilde\alpha_j \tilde\tau_j \tilde\beta_j$
    for nontrivial paths $\tilde\alpha_j$ and $\tilde\beta_j$.
    The $\tilde\tau_j$ are an increasing sequence of lifts of tiles
    whose union is a bi-infinite path $\tilde\lambda$ 
    that is fixed by $\tilde h_\sharp$.

    We claim that the projection $\lambda \subset \mathcal{G}$ realizes
    an element $\beta \in \mathcal{B}$ that is generic
    with respect to some element of $\mathcal{L}(\varphi)$.
    After replacing $m$ by a multiple if necessary,
    we may assume that the $h_\sharp$-image of any edge in $H_r$
    contains at least two edges in $H_r$.
    Define $\tilde\lambda_k$ to be the subpath of $\tilde\lambda$
    that begins with the $k$th lift of an edge of $H_r$ to the left of $\tilde E$
    and ends with the $k$th lift of an edge of $H_r$ to the right of $\tilde E$.
    Project to $\lambda_k \subset G_r$ and define $V_k = N(\lambda_k)$.
    \Cref{rttlemma} implies that 
    $\tilde h_\sharp(\tilde\lambda_k) \supset \tilde\lambda_{2k}$.
    Bounded cancellation in the form of \Cref{boundarybasics} part 3
    implies $h_\sharp(V_k) \subset V_{k+1}$ for all sufficiently large $k$.
    The $V_k$ are a neighborhood basis for $\lambda$,
    and so for sufficiently large $k$,
    $V_k$ is an attracting neighborhood of $\lambda$ for the action of $\varphi^m$.

    Since the difference between the number of edges in 
    $\tilde h_\sharp(\tilde\lambda_k)$
    and the number of edges in $\tilde\lambda_k$ grows without bound,
    the $\tilde\lambda_k$ cannot be subpaths of a single
    $\tilde h_\sharp$-invariant axis.
    In other words, $\lambda$ is not a circuit,
    and so cannot be carried by any $F_1$ or $C_2*C_2$ free factor.

    Note that by construction $\lambda$ has an exhaustion by tiles.
    We will show that $\lambda$ has a $k$-tiling for all $k\ge 1$.
    Notice that each $1$-tile has a $0$-tiling,
    and inductively each $(k+1)$-tile has a $k$-tiling
    and thus an $\ell$-tiling for $\ell \le k+1$.
    A $k$-tiling of $\tilde\tau_i$ determines and is determined by
    a finite set of vertices $\tilde V_i$ in $\tilde\tau_i \subset \tilde\lambda$.
    Given a vertex $\tilde v\in \tilde\lambda$,
    we may, after passing to a subsequence,
    assume that either $\tilde v \in \tilde V_i$ for all large $i$
    or $\tilde v\notin \tilde V_i$ for all large $i$.
    We then may consider (with this subsequence) another vertex of $\tilde\lambda$.
    The set of vertices that satisfy $\tilde v \in \tilde V_i$ for all large $i$
    determine a $k$-tiling of $\tilde\lambda$ and thus of $\lambda$.

    We have that \hyperlink{EG-i}{(EG-i)} implies that the first and last edges
    of any tile are contained in $H_r$.
    Thus each end of $\lambda$ must contain infinitely many edges in $H_r$.
    Bounded cancellation and the existence of $k$-tilings for all $k$
    imply that each tile occurs infinitely often in each end of $\lambda$.
    Since every finite subpath of $\lambda$ is contained in a tile,
    $\lambda$ is birecurrent.
    This shows that $\lambda$ is generic for some element of $\mathcal{L}(\varphi)$.
\end{proof}

\begin{lem}[\cite{BestvinaFeighnHandel} Lemma 3.1.10]
    \label{laminationproperties}
    Assume that $\beta \in \mathcal{B}$ is a generic leaf
    of some $\Lambda^+ \in \mathcal{L}(\varphi)$,
    that $f\colon \mathcal{G} \to \mathcal{G}$
    and $\varnothing = G_0 \subset G_1 \subset \dotsb \subset G_m = G$
    are a relative train track map and filtration
    representing $\varphi$ and that $\lambda$ is the realization 
    of $\beta$ in $\mathcal{G}$.
    \begin{enumerate}
        \item The highest stratum $H_r$ crossed by $\lambda$ is exponentially growing.
        \item The bi-infinite path $\lambda$ is $r$-legal.
    \end{enumerate}
    Define \emph{tiles} for $H_r$ as before the proof of \Cref{laminationconstruction}.
    \begin{enumerate}[resume]
        \item The bi-infinite path $\lambda$ has a $k$-tiling for all $k \ge 1$.
        \item The bi-infinite path $\lambda$ has an exhaustion by tiles.
    \end{enumerate}
\end{lem}

\begin{proof}
    We argue by induction on the structure of $F$.
    The base cases of $F = A_1$, $F_1$ and $C_2*C_2$ are trivial,
    so we assume the result holds for outer automorphisms of proper free factors of $F$.
    The rest of the proof follows as in \cite[Lemma 3.1.10]{BestvinaFeighnHandel}.

    Suppose first that $\lambda$ is contained in $G_{m-1}$.
    Let $s$ be the smallest positive integer so that the component $C$ of $G_{m-1}$
    containing the image of $\lambda$ is $f^s$-invariant.
    The inductive hypothesis applied to the restriction of $\varphi$ to $\pi_1(C)$
    completes the proof of item 1.

    Now suppose that $\lambda$ contains edges of $H_m$.
    Since $\lambda$ is birecurrent, it crosses some edge of $H_m$ infinitely many times.
    Now choose $s \ge 1$ so that $\lambda$ has an attracting neighborhood
    for the action of $\varphi^s$,
    and let $\alpha \ne \lambda$ be a periodic line that is weakly attracted to $\lambda$
    under the action of $\varphi^s_\sharp$.
    (Let us remark that it is here that we use the third point in the definition
    of an attracting lamination to ensure
    such a periodic line exists.)
    Since $f^{s\ell}_\sharp(\alpha) \to \lambda$ as $\lambda \to \infty$,
    $f^{\ell}_\sharp$ cannot act with period other than $1$ on $\alpha$.
    Since $\lambda \ne \alpha$, the map $f^{s}_\sharp$ cannot fix $\alpha$.
    Since $\alpha$ is periodic, it corresponds to a circuit in $\mathcal{G}$.
    
    Suppose the number of edges in the circuit $f^{s\ell}_\sharp(\alpha)$
    does not grow without bound but that $f^{s\ell}_\sharp(\alpha)$
    is not periodic.
    Since there are only finitely many underlying paths in $G$ of a given length,
    it follows that there exists $M \ge 1$
    such that $f^{sM\ell}_\sharp(\alpha)$ has the same underlying path
    as $\alpha$ but necessarily the vertex group elements differ.
    Consider a lift $\tilde\alpha$ of $\alpha$ to $\Gamma$,
    and a lift $\tilde h = \tilde f^{sM}$ such that
    $\tilde f^{sM}_\sharp(\tilde\alpha)$ and $\tilde\alpha$ share a subpath
    $\tilde\sigma$ with endpoints at points $\tilde v$ and $\tilde w$ of infinite valence in $\Gamma$
    with the property that the corresponding vertex group elements in
    the circuit $\alpha$ (which may be the same vertex group element)
    are not periodic.
    We claim that the sequence $\tilde h_\sharp^\ell(\tilde\alpha)$
    converges to $\tilde\sigma$ in $\tilde{\mathcal{B}}(\Gamma)$.
    Indeed, consider a point $\tilde y$ distinct from $\tilde v$ and $\tilde w$ in $\Gamma$.
    We will show that the initial endpoint of $\tilde h^\ell_\sharp(\tilde\sigma)$ 
    belongs to the same half-tree at $\tilde y$ as $\tilde v$ for large $\ell$.
    The same argument shows that the terminal endpoint of $\tilde h^\ell_\sharp(\tilde\sigma)$
    belongs to the same half-tree at $\tilde y$ as $\tilde w$ for large $\ell$.
    Consider the direction at $\tilde v$ determined by the path from $\tilde v$ to $\tilde y$.
    A sufficient condition for the initial endpoint of $\tilde h^\ell_\sharp(\tilde\sigma)$
    to be contained in the same half-tree at $\tilde y$ as $\tilde v$ is for this direction
    to be distinct from the direction at $\tilde v$ 
    determined by the ray from $\tilde v$ to the initial endpoint of $h^\ell_\sharp(\tilde\sigma)$.
    Since by assumption this latter direction is not periodic,
    this is indeed the case for large $\ell$.
    This shows that $\tilde h^\ell_\sharp(\tilde\alpha)$ converges to $\tilde\sigma$.
    This is a contradiction,
    as there are neighborhoods of $\lambda$ that do not contain $\sigma$
    (indeed, take some subpath of $\lambda$ longer than $\sigma$).

    Therefore the number of edges of $f^{s\ell}_\sharp(\alpha)$ grows without bound.
    Since $f^{s\ell}_\sharp(\alpha)$ converges to $\lambda$,
    which has infinitely many edges of $H_m$,
    the number of $H_m$-edges of $f^{s\ell}_\sharp(\alpha)$ grows without bound.
    This implies that $H_m$ is exponentially growing.

    Now suppose that $H_r$ is the highest exponentially growing stratum
    crossed by $\lambda$.
    To prove that $\lambda$ is $r$-legal,
    let $j$ be the number of illegal turns of the circuit $\alpha$ in $H_r$.
    Since $f_\sharp$ does not create new illegal turns in $H_r$,
    the number of illegal turns of $f^{s\ell}_\sharp(\alpha)$ in $H_r$
    is bounded by $j$.
    Take a finite subpath $\lambda_0$ of $\lambda$.
    By the definition of weak convergence,
    we have that $\lambda_0$ is a subpath of $f^{s\ell}_\sharp(\alpha)$
    for all $\ell$ sufficiently large.
    Since the length of $f^{s\ell}_\sharp(\alpha)$ increases without bound,
    for $\ell$ sufficiently large, the subpath $\lambda_0$
    is covered by two fundamental domains of the circuit $f^{s\ell}_\sharp(\alpha)$,
    so the number of illegal turns of $\lambda_0$ in $H_r$ is at most $2j$.
    But since $\lambda$ is birecurrent and we may choose $\lambda_0$ arbitrarily,
    this uniform bound actually implies that $\lambda$ is $r$-legal.

    Now consider item 3. Fix $k \ge 1$ and let $\tilde\lambda \subset \Gamma$
    be a lift of $\lambda$.
    Recall from the proof of \Cref{laminationconstruction}
    that a $k$-tiling of $\lambda$ corresponds to a subdivision of $\tilde\lambda$
    and thus the vertices of $\tilde\lambda$ 
    that are the endpoints of the subdivision pieces.
    Let $q$ be the number of edges in $\alpha$.
    Given a finite subpath $\lambda_0 \subset \lambda$,
    let $\lambda_1 \subset \lambda$ be a finite subpath
    that contains $2q+1$ copies of $\lambda_0$.
    (The path $\lambda_1$ exists by birecurrence.)
    As in the argument for item 2,
    if $\ell$ is sufficiently large,
    then $\lambda_1$ occurs as a subpath of the periodic line
    determined by $f^{s\ell}_\sharp(\alpha)$ that is covered by two fundamental domains.
    In particular, at least one copy of $\lambda_0$ occurs
    as a subpath of $f^{s\ell}_\sharp(E)$ for some edge $E$ of $G_r$.
    We conclude that $\lambda$ is an increasing union of finite subpaths
    that have $k$-tilings.
    The $k$-tilings of these subpaths correspond to finite sets $\tilde V_i$
    of vertices of $\tilde \lambda$.
    After passing to subsequences as in the proof of \Cref{laminationconstruction},
    we obtain a $k$-tiling of $\lambda$.

    To complete the proof of item 4, we need only show that each finite subpath
    $\lambda_0 \subset \lambda$ is actually a subpath of a tile.
    Choose a finite subpath $\lambda_0 \subset \lambda$.
    Birecurrence implies that there is a finite subpath $\lambda_1 \subset \lambda$
    that contains two disjoint copies of $\lambda_0$.
    Assume further that $\lambda_0$ contains an edge of $H_r$.
    By item 3, $\lambda$ has a $k$-tiling,
    where $k$ is chosen so large that each $k$-tile is longer than $\lambda_1$.
    In any $k$-tiling of $\lambda$ 
    there are at most two $k$-tiles that intersect $\lambda_1$;
    one of these must contain a copy of $\lambda_0$,
    completing the proof of item 4.
\end{proof}

\begin{cor}[\cite{BestvinaFeighnHandel} Corollary 3.1.11]
    \label{laminationunique}
    Assume that $f\colon \mathcal{G} \to \mathcal{G}$
    and $\varnothing \subset G_0 \subset G_1 \subset \dotsb \subset G_m = G$
    are a relative train track map and filtration representing $\varphi$,
    that $H_r$ is an aperiodic exponentially growing stratum
    and that \emph{tiles} 
    (see the paragraph before the proof of \Cref{laminationconstruction})
    are defined with respect to $H_r$.
    Assume further that $\beta \in \mathcal{B}$ is $\Lambda^+$-generic
    for $\Lambda^+ \in \mathcal{L}(\varphi)$
    and that $H_r$ is the highest stratum crossed by the realization 
    of $\beta \in \mathcal{G}$.
    Then $\{N(\tau) : \tau \text{ is a tile}\}$ is a neighborhood basis
    in $\mathcal{B}$ for $\beta$.
    It follows that all such $\beta$ have the same closure.
\end{cor}

\begin{proof}
    The proof is identical to \cite[Corollary 3.1.11]{BestvinaFeighnHandel}
    Let $\lambda \subset \mathcal{G}$ be the realization of $\beta$.
    By \Cref{laminationproperties} item 3,
    $\lambda$ has a $k$-tiling for all $k$.
    Suppose $k_0$ is such that $M^{k_0}$ is positive.
    Then the $f^{k_0}_\sharp$-image of an edge of $H_r$
    contains every edge of $H_r$,
    and thus each $(k_0 + \ell)$-tile contains every $\ell$-tile.
    More generally if $k - \ell \ge k_0$,
    each $k$-tile contains every $\ell$-tile.
    It follows that $\lambda$ contains every tile.
    Conversely, item 4 of \Cref{laminationproperties}
    implies that every subpath of $\lambda$ is contained in a tile in $\lambda$.
\end{proof}

The previous three results imply
that for any relative train track map representing $\varphi$
and any aperiodic exponentially growing stratum $H_r$,
there is a \emph{unique} attracting lamination $\Lambda^+ \in \mathcal{L}(\varphi)$
with the property that $H_r$ is the highest stratum
crossed by the realization $\lambda \subset \mathcal{G}$ of a $\Lambda^+$-generic leaf.
We will say that $H_r$ is \emph{the stratum determined by $\Lambda^+$}
and that $\Lambda^+$ is \emph{the attracting lamination associated to $H_r$.}

\begin{lem}[\cite{BestvinaFeighnHandel} Lemma 3.1.13]
    The set $\mathcal{L}(\varphi)$ is finite.
\end{lem}

\begin{proof}
    The proof is identical to \cite[Lemma 3.1.13]{BestvinaFeighnHandel}.
    Choose a relative train track map $f\colon \mathcal{G} \to \mathcal{G}$
    and filtration $\varnothing = G_0 \subset G_1 \subset \cdots \subset G_m = G$
    representing $\varphi$.
    If $f\colon \mathcal{G} \to \mathcal{G}$ is eg-aperiodic,
    then the paragraph before the lemma
    implies a one-to-one correspondence between the exponentially growing strata
    of $f$ and the attracting laminations in $\mathcal{L}(\varphi)$, so the lemma holds.

    Suppose on the other hand that some $H_r$ is not aperiodic.
    Then by general theory (see \cite{Seneta})
    there is a partition of the edges of $H_r$ into $s > 1$ sets
    $P_1,\ldots,P_s$ such that for each edge $E$ in $P_i$,
    the edge path $f_\sharp(E)$ only crosses edges in $G_{r-1}$ and $P_{i+1}$,
    where indices are taken mod $s$.
    The $s$th power of the transition matrix $M^s_r$ is not irreducible,
    so the filtration for $f$ must be enlarged to obtain a filtration for $f^s$.
    Replacing $f$ by $f^s$ has the effect of replacing $H_r$
    by $s$ exponentially growing strata.
    If we choose $s$ maximal, the transition matrix of each of these
    $s$ exponentially growing strata is aperiodic.
    Therefore some iterate of $\varphi$ is represented by an
    eg-aperiodic relative train track map.
    Since $\mathcal{L}(\varphi^p) = \mathcal{L}(\varphi)$ for all $p \ge 1$,
    we are reduced to the previous case.
\end{proof}

\begin{lem}[\cite{BestvinaFeighnHandel} Lemma 3.1.14]
    \label{aperiodicallaperiodic}
    The following are equivalent.
    \begin{enumerate}
        \item Each element of $\mathcal{L}(\varphi)$ is $\varphi$-invariant.
        \item Each element of $\mathcal{L}(\varphi)$ has an attracting neighborhood
            for $\varphi_\sharp$.
        \item Every relative train track map
            $f\colon \mathcal{G} \to \mathcal{G}$ representing $\varphi$
            is eg-aperiodic.
        \item Some relative train track map
            $f\colon \mathcal{G} \to \mathcal{G}$ representing $\varphi$
            is eg-aperiodic.
    \end{enumerate}
\end{lem}

\begin{proof}
    The proof is identical to \cite[Lemma 3.1.14]{BestvinaFeighnHandel}.
    It is clear that item 3 implies item 4.

    Suppose that $f\colon \mathcal{G} \to \mathcal{G}$ is an eg-aperiodic
    relative train track map for $\varphi$,
    that $\Lambda^+ \in \mathcal{L}(\varphi)$ is an attracting lamination
    and that $H_r$ is the exponentially growing stratum associated to $\Lambda^+$.
    If $\lambda$ in $G_r$ is $\Lambda^+$-generic,
    then $f_\sharp(\lambda)$ is $\varphi(\Lambda^+)$-generic.
    Since $H_r$ is the highest stratum crossed by $f_\sharp(\lambda)$,
    \Cref{laminationunique} implies that $\varphi(\Lambda^+) = \Lambda^+$,
    so item 4 implies item 1.

    Suppose that $f\colon \mathcal{G} \to \mathcal{G}$ is a relative train track map
    representing $\varphi$ and that $H_r$ is an exponentially growing stratum
    which is not aperiodic.
    As in the proof of the previous lemma,
    there is a partition of the edges of $H_r$ into $s > 1$ sets
    $P_1,\ldots,P_s$ such that for each edge $E$ of $P_i$,
    the edge path $f_\sharp(E)$ crosses only edges in $G_{r-1}$ and $P_{i+1}$
    with indices taken mod $s$.
    When we replace $f$ by $f^s$,
    the exponentially growing stratum $H_r$ divides into $s$
    exponentially growing strata, one for each $P_i$.
    By \Cref{laminationconstruction},
    these contribute elements to $\mathcal{L}(\varphi)$
    that do not have attracting neighborhoods for the action of $\varphi$.
    Therefore item 2 implies item 3.

    Finally suppose that $\Lambda^+$ is $\varphi$-invariant,
    that $\beta$ is a $\Lambda^+$-generic leaf
    and that $V$ is an attracting neighborhood for $\beta$
    with respect to the action of $\varphi^s$.
    Each $\varphi^i_\sharp(\beta) \in \varphi_\sharp^i(V)$
    is generic with respect to $\Lambda^+$.
    \Cref{laminationunique} implies that $\beta \in \varphi_\sharp^i(V)$.
    Therefore
    \[  U = V \cap \varphi_\sharp(V) \cap \cdots \cap \varphi_\sharp^{s-1}(V) \]
    is a neighborhood of $\beta$ that satisfies $\varphi_\sharp(U) \subset U$
    and $\varphi^s_\sharp(U) \subset V$.
    Therefore $U$ is an attracting neighborhood for $\beta$
    and we see that item 1 implies item 2.
\end{proof}

\begin{lem}[\cite{BestvinaFeighnHandel} Lemma 3.1.15]
    \label{laminationgenericcriterion}
    Assume that $H_r$ is an aperiodic exponentially growing stratum
    for a relative train track map $f\colon \mathcal{G} \to \mathcal{G}$,
    that $\Lambda^+ \in \mathcal{L}(\varphi)$ is associated to $H_r$
    and that $\sigma$ is a leaf of $\Lambda^+$ that is not entirely contained 
    in $G_{r-1}$.
    Then the closure of $\delta$ is all of $\Lambda^+$.
    If $\sigma$ is birecurrent, then it is $\Lambda^+$-generic.
\end{lem}

\begin{proof}
    The proof is identical to \cite[Lemma 3.1.15]{BestvinaFeighnHandel}.
    Fix $k \ge 1$. By \Cref{laminationproperties} item 3,
    each $\Lambda^+$-generic line has a $k$-tiling.
    Since $\sigma$ is a weak limit of $\Lambda^+$-generic leaves,
    $\sigma$ is an increasing union of finite subpaths that have $k$-tilings.
    (This proves that $\sigma$ is not a finite path between vertices with infinite vertex group)
    The argument in the proof of item 3 of \Cref{laminationproperties}
    shows that $\sigma$ has a $k$-tiling.
    If $\sigma \not\subset G_{r-1}$,
    then $\sigma$ must contain at least one $k$-tile.
    Since $k$ is arbitrary, \Cref{laminationunique}
    and the fact that each $k$-tile contains every $\ell$-tile
    for $k-\ell$ sufficiently large (see again the proof of \Cref{laminationunique})
    implies that the closure of $\sigma$ contains
    each $\Lambda^+$-generic line and thus contains $\Lambda^+$.
    It follows that $\sigma$ is not carried by any $F_1$ or $C_2 * C_2$ free factor.
    The fact that $\sigma$ has an attracting neighborhood for the action
    of some iterate of $\varphi$ follows from the fact
    that every neighborhood of a generic leaf is also a neighborhood of $\sigma$.
    If $\sigma$ is birecurrent,
    then all the items in the definition of an attracting lamination are satisfied
    and we conclude that $\sigma$ is $\Lambda^+$-generic.
\end{proof}

\begin{lem}[Lemma 3.1.16 of \cite{BestvinaFeighnHandel}]
    \label{laminationneveracircuit}
    A generic leaf of $\Lambda^+ \in \mathcal{L}(\varphi)$
    is never a circuit.
\end{lem}

\begin{proof}
    The proof is identical to \cite[Lemma 3.1.16]{BestvinaFeighnHandel}.
    A set in $\mathcal{B}$ consisting of a single circuit is closed.
    Therefore if a $\Lambda^+$-generic leaf $\beta$ is a circuit,
    then $\Lambda^+ = \{\beta\}$.
    Choose a relative train track map $f\colon \mathcal{G} \to \mathcal{G}$
    representing $\varphi$.
    Since $\mathcal{L}(\varphi)$ is finite and $\varphi$-invariant,
    the realization of $\lambda$ in $\mathcal{G}$ for $\beta$ is invariant
    for the action of some iterate of $f_\sharp$.
    But $\lambda$ is contained in $G_r$ for some exponentially growing stratum $H_r$,
    is $r$-legal, and crosses edges in $H_r$.
    Therefore the length of $\lambda$ (as a circuit)
    must both be bounded and grow without bound, a contradiction.
\end{proof}

Given a point $\xi \in \partial_\infty(F,\mathscr{A})$,
Guirardel and Horbez \cite[Definition 4.9]{GuirardelHorbez} define the \emph{limit set} of $\xi$
to be the lamination $\Lambda(\xi)$ defined as follows.
(Here we think of $\Lambda(\xi)$ as an $F$-invariant closed subset of $\tilde{\mathcal{B}}$.)
Given a line $\tilde\beta$ with endpoints $\alpha$ and $\omega$,
we have that $\tilde\beta \in \Lambda(\xi)$
if (up to swapping $\alpha$ and $\omega$),
there exists a sequence $\{g_n\}$ of elements of $F$ converging to $\alpha$
such that $g_n.\xi$ converges to $\omega$.
They remark \cite[Remark 4.10]{GuirardelHorbez}
that if we fix a marked graph of groups $\mathcal{G}$
with Bass--Serre tree $\Gamma$ and a point $\tilde x \in \Gamma$,
then if $\tilde\lambda$ realizes $\tilde\beta$ in $\Gamma$
we have that $\tilde\beta \in \Lambda(\xi)$
if and only if $\tilde\lambda$ is a limit of translates of the ray $\tilde R_{\tilde x,\xi}$.
Bounded cancellation implies that $\varphi_\sharp(\Lambda(\xi)) = \lambda(\hat\Phi(\xi))$.
If $\xi$ belongs to $\fix(\hat\Phi)$, then $\Lambda(\xi)$ is $\varphi_\sharp$-invariant.

We would like to relate the construction of attracting laminations
to the idea of attracting \emph{fixed points} at infinity for the action of
some lift $\tilde f\colon \Gamma \to \Gamma$.
This is accomplished in the following lemma.

\begin{lem}[cf. Lemma 2.13 of \cite{FeighnHandel}]
    \label{fixedpointstolaminations}
   Suppose that $H_r$ is an exponentially growing stratum
   of a relative train track map $f\colon \mathcal{G} \to \mathcal{G}$,
   that $\tilde f\colon \Gamma \to \Gamma$ is a lift
   and that $\tilde v \in \fix(\tilde f)$.
   \begin{enumerate}
       \item If $E$ is an oriented edge in $H_r$ 
           and $\tilde E$ is a lift that determines
           a fixed direction at $\tilde v$,
           then there is a unique ray $\tilde R \subset \Gamma$ that begins with
           $\tilde E$, intersects $\fix(\tilde f)$ only in $\tilde v$,
           converges to an attractor $\xi \in \fix(\hat f)$
           and whose image in $\mathcal{G}$ is $r$-legal and has height $r$.
           The limit set of $\xi$ is the unique attracting lamination associated to $H_r$.
        \item Suppose that $E'$ is another oriented edge in $H_r$,
            that $\tilde E' \ne \tilde E$ determines a fixed direction at $\tilde v$
            and that $\tilde R'$ is the ray associated to $\tilde E'$
            as in item 1.
            Suppose further that the projection of the turn 
            $(\bar {\tilde E},\tilde E')$ is contained in the path
            $f^k_\sharp(E'')$ for some $k \ge 1$ and some edge $E''$ of $H_r$.
            Then the line $\bar {\tilde R}\tilde R'$
            projects to a generic leaf of the attracting lamination $\Lambda^+$
            associated to $H_r$.
   \end{enumerate}
\end{lem}

\begin{proof}
    We follow \cite[Lemma 2.13]{FeighnHandel}.
    \Cref{rttlemma} and \hyperlink{EG-i}{(EG-i)}
    imply that $\tilde f(\tilde E) = \tilde E\cdot \tilde \mu_1$
    for some non-trivial $r$-legal path $\mu_1$ of height $r$
    that ends with an edge of $H_r$.
    If we apply \Cref{rttlemma} again,
    we have $\tilde f^2_\sharp(\tilde E) = \tilde E \cdot \tilde\mu_1 \cdot\tilde\mu_2$
    where $\mu_2$ is an $r$-legal path of height $r$ that ends with an edge of $H_r$.
    Iterating this produces a nested increasing sequence of paths
    $\tilde E \subset \tilde f(\tilde E) \subset \tilde f_\sharp(\tilde E) \subset \cdots$
    whose union is a ray $\tilde R$ that converges to a point $\xi \in \fix(\hat f)$.
    This point is an attractor:
    bounded cancellation shows that if we think of $\xi$ as an infinite word,
    we are in the ``superlinear attractor'' case of \Cref{GJLLprop}.

    If $\tilde R'$ is another ray whose image in $\mathcal{G}$ is $r$-legal
    and has height $r$ that begins with $\tilde E$
    and converges to some point $\xi' \in \fix(\hat f)$,
    then $\tilde R'$ has a splitting into terms that project to either edges in $H_r$
    or maximal subpaths in $G_{r-1}$.
    The edge $\tilde E$ is a term of this splitting,
    and since $\tilde f_\sharp(\tilde R') = \tilde R'$,
    one can argue by induction that $\tilde f^k_\sharp(\tilde E)$
    is an initial segment of $\tilde R'$ for all $k$,
    which implies that $\tilde R' = \tilde R$.
    This proves the uniqueness in item 1.

    Since $f_\sharp(E)$ contains $E$,
    the stratum $H_r$ is aperiodic,
    so there is a unique attracting lamination $\Lambda^+$ associated to $H_r$.
    Recall the definition of \emph{tiles} 
    from before the proof of \Cref{laminationconstruction}.
    Suppose $\tilde\lambda$ is the realization in $\Gamma$ 
    of a generic leaf of $\Lambda^+$.
    By \Cref{laminationunique}, $\tilde\lambda$ has a neighborhood basis that consists of lifts of tiles.
    Let $\tilde\tau\subset \tilde\lambda$ be such a tile.
    By construction, every tile occurs infinitely often in $\tilde R$,
    so some translate of $\tilde R$ belongs to $N(\tilde\gamma)$.
    It follows that $\lambda \in \Lambda(\xi)$, so $\Lambda^+ \subset \Lambda(\xi)$.
    We will show that $\xi$ is an endpoint
    of a leaf $\tilde\beta$ such that $\beta \in \Lambda^+$.
    By \Cref{laminationgenericcriterion}, since $\beta$ is not entirely contained in $G_{r-1}$,
    the closure of $\beta$ is all of $\Lambda^+$.
    By \cite[Lemma 4.12]{GuirardelHorbez},
    we have that $\Lambda(\xi) \subset \Lambda^+$,
    completing the proof of item 1.

    Arguing as in the proof of \Cref{egvalence},
    one can show that there is a legal turn based at $\tilde v$
    projecting into $G_r$ and
    in the image of $D\tilde f$,
    one of whose directions is $\tilde E$.
    Say the other direction is $\tilde E'$.
    If the direction $\tilde E'$ is periodic,
    we may iterate until it is fixed,
    determining a ray $\tilde R'$.
    If the original edge $\tilde E'$ belonged to $H_r$,
    this new edge, still called $\tilde E'$ belongs to $H_r$,
    and we are in the case of item 2.

    By construction, every subpath of $\bar{\tilde R}\tilde R'$
    is contained in a lift of a tile except possibly those subpaths crossing the turn
    $(\bar{\tilde E},\tilde E')$.
    Since we assume, as in item 2,
    that the projection of the turn $(\bar {\tilde E},\tilde E')$
    is contained in the path $f^k_\sharp(E'')$ for some $k \ge 1$
    and some edge $E''$ of $H_r$,
    then the line $\bar{\tilde R}\tilde R'$ is contained in $\Lambda^+$.
    In this situation, the line
    is birecurrent by construction.
    \Cref{laminationgenericcriterion} implies that it is $\Lambda^+$-generic,
    completing the proof of item 2.

    If $\tilde E'$ instead belongs to $G_{r-1}$,
    then the argument in the previous paragraph applies,
    except the line $\bar{\tilde R}\tilde R'$ is not birecurrent.
    Nonetheless it is contained in $\Lambda^+$.

    If finally the direction $\tilde E'$ is not periodic,
    then $\tilde v$ has infinite valence, and we may take $\tilde R$
    to be the line in question; again it is contained in $\Lambda^+$,
    but it is not birecurrent because it is not bi-infinite.
\end{proof}

For later use, we record the following lemma.

\begin{lem}
    \label{minimalfreefactorcarries}
    There is a unique free factor $F^i$ of minimal rank
    whose conjugacy class $[[F^i]]$ carries every line in $\Lambda^+$.
\end{lem}

\begin{proof}
    We are grateful to Lee Mosher for suggesting the following proof.
    Since an attracting lamination $\Lambda^+$ is the closure of a single line $\beta$,
    any free factor that carries $\beta$ carries every line in $\Lambda^+$.
    Since $[[F]]$ carries $\beta$,
    there is at least one free factor of minimal rank that carries $\beta$.
    Suppose that $\beta$ is carried by both $[[F^1]]$ and $[[F^2]]$.
    For $i = 1,2$ choose a marked thistle $\mathcal{G}_i$ with vertex $\star_i$
    and a subgraph $K_i$ so that the marking identifies $\pi_1(K_i,\star_i)$ with $F^i$.
    Then in $\Gamma_i$, the Bass--Serre tree for $\mathcal{G}_i$,
    the $F^i$-minimal subtree is disjoint from all its translates
    by elements of $F \setminus F^i$.
    It follows under the identification of $\partial\Gamma_i$ with $\partial(F,\mathscr{A})$
    that $\partial(F^i,\mathscr{A}|_{F^i}) \subset \partial(F,\mathscr{A})$
    is disjoint from all its translates by elements of $F \setminus F^i$.
    Each translate $c.\partial(F^i,\mathscr{A}|_{F^i})$ is stabilized
    by the corresponding conjugate $cF^ic^{-1}$.
    Let $\tilde\beta$ be a line in $\tilde{\mathcal{B}}$ that lifts $\beta$.
    Since $\beta$ is carried by $F^i$,
    there is a unique translate $c_i.\partial(F^i,\mathscr{A}|_{F^i})$
    that contains the endpoints of $\tilde\beta$.
    Let $F^3 = c_1F^1c_1^{-1} \cap c_2F^2c_2^{-1}$.
    The $F^3$-minimal subtree contains the realization of $\tilde\beta$ in $T$.
    It follows that $F^3$ is a free factor (of positive complexity)
    whose conjugacy class $[[F^3]]$ carries $\beta$.
    This contradicts minimality, so we conclude that $[[F^1]] = [[F^2]]$.
\end{proof}

If $\mathcal{F}$ is a free factor system $\mathcal{F} = \{[[F^1]],\ldots,[[F^\ell]]\}$,
define the \emph{complexity} of $\mathcal{F}$ to be zero if $\mathcal{F}$ is trivial,
and to be the complexities of the $F^i$ rearranged in non-increasing order
if $\mathcal{F}$ is nontrivial.
Order the complexities of free factor systems of $F$ lexicographically.

\begin{cor}
    \label{minimalcomplexityfreefactorsystem}
    For any subset $B \subset \mathcal{B}$,
    there is a unique free factor system $\mathcal{F}(B)$
    of minimal complexity that carries every element of $B$.
\end{cor}

\begin{proof}
    Since $[[F]]$ carries every element of $B$,
    there is at least one free factor system $\mathcal{F}_1$
    of minimal complexity that carries every element of $B$.
    Suppose $\mathcal{F}_2$ is another free factor system of minimal complexity
    that carries every element of $B$.
    Define $\mathcal{F}_1 \wedge \mathcal{F}_2$
    to be the set of free factors of positive complexity in the set
    $\{[[F^i \cap (F^j)^c]] : [[F^i]] \in \mathcal{F}_1,\ [[F^j]] \in \mathcal{F}_2,\ c \in F\}$.
    The Kurosh subgroup theorem implies that $\mathcal{F}_1 \wedge \mathcal{F}_2$
    is a possibly empty free factor system of $F$.
    The argument in the previous lemma implies that
    $\mathcal{F}_1 \wedge \mathcal{F}_2$ carries every element of $B$.

    Notice that if $\mathcal{F}_1 \wedge \mathcal{F}_2 \ne \mathcal{F}_1$,
    then its complexity is strictly smaller:
    each free factor $F^i \cap (F^j)^c$ of positive complexity is a free factor of $F^i$,
    and thus it either equals $F^i$ or has strictly smaller complexity.
    The minimality assumption therefore shows that $\mathcal{F}_1 = \mathcal{F}_2$,
    proving uniqueness.
\end{proof}

Our final goal in this section is the following proposition.
Suppose that $\Lambda^+$ is an attracting lamination 
for some element of $\out(F,\mathscr{A})$.
Define the \emph{stabilizer} of $\Lambda^+$ to be
\[  \stab(\Lambda^+) = \{\psi \in \out(F,\mathscr{A}) 
: \psi_\sharp(\Lambda^+) = \Lambda^+\}. \]

\begin{prop}[cf. Corollary 3.3.1 of \cite{BestvinaFeighnHandel}]
    \label{PFhomomorphism}
    There is a homomorphism \[\pf_{\Lambda^+}\colon \stab(\Lambda^+) \to \mathbb{Z}\]
    such that we have $\psi \in \ker(\pf_{\Lambda^+})$ if and only if
    $\Lambda^+ \notin \mathcal{L}(\psi)$ and $\Lambda^+ \notin \mathcal{L}(\psi^{-1})$.
\end{prop}

Assume that $f\colon \mathcal{G} \to \mathcal{G}$
and $\varnothing = G_0 \subset G_1 \subset \cdots \subset G_m = G$
are a relative train track map and filtration for an element of $\stab(\Lambda^+)$
and that $\Lambda^+$ is the attracting lamination associated to the
(necessarily aperiodic) exponentially growing stratum $H_r$.
For any path $\sigma \subset \mathcal{G}$,
define $\operatorname{EL}_r(\sigma)$ to be the edge length of $\sigma$
counting only the edges of $H_r$ that are contained in $\sigma$.
We say that $\psi \in \stab(\Lambda^+)$ 
\emph{asymptotically expands $\Lambda^+$ by the factor $\mu$}
if for every choice of $\varnothing = G_0 \subset G_1 \subset \cdots \subset G_m = G$
and $f\colon \mathcal{G} \to \mathcal{G}$ as above,
every topological representative $g\colon \mathcal{G} \to \mathcal{G}$ of $\psi$
and for all $\eta > 0$, we have
\[  \mu - \eta < \frac{\operatorname{EL}_r(g_\sharp(\sigma))}{\operatorname{EL}_r(\sigma)}
< \mu + \eta \]
whenever $\sigma$ is contained in a $\Lambda^+$-generic leaf
and $\operatorname{EL}_r(\sigma)$ is sufficiently large.

For the remainder of this section,
we fix the relative train track map $f\colon \mathcal{G} \to \mathcal{G}$,
the filtration $\varnothing = G_0 \subset G_1 \subset \cdots \subset G_m = G$,
the stratum $H_r$ and the attracting lamination $\Lambda^+$.
The following proposition has the useful corollary that
the Perron--Frobenius eigenvalue associated to an aperiodic exponentially growing stratum
of a relative train track map $f\colon \mathcal{G} \to \mathcal{G}$
depends only on $\varphi$ and the element of $\mathcal{L}(\varphi)$ that is associated
to the stratum.

\begin{prop}[cf. Proposition 3.3.3 of \cite{BestvinaFeighnHandel}]
    \label{PFexphomomorphism}
    We have the following.
    \begin{enumerate}
        \item Every $\psi \in \stab(\Lambda^+)$
            asymptotically expands $\Lambda^+$ by some factor $\mu = \mu(\psi)$.
        \item $\mu(\psi\psi') = \mu(\psi)\mu(\psi')$.
        \item $\mu(\psi) > 1$ if and only if $\Lambda^+ \in \mathcal{L}(\psi)$.
        \item If $\Lambda^+ \in \mathcal{L}(\psi)$
            and $f'\colon \mathcal{G} \to \mathcal{G}$ is a relative train track map
            for $\psi$, then $\mu(\psi) = \mu'_s$
            is the Perron--Frobenius eigenvalue for the transition matrix
            $M'_s$ of the exponentially growing stratum $H'_s$ associated to $\Lambda^+$.
    \end{enumerate}
\end{prop}

\begin{proof}[Proof of \Cref{PFhomomorphism}]
    The proof is identical to the proof of \cite[Corollary 3.3.1]{BestvinaFeighnHandel}.
    Define $\pf_{\Lambda^+}(\psi) = \log\mu(\psi)$.
    \Cref{PFexphomomorphism} and the observation that 
    \[\pf_{\Lambda^+}(\psi^{-1}) = -\pf_{\Lambda^+}(\psi)\]
    implies that each $\mu(\psi)$ (or its multiplicative inverse)
    other than $1$ occurs as the Perron--Frobenius eigenvalue
    for an irreducible matrix of uniformly bounded size.
    It follows (\cite[p. 37]{BestvinaHandel} 
    or the argument in the proof of \cite[Theorem 2.2]{Myself})
    that the image of $\pf_{\Lambda^+}$ is an infinite discrete subset of $\mathbb{R}$
    and thus is isomorphic to $\mathbb{Z}$.
    Abusing notation, identify the image with $\mathbb{Z}$
    and call the resulting homomorphism $\pf_{\Lambda^+}$.
    It is clear from \Cref{PFexphomomorphism}
    that $\psi \in \ker(\pf_{\Lambda^+})$
    if and only if $\Lambda^+ \notin \mathcal{L}(\psi)$
    and $\Lambda^+ \notin \mathcal{L}(\psi^{-1})$.
\end{proof}

Let $E_i$ be an edge of $H_r$ 
and let $\mu_r$ be the Perron--Frobenius eigenvalue
for $M_r$, the transition matrix associated to $H_r$.
The Perron--Frobenius theorem \cite{Seneta}
implies that the matrices $\mu_r^{-n}M_r^n$
converge to a matrix $M^\star$ with the property
that all of its columns are multiples of each other.
Let $A = (a_i)$ be the vector obtained from a column of $M^\star$
by multiplying so that the sum $\sum a_i$ is $1$;
call it a \emph{frequency vector.}
Let $\tau^k_i$ be the $k$-tile $f^k_\sharp(E_i)$.

A $0$-tiling of a $\Lambda^+$-generic leaf $\lambda$
is a decomposition into edges of $H_r$,
vertex group elements,
and maximal subpaths in $G_{r-1}$;
the leaf $\lambda$ has a unique $0$-tiling.
This yields a $1$-tiling of $f_\sharp(\lambda)$
and inductively a $k$-tiling of $f_\sharp^k(\lambda)$
called the \emph{standard $k$-tiling} of $f^k_\sharp(\lambda)$.
Notice that $\lambda = f^k_\sharp(\gamma_k)$
for some $\Lambda^+$-generic leaf $\gamma_k$.
It follows that every $\Lambda^+$-generic leaf
has a standard $k$-tiling for $k \ge 0$.

Suppose $\sigma$ is a finite subpath of a $\Lambda^+$-generic leaf $\lambda$.
Let $\alpha_{ik}(\sigma)$ be the proportion,
among those $k$-tiles in the standard $k$-tiling of $\lambda$
that are entirely contained within $\sigma$,
of the tiles that are equal to $\tau^k_i$.

\begin{lem}[cf. Lemma 3.3.5 of \cite{BestvinaFeighnHandel}]
    \label{laminationestimate}
    Given $\epsilon > 0$ and $k \ge 0$
    and a $\Lambda^+$-generic leaf $\lambda$,
    if $\sigma$ is a finite subpath of $\lambda$
    with $\operatorname{EL}_r(\sigma)$ sufficiently large,
    then $a_i - \epsilon < \alpha_{ik}(\sigma) < a_i + \epsilon$.
\end{lem}

\begin{proof}
    We follow the proof of \cite[Lemma 3.3.5]{BestvinaFeighnHandel}.
    Suppose first that $\sigma$ is the $\ell$-tile $\tau^\ell_j$
    and that $k = 0$.
    Observe that by \Cref{rttlemma} (cf. \cite[Lemma 3.1.8(4)]{BestvinaFeighnHandel}),
    the $(i,j)$-entry of $M^\ell_r$ is the number of times
    $\tau^\ell_j$ crosses the $i$th edge in either direction,
    and that $\alpha_{i0}(\tau^\ell_j)$
    is the $(i,j)$-entry of $M^\ell_r$ divided by the sum of the entries
    in the $j$th column of $M^\ell_r$.
    By the Perron--Frobenius theorem, we have that $\alpha_{i0}(\tau^{\ell}_j)$
    converges to $a_i$ as $\ell$ increases.
    Since $\operatorfont{EL}_r(\tau^\ell_j)$ increases with $\ell$,
    we see that the lemma holds in this case.
    Observe further that $\alpha_{ik}(\tau^{\ell+k}_j) = \alpha_{i0}(\tau^{\ell}_j)$
    by \Cref{rttlemma}, so for $\ell$ sufficiently large,
    the lemma holds for $\ell$-tiles with arbitrary $k \ge 0$.
    An easy calculation shows that if the results of the lemma hold
    with fixed $k$ and $\epsilon$ for each $\ell$-tile,
    then the lemma holds when $\sigma$ has an $\ell$-tiling
    induced by the standard $\ell$-tiling of $\lambda$.
    (That is, $\sigma$ is a union of $\ell$-tiles, vertex group elements
    and paths in $G_{r-1}$.)
    Finally for general $\sigma$,
    in the standard $\ell$-tiling of $\lambda$,
    at most two $\ell$-tiles intersect $\sigma$ 
    but are not entirely contained in $\sigma$.
    If $\operatorname{EL}_r(\sigma)$ is sufficiently large (compared with $\ell$)
    the contribution of these two tiles is negligible,
    so the lemma follows.
\end{proof}

\begin{lem}[cf. Lemma 3.3.6 of \cite{BestvinaFeighnHandel}]
    \label{nongenericbound}
    Suppose that $g\colon \mathcal{G} \to \mathcal{G}$ 
    is a topological representative of $\psi \in \out(F,\mathscr{A})$
    and that $\Lambda^+$ is $\psi$-invariant.
    Then there exists a constant $C_1$ depending only on $g$
    satisfying $\operatorname{EL}_r(g_\sharp(\delta)) < C_1$
    for any subpath $\delta \subset G_{r-1}$ of a $\Lambda^+$-generic leaf.
\end{lem}

\begin{proof}
    We follow the proof of \cite[Lemma 3.3.6]{BestvinaFeighnHandel}.
    Suppose that the lemma fails.
    Then there exist $\Lambda^+$-generic leafs $\lambda_j$
    and finite subpaths $\delta_j \subset G_{r-1}$ of $\lambda_j$
    such that $g_\sharp(\delta_j)$ contains at least one edge in $H_r$,
    even after the first and last $j$ edges have been removed.
    If we pass to a subsequence,
    we may assume that the paths $\delta_j$ form an increasing sequence
    whose union is a line $\delta^\star \subset G_{r-1}$
    with the property that $g_\sharp(\delta^\star) \not\subset G_{r-1}$.
    But note that $\delta^{\star}$ is a leaf of $\Lambda^+$
    whose closure is not all of $\Lambda^+$
    (since it contains no edges of $H_r$),
    so $g_\sharp(\delta^\star)$ is also a line in $\Lambda^+$ 
    whose closure is not all of $\Lambda^+$.
    But this contradicts \Cref{laminationgenericcriterion}
    since $g_\sharp(\delta^\star)$ is not entirely contained in $G_{r-1}$.
\end{proof}

\begin{proof}[Proof of \Cref{PFexphomomorphism}]
    We follow the proof of \cite[Proposition 3.3.3]{BestvinaFeighnHandel}.
    Let us recall our standing assumption:
    $f\colon \mathcal{G} \to \mathcal{G}$ and
    $\varnothing = G_0 \subset G_1 \subset \cdots \subset G_m = G$
    are a relative train track map and associated filtration,
    $H_r$ is an aperiodic exponentially growing stratum
    with associated lamination $\Lambda^+$,
    and $g\colon \mathcal{G} \to \mathcal{G}$
    is a topological representative of $\psi \in \stab(\Lambda^+)$.
    Let $a_i$ be the frequency vector coefficient 
    corresponding to the edge $E_i$ of $H_r$,
    and let $\tau^k_i$ be the $k$-tile $f^k_\sharp(E_i)$.
    Suppose $\sigma$ is a finite subpath of a $\Lambda^+$-generic leaf.
    Define
    \[  \mu_k = \frac{\sum_i a_i\operatorname{EL}_r
    (g_\sharp(\tau^k_i))}{\sum_i a_i\operatorname{EL}_r(\tau^k_i)}.\]
    Fixing $\epsilon > 0$, we will show that if $k$ is sufficiently large
    (depending on $\epsilon$)
    and $\operatorname{EL}_r(\sigma)$ is sufficiently large
    (depending on $\epsilon$ and $k$), then
    \[  (1-\epsilon)\mu_k \le 
        \frac{\operatorname{EL}_r(g_\sharp(\sigma))}{\operatorname{EL}_r(\sigma)} 
    \le (1 + \epsilon)\mu_k. \]
    It follows that the $\mu_k$ form a Cauchy sequence
    and thus converge to a limit $\mu = \lim_{k\to\infty}\mu_k$.
    Recall that in order to show that $\psi$ asymptotically expands
    $\Lambda^+$ by the factor $\mu$, we must show that for every
    choice of relative train track map $f\colon \mathcal{G} \to \mathcal{G}$
    and filtration $\varnothing = G_0 \subset G_1 \subset \cdots \subset G_m = G$
    and every topological representative $g\colon \mathcal{G} \to \mathcal{G}$
    of $\psi$, we have for all $\eta > 0$,
    \[  \mu - \eta 
        < \frac{\operatorname{EL}_r(g_\sharp(\sigma))}{\operatorname{EL}_r(\sigma)} 
    < \mu + \eta \]
    whenever $\sigma$ is contained in a $\Lambda^+$-generic leaf
    and $\operatorname{EL}_r(\sigma)$ is sufficiently large.
    The first equation shows that the second holds
    for $\sigma \subset \lambda$ sufficiently long
    relative to $\eta$ and our particular choice of $f$ and $g$.

    Let $\operatorname{bcc}(g)$ be a bounded cancellation constant 
    for $g\colon \mathcal{G} \to \mathcal{G}$
    and let $C_1$ be the constant from \Cref{nongenericbound}.
    Given $\epsilon > 0$,
    write $x \sim_\epsilon y$ to mean that $|x - y|$ is small relative to $\epsilon$.
    We may choose $k$ so large that for all $i$, we have
    \begin{enumerate}
        \item $\displaystyle\frac{C_1}{\operatorname{EL}_r(\tau^k_i)} \sim_\epsilon 0$
            and
        \item $\displaystyle\frac{\operatorname{bcc}(g)}{\operatorname{EL}_r(\tau^k_i)}
            \sim_\epsilon 0$.
    \end{enumerate}
    We may furthermore choose $\sigma \subset \lambda$ so long that
    \begin{enumerate}[resume]
        \item $\alpha_{ik}(\sigma) \sim_\epsilon a_i$ and
        \item $\displaystyle
            \frac{\operatorname{EL}_r(\tau^k_i)}{\operatorname{EL}_r(\sigma)}
            \sim_\epsilon 0$.
    \end{enumerate}

    As in \cite{BestvinaFeighnHandel},
    we make simplifying assumptions in the calculation
    that introduce small errors.
    First, we assume that $\sigma$ begins and ends with a complete $k$-tile 
    from the standard $k$-tiling of $\lambda$;
    that is, $\sigma$ is a concatenation of $k$-tiles $\gamma_j$,
    vertex group elements
    and maximal subpaths $\delta_\ell \subset G_{r-1}$.
    The error in making this assumption is at most twice the number of $H_r$-edges
    in a $k$-tile, so is controlled by item 4.
    The second assumption is that the approximation in item 3 is exact.
    The third assumption is that $\operatorname{EL}_r(g_\sharp(\delta_\ell)) = 0$
    for each maximal subpath $\delta_\ell \subset G_{r-1}$.
    Here the error is controlled by item 1.
    (The true upper bound is $C_1$,
    and by assuming that $C_1$ is much smaller than $\operatorname{EL}_r(\tau^k_i)$,
    we are in effect able to ignore each $\operatorname{EL}_r(g_\sharp(\delta_\ell))$.)
    The final assumption is that $g_\sharp(\sigma) \subset g_\sharp(\lambda)$
    is a tight concatenation of the paths $g_\sharp(\gamma_j)$, 
    vertex group elements
    and the paths $g_\sharp(\delta_\ell)$
    with no cancellation beyond multiplication in vertex groups.
    This produces an error bounded by $2\operatorname{bcc}(g) N(\sigma)$,
    where $N(\sigma)$ is the number of $k$-tiles $\sigma$ contains.
    This error is controlled by item 2.

    Given these assumptions, we have
    \[  \operatorname{EL}_r(g_\sharp(\sigma)) = 
    \sum_i N(\sigma)a_i \operatorname{EL}_r(g_\sharp(\tau^k_i)). \]
    If we let $g$ be the identity, we see that our assumptions calculate
    \[  \operatorname{EL}_r(\sigma) =
    \sum_i N(\sigma)a_i \operatorname{EL}_r(\tau^k_i), \]
    so we see that
    \[  \frac{\operatorname{EL}_r(g_\sharp(\sigma))}{\operatorname{EL}_r(\sigma)}
    = \mu_k, \]
    verifying the equation above.

    If $g^\star\colon \mathcal{G} \to \mathcal{G}$ is another topological
    representative of $\psi$, then there are lifts
    $\tilde g \colon \Gamma \to \Gamma$ and $\tilde g^\star\colon \Gamma \to \Gamma$
    such that the distance between $\tilde g(\tilde x)$ and $\tilde g^\star(\tilde x)$
    is bounded independently of $\tilde x$,
    from which it follows that 
    $\operatorname{EL}_r(g_\sharp(\sigma)) 
    - \operatorname{EL}_r(g^\star_\sharp(\sigma))$
    is bounded independently of $\sigma$,
    so hence $\mu$ does not depend on the choice of $g$.

    If instead $f^\star \colon \mathcal{G}^\star \to \mathcal{G}^\star$
    and $\varnothing = G_0^\star \subset G_1^\star \subset \cdots \subset G_M^\star 
    = G^\star$
    is another relative train track map and filtration representing $\varphi$,
    $\Lambda^+$ is the attracting lamination associated to $H_s^\star$
    and $\operatorname{EL}_s^\star$ is the edge length function
    that counts edges of $H^\star_s$ in $\mathcal{G}^\star$,
    we will argue that the following holds.

    Choose a homotopy equivalence $h\colon \mathcal{G} \to \mathcal{G}^\star$
    that respects the markings and maps edges to tight edge paths.
    One  argues exactly as above that there exists a positive constant $\nu$
    such that for all $\epsilon > 0$, we have
    \[  (1-\epsilon)\nu < 
    \frac{\operatorname{EL}_s^\star(h_\sharp(\sigma))}{\operatorname{EL}_r(\sigma)}
    < (1+\epsilon)\nu \]
    whenever $\sigma$ is contained in a $\Lambda^+$-generic leaf
    and $\operatorname{EL}_r(\sigma)$ is sufficiently large.
    We leave the details to the reader.

    Suppose that $\hat g\colon \mathcal{G}^\star \to \mathcal{G}^\star$
    represents $\psi$.
    Given any finite path $\sigma \subset G_r$,
    we have that $h_\sharp g_\sharp(\sigma)$ and $\hat g_\sharp h_\sharp(\sigma)$
    differ by initial and terminal segments of uniformly bounded size.
    This, together with the equation above, shows that
    \[  \frac{\operatorname{EL}_s^\star
    (\hat g_\sharp h_\sharp(\sigma))}{\operatorname{EL}_s^\star(h_\sharp(\sigma))}
    \sim \frac{\operatorname{EL}_s^\star
    (h_\sharp g_\sharp(\sigma))}{\operatorname{EL}_s^\star(h_\sharp(\sigma))}
    \sim \frac{\operatorname{EL}_r(g_\sharp(\sigma))}{\operatorname{EL}_r(\sigma)}, \]
    where the error of approximation in each $\sim$ goes to $0$
    as $\operatorname{EL}_r(\sigma)$ or equivalently 
    $\operatorname{EL}_s^\star(h_\sharp(\sigma))$ grows.
    We conclude that $\mu$ is independent of the choice of
    relative train track map and filtration
    and thus that $\psi$ asymptotically expands $\Lambda^+$ by the factor $\mu$,
    proving part 1 of \Cref{PFexphomomorphism}.

    Now suppose that $g\colon \mathcal{G} \to \mathcal{G}$
    and $g'\colon \mathcal{G} \to \mathcal{G}$
    are topological representatives for $\psi$ and $\psi'$ respectively.
    If $\sigma \subset \mathcal{G}$ is contained in a $\Lambda^+$-generic leaf,
    then (by bounded cancellation) there is a subpath $\sigma'$ of $g_\sharp(\sigma)$
    that is contained in a $\Lambda^+$-generic leaf
    and that differs from $g_\sharp(\sigma)$ only
    in an initial and terminal segment of uniformly bounded length.
    Similarly $g'_\sharp g_\sharp(\sigma)$ differs from $g'_\sharp(\sigma')$
    only in an initial and terminal segment of uniformly bounded length.
    Therefore $\mu(\psi\psi') = \mu(\psi)\mu(\psi')$, proving part 2.

    To prove part 3, we will show that $\mu = \mu(\psi) > 1$ holds if and only if we have
    $\Lambda^+ \in \mathcal{L}(\psi)$.
    Suppose first that $\mu > 1$, and that $\sigma_0$ is a subpath of a $\Lambda^+$-generic leaf.
    Let $\sigma_1$ be the finite subpath obtained from $g_\sharp(\sigma_0)$
    by removing initial and terminal segments of length $\operatorname{bcc}(g)$.
    Bounded cancellation implies that $g_\sharp(N(\sigma_0)) \subset N(\sigma_1)$.
    Since $\mu > 1$,
    if $\operatorname{EL}_r(\sigma_0)$ is sufficiently large,
    we have $\operatorname{EL}_r(\sigma_1) > \operatorname{EL}_r(\sigma_0)$.
    Inductively we may produce subpaths $\sigma_k$
    such that $\operatorname{EL}_r(\sigma_k)$ is increasing and such that
    $g_\sharp(N(\sigma_{k-1})) \subset N(\sigma_k)$,
    from which it follows that $g^k_\sharp(N(\sigma_0)) \subset N(\sigma_k)$.
    Since $g^k_\sharp(\lambda)$ is $\Lambda^+$-generic,
    \Cref{laminationunique} implies that $\lambda \in N(\sigma_k)$ for all $k$.

    Since $\lambda$ has an exhaustion by tiles by \Cref{laminationproperties},
    $\sigma_0$ is contained in an $\ell$-tile in $\lambda$ for some $\ell$.
    If $k$ is sufficiently large, since $\operatorname{EL}_r(\sigma_k)$ grows with $k$,
    by \Cref{laminationestimate}, $\sigma_k$ contains every $\ell$-tile at least once.
    Therefore in particular $N(\sigma_k) \subset N(\sigma_0)$.
    Recall from \Cref{laminationunique} that the tiles form a neighborhood basis for $\lambda$.
    The same argument shows that given $\ell'$, there is $k'$ sufficiently large
    so that $\sigma_{k'}$ contains every $\ell'$-tile,
    so the neighborhoods $N(\sigma_j)$ as $j$ varies form a neighborhood basis for $\lambda$.
    It follows that $N(\sigma_0)$ is an attracting neighborhood for the action of $\psi^k_\sharp$.
    This shows that $\Lambda^+ \in \mathcal{L}(\psi)$.

    Finally, if $\Lambda^+ \in \mathcal{L}(\psi)$, we may by the independence of choices in item 1
    assume that the relative train track map $f\colon \mathcal{G} \to \mathcal{G}$
    used to compute $\mu = \mu(\psi)$ represents $\psi$,
    that $g = f$, and that $\gamma$ is a $k$-tile $\tau^k_i$ for some large $k$.
    Under these assumptions, we have that
    $\operatorname{EL}_r(g_\sharp(\sigma))$ is the sum of the $i$th column of
    $M^{k+1}_r$, the $(k+1)$-st power of the transition matrix.
    Meanwhile $\operatorname{EL}_(\sigma)$ is the $i$th column sum of $M^k_r$.
    The Perron--Frobenius theorem implies that 
    \[\frac{\operatorname{EL}_r(g_\sharp(\sigma))}{\operatorname{EL}_r(\sigma)} \to \mu_r, \]
    the Perron--Frobenius eigenvalue of $M_r$,
    as $k \to \infty$.
    Thus $\mu = \mu_r > 1$, completing the proof of item 3 and also proving item 4.
\end{proof}

\section{Principal automorphisms}\label{principalsection}
The purpose of this section is to develop the so-called ``Nielsen theory''
for outer automorphisms of free products.
We define principal automorphisms and rotationless outer automorphisms,
and show that in cases of particular interest every outer automorphism
has a rotationless power.
(Rotationless outer automorphisms were
originally called ``forward rotationless'' in~\cite{FeighnHandel},
but see~\cite{FeighnHandelAlg}.)

\paragraph{Principal automorphisms.}
Given an outer automorphism $\varphi\in \out(F,\mathscr{A})$
and an automorphism $\Phi\colon (F,\mathscr{A}) \to (F,\mathscr{A})$
representing it,
let $\fix_N(\hat\Phi) \subset \fix(\hat\Phi)$
denote the subset of the set of fixed points of $\hat\Phi$ 
containing all fixed points in $V_\infty(F,\mathscr{A})$
along with those fixed points in $\partial_\infty(F,\mathscr{A})$
that are not repellers.
We say that $\Phi$ is a \emph{principal automorphism}
if either of the following conditions holds.
\begin{enumerate}
    \item $\fix_N(\hat\Phi)$ contains at least three points.
    \item $\fix_N(\hat\Phi)$ is a two-point set that is neither
        the endpoints of an axis nor the endpoints of a lift
        of a generic leaf $\beta$ of an attracting lamination 
        $\Lambda^+ \in \mathcal{L}(\varphi)$.
\end{enumerate}
If $f\colon \mathcal{G} \to \mathcal{G}$ is a relative train track map
representing $\varphi$,
then we call the lift $\tilde f\colon \Gamma \to \Gamma$ 
corresponding to a principal automorphism $\Phi$ a \emph{principal lift.}

\paragraph{Nielsen almost equivalence and isogredience.}
In \cite{FeighnHandel}, Feighn and Handel
prove a kind of dictionary between certain equivalence classes
in $\fix(f)$ called \emph{Nielsen classes}
and \emph{isogredience classes} of principal automorphisms.
A pair of automorphisms $\Phi_1$ and $\Phi_2$ are \emph{isogredient}
if there exists $c \in F$
such that $\Phi_2 = i_c\Phi_1i_c^{-1}$,
where $i_c$ is the inner automorphism $x \mapsto cxc^{-1}$.
In the language of lifts,
we say that $\tilde f_1$ is isogredient to $\tilde f_2$
if $\tilde f_2 = T_c\tilde f_1T_{c}^{-1}$.
The idea is that if $\tilde f_1$ and $\tilde f_2$
are lifts of $f$ that have nonempty fixed sets,
then they are isogredient if and only if those fixed sets
project to the same Nielsen class in $\fix(f)$.

For free products, the situation is complicated by the following facts.
Firstly, a path in the Bass--Serre tree $\Gamma$ with endpoints
in $\fix(\tilde f_1)$ need not project to an honest Nielsen path for $f$,
merely an almost Nielsen path.
Second,  if $\tilde v$ is a vertex of $\Gamma$
that projects to a vertex $v$ with nontrivial vertex group,
then two  lifts $\tilde f_1$ and $\tilde f_2$ may fix $\tilde v$
without being isogredient.
What's more, we will show that if $\fix(\tilde f) = \{\tilde v\}$
and $\tilde f$ is principal, then $D\tilde f$ must fix a direction at $\tilde v$,
so not all lifts of $f$ fixing $\tilde v$ are principal.
To get the correct division of $\fix(f)$,
we have to make the following somewhat artificial-seeming definition.

Let $f\colon \mathcal{G} \to \mathcal{G}$ be a relative train track map.
Two points $x$ and $y$ in $\fix(f)$ are \emph{Nielsen almost equivalent}
or belong to the same \emph{almost Nielsen class}
if they are the endpoints of an almost Nielsen path for $f$.
We must warn the reader: Nielsen almost equivalence is \emph{not}
an equivalence relation, since it fails to be transitive in general.
A vertex may belong to \emph{multiple} almost Nielsen classes.

Recall that a direction $(g,e)$ based at a vertex $v$ of $\mathcal{G}$
is \emph{almost fixed} by $f$ if $Df(g,e) = (g',e)$.
It is \emph{almost periodic} if it is almost fixed by some iterate of $Df$.
If we have $g = g'$ in the above equation,
of course,
the direction is \emph{fixed} or \emph{periodic.}

To account for principal lifts satisfying $\fix(\tilde f) = \{\tilde v\}$,
we will add yet more almost Nielsen classes 
associated to a vertex $v$
with nontrivial vertex group:
let $e_1,\ldots,e_m$ be the oriented edges beginning at $v$
determining almost fixed directions for $Df$.
That is, $Df(1,e_i) = (g_i,e_i)$ for some vertex group element $g_i \in \mathcal{G}_v$.
Define an equivalence  relation on the set $\{e_1,\ldots,e_m\}$
where $e_i \sim e_j$ if there exists $g_i$ and $g_j$ in $\mathcal{G}_v$
such that $Df(1,e_i) = (g_i,e_i)$ and $g_i^{-1}Df(g_j,e_j) = (g_j,e_j)$.
The vertex $v$ should belong to an almost Nielsen class for each equivalence class
$[e_i]$.
For some of these $e_i$, there is an almost Nielsen path based at $v$
that begins $e_i\cdots$; for these $e_i$
we do not need to an add an almost Nielsen class;
for the remainder of these $e_i$,
the almost Nielsen class is $\{v\}$.
For technical reasons,
we also need to add an \emph{exceptional} almost Nielsen class $\{v\}$
that does not correspond to any $[e_i]$ above.
Thus a \emph{non-exceptional} almost Nielsen class
either corresponds to a non-trivial almost Nielsen path,
is a single point $w$ (possibly in the interior of an edge)
that has trivial vertex group if $w$ is a vertex,
or is determined by a vertex $v$ with nontrivial vertex group
and at least one almost fixed direction at $v$.

\begin{lem}
    There are finitely many almost Nielsen classes.
\end{lem}

\begin{proof}
    Since $G$ is a finite graph,
    $\fix(f)$ has finitely many connected components.
    Each component of $\fix(f)$ that contains an edge
    splits into finitely many almost Nielsen classes,
    since points in the interior of the same almost fixed edge
    are Nielsen almost equivalent.
    Each vertex belongs to finitely many almost Nielsen classes,
    since there are finitely many almost Nielsen paths for $f$ 
    and finitely many oriented edges incident to that vertex.
    An isolated point in $\fix(f)$ in the interior of an edge
    belongs to a single almost Nielsen class.
\end{proof}

Suppose $\tilde f\colon \Gamma \to \Gamma$
is a lift of $f\colon \mathcal{G} \to \mathcal{G}$.
Any tight path $\tilde\alpha$ in $\Gamma$ with endpoints in $\fix(\tilde f)$
projects to an almost Nielsen path $\alpha$ for $f$.
Conversely, if $\alpha$ is an almost Nielsen path for $f$
and $\tilde f$ fixes an endpoint of $\tilde\alpha$ with trivial stabilizer,
then $\tilde f$ also fixes the other endpoint.
If $\fix(\tilde f)$ is a single vertex $\tilde v$ that projects to a vertex $v$
with nontrivial vertex group,
we say that $\fix(\tilde f)$ projects to the exceptional almost Nielsen class $\{v\}$
if there is no fixed direction for $\tilde f$ at $\tilde v$.
Otherwise, suppose $\tilde e$ determines a fixed direction for $D\tilde f$ based at $\tilde v$
corresponding to the direction $(g,e)$ based at $v$.
Then $(g,e)$ is almost fixed for $Df$;
say $Df(1,e) = (h,e)$, and $D\tilde f(x,e) = gh^{-1}f_v(g^{-1})Df(x,e)$ for $x \in \mathcal{G}_v$.
If $e \sim e'$, let $h'$ be such that $h^{-1}Df(h',e) = (h',e)$.
Then
\[  D\tilde f(gh',e') = gh^{-1}f_v(g^{-1}) Df(gh',e') = gh^{-1} Df(h',e') = (gh',e'). \]
Thus $\fix(\tilde f)$ is either empty or projects to a single almost Nielsen class
in $\fix(f)$.
\begin{lem}[cf. Lemma 3.8 of \cite{FeighnHandel}]
    Suppose that $f\colon \mathcal{G} \to \mathcal{G}$
    represents $\varphi \in \out(F,\mathscr{A})$
    and that $\tilde f_1$ and $\tilde f_2$ are lifts of $f$
    with nonempty fixed point sets
    that project to non-exceptional almost Nielsen classes.
    Then $\tilde f_1$ and $\tilde f_2$ belong to the same isogredience class
    if and only if $\fix(\tilde f_1)$ and $\fix(\tilde f_2)$
    project to the same almost Nielsen class in $\fix(f)$.
\end{lem}

\begin{proof}
    If $\tilde f_2 = T_c\tilde f_1T_c^{-1}$,
    then $\fix(\tilde f_2) = T_c\fix(\tilde f_1)$.
    Furthermore if $\tilde E$ determines a fixed direction for $D\tilde f_1$ 
    at $\tilde v \in \fix(\tilde f_1)$,
    then $T_c(\tilde E)$ determines a fixed direction for $D\tilde f_2$ 
    at $T_c(\tilde v) \in \fix(\tilde f_2)$.
    Therefore
    $\fix(\tilde f_2)$ and $\fix(\tilde f_1)$ project to the same
    almost Nielsen class in $\fix(f)$.
    We remark that this holds 
    without assumption on $\fix(\tilde f_1)$ and $\fix(\tilde f_2)$.

    Conversely if $\fix(\tilde f_2)$ and $\fix(\tilde f_1)$
    have the same non-exceptional almost Nielsen class as a projection,
    then there exists $\tilde x \in \fix(\tilde f_2)$ and $T_c$
    such that $T_c^{-1}(\tilde x) \in \fix(\tilde f_1)$.
    If $\tilde x$ projects to a point with trivial associated group,
    then  we are done,
    as $\tilde f_2$ and $T_c\tilde f_1T_c^{-1}$  are lifts of $f$
    that agree on a point and are thus equal.

    If instead $\tilde x$ projects to a vertex with nontrivial vertex group,
    then by the assumption that $\fix(\tilde f_2)$
    projects to a non-exceptional almost Nielsen class,
    there is either a second point $\tilde y$ in $\fix(\tilde f_2)$
    or
    there is an edge $\tilde e$ incident to $\tilde x$ determining a fixed direction
    for $\tilde f_2$.
    In the former case, we may choose $c$ such that $T_c^{-1}(\tilde y) \in \fix(\tilde f_1)$
    and we are done, as $\tilde f_2$ and $T_c\tilde f_1 T_c^{-1}$ agree on two points and are thus equal.
    In the contrary case,
    since $\fix(\tilde f_1)$ projects to the same almost Nielsen class as $\fix(\tilde f_2)$,
    we may choose $c$ such that $T_c^{-1}(\tilde e)$ determines a fixed direction
    for $\tilde f_1$.
    Then $T_c\tilde f_1T_c^{-1}$ and $\tilde f_2$
    agree on a point and a direction at that point and are thus equal.
\end{proof}

We show below that principal lifts $\tilde f$
have nonempty fixed sets which project to non-exceptional almost Nielsen classes.
Since there are finitely many non-exceptional almost Nielsen classes,
it follows that there are finitely many isogredience classes
of principal automorphisms $\Phi$ representing each $\varphi \in \out(F,\mathscr{A})$.

\paragraph{Rotationless automorphisms.}
Let $P(\varphi)$ denote the set of principal automorphisms
$\Phi\colon (F,\mathscr{A}) \to (F,\mathscr{A})$
representing $\varphi \in \out(F,\mathscr{A})$.
Note that if $\Phi \in P(\varphi)$ is  principal and $k \ge 1$,
then $\Phi^k$ is also principal for $\varphi^k$.
Let $\per_N(\hat\Phi)$ denote the set of non-repelling periodic points for the action
of $\hat\Phi$ on $\partial(F,\mathscr{A})$.
(That is, the set of periodic points that are not repellers in $\partial_\infty(F,\mathscr{A})$
for some positive iterate of $\hat\Phi$.)

An outer automorphism $\varphi$ is \emph{rotationless}
if for all $\Phi \in P(\varphi)$, we have $\per_N(\hat\Phi) = \fix_N(\hat\Phi)$
and for all $k \ge 1$, the map $\Phi\mapsto \Phi^k$
defines a bijection $P(\varphi) \to P(\varphi^k)$.

We assume throughout the remainder of the paper that $F \ne F_1$ or $C_2*C_2$,
and that the free product decomposition of $F$ has positive complexity, so $F \ne A_1$.
The Bowditch boundary of each of the former free products is two points
which are the endpoints of an axis,
so there are no principal automorphisms of $F_1$ or $C_2*C_2$.
Nevertheless for notational convenience,
we will say that the identity outer automorphism
and either outer automorphism of these groups are rotationless respectively.

\paragraph{}
We now turn to showing that principal lifts have nonempty fixed point sets
that project to non-exceptional almost Nielsen classes.
Suppose that $f\colon \mathcal{G}\to \mathcal{G}$ 
represents $\varphi \in \out(F,\mathscr{A})$
and that $\tilde f\colon \Gamma \to \Gamma$
is a lift to the Bass--Serre tree.
We say that $\tilde f$ \emph{moves $\tilde z \in \Gamma$ towards $P \in \fix(\hat f)$}
if the tight ray from $\tilde f(\tilde z)$ to $P$ does not contain $\tilde z$.
Similarly we say that $\tilde f$ \emph{moves $\tilde y_1$ and $\tilde y_2$
away from  each other} if the tight path in $\Gamma$
connecting $\tilde f(\tilde y_1)$ to $\tilde f(\tilde y_2)$
contains $\tilde y_1$ and $\tilde y_2$
and if $\tilde f(\tilde y_1) < \tilde y_1 < \tilde y_2 < \tilde f(\tilde y_2)$
in the order induced by the orientation on that path.

Recall that $\partial(F,\mathscr{A})$ 
is identified with the Bowditch boundary of $\Gamma$.
Thus it makes sense to say that points in $\Gamma$
are close  to $P\in \partial(F,\mathscr{A})$
or that $P$ is the limit of points in $\Gamma$.

\begin{lem}[cf. Lemma 3.15 of \cite{FeighnHandel}]
    \label{movingtowardsattractors}
    Suppose that $P \in \partial_\infty(F,\mathscr{A})$ is a point in $\fix(\hat f)$
    and that there does not exist $c \in \Gamma$ such that
    $\fix(\hat f) = \{\hat T_c^{\pm}\}$.
    \begin{enumerate}
        \item If $P$ is an attractor  for $\hat f$
            then $\tilde z$ moves toward $P$ under the action of $\tilde f$
            for all $\tilde z \in \Gamma$ that are sufficiently close to $P$.
        \item If $P$ is the endpoint of an axis $A_c$ or if $P$
            is the limit of points in $\Gamma$
            that are  either fixed by $\tilde f$
            or move toward $P$ under the action of $\tilde f$,
            then $P \in \fix_N(\hat f)$.
    \end{enumerate}
\end{lem}

\begin{proof}
    The proof follows \cite[Lemma 3.15]{FeighnHandel}.
    If $P$ is not the endpoint of an axis $A_c$,
    then the statement follows from the proof of \Cref{GJLLprop}
    and bounded cancellation.

    If $P$ is the endpoint of an axis $A_c$,
    then $\fix(\hat f)$ contains $\hat T_c^{\pm}$ and at least one other point.
    It follows from \Cref{boundarybasics} that 
    $P \in \partial(\mathbb{F},\mathscr{A}|_{\mathbb{F}})$
    and from \Cref{GJLLprop} that this point is not isolated in $\fix(\hat f)$;
    thus it is neither an attractor nor a repeller,
    so $P \in \fix_N(\hat f)$.
\end{proof}

\begin{lem}[cf.~Lemma 3.16 of \cite{FeighnHandel}]
    \label{producingafixedpoint}
    If $\tilde f$ moves $\tilde y_1$ and $\tilde y_2$ away from each other,
    then $\tilde f$ fixes a point 
    in the interval bounded by $\tilde y_1$ and $\tilde y_2$.
\end{lem}

\begin{proof}
    We follow the proof of \cite[Lemma 3.16]{FeighnHandel}.
    Let $\alpha_0$ denote the oriented tight path
    connecting $\tilde y_1$ to $\tilde y_2$
    and let $\alpha_1$ denote the oriented tight path
    connecting $\tilde f(\tilde y_1)$ to $\tilde f(\tilde y_2)$.
    Let $r\colon \Gamma \to \alpha_1$
    be the retraction onto the nearest point in $\alpha_1$
    and let $\tilde g = r\tilde f \colon \alpha_0 \to \alpha_1$.
    Since $\tilde f$ moves $\tilde y_1$ and $\tilde y_2$ away from each other,
    $\tilde\alpha_0$ is a proper subpath of $\tilde\alpha_1$
    and $\tilde g$ is a surjection.
    If $\tilde y$ is  the first point in $\tilde\alpha_0$ 
    such that $\tilde g(\tilde y) = \tilde y$,
    then $\tilde y_1 < \tilde y < \tilde y_2$
    and $\tilde g(\tilde z) < \tilde g(\tilde y)$ 
    for $\tilde y_1 < \tilde z < \tilde y$.
    It follows  that $\tilde f(\tilde y)$ belongs to $\tilde\alpha_1$,
    so $\tilde y$ is fixed by $\tilde f$.
\end{proof}

\begin{cor}[cf.~Corollary 3.17 of \cite{FeighnHandel}]
    \label{principalnonemptyfixed}
    If $\tilde f$ is a principal lift,
    then $\fix(\tilde f)$ is nonempty
    and projects to a non-exceptional almost Nielsen class in $\fix(f)$.
\end{cor}

\begin{proof}
    The proof follows \cite[Corollary3.17]{FeighnHandel}.
    Suppose first that there is a nonperipheral element $c$
    such that $A_c$ has its endpoints in $\fix_N(\hat f)$
    and so $T_c$ commutes with $\tilde f$.
    Either some point of the axis $A_c$ is fixed by $\tilde f$ or not.
    If there is a fixed point $\tilde x$, then note that
    \[  \tilde fT_c(\tilde x) = T_c\tilde f(\tilde x) = T_c(\tilde x), \]
    so $\fix(\tilde f)$ projects 
    to a non-exceptional almost Nielsen class in $\fix(f)$.

    If there is no fixed point on $A_c$,
    then there is a point in $A_c$ that moves toward one of the endpoints of $A_c$,
    say to $P$, for
    if we assume the contrary, then $\tilde f$ moves any two points of $A_c$ away from each other,
    contradicting our assumed lack of fixed point.
    Since $\tilde f$ commutes with $T_c$,
    there are points in $\Gamma$ arbitrarily close to $P$ that move toward $P$.
    The same property holds for an attractor $P \in \fix(\hat f)$
    by \Cref{movingtowardsattractors}.

    Therefore if $\fix_N(\hat f)$ contains $\{\hat T_c^\pm\}$ and at least one other point
    in $\partial_\infty(F,\mathscr{A})$,
    but we assume that there is no fixed point on each axis $A_c$
    satisfying $\{\hat T_c^\pm\} \subset \fix_N(\hat f)$,
    then there are distinct $P_1$ and $P_2$  in $\fix_N(\hat f)$
    and $\tilde x_1$ and $\tilde x_2$ in $\Gamma$
    such that $\tilde x_i$ is close to and moves toward $P_i$.
    It follows that $\tilde f$ moves $\tilde x_1$ and $\tilde x_2$
    away from each other.
    \Cref{producingafixedpoint} produces a fixed point $\tilde y$.

    If $\fix_N(\hat f)$ contains $\{\hat T_c^\pm\}$ and at least one other point $\tilde y$
    in $V_\infty(F,\mathscr{A})$, then $\tilde y \in \fix(\tilde f)$.
    The argument in the first paragraph shows that $T_c(\tilde y)$ is also fixed,
    so $\fix(\tilde f)$ projects to a non-exceptional almost Nielsen class in $\fix(f)$.

    If $\tilde f$ is a principal lift 
    corresponding to $\Phi$ but $\fix_N(\hat f)$ contains only one point 
    in $\partial_\infty(F,\mathscr{A})$
    then there is a conjugate of some infinite vertex group $A_i$
    sent to itself by $\Phi$.
    This conjugate fixes a unique point $\tilde x$ in $\Gamma$,
    which must (by $\Phi$-twisted equivariance) be fixed by $\tilde f$.
    Let $P \in \partial_\infty(F,\mathscr{A})$ be in $\fix_N(\hat f)$. 
    It is an attractor, so there exist points on the ray from $\tilde x$ to $P$
    close to $P$ that move toward $P$ under $\tilde f$.
    Either the direction determined by the tight ray from $\tilde x$ to $P$
    is fixed by $\tilde f$,
    or there are points on the tight ray from $\tilde x$ to $P$
    moved away from each other by $\tilde f$.
    In either case, $\fix(\tilde f)$ projects 
    to a non-exceptional almost Nielsen class in $\fix(f)$.

    Finally if $\tilde f$ is a principal lift
    corresponding to $\Phi$ but $\fix(\hat f)$ contains no points in $\partial_\infty(F,\mathscr{A})$
    then there are two conjugates of (not necessarily distinct) $A_i$ sent to themselves by $\Phi$.
    These conjugates fix unique points $\tilde x$ and $\tilde y$ in $\Gamma$,
    which must be fixed by $\tilde f$.
    The tight path from $\tilde x$ to $\tilde y$ projects to an almost Nielsen path,
    so $\fix(\tilde f)$ projects to a non-exceptional almost Nielsen class in $\fix(f)$.
\end{proof}

There are two cases in which a lift $\tilde f$
whose fixed point set projects to a non-exceptional
almost Nielsen class is not a principal lift.
The first is when $\fix(\hat f)$ is the endpoints of
a generic leaf of an attracting lamination $\Lambda^+ \in \mathcal{L}(\varphi)$.
In this case $\fix(\tilde f)$ is a single point $\tilde x$
and there are exactly two periodic directions at $\tilde x$.
The second case is when $\fix(\hat f)$
is the endpoints of an axis $A_c$.
As it happens, this latter case occurs when $\fix(\tilde f) = A_c$.
In this latter case, the axis $A_c$
projects to either a topological circle in $\fix(f)$
or to the quotient of $\mathbb{R}$ by the standard action of $C_2*C_2$,
i.e.~the quotient is a subgraph of subgroups of $\mathcal{G}$
that is topologically an interval with $C_2$ vertex groups at the endpoints.
In both cases, the lift becomes principal after passing to an iterate unless
there are exactly two periodic directions
at each point of this component of $\fix(f)$.
In the case where $\fix(\tilde f)$ projects to the quotient of $\mathbb{R}$
by the standard action of $C_2*C_2$, notice that if a vertex group $\mathcal{G}_v$ in question
is finite but not equal to $C_2$,
then there are more than two periodic directions at any lift $\tilde v$ in $\fix(\tilde f)$.
If $\mathcal{G}_v$ is infinite, then $\fix_N(\hat f)$ is infinite,
so $\tilde f$ is principal.

Points in $\per(f)$ are \emph{Nielsen almost equivalent}
if they are Nielsen almost equivalent as fixed points for some iterate of $f$.
Thus two periodic points $x$ and $y$ are Nielsen almost equivalent
if they are the endpoints of a periodic almost Nielsen path.
We say that a point $x \in \per(f)$ is \emph{principal}
if neither of the following conditions are satisfied.
\begin{enumerate}
    \item The point $x$ has finite vertex group if it is a vertex, 
        is not the endpoint of a periodic almost Nielsen path
        and there are exactly two periodic directions at $x$,
        both of which are contained in the same exponentially growing stratum.
    \item The point $x$ is contained in a component $C$ of $\per(f)$
        which is either a topological circle
        or isomorphic as a graph of groups to the quotient of $\mathbb{R}$
        by the standard action of $C_2 * C_2$
        and each point in $C$ has exactly two periodic directions.
\end{enumerate}

If $x\in \per(f)$ is a vertex with nontrivial vertex group,
then it may be principal in multiple ways,
according to the ways the edges incident to $x$ determine different
non-exceptional almost Nielsen classes.
Note that a vertex with nontrivial vertex group is always principal
unless that vertex group is $C_2$.

(Note that we follow Definition 3.5 in \cite{FeighnHandelAlg},
which is a corrected version of \cite[Definition 3.18]{FeighnHandel}.)
Lifts of principal periodic points in $\mathcal{G}$ to $\Gamma$ are \emph{principal}.

We say that $f\colon \mathcal{G} \to \mathcal{G}$ is \emph{rotationless}
if it satisfies the following conditions.
\begin{enumerate}
    \item Each principal vertex is fixed.
    \item Each almost periodic direction at a principal vertex
        is almost fixed.
    \item Suppose $v$ is a principal vertex with nontrivial vertex group
        and that $(1,e)$ is an almost fixed direction at $v$, say $Df(1,e) = (g,e)$.
        If a direction $d$ is periodic for the map $d \mapsto g^{-1}.Df(d)$, it is fixed.
\end{enumerate}

In practice we apply these definitions to relative train track maps
$f\colon \mathcal{G} \to \mathcal{G}$
that satisfy the conclusions of \Cref{improvedrelativetraintrack}.
Principal periodic points are thus
either contained in almost periodic edges or are vertices.
We have that any endpoint of an indivisible periodic almost Nielsen path
(containing an edge, of course) is principal,
as is the initial endpoint of any non-exponentially growing
edge that is not almost periodic.
This latter property implies that 
for a rotationless relative train track map 
each non-exponentially growing stratum
that is not periodic 
is a single edge.

\begin{lem}
    \label{rotationlessdirections}
    Suppose that $f\colon \mathcal{G} \to \mathcal{G}$ is a rotationless relative train track map,
    that $v \in G$ is a principal vertex with nontrivial vertex group
    and that $g = f^k$ for some $k \ge 1$.
    Let $\tilde v$ be a lift of $v$ fixed by a lift $\tilde g$ of $g$.
    Suppose that $\tilde g$ fixes a direction at $\tilde v$.
    There is a lift $\tilde f$ of $f$ fixing $\tilde v$ and all the same directions
    at $\tilde v$ as $\tilde g$.
\end{lem}

\begin{proof}
    Using the $\mathcal{G}_v$-equivariant bijection from \Cref{basicssection},
    we identify the set of directions at $\tilde v$
    with
    \[  \coprod_{e \in \st(v)} \mathcal{G}_v \times \{e\}. \]
    Since $f$ is rotationless and $v$ is principal,
    each $Df$-almost periodic direction $(1,e_i)$ at $v$ is almost fixed,
    say $Df(1,e_i) = (g_i,e_i)$.
    By the discussion in \Cref{basicssection},
    the map $D\tilde f$ takes the form $D\tilde f(x,e_i) = (hf_v(x)g_i,e_i)$
    for some element $h = h(\tilde f) \in \mathcal{G}_v$.
    Suppose $(x,e_i)$ is fixed by $D\tilde g$.
    Take $h = xg_i^{-1}f_v(x^{-1})$, and observe that
    \[D\tilde f(x,e_i) = (xg_i^{-1}f_v(x^{-1})f_v(x)g_i,e_i) = (x,e_i). \]
    Furthermore, by definition there exists $h' = h'(\tilde g) \in \mathcal{G}_v$
    such that
    \[  D\tilde g(x,e_i) = (h'f_v^k(x)f_v^{k-1}(g_i)f_v^{k-2}(g_i)\cdots f_v(g_i)g_i,e_i),\]
    and the condition that $D\tilde g(x,e_i) = (x,e_i)$ says that
    \[  h' = x g_i^{-1}f_v(g_i^{-1})\cdots f_v^{k-2}(g_i^{-1})f_v^{k-1}(g_i^{-1})f_v^k(x^{-1})
     = h f_v(h) \cdots f_v^{k-1}(h). \]
     In other words, $D\tilde g = D\tilde f^k$.

     Therefore to complete the proof, we need only show that if a direction has $D\tilde f$-period $k$,
     it is fixed.
     Observe that if a direction $(y,e_j)$ has period $k$ under $D\tilde f$,
     then $(x^{-1}y,e_j)$ has period $k$ under $D(T_{x^{-1}}\tilde fT_x)$.
     If the direction $(1,e_j)$ satisfies $Df(1,e_j) = (g_j,e_j)$,
     then $D(T_{x^{-1}}\tilde fT_x)(y,e_j) = (g_i^{-1}f_v(y)g_j,e_j)$.
     By the definition of rotationless,
     if a direction is periodic for $D(T_{x^{-1}}\tilde fT_x)$, it is fixed,
     so the same is true for $D\tilde f$.
\end{proof}

The following proposition provides a criterion for showing 
that a relative train track map $f\colon \mathcal{G} \to \mathcal{G}$
has a rotationless iterate.

\begin{prop}
    \label{rotationlesscriterion}
    Let $f\colon \mathcal{G} \to \mathcal{G}$ be a relative train track map.
    Suppose that each infinite vertex group $\mathcal{G}_v$ for $v \in G$
    has a bound on the order of finite order elements.
    If $(1,e_i)$ is an almost periodic direction at $v$ satisfying
    say $Df(1,e_i) = (g_i,e_j)$,
    suppose additionally that there exists $k > 0$ such that
    the automorphism $f_{v,i}\colon \mathcal{G}_v \to \mathcal{G}_v$
    defined as $f_{v,i}(x) = g_i^{-1}f_v(x)g_i$
    has the property that if an element $x$ is $f_{v,i}^k$-periodic, it is fixed.
    If the automorphism $f_v$ also has this property,
    then $f\colon \mathcal{G} \to \mathcal{G}$ has a rotationless iterate.
\end{prop}

\begin{proof}
    Since there are finitely many principal vertices
    and finitely many directions at principal vertices with finite vertex group,
    and finitely many almost periodic directions of the form $(1,e_i)$
    at vertices with infinite vertex group,
    there is $N > 0$ such that each principal vertex for $f^N\colon \mathcal{G} \to \mathcal{G}$
    is fixed,
    each almost periodic (and hence periodic) direction at a vertex with finite vertex group is fixed,
    and each almost periodic direction at a vertex with infinite vertex group is almost fixed.
    By increasing $N$ we may additionally assume that the condition in the proposition
    holds for each almost periodic direction at each vertex with infinite vertex group.

    Write $g = f^N$.
    Suppose that $v$ is a vertex with infinite vertex group
    and that a direction $(x,e_i)$ is $Dg$-periodic.
    Then $(1,e_i)$ is $Dg$-almost periodic, hence $Dg(1,e_i) = (g_i,e_i)$
    for some $g_i \in \mathcal{G}_v$.
    Suppose $Dg^k(x,e_i) = (x,e_i)$.
    By definition, we have
    \[  Dg^k(x,e_i) = (g_v^k(x)g_v^{k-1}(g_i)g_v^{k-2}(g_i)\cdots g_v(g_i)g_i, e_i), \]
    so writing $h = g_v(x)g_ix^{-1}$, we see that
    \[  g^k_v(x)g^{k-1}_v(g_i)g_v^{k-2}(g_i)\cdots g_v(g_i)g_i x^{-1} 
    = g_v^{k-1}(h)g_v^{k-2}(h)\cdots g_v(h)h = 1. \]
    Again by definition we have
    \[  Dg^{k+1}(x,e_i) = (g_v^k(h)g_v^{k-1}(h)g_v^{k-2}(h)\cdots g_v(h)hx, e_i) = (hx,e_i), \]
    so it follows that $h$ is $g_v$-periodic and hence fixed,
    since we have
    \[  h = g_v^k(h) g_v^{k-1}(h) \cdots g_v(h) h = g_v^k(h). \]
    But then $Dg^k(x,e_i) = (h^kx,e_i) = (x,e_i)$, so $h$ has finite order.
    In fact, the period of $(x,e_i)$ is the order of $h$.

    Now suppose the direction $d$ is periodic for the map
    $d \mapsto g_i^{-1}.Dg(d)$.
    Write $d = (x,e_j)$ and suppose $(g_i^{-1}.Dg)^k(d) = d$.
    Then the direction $(1,e_j)$ is almost periodic, so satisfies $Dg(1,e_j) = (g_j,e_j)$
    for some $g_j \in \mathcal{G}_v$.
    By definition we have
    \[  g_i^{-1}.Dg(x,e_j) = (g_i^{-1}g_v(x)g_j,e_j) = (g_i^{-1}g_j g_{v,j}(x),e_j), \]
    where $f_{v,j} \colon \mathcal{G}_v \to \mathcal{G}_v$ is, as in the statement,
    the automorphism $x \mapsto g_j^{-1}g_v(x)g_j$.
    Write $h' = x^{-1}g_i^{-1}g_jg_{v,j}(x)$.
    For $\ell \ge 1$, we have
    \[  (g_i^{-1}.Dg)^\ell(x,e_j) = (xhg_{j,v}(h')\cdots g_{j,v}^{\ell-1}(h'),e_j).\]
    The condition that $(g_i^{-1}.Dg)^k(x,e_j) = (x,e_j)$ says that
    $h'g_{v,j}(h')\cdots g_{v,j}^{k-1}(h') = 1$,
    and the argument above shows that $h'$ is $g_{v,j}$-periodic and hence fixed,
    and thus that it must have finite order.
    In fact, the period of $(x,e_j)$ is the order of $h'$.

    By assumption, there exists $M \ge 1$ such that if $x$ is an element of an infinite vertex group
    of finite order, then $x^M = 1$.
    The arguments in the previous paragraphs show that $g^{NM}$ is rotationless.
\end{proof}

\begin{cor}
    \label{freeproductsoffiniteandycycliccor}
    If each infinite vertex group of $\mathcal{G}$ is a free product
    of the form $B_1 * \cdots * B_m * F_\ell$,
    where the $B_i$ are finite groups and $F_\ell$ is free of finite rank $\ell$,
    every relative train track map $f\colon \mathcal{G} \to \mathcal{G}$
    has a rotationless iterate.
\end{cor}

In principal, it appears the statement should hold for vertex groups which are
virtually finite-rank free.

\begin{proof}
    Suppose $\mathcal{G}_v$ is a free product of the form in the statement.
    Any finite order element $x$ of $\mathcal{G}_v$ is conjugate into some $B_i$,
    so there is a bound on the order of $x$.
    Therefore by \Cref{rotationlesscriterion} it suffices to show that
    there exists $k > 0$ such that for
    each automorphism $\Phi$ of $\mathcal{G}_v$,
    if an element of $\mathcal{G}_v$ is $\Phi^k$-periodic, it is fixed.
    Let $\Pi(\Phi)$ denote the subgroup of $\mathcal{G}_v$
    consisting of all $\Phi$-periodic elements.
    Then $\Pi(\Phi)$ is again a free product of finite and cyclic groups
    and the restriction $\Phi|_{\Pi(\Phi)}$ is a periodic automorphism.
    Suppose toward a contradiction that $x_0,\ldots,x_{m+\ell}$ 
    are $m+\ell+1$ elements of distinct free factors for $\Pi(\Phi)$.
    There exists $N \ge 1$ such that $x_0,\ldots,x_n$ are all $\Phi$-periodic
    of period dividing $N$.
    Thus they belong to the fixed subgroup of $\Phi^N$,
    which has rank bounded by $m+\ell$ by the main theorem of \cite{CollinsTurner}
    (the Scott conjecture),
    a contradiction.
    Therefore $\Pi(\Phi)$ has a finite free product decomposition with factors
    that are subgroups of the $B_i$.
    It follows that
    the free product $\Pi(\Phi)$ takes on finitely many values as $\Phi$ varies,
    and that the minimum index of a finite-index free subgroup of $\Pi(\Phi)$
    and the rank of this minimum-index free subgroup are both bounded independent of $\Phi$.
    We see that some uniform power of $\Phi$ 
    restricts to a periodic automorphism of a free group of finite rank.
    The order of this restriction is bounded depending only on the rank,
    and we claim that except in the case where $\Pi(\Phi) = C_2*C_2$,
    an automorphism of $\Pi(\Phi)$ fixing this finite-index free subgroup fixes the whole group.
    Since periodic automorphisms of $C_2*C_2$ clearly have bounded order, this will complete the proof.

    We are grateful to Sami Douba for explaining the following argument.
    Write $G = \Pi(\Phi)$.
    The statement is true for $F_1$, so assume that $G \ne F_1$.
    Since $G \ne C_2*C_2$ nor $F_1$ and $G$ is a free product of finite and cyclic groups,
    it is a hyperbolic group that acts effectively on its Gromov boundary $\partial G$.
    Let $H < G$ be a subgroup of finite index.
    We claim that if $\Psi\colon G \to G$ is an automorphism such that $\Psi|_H$ is the identity,
    then $\Psi$ is the identity.
    There exists $M > 0$ such that if
    $g \in G$ has infinite order,
    then $g^M \in H$.
    Each such $g$ has a unique attracting fixed point $\hat T_g^+$ in $\partial G$,
    and $\Psi$ induces a $\Psi$-twisted equivariant homeomorphism 
    $\hat\Psi \colon \partial G \to \partial G$.
    We have
    \[  T_g^+ = T_{g^M}^+ = T_{\Psi(g^M)}^+ = T_{\Psi(g)^M}^+ = T_{\Psi(g)}^+ = \hat\Psi(T_g^+). \]
    Since the set of attracting fixed points of elements of infinite order is dense in $\partial G$,
    we have that $\hat\Psi$ is the identity.
    Therefore for any element $g \in G$ and $\xi \in \partial G$, we have
    \[  g.\xi = \hat\Psi(g.\xi) = \Psi(g).\hat\Psi(\xi) = \Psi(g).\xi. \]
    Because the action of $G$ on $\partial G$ is effective,
    we conclude that $g = \Psi(g)$.
\end{proof}

\begin{lem}[cf.~Lemma 3.19 of \cite{FeighnHandel}]
    \label{principalEG}
    Suppose that $f\colon \mathcal{G} \to \mathcal{G}$
    is a relative train track map satisfying the conclusions of
    \Cref{improvedrelativetraintrack}.
    For every exponentially growing stratum $H_r$,
    there is a principal vertex whose link contains a periodic
    or almost periodic direction in $H_r$.
\end{lem}

\begin{proof}
    If $v$ has nontrivial vertex group not equal to $C_2$, we are done,
    so suppose all vertices of $H_r$ have vertex group trivial or $C_2$.
    If some vertex $v \in H_r$ belongs to a non-contractible component of $G_{r-1}$
    we know that $v$ is periodic by \Cref{trivialedgegroupsEG2}.
    By \hyperlink{EG-i}{(EG-i)}, some iterate of $Df$ induces a self-map
    of the directions based at $v$ in $H_r$ and a self-map of the directions based
    at $v$ in $G_{r-1}$.
    Therefore there are at least two periodic directions based at $v$,
    at least one in $H_r$ and one out of $H_r$, so $v$ is principal.

    So suppose no vertex of $v$ belongs to a non-contractible component of $G_{r-1}$.
    Then $H^z_r$ is a union of components of $G_r$,
    and in fact $H^z_r = H_r$.
    We assume as above that all nontrivial vertex groups are $C_2$.
    We may also suppose that each periodic point of $H_r$
    is not the endpoint of a periodic almost Nielsen path---that is, 
    $H_r$ contains no periodic almost Nielsen paths.

    The argument is similar to \cite[Lemma 5.2]{BestvinaFeighnHandelSolvable}.
    Let $\tau_1$ and $\tau_2$ be legal paths in $H_r$
    so that the initial ends of $\bar\tau_1$ and $\tau_2$
    determine an illegal but nondegenerate turn.
    By \Cref{kprotectedsplitting} and the proof of \Cref{finitelymanynielsenpaths},
    if $\tau_1$ and $\tau_2$ are sufficiently long
    then for some $k > 0$, we have that
    $f^k_\sharp(\tau_1\tau_2)$ is legal and not all of
    $f^k_\sharp(\bar\tau_1)$ and $f^k_\sharp(\tau_2)$ 
    is canceled when $f^k_\sharp(\tau_1)f^k_\sharp(\tau_2)$
    is tightened to $f^k_\sharp(\tau_1\tau_2)$.
    We may write $f^k_\sharp(\tau_1) = \mu_1\eta$ and 
    $f^k_\sharp(\tau_2) = \bar\eta\mu_2$
    such that $f^k_\sharp(\tau_1\tau_2) = \mu_1\mu_2$ for legal paths
    $\mu_1$, $\mu_2$ and $\eta$.
    By assumption, each of the three turns determined by these paths is a legal turn.
    By assumption there are finitely many directions in $H_r$,
    so we may iterate until each of the directions determined by
    $\bar\mu_1$, $\bar\mu_2$  and $\eta$ are fixed; they remain distinct.
    Each direction determines an attractor for $\hat f$ in a way that we now describe.
    This shows that the common vertex of each of these fixed directions is principal.

    Suppose that an edge $E$ of $H_r$ determines a fixed direction,
    that $\tilde E$ is a lift of $E$
    and that $\tilde f\colon \Gamma \to \Gamma$ is a lift of $f$
    that fixes the direction determined by $\tilde E$.
    Then $f(E) = Eu$.
    Since $f|_{H_r}$ is an irreducible train track map,
    for all $k \ge 1$,
    we have $[f^k(E)] = Euf(u)f^2(u)\ldots f^{k-1}(u)$.
    Let $\tilde u$ denote the lift of $u$ 
    that begins at the terminal endpoint of $\tilde E$
    and let
    \[  \tilde\gamma  = \tilde E\tilde u\tilde f(\tilde u)\tilde f^2(\tilde u)\ldots.\]
    Again by irreducibility of $f|_{H_r}$, the length of $f^k(u)$
    goes to infinity as $k$ goes to infinity,
    so by \Cref{GJLLprop}, the endpoint $\xi$ of $\tilde\gamma$
    is a (superlinear) attractor for $\tilde f$.
\end{proof}

Notice that \Cref{principalEG} implies that the transition matrix
of an exponentially growing stratum
of a rotationless relative train track map $f\colon \mathcal{G}\to \mathcal{G}$
representing $\varphi \in \out(F,\mathscr{A})$
and satisfying the conclusions of \Cref{improvedrelativetraintrack}
has at least one nonzero diagonal entry and so is aperiodic.
This defines a bijection between $\mathcal{L}(\varphi)$ and the set of exponentially
growing strata of $f$.

\paragraph{} Our next goal is to relate principal points in $\Gamma$
to principal automorphisms,
with the eventual goal of showing that a rotationless relative train track map
represents a rotationless outer automorphism.

\begin{lem}[cf. Lemma 3.21 of \cite{FeighnHandel}]
    \label{fixeddirectionsproperties}
    Suppose that $\tilde f\colon \Gamma \to \Gamma$
    is a principal lift of a relative train track map $f\colon \mathcal{G} \to \mathcal{G}$.
    \begin{enumerate}
        \item For each attractor $P \in \fix_N(\hat f) \cap \partial_\infty(F,\mathscr{A})$ 
            there is a
            (not necessarily unique)
            fixed point $\tilde x \in \fix(\tilde f)$
            such that the interior of the ray $\tilde R_{\tilde x,P}$
            that starts at $\tilde x$ and ends at $P$
            is fixed point free.
        \item If $P \in \fix_N(\hat f) \cap \partial_\infty(F,\mathscr{A})$ is an attractor,
            $\tilde x$ is a fixed point
            and if the interior of the ray $\tilde R_{\tilde x,P}$
            is fixed point free,
            then no point in the interior of $\tilde R_{\tilde x,P}$
            is mapped to $\tilde x$ by $\tilde f$.
            It follows that the initial direction
            determined by $\tilde R_{\tilde x,P}$
            is fixed.
        \item If $P$ and $Q$ are distinct attractors in $\fix_N(\hat f)$,
            if $\tilde x$ is a fixed point
            and if the interiors of both rays $\tilde R_{\tilde x,P}$ and $\tilde R_{\tilde x,Q}$
            are fixed point free
            then the directions determined by $\tilde R_{\tilde x,P}$ and $\tilde R_{\tilde x,Q}$
            are distinct.
    \end{enumerate}
\end{lem}

\begin{proof}
    To find $\tilde x \in \fix(\tilde f)$ and $\tilde R_{\tilde x,P}$ as in item 1,
    begin with any ray $\tilde R'$
    that begins in $\fix(\tilde f)$ and converges to $P$.
    Since $P$ is not in the boundary of the fixed subgroup of $f_\sharp$,
    there is a last point $\tilde x$ of $\fix(\tilde f)$
    in $\tilde R'$.
    Item 3 follows from \Cref{producingafixedpoint};
    the line from $P$ to $Q$ in $\Gamma$ contains a fixed point.
    That fixed point must be $\tilde x$,
    thus it follows that $\tilde R_{\tilde x,P}$ and $\tilde R_{\tilde x,Q}$
    determine distinct directions at $\tilde x$.
    Similarly if some point $\tilde y$ in the interior of $\tilde R_{\tilde x,P}$
    maps to $\tilde x$ by $\tilde f$,
    then there is a fixed point in the interior of $\tilde R_{\tilde x,P}$,
    so item 2 follows.
\end{proof}

\begin{cor}[cf. Corollary 3.22 of \cite{FeighnHandel}]
    \label{principalliftprincipalfixed}
    Assume that $f\colon \mathcal{G} \to \mathcal{G}$
    is a relative train track map
    satisfying the conclusions of \Cref{improvedrelativetraintrack}.
    If $\tilde f$ is a principal lift of $f$,
    then each element of $\fix(\tilde f)$ is principal.
\end{cor}

\begin{proof}
    Let $\Phi$ be the automorphism corresponding to $\tilde f$.
    If $\fix(\Phi)$ is not $F_1$, $C_2*C_2$ or some subgroup of an $A_i$,
    then $\fix(\tilde f)$ is neither a single point nor a single axis and we are done.
    If $\fix(\Phi)$ is either $F_1$ or $C_2*C_2$,
    then $\fix(\tilde f)$ is infinite
    and $\fix_N(\hat f)$ contains an attractor or a point in $V_\infty(F,\mathscr{A})$.
    In either case (with the former being a consequence of \Cref{fixeddirectionsproperties})
    some $\tilde x \in \fix(\tilde f)$
    has a fixed direction that does not come from a fixed edge
    or an almost Nielsen path with one endpoint in $V_\infty(F,\mathscr{A})$
    and we are done.
    In the remaining case, 
    $\fix(\Phi)$ is a possibly trivial subgroup of some $A_i$
    and $\fix(\hat f)\cap \partial_\infty(F,\mathscr{A})$
    is a set of attractors and does not contain the endpoints of any axis.
    Clearly $\fix(\tilde f)$ is not an axis.
    If $\fix(\tilde f)$ is a single vertex $\tilde x$ 
    with infinite stabilizer in $\mathbb{T}$,
    \Cref{fixeddirectionsproperties} implies that
    $\tilde x$ has a fixed direction
    and we are done.
    So suppose $\tilde x$ has finite (possibly trivial) stabilizer.
    The only case where $\tilde x$ is \emph{not} principal
    is if there are only two periodic directions at $\tilde x$
    and these two directions are determined by lifts of oriented edges 
    $\tilde E_1$ and $\tilde E_2$
    that belong to the same exponentially growing stratum.
    Then \Cref{fixedpointstolaminations}
    together with \Cref{fixeddirectionsproperties}
    implies that $\fix_N(\hat f)$ 
    is the endpoint set of a lift of a generic leaf of an attracting lamination,
    contradicting the assumption that $\tilde f$ is principal.
    Therefore we conclude that $\tilde x$ is principal.
\end{proof}

\begin{lem}[cf. Lemma 3.23 of \cite{FeighnHandel}]
    \label{emptyfixedsetfixedboundarypoint}
    Suppose that $\fix(\tilde f)$ is empty.
    Then there is a ray $\tilde R \subset \Gamma$
    converging to an element $P \in \fix(\hat f)$
    and there are points in $\tilde R$ arbitrarily close to $P$ that move toward $P$.
\end{lem}

\begin{proof}
    The proof is identical to \cite[Lemma 3.23]{FeighnHandel}.
    Given a vertex $\tilde y$ of $\Gamma$,
    say that the initial edge of the tight path from $\tilde y$ to $\tilde f(\tilde y)$
    is \emph{preferred} by $\tilde y$.
    Begin with some vertex $\tilde y_0$ and inductively define $\tilde y_{i+1}$
    to be the other endpoint of the edge preferred by $\tilde y_i$.
    If some edge $\tilde E$ is preferred by \emph{both} of its endpoints,
    then $\tilde f$ maps some subinterval of $\tilde E$ over all of $\tilde E$, reversing orientation.
    It follows that $\fix(\tilde f) \ne \varnothing$.
    Therefore we conclude that the vertices $\tilde y_i$
    are contained in a ray that converges to some $P \in \fix(\hat f)$,
    and that $\tilde y_i$ moves toward $P$.
\end{proof}

The main consequence of \Cref{emptyfixedsetfixedboundarypoint} 
is the following result,
for which we set notation.

\paragraph{Restricting to $G_{r-1}$ for non-exponentially growing strata.}
Suppose that $f\colon \mathcal{G} \to \mathcal{G}$
is a relative train track map 
satisfying the conclusions of \Cref{improvedrelativetraintrack},
that $H_r$ is a single edge $E_r$,
and that $f(E_r) = E_ru$ for some nontrivial path $u \subset G_{r-1}$.
Fix a lift $\tilde E_r$ of $E_r$
and let $\tilde f\colon \mathcal{G} \to \mathcal{G}$
be the lift of $f$ that fixes the initial endpoint of $\tilde E_r$.
By property \hyperlink{NEG}{(NEG)},
the component $C$ of $G_{r-1}$ containing the terminal endpoint of $E_r$ is non-contractible.
Denote the copy of the Bass--Serre tree for $C$
that contains the terminal endpoint of $\tilde E_r$ by $\Gamma_{r-1}$
and the restriction of $\tilde f$ to $\Gamma_{r-1}$ by $h\colon \Gamma_{r-1} \to \Gamma_{r-1}$.

The elements of $F$ (thought of as automorphisms of the projection $\Gamma \to \mathcal{G}$)
that preserve $\Gamma_{r-1}$ define a free factor $F(C)$ of positive complexity
such that $[[F(C)]] = [[\pi_1(C)]]$.
The closure in $\partial(F,\mathscr{A})$ of $\{\hat T_c^{\pm} : c \in F(C)\}$
is naturally identified with $\partial(F(C),\mathscr{A}|_{F(C)})$
and with the Bowditch boundary of $\Gamma_{r-1}$.
We have $\hat h = \hat f|_{\partial\Gamma_{r-1}} 
\colon \partial(F(C),\mathscr{A}|_{F(C)}) \to \partial(F(C),\mathscr{A}|_{F(C)})$.

The following lemma is an immediate application of \Cref{emptyfixedsetfixedboundarypoint}.

\begin{lem}[cf. Lemma 3.25 of \cite{FeighnHandel}]
    \label{negfixedboundarypoint}
    Assuming notation as in the previous paragraphs,
    if $\fix(h) = \varnothing$
    then there is a ray $\tilde R \subset \Gamma_{r-1}$
    converging to an element $P \in \fix(\hat h)$
    and points in $\tilde R$ arbitrarily close to $P$ that move toward $P$.
\end{lem}

The following lemma is the key step in proving the correspondence between principal points in $\Gamma$
and principal automorphisms.

\begin{lem}[cf. Lemma 3.26 of \cite{FeighnHandel}]
    \label{fixedpointsfromfixeddirections}
    Suppose that $f\colon \mathcal{G} \to \mathcal{G}$
    is a rotationless relative train track map
    satisfying the conclusions of \Cref{improvedrelativetraintrack},
    that $\tilde f\colon \Gamma \to \Gamma$ is a lift of $f$,
    that $\tilde v$ is a fixed point for $\tilde f$
    and that $D\tilde f$ fixes the direction at $\tilde v$
    determined by an edge $\tilde E$ which is a lift of an edge $E \subset H_r$.
    Suppose that $H_r$ is not the bottom half of a dihedral pair.
    There is a fixed point $P \in \fix(\hat f)$
    such that the ray $\tilde R$ from the initial endpoint of $\tilde E$ to $P$
    contains $\tilde E$, projects into $G_r$ and has the following properties
    provided that $P \in \partial_\infty(F,\mathscr{A})$.
    \begin{enumerate}
        \item There are points in $\tilde R$ arbitrarily close to $P$
            that either move toward $P$ or are fixed by $\tilde f$.
            If $\fix(\hat f) \ne \{\hat T_c^{\pm}\}$ for any nonperipheral $c \in F$,
            then $P$ belongs to $\fix_N(\hat f)$.
        \item If $H_r$ is an exponentially growing stratum,
            then $P$ is an attractor whose limit set is the unique attracting lamination of height $r$,
            the interior of $\tilde R$ is fixed point free and $\tilde R$
            projects to an $r$-legal ray in $G_r$.
        \item If $H_r$ is non-exponentially growing but not almost fixed
            then $\tilde R \setminus \tilde E$
            projects into $G_{r-1}$.
        \item No point in the interior of $\tilde R$ is mapped to $\tilde v$
            by any iterate of $\tilde f$.
            This item holds without the assumption that $P \in \partial_\infty(F,\mathscr{A})$.
    \end{enumerate}
\end{lem}

\begin{proof}
    We follow the argument in \cite[Lemma 3.26]{FeighnHandel}.
    The second statement in item 1 follows 
    from the first and \Cref{movingtowardsattractors}.

    We proceed by induction on $r$, beginning with $r= 1$.
    If $G_1$ is not almost fixed,
    then it is exponentially growing.
    \Cref{fixedpointstolaminations} and \Cref{fixeddirectionsproperties}
    imply the existence of $P$ satisfying the conclusions of the lemma.

    If $G_1$ is almost fixed,
    then since we assume that $H_1$ is not the bottom half of a dihedral pair,
    it consists of a single edge,
    and by property \hyperlink{F}{(F)}, we have that $G_1$ is its own core.
    Write $\mathbb{F} = \fix(\Phi|_{\pi_1(G_1,v)})$.
    If the splitting induced by $G_1$ of $\mathbb{F}$ is nontrivial,
    then we may choose $P$ to be the endpoint of any ray $\tilde R$
    that begins with $\tilde E$ and is contained in $\fix(\tilde f)$.
    If the splitting of $\mathbb{F}$ is trivial,
    then since $f$ is rotationless,
    at least one vertex group of $G_1$ is infinite,
    and we may take $P$ to be a fixed point for $\tilde f$
    with infinite valence connected to $\tilde v$ by a path of length at most two.
    This completes the $r=1$ case,
    so assume the lemma holds for edges with height less than $r$.

    If $H_r$ is exponentially growing, then again the existence of $P$ follows
    from \Cref{fixedpointstolaminations} and \Cref{fixeddirectionsproperties}.
    Therefore we may assume that $H_r$ is non-exponentially growing.
    Suppose first that $H_r$ is not almost fixed.
    Let $h\colon \Gamma_{r-1} \to \Gamma_{r-1}$ be as in the paragraph
    ``Restricting to $G_{r-1}$ for non-exponentially growing strata.''
    If $\fix(h)$ is nonempty, then the initial endpoint of $\tilde E$ and some $\tilde x \in \fix(h)$
    cobound an indivisible almost Nielsen path.
    If $\tilde x$ has infinite valence then it is the endpoint $P$.
    If $\tilde x$ has finite valence, then $\tilde x$ is principal,
    there is a fixed direction in $\Gamma_{r-1}$ based at $\tilde x$,
    and the existence of $P$ follows from the inductive hypothesis.
    The case where $\fix(h) = \varnothing$ follows from \Cref{negfixedboundarypoint}.

    If $H_r$ is almost fixed but not a forest, 
    then the argument for the base case applies.
    So suppose $H_r$ is an almost fixed forest,
    that is not the bottom half of a dihedral pair,
    i.e.~a single edge $E$ with distinct endpoints $v$ and $w$
    and at least one incident vertex group trivial.
    If $w$ has infinite vertex group, the endpoint $\tilde w$ of $\tilde E$
    is $P$ and the ray $\tilde R$ is the edge $\tilde E$.
    So suppose $w$ has finite, possibly trivial vertex group.
    By \hyperlink{F}{(F)} and \Cref{propertyPconsequence},
    each vertex with trivial vertex group is contained 
    in a non-contractible component $C$ of $G_{r-1}$ and is thus principal for $f$.
    Since $f$ is rotationless, $\tilde v$ is connected to a lift of this principal vertex
    by a path of fixed edges for $\tilde f$ beginning with $\tilde E$ of length at most two.
    The existence of $\tilde R$ then follows from the inductive hypothesis.
\end{proof}

Suppose $f\colon \mathcal{G} \to \mathcal{G}$ 
is a topological representative
with non-exponentially growing stratum $H_i$
which is not almost periodic and consists of a single edge $E_i$.
A \emph{basic path of height $i$} is a path of the form $E_i\gamma$ or $E_i\gamma\bar E_i$,
where $\gamma \subset G_{i-1}$ is a nontrivial tight path.

\begin{lem}[cf. Lemma 4.1.4 of \cite{BestvinaFeighnHandel}]
    \label{negsplitting}
    Suppose that $f\colon \mathcal{G} \to \mathcal{G}$ and $E_i$ are as above,
    and that $\sigma \subset G_i$ is a tight path
    that meets the interior of $H_i$
    and whose endpoints are not contained in the interior of $E_i$.
    The path $\sigma$ has a splitting whose pieces are basic paths of height $i$
    or are contained in $G_{i-1}$.
\end{lem}

\begin{proof}
    The proof is identical to \cite[Lemma 4.1.4]{BestvinaFeighnHandel}.
    Choose a lift $\tilde f\colon \Gamma \to \Gamma$ and a lift $\tilde\sigma \subset\Gamma$.
    A splitting of $\sigma$ is determined by a set of points in $\tilde\sigma$.
    Fix $k > 0$. 
    Suppose that $\tilde x$ is a point of $\tilde\sigma$ decomposing $\tilde\sigma$
    into $\tilde\sigma_1$ and $\tilde\sigma_2$ and that
    $f^k(\tilde x)$ belongs to $\tilde f^k_\sharp(\tilde\sigma)$.
    Then $\tilde f^k_\sharp(\tilde\sigma) 
    = \tilde f^{k}_\sharp(\tilde\sigma_1)\tilde f^k_\sharp(\tilde\sigma_2)$.
    Observe further that the set of points $\tilde x$ satisfying
    $\tilde f^k(\tilde x) \in \tilde f^k_\sharp(\tilde\sigma)$ is closed.

    Notice that there is an initial segment $E^k_i$ of $E_i$
    such that $f^k(E_i) = E_i$,
    and that no other points of $G_i$ are mapped into the interior of $E_i$ by $f^k$.
    Suppose a copy of $E_i$ cancels with a copy of $\bar E_i$
    when $f^k(\sigma)$ is tightened to $f^k_\sharp(\sigma)$.
    Then there is a subpath $\mu$ of $\sigma$ connecting a copy of $E^k_i$ to a copy of $\bar E^k_i$
    such that $f^k_\sharp(\mu)$ is trivial (with trivial vertex group element).
    But since $\mu$ is a closed path and $f$ is a homotopy equivalence,
    this is impossible.
    Therefore no such cancellation occurs.
    The argument in the previous paragraph shows that if $\tilde x$
    is a point in the interior of a lift of $E^k_i$ or $\bar E^k_i$
    decomposing $\tilde\sigma$ into $\tilde\sigma_1\tilde\sigma_2$,
    then 
    $\tilde f^k_\sharp(\tilde\sigma) 
    = \tilde f^k_\sharp(\tilde\sigma_1)\tilde f^k_\sharp(\tilde\sigma_2)$.
    In fact, the same holds true if $\tilde x$ is the initial vertex of any lift
    of $\tilde E_i$ or the terminal vertex of any lift of $\bar E_i$.
    Since $k$ was arbitrary, $\tilde\sigma$ can be split at these points.
    It follows from the definition of a splitting that $\tilde\sigma$ can be split
    at all of these points simultaneously, and the result follows.
\end{proof}

\begin{lem}
    \label{nielsenpathsfixeddirection}
    Suppose $f\colon \mathcal{G} \to \mathcal{G}$ is rotationless
    and that $\gamma$ is an indivisible almost Nielsen path.
    Then at least one of the initial and terminal directions of $\gamma$ is fixed.
\end{lem}

\begin{proof}
    Suppose $r$ is the highest stratum such that $\gamma$ meets the interior of $H_r$.
    $H_r$ cannot be a zero stratum,
    nor can it be almost fixed
    since $\gamma$ is assumed to be indivisible.
    If $H_r$ is exponentially growing, the lemma follows from \Cref{finitelymanynielsenpaths}.
    If $H_r$ is non-exponentially growing,
    then it is a single edge $E_r$
    and the lemma follows from \Cref{negsplitting}:
    up to reversing the orientation of $\gamma$ and multiplying by vertex group elements at the ends,
    we have $\gamma = E_r\sigma$ or $\gamma = E_r\sigma \bar E_r$ for a nontrivial path
    $\sigma \subset G_{r-1}$.
    In the former case the initial direction of $\gamma$ is fixed;
    in the latter both the initial and terminal directions are fixed.
\end{proof}

Now, under the assumption that $f\colon \mathcal{G} \to \mathcal{G}$ is rotationless,
we can prove the converse of \Cref{principalliftprincipalfixed}.

\begin{cor}[cf. Corollary 3.27 of \cite{FeighnHandel}]
    \label{principalfixedprincipallift}
    Suppose that $f\colon \mathcal{G} \to \mathcal{G}$ is a relative train track map
    satisfying the conclusions of \Cref{improvedrelativetraintrack}
    and is rotationless.
    If some, and hence every, $\tilde x \in \fix(\tilde f)$ is principal
    and $\fix(\tilde f)$ projects to a non-exceptional almost Nielsen class,
    then $\tilde f$ is principal.
\end{cor}

\begin{proof}
    We follow the proof in \cite[Corollary 3.27]{FeighnHandel}.
    Assume that each point in $\fix(\tilde f)$ is principal.
    If $\fix(\tilde f)$ contains two points of infinite valence,
    then $\tilde f$ is principal.

    Assume that $\fix(\tilde f)$ contains a vertex $\tilde v$ of infinite valence.
    If there is a fixed direction based at $\tilde v$,
    \Cref{fixedpointsfromfixeddirections} produces a second point in $\fix_N(\hat f)$,
    and we see that $\tilde f$ is principal.
    If there is no fixed direction at $\tilde v$,
    by the assumption that $\fix(\tilde f)$ projects to a non-exceptional almost Nielsen class,
    there is a second point $\tilde x$ in $\fix(\tilde f)$, necessarily of finite valence.
    The tight path $\tilde\gamma$ from $\tilde v$ to $\tilde x$
    projects to an almost Nielsen path $\gamma$ for $f$.
    Since we assume that there is no fixed direction at $\tilde v$,
    by re-choosing $\tilde x$ we may assume that $\gamma$ is indivisible,
    and therefore by \Cref{nielsenpathsfixeddirection} 
    that the terminal direction of $\tilde\gamma$ is fixed by $D\tilde f$.

    Let $x$ be the vertex of $G$ that $\tilde x$ projects to.
    Since $f\colon \mathcal{G} \to \mathcal{G}$ is rotationless,
    the restriction of $Df$ to the set of almost periodic directions based at $x$ is the identity.
    We saw in \Cref{basicssection} that the map $D\tilde f$ is $Df$ followed by left-multiplication
    by some element of $\mathcal{G}_x$;
    in other words, if the edge $\tilde e$ corresponds to an almost periodic direction
    $(x,e)$ based at $x$,
    the direction $D\tilde f(\tilde e)$ corresponds to $(hx,e)$ for some element $h \in \mathcal{G}_x$.
    Since there is a $D\tilde f$-fixed direction based at $\tilde x$,
    we must have $h =1$.
    It follows by \Cref{legalturn} that there is a second periodic and hence fixed direction
    based at $\tilde x$, producing by \Cref{fixedpointsfromfixeddirections} 
    a second point in $\fix_N(\hat f)$, and we are done.

    Suppose every fixed point $\tilde x$ has finite valence.
    The argument in the previous paragraphs produces a fixed point $\tilde x$
    with a fixed direction.
    By \Cref{legalturn}, there is a legal turn based at $\tilde x$,
    so there are at least two periodic and hence fixed directions at $\tilde x$.
    As in \cite{FeighnHandel}, observe that if some $\tilde x \in \fix(\tilde f)$
    has three fixed directions,
    then \Cref{fixedpointsfromfixeddirections} produces at least three points in $\fix_N(\hat f)$,
    and we are done.
    So suppose that there are either zero or 
    exactly two fixed directions at each $\tilde x \in \fix(\tilde f)$.
    By the assumption that each $\tilde x$ is principal,
    if $\fix(\tilde f)$ contains an edge,
    then $\fix(\tilde f)$ may not be a line, so some edge of $\fix(\tilde f)$
    is incident to a valence-one vertex $\tilde w$ of $\fix(\tilde f)$.
    By \Cref{propertyPconsequence}, $\tilde w$
    projects to a valence-one vertex $w$ of some fixed stratum,
    so it is contained in a lower filtration element;
    by property \hyperlink{F}{(F)}, in proving \Cref{propertyPconsequence}
    there is no loss in assuming that this filtration element $G_k$ is its own core.
    But then \Cref{legalturn} implies that there is a legal turn based at $w$ in $G_k$,
    and thus $\tilde w$ has at least three fixed directions,
    since we assume $\tilde w$ has finite valence.
    Thus by this contradiction we may assume that $\fix(\tilde f)$ contains no edges.

    Choose an edge $E$ in $H_r$ and a lift $\tilde E$ whose initial direction is fixed
    and based at some $\tilde x \in \fix(\tilde f)$.
    Let $\tilde R$ be the ray that begins with $\tilde E$ 
    and ends at some fixed point $P \in \fix(\hat f)$
    as in \Cref{fixedpointsfromfixeddirections}.
    Since we assume every fixed point in $\fix(\tilde f)$ has finite valence,
    we have that $P \in \partial_\infty(F,\mathscr{A})$.
    If $H_r$ is exponentially growing,
    then the limit set of $P$ is an attracting lamination,
    which implies by \Cref{laminationneveracircuit} that $P$ is not the endpoint of an axis.
    If $H_r$ is non-exponentially growing,
    then the ray $\tilde R \setminus \tilde E$ is contained in $G_{r-1}$,
    so $P$ is not the endpoint of an axis that contains $\tilde E$.
    It follows that the line composed of $\tilde R$ and the ray determined
    by the second fixed direction at $\tilde x$ is not an axis.
    It follows that $\fix(\hat f)$ is not equal to $\{\hat T_c^{\pm}\}$
    for any nonperipheral $c \in F$
    and that every point in $\fix(\hat f)$ produced by \Cref{fixedpointsfromfixeddirections}
    is contained in $\fix_N(\hat f)$.
    We have shown that $\fix_N(\hat f)$ contains at least two points
    and is not the endpoint set of an axis.

    So assume arguing toward a contradiction that $\fix_N(\hat f)$ is the endpoint set
    of a lift $\tilde\lambda$ of a generic leaf of an attracting lamination.
    Since $\lambda$ is birecurrent and contains $E$, the stratum $H_r$
    is exponentially growing and the second fixed direction based at $\tilde x$
    comes from an edge of $H_r$.
    Item 2 of \Cref{fixedpointsfromfixeddirections} implies that $\lambda$ is $r$-legal
    and hence does not contain any indivisible almost Nielsen paths of height $r$.
    It follows that $\tilde x$ is the only fixed point in $\tilde\lambda$.
    Since $\fix(\tilde f)$ is principal
    it must contain a point other than $\tilde x$;
    that point would have a fixed direction 
    that does not come from the initial edge
    of a ray converging to an endpoint of $\tilde\lambda$.
    This contradiction completes the proof.
\end{proof}

We now turn to proving that rotationless relative train track maps 
satisfying the conclusions of \Cref{improvedrelativetraintrack}
represent rotationless outer automorphisms and vice versa.

\begin{lem}[cf.~Lemma 3.28 of \cite{FeighnHandel}]
    \label{rotationlessnielsenpaths}
    Suppose that $f\colon \mathcal{G} \to \mathcal{G}$ is a rotationless relative train track map
    satisfying the conclusions of \Cref{improvedrelativetraintrack}.
    Every periodic almost Nielsen path with principal endpoints
    is an almost Nielsen path (i.e.~has period one).
\end{lem}

\begin{proof}
    The proof follows \cite[Lemma 3.28]{FeighnHandel}.
    We may assume that our periodic almost Nielsen path $\sigma$ 
    is either a single (almost periodic) edge
    or an indivisible periodic almost Nielsen path.
    In the former case, $\sigma$ is an almost periodic edge with a principal endpoint,
    so is almost fixed.
    Therefore we may assume that $\sigma$ is indivisible.

    The proof is by induction on the height $r$ of $\sigma$.
    The base case of $r = 0$ is vacuous.
    The case that $H_r$ is exponentially growing follows from the second paragraph of
    \Cref{finitelymanynielsenpaths}.

    Therefore we may assume that $H_r$ is a single non-exponentially growing edge $E_r$
    which is not almost fixed.
    \Cref{negsplitting} implies that up to reversing the orientation of $\sigma$
    and multiplying by vertex group elements at the ends
    we have $\sigma = E_r\mu$ or $E_r\mu\bar E_r$ for some path $\mu \subset G_{r-1}$.
    Let $\tilde E_r$ be a lift of $E_r$ with initial endpoint $\tilde v$
    and $\tilde f\colon \Gamma \to \Gamma$ the lift that fixes $\tilde v$
    and the initial direction of $\tilde E_r$.
    Let $h\colon \Gamma_{r-1} \to \Gamma_{r-1}$ be as in the paragraph
    ``Restricting to $G_{r-1}$''.
    By \Cref{principalfixedprincipallift}, $\tilde f$ is principal.
    Let $\tilde w$ be the terminal endpoint of the lift of $\tilde\sigma = \tilde E_r\tilde\mu$
    that begins with $\tilde E_r$.

    If $\sigma = E_r\mu$, then $\tilde w $ belongs to $\Gamma_{r-1}$.
    If the period $p$ of $\sigma$ is not 1,
    then the tight path $\tilde\tau = h_\sharp(\tilde\mu)$ connecting $\tilde w$ to $h(\tilde w)$
    projects to a nontrivial periodic almost Nielsen path $\tau \subset G_{r-1}$
    that is closed because $\tilde w$ projects to $w \in \fix(f)$.
    Since $\tilde w$ is principal,
    the inductive hypothesis implies that $\tau$ has period one.
    The projection of the closed path $\tilde\tau h(\tilde\tau)\ldots h^{p-1}(\tilde\tau)$
    is therefore homotopic to $\tau^p$.
    But $\tau$ and therefore $\tau^p$ determine non-peripheral conjugacy classes in $F$,
    while $\tilde\sigma\tilde\tau h(\tilde \tau)\ldots h^{p-1}(\tilde\tau)$
    is homotopic to $\tilde\sigma$,
    so this is a contradiction;
    we conclude that $p = 1$ in the case that $\sigma = E_r\mu$.

    Suppose now that $\sigma = E_r\mu\bar E_r$.
    If $\fix(h^p)$ is nonempty,
    then the tight paths $\tilde\sigma_1$ connecting $\tilde v$ to $\tilde x \in \fix(h^p)$
    and $\tilde \sigma_2$ connecting $\tilde x$ to $\tilde w$ are periodic almost Nielsen paths
    whose concatenation is homotopic to $\tilde\sigma$.
    By the preceding case $\sigma_1$ and $\sigma_2$ and hence $\sigma$ have period one.
    Therefore assume that $\fix(h^p) = \varnothing$.

    Let $T_c\colon \Gamma \to \Gamma$ be the automorphism of the natural projection
    satisfying $T_c(\tilde v) = \tilde w$ and taking
    the initial direction  of $\tilde\sigma$ to the terminal direction  of $\tilde\sigma$.
    Then $T_c$ is a nonperipheral element of $F$,
    it commutes with $\tilde f^p$ and its axis $A_c$ is contained in $\Gamma_{r-1}$.
    \Cref{boundarybasics} implies that $\hat T_c^{\pm} \in \fix(\hat h^p)$.
    If $\Phi$ is the principal automorphism corresponding to $\tilde f$,
    then we have that $\hat T_{\Phi(c)} = \hat f\hat T_c\hat f^{-1}$,
    which implies that $\hat T_{\Phi(c)}^\pm = \hat h(\hat T_c^\pm) \in \fix(\hat h^p)$
    and that $A_{\Phi(c)}$ is contained in $\Gamma_{r-1}$.
    If $\{\hat T_c^\pm\}$ is not $\hat h$-invariant,
    then $\fix_N(\hat h^p)$ contains the four points $\{\hat T_c^\pm\} \cup \{\hat h(\hat T_c^\pm)\}$,
    and $h^p$ is a principal lift of $f|_C$,
    where $C$ is the component of $G_{r-1}$ that contains the terminal endpoint of $E_r$.
    But $\fix(h) = \varnothing$, which contradicts \Cref{principalnonemptyfixed},
    so we conclude $\hat T^\pm_c \in \fix(\hat h)$.
    It follows that $\tilde f$ commutes with $T_c$,
    and hence that $\tilde w \in \fix(\tilde f)$,
    so $p=1$.
    This completes the inductive step.
\end{proof}

The following is the main result of this section. 

\begin{prop}[cf.~Proposition 3.29 of \cite{FeighnHandel}]
    \label{rotationlessisrotationless}
    Suppose that $f\colon \mathcal{G} \to \mathcal{G}$ is a relative train track map
    representing $\varphi \in \out(F,\mathscr{A})$
    and satisfying the conclusions of \Cref{improvedrelativetraintrack}.
    Then if $f\colon \mathcal{G} \to \mathcal{G}$ is rotationless,
    so is $\varphi$.
    If we assume additionally that each such $f$ has a rotationless iterate,
    then if $\varphi$ is rotationless,
    so is $f\colon \mathcal{G} \to \mathcal{G}$.
\end{prop}

\begin{proof}
    We follow the proof of \cite[Proposition 3.29]{FeighnHandel}.
    Suppose that $f\colon \mathcal{G} \to \mathcal{G}$ is rotationless,
    write $g = f^k$ for some $k \ge 1$,
    and suppose that $\tilde g\colon \Gamma \to \Gamma$
    is a principal lift of $g$.
    \Cref{principalnonemptyfixed} and \Cref{principalliftprincipalfixed}
    imply that $\fix(\tilde g)$ is a nonempty set of principal fixed points
    projecting to a non-exceptional almost Nielsen class.
    Since $f$ is rotationless,
    for each $\tilde v \in \fix(\tilde g)$,
    by \Cref{rotationlessdirections}
    there is a lift $\tilde f\colon \Gamma \to \Gamma$ that fixes $\tilde v$
    and all the same directions at $\tilde v$ that $\tilde g$ does.

    Observe that if $\Phi_1$ and $\Phi_2$ are distinct representatives of $\varphi$,
    then $\fix_N(\hat\Phi_1) \cap \fix_N(\hat\Phi_2)$
    is contained in $\fix(\hat \Phi^{-1}_1\hat\Phi_2) = \fix(\hat T_c)$,
    which is either empty, a single point in $V_\infty(F,\mathscr{A})$
    or the endpoints $\{\hat T_c^{\pm}\}$ of an axis.
    Therefore, if $\Phi_1$ and $\Phi_2$ are distinct and principal,
    then $\fix_N(\hat\Phi_1) \ne \fix_N(\hat\Phi_2)$.

    It therefore suffices to show that $\fix_N(\hat f) = \fix_N(\hat g)$,
    since the argument in the previous paragraph shows that this implies $\tilde f^k = \tilde g$,
    so the map $\Phi \mapsto \Phi^k$ will be a surjection $P(\varphi) \to P(\varphi^k)$,
    and thus (again by the previous paragraph) a bijection.

    Suppose $\tilde\sigma$ is the tight path connecting $\tilde v$
    to another point in $\fix(\tilde g)$.
    This path projects to an almost Nielsen path for $g = f^k$,
    and hence by \Cref{rotationlessnielsenpaths} to an almost Nielsen path for $f$.
    Thus we see that $\fix(\tilde f) = \fix(\tilde g)$,
    and thus $\hat f$ and $\hat g$ have the same fixed points in $V_\infty(F,\mathscr{A})$.
    It also follows that $\tilde g$ and $\tilde f$ commute
    with the same nonperipheral automorphisms of the natural projection,
    and \Cref{GJLLprop} implies that $\fix_N(\hat f)$ and $\fix_N(\hat g)$
    have the same non-isolated fixed points in $\partial_\infty(F,\mathscr{A})$.

    Each isolated point $P \in \fix_N(\hat g)\cap \partial_\infty(F,\mathscr{A})$
    is an attractor for $\hat g$.
    It suffices to show that $P \in \fix_N(\hat f)$.
    By \Cref{fixeddirectionsproperties}, there is a ray $\tilde R$
    that terminates at $P$ that intersects $\fix(\tilde g)$ only in its initial endpoint
    and whose initial direction is fixed by $D\tilde g$ and thus by $D\tilde f$.
    Let us assume that the height $r$ of the initial edge $\tilde E$ of $\tilde R$
    is minimal among all choices of $\tilde R$ converging to $P$.
    By \Cref{fixedpointsfromfixeddirections}, there is a ray $\tilde R'$
    that converges to some $P' \in \fix(\hat f)$.
    It suffices to show that $P = P'$, since a repeller for $\hat f$ cannot be an
    attractor for $\hat g$.
    If $H_r$ is exponentially growing, then item 2 of \Cref{fixedpointsfromfixeddirections}
    implies that $P'$ is an attractor for $\hat f$ (and hence for $\hat g$)
    and item 3 of \Cref{fixeddirectionsproperties} applied to $\hat g$ implies that $P = P'$.
    We may therefore assume that $H_r$ is non-exponentially growing.
    If $\tilde E$ is fixed, or
    if there is a $\tilde g$-fixed point $\tilde x$ in $\tilde R' \setminus \tilde E$,
    then the ray connecting $\tilde x$ to $P$ is in $G_{r-1}$,
    in contradiction to our choice of $\tilde R$.
    Therefore we may assume that $\fix(\tilde g)$ intersects the interior of $\tilde R'$ trivially.
    Item 1 of \Cref{fixedpointsfromfixeddirections} implies there exists $\tilde x \in \tilde R'$
    that is moved towards $P'$ by $f$.
    Then \Cref{producingafixedpoint} implies that $P = P'$.
    Therefore if $f\colon \mathcal{G} \to \mathcal{G}$ is rotationless,
    then $\varphi$ is rotationless.

    Now suppose $\varphi$ is rotationless and that $f\colon \mathcal{G} \to \mathcal{G}$
    is a relative train track map satisfying the conclusions of \Cref{improvedrelativetraintrack}
    and representing $\varphi$.
    By assumption there is a rotationless iterate $g$ satisfying $g = f^k$ for some $k > 0$.
    Given a principal vertex $v \in \fix(g)$,
    there is a lift $\tilde v$ of $v$ and a principal lift $\tilde g$ of $g$ that fixes $\tilde v$.
    Since $\varphi$ is rotationless,
    there is a (principal) lift $\tilde f$ of $f$ such that $\tilde f^k = \tilde g$
    satisfying $\fix_N(\hat f) = \fix_N(\hat g)$.
    To complete the proof we must show that $\tilde v$ is fixed by $\tilde f$
    and that each direction fixed by $D \tilde g$ is fixed by $D \tilde f$.
    Note that the assumption that $\fix_N(\hat g) = \fix_N(\hat f)$
    implies that if $\tilde x$ is a vertex with infinite stabilizer fixed by $\tilde g$,
    then it is fixed by $\tilde f$.

    Suppose that there is at most a single $D\tilde g$-fixed direction at $\tilde v$.
    Then $\tilde v$ has infinite stabilizer and is thus fixed by $\tilde f$.
    There is at most one $D\tilde f$-periodic direction at $\tilde v$.
    Suppose there is a $D\tilde g$-fixed direction $\tilde d$ at $\tilde v$.
    The edge determined by $\tilde d$ extends to a ray $\tilde R$
    converging to a fixed point $P \in \fix_N(\hat g) = \fix_N(\hat f)$.
    The ray $\tilde R$ satisfies $f_\sharp(\tilde R) = \tilde R$.
    If the direction $\tilde d$ is not fixed by $D\tilde f$,
    then there is a point $\tilde x$ in the interior of $\tilde R$
    satisfying $\tilde f(\tilde x) = \tilde v$,
    contradicting item 4 of \Cref{fixedpointsfromfixeddirections} applied to $\tilde g = \tilde f^k$,
    so we conclude the direction $\tilde d$ is fixed by $D\tilde f$.

    Now suppose there are at least two $D\tilde g$-fixed directions 
    $\tilde d_1$ and $\tilde d_2$
    at $\tilde v$.
    (We no longer assume that $\tilde v$ has infinite stabilizer.)
    The edges determining these directions extend to rays $\tilde R_1$ and $\tilde R_2$
    that converge to $P_1$ and $P_2$ in $\fix_N(\hat g) = \fix_N(\hat f)$;
    denote the line connecting $P_1$ to $P_2$ by $\tilde\gamma$.
    We have $\tilde f_\sharp(\tilde\gamma) = \tilde\gamma$
    and the turn $(\tilde d_1,\tilde d_2)$ is legal for $\tilde g$ and hence for $\tilde f$.
    If it is not the case that $\tilde f(\tilde v)$ belongs to $\tilde \gamma$,
    then there exists $\tilde y \in \tilde\gamma$ such that $\tilde f(\tilde v) = \tilde f(\tilde y)$.
    But then we have $\tilde f^k(\tilde y) = \tilde f^k(\tilde v) = \tilde v$,
    which contradicts item 4 of \Cref{fixedpointsfromfixeddirections} applied to $\tilde g$.
    Therefore $\tilde f(\tilde v)$ belongs to $\tilde\gamma$.
    Now suppose that $\tilde f(\tilde v) \ne \tilde v$.
    Write $\tilde v_0 = \tilde v$ and orient $\tilde\gamma$ so that $\tilde v < \tilde f(\tilde v)$
    in the order induced from the orientation.
    There exist $\tilde v_i \in \tilde \gamma$ for $1 \le i \le k$
    such that $\tilde v_i < \tilde v_{i-1}$ and $\tilde f(\tilde v_i) = \tilde v_{i-1}$.
    But then $\tilde f^k(\tilde v_k) = \tilde v$, 
    again contradicting item 4 of \Cref{fixedpointsfromfixeddirections}.
    Therefore $\tilde f(\tilde v) = \tilde v$.
    A final application of item 4 of \Cref{fixedpointsfromfixeddirections}
    implies that the directions $\tilde d_i$ are fixed by $D\tilde f$.
\end{proof}

\begin{prop}[cf.~Lemma 3.30 of \cite{FeighnHandel}]
    \label{rotationlessproperties}
    Suppose $\varphi \in \out(F,\mathscr{A})$ is rotationless.
    \begin{enumerate}
        \item Each periodic non-peripheral conjugacy class is fixed
            and each representative of that conjugacy class is fixed
            by some principal automorphism representing $\varphi$.
        \item Each attracting lamination $\Lambda^+$ in $\mathcal{L}(\varphi)$ is $\varphi$-invariant.
        \item A free factor of positive complexity that is invariant under an iterate of $\varphi$
            is $\varphi$-invariant.
    \end{enumerate}
\end{prop}

\begin{proof}
    We follow in outline the argument in \cite[Lemma 3.30]{FeighnHandel}.

    Item 2 follows from \Cref{principalEG}, \Cref{rotationlessisrotationless}
    and \Cref{aperiodicallaperiodic}.

    For item 1, suppose $c$ is a non-peripheral element whose conjugacy class is $\varphi^k$-invariant
    for some $k \ge 1$.
    We will show that there is a principal automorphism $\Phi_k \in P(\varphi^k)$ that fixes $c$.
    By \Cref{boundarybasics}, this is equivalent to the condition that
    $\fix_N(\hat\Phi_k)$ contains $\hat T_c^{\pm}$.
    Since $\varphi$ is rotationless, we may assume that $k = 1$, completing the proof of item 1.

    Let $f\colon \mathcal{G} \to \mathcal{G}$ be a relative train track map
    satisfying the conclusions of \Cref{improvedrelativetraintrack} and
    representing $\varphi^k$, and let $\tilde f\colon \Gamma \to \Gamma$ be
    a lift commuting with $T_c$.
    Since $\varphi$ is rotationless, so is $\varphi^k$,
    and \Cref{rotationlessisrotationless} implies $f$ is rotationless.
    We show that we may assume $\fix(\tilde f)$ is nonempty,
    following \cite[Lemma 4.1.2]{BestvinaFeighnHandel}.
    We have $\tilde f_\sharp(A_c) = A_c$.
    The set $\tilde S_\ell = \{ \tilde x \in A_c : \tilde f^\ell(\tilde x) \in A_c\}$
    is closed.
    An induction argument reveals that $\tilde f^N$ maps $\bigcap_{\ell=1}^N \tilde S_\ell$ onto $A_c$
    for all $N \ge 1$.
    Since $\bigcap_{\ell=1}^N \tilde S_\ell$ is $T_c$-invariant and nonempty,
    it intersects each fundamental domain of $A_c$.
    It follows that the $T_c$-invariant set $\bigcap_{\ell=1}^\infty \tilde S_\ell$ is nonempty,
    and any $T_c$-orbit in $\bigcap_{\ell=1}^\infty \tilde S_\ell$ determines a splitting of
    the circuit $\sigma$ (a splitting into a path) that $A_c$ projects to in $\mathcal{G}$.
    This path is an almost Nielsen path for $f$,
    and we conclude there is a lift $\tilde f$ commuting with $T_c$ with $\fix(\tilde f)$ nonempty.
    In particular by property \hyperlink{V}{(V)}, $\tilde f$ fixes a vertex of $\Gamma$.

    If some fixed vertex $\tilde v$ has infinite valence, we are done, as $\tilde f$ is principal.
    So assume all fixed vertices $\tilde v$ have finite valence.
    We argue as in \cite[Lemma 5.2]{BestvinaFeighnHandelSolvable}.
    By \Cref{legalturn}, there are at least two fixed directions based at $\tilde v$.
    If some edge $\tilde E$ based at $\tilde v$ determines a fixed direction
    but not a fixed edge, then \Cref{fixedpointsfromfixeddirections} produces a fixed point
    $P \in \fix(\hat f)$ that is not the endpoint of $A_c$.
    Suppose then that every fixed direction based at $\tilde v$ is determined by a fixed edge.
    
    Thus there is no loss in supposing that $\fix(\tilde f) = A_c$.
    The image of $A_c$ in $\mathcal{G}$ is a subgraph of groups that is a component of $\fix(f)$
    that is either a topological circle 
    or the quotient of $\mathbb{R}$ by the standard action of $C_2*C_2$.
    After reordering the filtration, we may assume that this subgraph of groups
    is a filtration element $G_i$.
    By our standing assumption that $F \ne F_1$ or $C_2*C_2$,
    there is a lowest stratum $H_j$ above $G_i$ that contains an edge incident to a vertex $v$ of $G_i$.
    Because by assumption the edges of $H_j$ do not determine fixed directions at $v$,
    it follows that $H_j$ is an edge $E_j$ and that $f(E_j) = E_j\sigma^m$ for some nonzero $m \ne 0$
    by \hyperlink{NEG}{(NEG)}.
    Let $\tilde E_j$ be a lift of $E_j$ with terminal endpoint in $A_c$
    and replace $\tilde f$ with the lift that fixes the initial endpoint of $\tilde E_j$.
    For this lift, $A_c$ is setwise but not pointwise fixed,
    and either the initial vertex of $\tilde E_j$ has infinite valence,
    or by \Cref{legalturn} and \Cref{fixedpointsfromfixeddirections} 
    there is a fixed point $P \in \fix(\hat f)$ that is not an endpoint of $A_c$.
    This completes the proof of the claim.

    We now turn to the proof of item 3,
    following the argument in \cite[Lemma 3.30]{FeighnHandel}.
    Suppose that a free factor $B$ of positive complexity is $\varphi^k$-invariant for some $k \ge 1$.
    If $B = F_1$, then it is $\varphi$-invariant by the first item of this lemma.
    If $B = C_2*C_2$, let $c$ be a generator of the index-two $F_1$ subgroup of $B$.
    By item 1, the conjugacy class of $c$ is $\varphi$-invariant,
    and it follows that the conjugacy class of $B$ is $\varphi$-invariant.
    So assume $B$ is neither $F_1$ nor $C_2*C_2$.
    Let $\mathscr{C}$ be the set of lines $\gamma$ in $\mathcal{B}$ that are carried by $B$
    and for which there exists a principal lift $\Phi$ of an iterate of $\varphi$
    and a lift $\tilde\gamma$ of $\gamma$ whose endpoints are contained in $\fix_N(\hat\Phi)$.
    Since $\varphi$ is rotationless, each $\gamma$ is $\varphi_\sharp$-invariant,
    so $\mathscr{C}$ is $\varphi$-invariant.
    It is clearly carried by $B$.
    It follows by uniqueness of $\mathcal{F}(\mathscr{C})$ as in \Cref{minimalcomplexityfreefactorsystem}
    that $\mathcal{F}(\mathscr{C})$ is $\varphi$-invariant.
    So to complete the proof, we need to show that no proper $\varphi$-invariant free factor system
    of $B$ carries $\mathscr{C}$.

    Suppose to the contrary that such a free factor system $\mathcal{F}$ exists.
    By \Cref{improvedrelativetraintrack} there is a relative train track map
    $g\colon \mathcal{G}' \to \mathcal{G}'$ representing $\varphi^k|_B$
    in which $\mathcal{F}$ is represented by a proper filtration element $G'_r \subset G'$.
    After replacing $\varphi|_B$ and $g$ by iterates, we may assume they are rotationless.
    We claim that there is a principal vertex $v \in G'$
    whose link contains an edge $E$ of $G' \setminus G'_r$ that determines an almost fixed direction.
    If $G'\setminus G'_r$ contains an exponentially growing stratum,
    this follows from \Cref{principalEG}.
    If not, then the highest stratum is non-exponentially growing,
    thus it consists of a single edge whose initial vertex is principal.
    We claim that there is a principal lift $\tilde g \colon \Gamma' \to \Gamma'$
    and a line $\tilde \gamma$ whose endpoints are contained in $\fix_N(\hat g)$
    whose projected image $\gamma$ crosses $E$ and so is not carried by $G'_r$.
    This follows from \Cref{fixedpointsfromfixeddirections}:
    one endpoint of $\tilde \gamma$ is $\tilde v$ if $\mathcal{G}'_v$ is infinite,
    otherwise there are at least two periodic and hence fixed directions at $v$,
    one of which is determined by $E$.
    The automorphism $\Phi' \in P(\varphi^k|_{B})$ determined by $\tilde g$
    extends to a principal automorphism $\Phi \in P(\varphi^k)$
    with $\fix_N(\hat\Phi') \subset \fix_N(\hat\Phi)$.
    Thus we conclude $\gamma \in \mathscr{C}$ in contradiction to our choice of $\mathcal{F}$.
\end{proof}

\begin{cor}
    If $\varphi$ is rotationless and $B$ is a $\varphi$-invariant free factor of positive complexity,
    then $\theta = \varphi|_B$ is rotationless.
\end{cor}

\begin{proof}
    The proof is identical to \cite[Corollary 3.31]{FeighnHandel}.
    Item 1 of the previous lemma handles the case $B = F_1$ and $B = C_2*C_2$.
    Let $f\colon \mathcal{G} \to \mathcal{G}$ be a relative train track map
    satisfying the conclusions of \Cref{improvedrelativetraintrack}
    with associated filtration $\varnothing = G_0 \subset G_1 \subset \cdots \subset G_m = G$
    such that the conjugacy class of $B$ is realized by $G_i$ for some $i$.
    (We may assume that $G_i$ is its own core by \hyperlink{F}{(F)}.)
    \Cref{rotationlessisrotationless} implies that $f\colon \mathcal{G} \to \mathcal{G}$
    is rotationless.
    The restriction of $f$ to $G_i$ is a rotationless relative train track map
    representing $\theta$ and satisfying the conclusions of \Cref{improvedrelativetraintrack},
    so by \Cref{rotationlessisrotationless}, we conclude that $\theta$ is rotationless.
\end{proof}

\section{CTs for free products}
\label{CTsection}
In this section we construct CTs for free products.
We begin by recalling the necessary definitions from \cite[Section 4]{FeighnHandel}.

\paragraph{Almost linear edges}
Given a non-peripheral element $c \in F$,
let $[c]_u$ be the \emph{unoriented conjugacy class determined by $c$.}
That is, $[c]_u = [d]_u$ if $d$ is conjugate to $c$ or $c^{-1}$.
If $\sigma$ is a closed path,
write $[\sigma]_u$ be the \emph{unoriented conjugacy class determined by $\sigma$}
thought of as a circuit.

Suppose that $f\colon \mathcal{G} \to \mathcal{G}$ is a rotationless relative train track map
satisfying the conclusions of \Cref{improvedrelativetraintrack}.
Each non-almost periodic non-exponentially growing stratum $H_i$
is a single edge $E_i$ satisfying $f(E_i) = g_iE_iu_i$
for a vertex group element $g_i$ and 
for a nontrivial path $u_i \subset G_{i-1}$.
If we assume additionally that the terminal vertex of $E_i$ is fixed,
(this will be part of the definition of a CT)
then the path $u_i$ is closed.
Feighn and Handel note that $u_i$ is sometimes called the \emph{suffix} for $E_i$.
If $u_i$ is an almost Nielsen path,
then we say that $E_i$ is an \emph{almost linear edge,}
and a \emph{linear edge} if it is a Nielsen path (that is, $f_\sharp(u_i) = u_i$ on the nose).
In the situation of a linear edge,
we define the \emph{axis} for $E_i$ to be $[w_i]_u$,
where $w_i$ is root free and $u_i = w_i^{d_i}$ for some nonzero integer $d_i$.

Suppose that $E_i$ is an almost linear edge with suffix $u_i$ satisfying $f_\sharp(u_i) = gu_i h$
for vertex group elements $g$ and $h$ in $\mathcal{G}_v$.
Suppose that $f(E_i) = x_i E_i \cdot u_i$ is a splitting and that
that $E_i$ is also a linear edge for $f^k_\sharp$,
i.e.~that $f^k_\sharp(u_i\cdot f_\sharp(u_i)\cdots  f^{k-1}_\sharp(u_i)) 
= u_i \cdot f_\sharp(u_i) \cdots f^{k-1}_\sharp(u_i)$.
Let $g_0 = h_0 = 1$ and inductively define $g_\ell = f_v(g_{\ell-1})g$
and $h_\ell = hf_v(h_{\ell-1})$.
Then we have
\[  u_1 \cdot f_\sharp(u_i) \cdots f^{k-1}_\sharp(u_i) = g_0u_i h_0g_1 u_i \ldots g_{k-1}u_ih_{k-1}. \]
The condition that $E_i$ is a linear edge for $f^k_\sharp$ implies that $g_{k} = 1$.
An easy calculation shows that $g_{k + \ell} = f^\ell_v(g_k)g_\ell = g_\ell$.
Since we have $h_{k+\ell}g_{k+\ell+1} = h_\ell g_{\ell +1}$ for $0\le \ell \le k-2$,
this implies in particular that $h_{k} = h_0$ and hence that $h_{k+\ell} = h_\ell$ for all $\ell \ge 0$.
The argument in the proof of \Cref{rotationlesscriterion}
implies that if $f$ is rotationless,
the condition that $g_k = h_k = 1$ implies that actually $g = h = 1$,
so $E_i$ was a linear edge to begin with.

If the terminal vertex of $E_i$ is the center vertex of a non-fixed dihedral pair
and $u_i$ is contained in that dihedral pair,
then $u_i$ is a periodic Nielsen path of period two,
as is $u_i f_\sharp(u_i)$. 
In this situation, we say that $E_i$ is a \emph{dihedral linear edge,}
and we define the \emph{axis} for $E_i$ to be $[w_i]_u$
where $w_i$ is root free and $u_if_\sharp(u_i) = w_i^{d_i}$ for some nonzero integer $d_i$.
In this case, if $E$ and $E'$ are the edges of the dihedral pair,
we may (after reversing the orientation of $w_i$)
write $w_i = \sigma\tau$, where $\sigma = Eg\bar E$, $\tau = E'g'\bar E'$
and $g$ and $g'$ denote the nontrivial elements of the respective copies of $C_2$.
We have $f_\sharp(\sigma) = \tau$ and vice versa,
and $u_i$ is equal to some alternating product of $\sigma$ and $\tau$,
say $u_i = (\sigma\tau)^{d_i}$ (in which case we say $u_i$ is \emph{even})
or $(\sigma\tau)^{d_i}\sigma$ (in which case we say $u_i$ is \emph{odd}),
for some integer $d_i$, which must be nonzero in the even case.

If $E_i$ and $E_j$ are linear edges
such that there exists nonzero integers $d_i$ and $d_j$
and a closed root-free almost Nielsen path $w$
such that the suffixes $u_i$ and $u_j$ satisfy $u_i = w^{d_i}$ and $u_j = w^{d_j}$,
where $d_i$ and $d_j$ have the same sign,
then a path of the form
$gE_i w^p \bar E_jh$ for integer $p$ 
and vertex group elements $g$ and $h$
is called an \emph{exceptional path.}
Notice that if $E_i w^p \bar E_j$ is an exceptional path, then for $k \ge 0$ we have
\[  f^k_\sharp(E_i w^p \bar E_j) = g_kE_i w^{p + k(d_i - d_j)} \bar E_jh_k \]
for vertex group elements $g_k$ and $h_k$.

If $E_i$ and $E_j$ are dihedral linear edges
with the same axis $w = \sigma\tau$
so that the suffixes $u_i$ and $u_j$ are alternating products of $\sigma$ and $\tau$,
then a path of the form
$g E_i \alpha \bar E_j h$,
where $\alpha$ is an alternating product of $\sigma$ and $\tau$,
and for vertex group elements $g$ and $h$ is sometimes exceptional 
and sometimes not according to the parities of $u_i$, $u_j$ and $\alpha$.
We have the following cases.
We always  assume that $d_i$ and $d_j$ have the same sign.
To ease notation, we assume that $f_\sharp(E_i) = E_iu_i$ and $f_\sharp(E_j) = E_ju_j$;
the general calculation is essentially identical.
\begin{enumerate}
    \item If $u_i$ and $u_j$ are even,
        i.e.~$u_i = (\sigma\tau)^{d_i}$ and $u_j = (\sigma\tau)^{d_j}$,
        then $E_i(\sigma\tau)^p\bar E_j$ is an exceptional path,
        $E_i(\sigma\tau)^p\sigma\bar E_j$ is not and we have
        \[  f_\sharp(E_i(\sigma\tau)^p\bar E_j) = E_i(\tau\sigma)^{p-d_i+d_j}\bar E_j. \]
    \item If $u_i$ is even and $u_j$ is odd,
        i.e.~$u_j = (\sigma\tau)^{d_j}\sigma$
        then both $E_i(\sigma\tau)^p\bar E_j$ and $E_i(\sigma\tau)^p\sigma\bar E_j$ are exceptional
        and we have
        \[  f_\sharp(E_i(\sigma\tau)^p\bar E_j) = E_i(\tau\sigma)^{p-d_i-d_j-1}\tau \bar E_j
            \text{ and }
        f_\sharp(E_i(\sigma\tau)^p\sigma\bar E_j) = E_i(\tau\sigma)^{p-d_i+d_j+1}\bar E_j. \]
    \item If $u_i$ and $u_j$ are both odd,
        then $E_i(\sigma\tau)^p\bar E_j$ is an exceptional path,
        $E_i(\sigma\tau)^p\sigma\bar E_j$ is not and we have
        \[  f_\sharp(E_i(\sigma\tau)^p\bar E_j) = E_i(\sigma\tau)^{p+d_i-d_j}\bar E_j. \]
\end{enumerate}

Therefore we have that $f_\sharp$ induces a height-preserving bijection on the set of exceptional paths,
that $E_i w^p \bar E_j$ is a (periodic) almost Nielsen path 
if and only if ($u_i$ and $u_j$ have the same parity and)
$d_i = d_j$
and that the interior of $E_i w^p \bar E_j$ is an increasing union of pre-trivial paths.
In fact, if we assume that $f\colon \mathcal{G} \to \mathcal{G}$
satisfies the conclusions of \Cref{improvedrelativetraintrack}
and is rotationless,
it follows by \Cref{rotationlessnielsenpaths} that all dihedral linear edges have odd suffix
(and that there is no version of an odd-suffix dihedral linear edge 
in the case that the dihedral pair is fixed).

\paragraph{Reduced filtration}
A filtration $\varnothing = G_0 \subset G_1 \subset \cdots \subset G_m = G$
is \emph{reduced with respect to $\varphi \in \out(F,\mathscr{A})$}
if whenever a free factor system $\mathcal{F}'$ is $\varphi^k$-invariant for some $k > 0$
and $\mathcal{F}(G_{r-1}) \sqsubset \mathcal{F}' \sqsubset \mathcal{F}(G_r)$,
then either $\mathcal{F}' = \mathcal{F}(G_{r-1})$ or $\mathcal{F}' = \mathcal{F}(G_r)$.

\paragraph{Taken paths, complete splittings}
If $E$ is an edge in an irreducible stratum $H_r$ and $k > 0$,
then a maximal subpath $\sigma$ of $f^k_\sharp(E)$ in a zero stratum $H_i$
is said to be \emph{$r$-taken} (or simply \emph{taken}).
If the zero stratum in question is enveloped by an exponentially growing stratum,
then $\sigma$ is a connecting path.
Recall that if $f\colon \mathcal{G} \to \mathcal{G}$
satisfies the conclusions of \Cref{improvedrelativetraintrack},
then each zero stratum $H_i$ of $f\colon \mathcal{G} \to \mathcal{G}$
is a wandering component of $G_i$ and hence not incident to vertices with nontrivial vertex group.

A path $\sigma$ is \emph{completely split} if it has a splitting,
called a \emph{complete splitting,} into subpaths
each of which is either a single edge in an irreducible stratum 
(possibly with vertex group elements on either end),
an indivisible almost Nielsen path,
an exceptional path,
or a connecting path in a zero stratum that is both maximal and taken.

A relative train track map $f\colon \mathcal{G} \to \mathcal{G}$
is \emph{completely split} if
\begin{enumerate}
    \item the path $f(E)$ is completely split for each edge $E$ in an irreducible stratum.
    \item If $\sigma$ is a taken connecting path in a zero stratum,
        then $f_\sharp(\sigma)$ is completely split.
\end{enumerate}

\begin{lem}[cf.~Lemma 4.6 of \cite{FeighnHandel}]
    \label{completelysplittocompletelysplit}
    If $f\colon \mathcal{G} \to \mathcal{G}$ is a completely split relative train track map
    for which each zero stratum $H_i$ is a wandering component of $G_i$
    and $\sigma$ is a completely split path,
    then $f_\sharp(\sigma)$ is completely split.
    If $\sigma = \sigma_1\cdots \sigma_k$ is a complete splitting,
    then $f_\sharp(\sigma)$ has a complete splitting
    which refines $f_\sharp(\sigma) = f_\sharp(\sigma_1)\cdots f_\sharp(\sigma_k)$.
\end{lem}

The assumption that zero strata are wandering allows us to use the definition of \emph{taken} above.
This is not a serious assumption, 
since we shall work with relative train track maps 
satisfying the conclusions of \Cref{improvedrelativetraintrack}.

\begin{proof}
    Suppose $\sigma_i$ is a term in a complete splitting of $\sigma$.
    In all cases, $f_\sharp(\sigma_i)$ is completely split,
    since $f_\sharp$ carries indivisible almost Nielsen paths to indivisible almost Nielsen paths
    and exceptional paths to exceptional paths.
    If $\sigma_i$ is $gE_ih$, where $E_i$ is an edge in an irreducible stratum
    and $g$ and $h$ are vertex group elements with $h$ nontrivial,
    the final term in the complete splitting of $E_i$
    is not a connecting path in a zero stratum.
    The path $f_\sharp(\sigma)$ therefore has a complete splitting that refines
    $f_\sharp(\sigma) = f_\sharp(\sigma_1)\cdots f_\sharp(\sigma_k)$,
    since each maximal subpath of $f_\sharp(\sigma)$ in a zero stratum
    is contained in a single $f_\sharp(\sigma_i)$.
\end{proof}

\paragraph{Folding the indivisible almost Nielsen path.}
The final definition and lemma we need before defining CTs
have to do with indivisible almost Nielsen paths in exponentially growing strata.
Suppose that $H_r$ is an exponentially growing stratum
of a relative train track map $f\colon \mathcal{G} \to \mathcal{G}$
and that $\rho$ is an indivisible almost Nielsen path of height $r$.
Suppose further that the map $f$ satisfies the following properties.
\begin{enumerate}
    \item The topological representative $f$ represents $\varphi \in \out(F,\mathscr{A})$.
    \item The filtration on $f\colon \mathcal{G} \to \mathcal{G}$
        realizes a nested sequence $\mathscr{C}$ of $\varphi$-invariant free factor systems.
    \item Each contractible component of a filtration element is a union of zero strata.
    \item The endpoints of all indivisible almost Nielsen paths of exponentially growing height
        are vertices.
\end{enumerate}
We will define an operation called \emph{folding $\rho$}
that produces a new relative train track map $f' \colon \mathcal{G}' \to \mathcal{G}'$
and show that this new relative train track map 
satisfies the above properties.

Write $\rho = \alpha\beta$ as a concatenation of maximal $r$-legal subpaths as in
\Cref{finitelymanynielsenpaths} and let $g_1E_1$ and $g_2E_2$ in $H_r$
be the initial edges and vertex group elements of $\bar\alpha$ and $\beta$ respectively.
If one of the edge paths $f(g_iE_i)$ for $i = 1$ or $2$
is an initial subpath of the other,
then we say that \emph{the fold at the illegal turn of $\rho$ is a full fold}
and that it is a \emph{partial fold} otherwise.
In the case of a full fold, if $f(g_1E_1) \ne f(g_2E_2)$,
then the full fold is \emph{proper,} otherwise it is \emph{improper.}

Suppose that $f(g_1E_1)$ is a proper initial subpath of $f(g_2E_2)$,
so that we are in the case of a proper full fold.
Write $\bar\alpha = g_1E_1bE_3$ where $b$ is a possibly trivial subpath of $\bar\alpha$ in $G_{r-1}$
or a single vertex group element
and $E_3$ is an edge of $H_r$.
The initial edge of $f(E_3)$ and the first edge of $f(\beta)$ that is not canceled with $f(g_1E_1b)$
when $f(\alpha)f(\beta)$ is tightened to $g\alpha\beta h$ belong to $H_r$.
We may decompose the edge $E_2$ into subpaths $E_2 = E''_2E'_2$
such that $f(g_2E''_2) = f(E_1b)$ and such that the first edge in $f(E'_2)$ is in $H_r$.
Form a new graph of groups $\mathcal{G}'$ by identifying $E''_2$ with $E_1b$.
The quotient map $F\colon \mathcal{G} \to \mathcal{G}'$ is called the
\emph{extended fold determined by $\rho$.}

We may think of $G \setminus E_2$ as a subgraph of groups of $\mathcal{G}'$ on which
the map $F$ is the identity.
By construction, we have that $F(E_2) = E_1bE'_2$.
The strata of the filtration on $\mathcal{G}'$ is defined so that
$H'_i = H_i$ if $i \ne r$ and $H'_r = (H_r\setminus E_2)\cup E'_2$.
The map $f$ factors as $gF$ for some map $g\colon \mathcal{G}' \to \mathcal{G}$.
The map $g\colon \mathcal{G}' \to \mathcal{G}$ is called
\emph{the map induced by the extended fold.}
Define $f'\colon \mathcal{G}' \to \mathcal{G}'$ from $Fg\colon \mathcal{G}' \to \mathcal{G}'$
by tightening the images of edges.
We say that $f'\colon \mathcal{G}' \to \mathcal{G}'$ 
is obtained from $f\colon \mathcal{G} \to \mathcal{G}$
by  \emph{folding $\rho$} and that $\rho' = F_\sharp(\rho)$
is \emph{the indivisible almost Nielsen path determined by $\rho$.}
If the fold at the illegal turn of $\rho'$ is itself proper,
then we may iterate this process,
which is referred to as \emph{iteratively folding $\rho$.}

\begin{lem}[cf.~Lemma 2.22 of \cite{FeighnHandel}]
    \label{properfoldrelativetraintrackmap}
    Given notation as above, the map $f'\colon \mathcal{G}' \to \mathcal{G}'$
    is a relative train track map that satisfies 
    properties 1 through 4 above.
\end{lem}

\begin{proof}
    We follow the proof of \cite[Lemma 2.22]{FeighnHandel},
    assuming explicitly that $f\colon \mathcal{G} \to \mathcal{G}$ is a topological representative.
    By construction, we have $f'|_{G'_{r-1}} = f|_{G_{r-1}}$.
    If $E$ is an edge in $H_r$, then $f(E)$ does not cross the illegal turn in $\rho$,
    and if $E \ne E_2$, then the path $f'(E)$ is obtained from $f(E)$
    replacing each occurrence of $E_2$ with $E_1bE'_2$ (and by tightening,
    which is not necessary if we assume $f$ is a topological representative).
    We also have that $f'(E'_2)$ is obtained from $f(E'_2)$ by replacing each occurrence
    of $E_2$ with $E_1bE'_2$.
    Therefore $H'_r$ satisfies properties \hyperlink{EG-i}{(EG-i)} through \hyperlink{EG-iii}{(EG-iii)}.

    So suppose $H_i$ is a stratum above $H_r$.
    If $H_i$ is non-exponentially growing, then $H'_i$ is non-exponentially growing.
    If $H_i$ is a zero stratum,
    then for each edge $E_i$ of $H_i$, we have $f(E_i) = f_\sharp(E_i)$ 
    is nontrivial because we assume that $f$ is a topological representative.
    Observe that $F$ does not identify points that are not identified by $f$.
    Since $F|_{E_i}$ is the identity, we have $(Fg)_\sharp(E_i) = (Ff)_\sharp(E_i)$,
    and this path is nontrivial.
    This shows that no edges are collapsed when $Fg$ is tightened to $f'$
    and that if $\sigma \subset H_i$ is a path with endpoints at vertices,
    then $f'_\sharp(\sigma)$ is nontrivial.

    So suppose that $H_i$ is exponentially growing,
    that $E_i$ is an edge of $H_i$,
    and that $f(E_i) = \sigma_1\mu_1\sigma_2\ldots\mu_\ell\sigma_{\ell+1}$
    is the decomposition of $f(E_i)$ into subpaths $\sigma_j$ in $H_m$
    and $\mu_k$ in $G_{i-1}$.
    Then we have
    \[  f(E_i) = (Fg)_\sharp(E_i) = (Ff)_\sharp(E_i) 
    = \sigma_1\mu'_1\sigma_2\ldots \mu'_\ell\sigma_{\ell+1} \]
    where $\mu'_j = F_\sharp(\mu_j)$ is nontrivial because $f_\sharp(\mu_j)$ is nontrivial.
    We conclude that $H_i$ satisfies \hyperlink{EG-i}{(EG-i)} and \hyperlink{EG-iii}{(EG-iii)}.

    Suppose $\sigma'$ is a connecting path for $H'_i$.
    If $\sigma'$ is contained in a contractible component of $G'_{i-1}$,
    it is contained in a zero stratum $H'_k$,
    then it is disjoint from $G_r$
    and thus identified by $F$ with a connecting path $\sigma$ in $H_k$,
    and we see that $f'_\sharp(\sigma')$ is nontrivial because $f_\sharp(\sigma)$ is.
    If $\sigma'$ is contained in a noncontractible component of $G'_{i-1}$,
    by property \hyperlink{Z}{(Z)} for $f\colon \mathcal{G} \to \mathcal{G}$
    it is contained in a non-contractible component $C$ of $G'_{i-1}$.
    Vertices in $H_i \cap C$ are periodic for $f$ by \Cref{trivialedgegroupsEG2}.
    Since $F$ does not identify points that are not identified by $f$,
    we see that these vertices are periodic for $f'$,
    proving \hyperlink{EG-ii}{(EG-ii)}.
    Therefore $f'\colon \mathcal{G}' \to \mathcal{G}'$ is a relative train track map.
    The additional properties listed above follow from the corresponding properties for $f$.
\end{proof}

\paragraph{CTs.}
A relative train track map $f\colon \mathcal{G} \to \mathcal{G}$
and filtration $\varnothing = G_0 \subset G_1 \subset \cdots \subset G_m = G$
is a \emph{CT} (for \emph{completely split improved relative train track map})
if it satisfies the following properties.
Compare \cite[Definition 4.7]{FeighnHandel}.
\begin{enumerate}
    \item \hypertarget{Rotationless}{(\textbf{Rotationless})}
        The map $f\colon \mathcal{G} \to \mathcal{G}$ is rotationless.
    \item \hypertarget{CompletelySplit}{(\textbf{Completely Split})}
        The map $f\colon \mathcal{G} \to \mathcal{G}$ is completely split.
    \item \hypertarget{Filtration}{(\textbf{Filtration})}
        The filtration $\varnothing = G_0 \subset G_1 \subset \cdots \subset G_m = G$
        is reduced.
        The core of a filtration element $G_r$ is a filtration element
        unless $H_r$ is the bottom half of a dihedral pair,
        in which case $G_{r-1}$ and $G_{r+1}$ are their own core.
    \item \hypertarget{Vertices}{(\textbf{Vertices})}
        The endpoints of all indivisible periodic (necessarily fixed) almost Nielsen paths
        are (necessarily principal) vertices.
        The terminal endpoint of each non-almost fixed non-exponentially growing edge
        is fixed and is either principal or the center vertex of a dihedral pair.
        In the latter case, the suffix of the non-exponentially growing edge
        is contained in the dihedral pair.
    \item \hypertarget{AlmostPeriodicEdges}{(\textbf{Almost Periodic Edges})}
        Each almost periodic edge is either almost fixed
        or belongs to a dihedral pair.
        If an almost fixed edge $E_r$ is not part of a dihedral pair,
        each endpoint of $E_r$ is principal,
        and if an endpoint of $E_r$
        has trivial vertex group and $E_r$ does not form a loop,
        then $G_{r-1}$ is its own core
        and that endpoint of $E_r$ is contained in $G_{r-1}$.
    \item \hypertarget{ZeroStrata}{(\textbf{Zero Strata})}
        If $H_i$ is a zero stratum,
        then $H_i$ is enveloped by an exponentially growing stratum $H_r$,
        each edge in $H_i$ is $r$-taken
        and each vertex in $H_i$ is contained in $H_r$ and has link contained in $H_i \cup H_r$.
    \item \hypertarget{LinearEdges}{(\textbf{Linear Edges})}
        For each linear edge $E_i$ there is a closed, root-free Nielsen path $w_i$
        such that $f(E_i) =g_iE_iw_i^{d_i}$ for some nonzero integer $d_i$.
        If $E_i$ and $E_j$ are distinct linear edges with the same axis,
        then $w_i = w_j$ and $d_i \ne d_j$.
        For each dihedral linear edge $E_i$ with axis $w_i = \sigma\tau$,
        we have $f(E_i) = g_iE_i(\sigma\tau)^{d_i}\sigma$ for some integer $d_i$.
        If $E_i$ and $E_j$ are distinct dihedral linear edges with the same axis,
        then $d_i \ne d_j$.
    \item \hypertarget{NEGAlmostNielsenPaths}{(\textbf{NEG Almost Nielsen Paths})}
        If the highest edges in an indivisible almost Nielsen path $\sigma$
        belong to a non-exponentially growing stratum,
        then there is a linear or dihedral linear edge $E_i$ 
        with $w_i$ as in \hyperlink{LinearEdges}{(Linear Edges)}
        and there exists $k \ne 0$ such that $\sigma = gE_iw_i^k\bar E_ih$
        for vertex group elements $g$ and $h$.
    \item \hypertarget{EGAlmostNielsenPaths}{(\textbf{EG Almost Nielsen Paths})}
        If $H_r$ is exponentially growing and $\rho$ is an indivisible almost Nielsen path of height $r$,
        then $f|_{G_r} = \theta\circ f_{r-1}\circ f_r$ where
        \begin{enumerate}
            \item $f_r\colon G_r \to \mathcal{G}^1$ is a composition of proper extended folds
                defined by iteratively folding $\rho$,
            \item $f_{r-1}\colon \mathcal{G}^1 \to \mathcal{G}^2$
                is a composition of folds involving edges in $G_{r-1}$, and
            \item $\theta\colon \mathcal{G}^2 \to G_r$ is an isomorphism of graphs of groups.
        \end{enumerate}
\end{enumerate}

Observe that if $f\colon \mathcal{G} \to \mathcal{G}$ is a CT,
\hyperlink{Vertices}{(Vertices)} and \Cref{trivialedgegroupsEG2}
imply that a vertex whose link contains edges in more than one irreducible stratum is principal
unless it is the center vertex of a dihedral pair.

The remainder of the section is given to studying  properties of CTs
with an eye towards their construction in the proof of \Cref{preciseCTtheorem}.
Our first lemma says, in particular, that the complete splitting of a path is unique
up to ``shuffling'' vertex  group elements between the terms.

\begin{lem}[cf.~Lemma 4.11 of \cite{FeighnHandel}]
    \label{completesplittingunique}
    Suppose that $f\colon \mathcal{G} \to \mathcal{G}$ is a CT,
    that $\sigma$ is a circuit or a path and that $\sigma = \sigma_1\ldots \sigma_\ell$
    is a decomposition into subpaths,
    each of which is either a single edge in an irreducible stratum
    (possibly with vertex group elements on either end),
    an indivisible almost Nielsen path,
    an exceptional path,
    or a connecting path in a zero stratum that is maximal and taken.
    Suppose additionally that each turn $(\bar\sigma_i,\sigma_{i+1})$ is legal. 
    Then the following hold.
    \begin{enumerate}
        \item The decomposition $\sigma = \sigma_1\ldots \sigma_{\ell}$
            is the unique complete splitting of $\sigma$
            up to the equivalence relation where $\sigma_1\ldots \sigma_\ell$
            is equivalent to $\sigma'_1\ldots \sigma'_{\ell'}$
            if $\ell = \ell'$ and $\sigma_i = g_i\sigma'_i h_i$
            for (possibly trivial) vertex group elements $g_i$ and $h_i$.
        \item Each pretrivial subpath $\tau$ of $\sigma$ is contained in a single $\sigma_i$.
        \item A subpath of $\sigma$ that has the same height as $\sigma$
            and is either an almost fixed edge or an indivisible almost Nielsen path
            is equal to $g_i\sigma_i h_i$ for some $i$ and vertex group elements $g_i$ and $h_i$.
    \end{enumerate}
\end{lem}

\begin{proof}
    We follow \cite[Lemma 4.11]{FeighnHandel}.
    There is no loss in assuming that $f$ is a topological representative,
    so as in the original, we shall.
    Let $\tilde\sigma = \tilde\sigma_1\ldots\tilde\sigma_\ell$
    be a lift of $\sigma$
    and let $\tilde f\colon \Gamma \to \Gamma$ be a lift of $f\colon \mathcal{G} \to \mathcal{G}$.
    We first establish the following property.
    \begin{enumerate}
        \item[4.] If $\sigma_i$ is not a maximal taken connecting path in a zero stratum,
            then for each $k \ge 0$,
            there exist nontrivial initial and terminal subpaths $\tilde\alpha_{i,k}$
            and $\tilde\beta_{i,k}$ of $\tilde\sigma_i$
            such that $\tilde f^{-k}(\tilde f^k(\tilde x))\cap \tilde \sigma = \{\tilde x\}$
            for each point $\tilde x$ in the union of the interiors of
            $\tilde \alpha_{i,k}$ and $\tilde \beta_{i,k}$.
    \end{enumerate}
    We prove item 4 by induction, first on $k$ and then on $\ell$.
    Recall that we assume $\tilde f$ acts linearly with respect to some metric on $\Gamma$.
    It is clear that item 4 holds for $k = 0$,
    so assume that item 4 holds for any iterate of $\tilde f$ (and all $\ell$) less than $k$.

    Now assume that $\ell = 1$, i.e.~that $\sigma = \sigma_1$.
    If $\sigma_1$ is an exceptional path 
    or an indivisible almost Nielsen path of non-exponentially growing height,
    then by \hyperlink{NEGAlmostNielsenPaths}{(NEG Almost Nielsen Paths)},
    the first and last edges of $\sigma_1$ are non-exponentially growing, and
    the existence of the initial and terminal segments $\tilde\alpha_{i,k}$ and $\tilde\beta_{i,k}$
    mapping over theses edges follows
    as in the proof of \Cref{negsplitting}.
    If $\sigma_1$ is an indivisible almost Nielsen path of exponentially growing height,
    the statement follows from the proof of \Cref{finitelymanynielsenpaths}.
    The final possibility is that $\sigma_1$ is an edge $E$ in an irreducible stratum.
    By \hyperlink{ZeroStrata}{(Zero Strata)}, the first and last terms
    in any complete splitting of $f(E)$ are not connecting paths in zero strata.
    By the inductive hypothesis, there exist initial and terminal subpaths
    $\tilde\alpha'$ and $\tilde\beta'$ of $\tilde f(\tilde E)$ such that
    \[  \tilde f^{-(k-1)}(\tilde f^{k-1}(\tilde x)) \cap \tilde f(\tilde E) = \{\tilde x\} \]
    for all $\tilde x$ in the interior of $\tilde\alpha'$ or $\tilde\beta'$.
    Since we assume that $f$ is a topological representative,
    $\tilde f|_{\tilde E}$ is an embedding,
    so we may pull back $\tilde\alpha'$ and $\tilde\beta'$
    to initial and terminal subpaths $\tilde\alpha_{i,k}$ and $\tilde\beta_{i,k}$ of $\tilde E$
    that satisfy item 4.
    This completes the case $\ell = 1$.

    Now suppose that item 4 holds for $k$ and $\sigma$ if the decomposition of $\sigma$ in the statement
    has fewer than $\ell \ge 2$ terms.
    There are three cases, depending on whether $\sigma_1$ or $\sigma_2$ are
    taken maximal connecting paths in zero strata.
    Suppose first that $\sigma_1$ is such a path in a zero stratum $H_p$.
    By \hyperlink{ZeroStrata}{(Zero Strata)}, $\sigma_2$ is an edge in an exponentially growing stratum
    $H_r$ with $r > p$.
    Define $\tilde\alpha_{i,k}$ and $\tilde\beta_{i,k}$ using $\tilde\sigma_2\ldots\tilde\sigma_\ell$
    in place of $\tilde\sigma$.
    Observe that the intersection of $\tilde f^k(\tilde\sigma_1)$ with the interior of
    $\tilde f^k(\tilde\alpha_{2,k})$ is empty,
    because $\tilde f^k(\tilde\sigma_1)$ projects to a path of height less than $r$
    and $\tilde f^k(\tilde\alpha_{2,k})$ is an embedded path
    whose initial direction projects to a direction of height $r$.
    Furthermore, the interior of $\tilde f^k(\tilde\alpha_{2,k})$ separates
    $\tilde f^k(\tilde\sigma_1)$ from each $\tilde f^k(\tilde\beta_{i,k})$
    and from each $\tilde f^{k}(\tilde\alpha_{i,k})$ if $i >2$,
    so it follows that $\tilde f^k(\tilde\sigma_1)$ is disjoint from these sets,
    proving that $\tilde\alpha_{i,k}$ and $\tilde\beta_{i,k}$ satisfy item 4
    with respect to $\sigma$.

    Suppose next that $\sigma_2$ is a taken maximal connecting path in a zero stratum $H_p$.
    Again by \hyperlink{ZeroStrata}{(Zero Strata)}, we have that $\sigma_1$ and $\sigma_3$
    (if $\ell \ge 3$)
    are edges in an exponentially growing stratum $H_r$ with $r > p$.
    Define $\tilde\alpha_{1,k}$ and $\tilde\beta_{1,k}$ using $\tilde\sigma_1$ in place of $\tilde\sigma$.
    For $i > 2$, define $\tilde\alpha_{i,k}$ and $\tilde\beta_{i,k}$ 
    using $\tilde\sigma_2\ldots\tilde\sigma_\ell$ in place of $\tilde\sigma$.
    As in the previous case, we have that $\tilde f^k(\tilde\sigma_2)$
    is disjoint from the interior of $\tilde f^k(\tilde\beta_{1,k})$
    and the interior of $\tilde f^k(\tilde\alpha_{3,k})$.
    Also as in the previous case, this implies that $\tilde\alpha_{i,k}$ and $\tilde\beta_{i,k}$
    satisfy item 4 with respect to $\sigma$.

    As a final case, suppose that neither $\sigma_1$ or $\sigma_2$
    are taken maximal connecting paths in zero strata.
    Define $\tilde\alpha_{1,k}$ and $\tilde\beta_{1,k}$ using $\tilde\sigma_1$ in place of $\tilde\sigma$.
    For $i \ge 2$, define $\tilde\alpha_{i,k}$ and $\tilde\beta_{i,k}$
    using $\tilde\sigma_2\ldots\tilde\sigma_\ell$ in place of $\tilde\sigma$.
    Because the turn $(\bar\sigma_1,\sigma_2)$ is legal,
    the interiors of $\tilde f^k(\tilde\alpha_{1,k})$ and $\tilde f^k(\tilde\beta_{2,k})$ are disjoint.
    The proof concludes as in the previous two cases.
    This completes the inductive step and with it the proof of item 4.

    If $\tau$ is a pretrivial path, choose $k > 0$ so that $f^k_\sharp(\tau)$ is trivial.
    For each point $\tilde x \in \tilde\tau$ there exists $\tilde y \ne \tilde x$
    such that $\tilde f^k(\tilde x) = \tilde f^k(\tilde y)$.
    If $\sigma_i$ is not a taken maximal connecting path in a zero stratum
    and if $\tau$ intersects the interior of $\sigma_i$,
    then we must have that $\tau$ is contained in the interior of $\sigma_i$ by item 4.
    No two consecutive $\sigma_i$ are taken maximal connecting paths in zero strata,
    so this proves item 2.
    It also follows that $\sigma = \sigma_1\ldots\sigma_m$ is a splitting,
    hence a complete splitting,
    of $\sigma$.

    Suppose that $\sigma = \sigma'_1\ldots \sigma'_1$ is a complete splitting as well.
    If $\sigma'_i$ is an exceptional path or an indivisible almost Nielsen path
    then observe that the interior of $\sigma'_i$ is an increasing union of pre-trivial subpaths.
    (This was remarked after the definition of exceptional paths,
    of which indivisible almost Nielsen paths of non-exponentially growing height
    are examples by \hyperlink{NEGAlmostNielsenPaths}{(NEG Almost Nielsen Paths)},
    and follows from the proof of \Cref{finitelymanynielsenpaths} for indivisible almost Nielsen paths
    of exponentially growing height.)
    Item 2 implies that $\sigma'_i$ is contained in some $\sigma_j$.
    Since $\sigma_j$ is not a single edge and is not contained in a zero stratum,
    it must be an indivisible almost Nielsen path or an exceptional path.
    By symmetry we conclude that $\sigma'_i = g\sigma_jh$ for vertex group elements $g$ and $h$.
    The terms that are taken maximal connecting paths in zero strata
    are the maximal subpaths of $\sigma$ in the complement of the indivisible almost Nielsen paths
    and exceptional paths,
    contained in zero strata, so these subpaths are the same in each decomposition.
    The remaining edges determine terms in the complete splitting.
    This proves that complete splittings are unique up to the equivalence relation in item 1.

    Finally, an almost fixed edge of maximal height in $\sigma$
    is clearly not contained in a zero stratum,
    is not an indivisible almost Nielsen path
    nor an exceptional path in $\sigma$
    and so must determine a term in the complete splitting of $\sigma$.
    By item 2, an indivisible almost Nielsen path in $\sigma$
    must be contained in a single $\sigma_i$ by item 2.
    If it has maximal height,
    then again we conclude that it must be all of $\sigma_i$.
    This proves item 3.
\end{proof}

\begin{cor}[cf.~Corollary 4.12 of \cite{FeighnHandel}]
    \label{splittinginitialsegments}
    Let $f\colon \mathcal{G} \to \mathcal{G}$ be a CT and $\sigma$
    a completely split path with complete splitting $\sigma = \sigma_1\ldots\sigma_s$.
    If $\tau$ is an initial segment of $\sigma$ with terminal endpoint in $\sigma_j$,
    then $\tau = \sigma_1\ldots \sigma_{j-1}\cdot \mu_j$ is a splitting,
    where $\mu_j$ is the initial segment of $\sigma_j$ that is contained in $\tau$.
    If $\tau$ is a nontrivial almost Nielsen path
    (or more generally, if $f_\sharp(\tau)$ has the same underlying path in $G$ as $\tau$)
    then $\sigma_i$ is an almost Nielsen path for $i \le j$
    and if $\sigma_j$ is not a single almost fixed edge, then $\mu_j = \sigma_j$.
\end{cor}

\begin{proof}
    The proof is identical to \cite[Lemma 4.12]{FeighnHandel};
    the main statement follows immediately from item 2 of \Cref{completesplittingunique}.
    The statement about almost Nielsen paths follows from the fact
    that a proper initial segment of $\sigma_j$ is an almost Nielsen path
    only when $\sigma_j$ is an almost fixed edge.
\end{proof}

\begin{lem}[cf.~Lemma 4.36 of \cite{FeighnHandel}]
    \label{CTrays}
    Suppose $f\colon \mathcal{G} \to \mathcal{G}$ is a CT
    and that $\tilde f \colon \Gamma \to \Gamma$ is a principal lift.
    The following hold.
    \begin{enumerate}
        \item If the vertex $\tilde v$ belongs to $\fix(\tilde f)$ and $\tilde E$ is a non-fixed edge
            determining a fixed direction at $\tilde v$,
            then $\tilde E \subset \tilde f_\sharp(\tilde E) \subset \tilde f^2_\sharp(\tilde E) 
            \subset \cdots$
            is an increasing sequence of paths whose union is a ray $\tilde R$
            that converges to some fixed point $P \in \fix_N(\hat f) \cap \partial_\infty(F,\mathscr{A})$.
            The interior of $\tilde R$ is fixed point free.
        \item For every isolated point $P \in \fix_N(\hat f) \cap \partial_\infty(F,\mathscr{A})$,
            there exists $\tilde E$ and $\tilde R$ as in Item 1
            that converges $P$. The edge $E$ is non-linear.
    \end{enumerate}
\end{lem}

\begin{proof}
    The proof is identical to \cite[Lemma 4.36]{FeighnHandel}.
    Given $\tilde E$ as Item 1 and for $m > 0$,
    \Cref{completelysplittocompletelysplit} implies that
    $\tilde E \subset \tilde f_\sharp(\tilde E) \subset \cdots \subset \tilde f^m_\sharp(\tilde E)$
    is a nested sequence of completely split paths
    which define a ray $\tilde R$ that converges to a fixed point $P$.
    Let $\tilde w$ be the terminal endpoint of $\tilde E$.
    Since the sequence of points $\tilde f^m(\tilde w)$ limits to $P$
    and each point in the sequence moves toward $P$ under the action of $\tilde f$,
    we have $P \in \fix_N(\hat f)$ by \Cref{movingtowardsattractors}.
    By \Cref{splittinginitialsegments}, the ray $\tilde R$ intersects $\fix(\tilde f)$
    only in its initial endpoint, proving item 1.

    If $P \in \fix_N(\hat f) \cap \partial_\infty(F,\mathscr{A})$ is isolated,
    then \Cref{GJLLprop} implies that $\tilde f$ moves points that are sufficiently close to $P$
    toward $P$.
    Therefore we may choose a ray $\tilde R$ that converges to $P$
    and intersects $\fix(\tilde f)$ only in its initial endpoint.
    The initial edge $\tilde E$ of $\tilde R$ determines a fixed direction by \Cref{producingafixedpoint},
    so by item 1 extends to a fixed-point-free ray $\tilde R'$
    converging to a point $Q \in \fix_N(\hat f)$.
    \Cref{producingafixedpoint} implies that in fact $P = Q$.
    \Cref{boundarybasics} and item 1 of \Cref{GJLLprop} imply that $P$ is not
    an endpoint of the axis of $T_c$ for some non-peripheral $c \in F$,
    so it follows by \Cref{negctlinearcondition} below that $E$ is not a linear edge.
\end{proof}

\begin{lem}[cf.~Lemma 4.21 of \cite{FeighnHandel}]
    \label{negctsplitting}
    Suppose $f\colon \mathcal{G} \to \mathcal{G}$ is a CT
    and $H_i$ is a non-exponentially growing stratum
    that is not a dihedral pair.
    Then $H_i$ is a single edge $E_i$.
    If $E_i$ is not almost periodic,
    then there is a nontrivial closed path $u_i \subset G_{i-1}$
    such that $f(E_i) = E_i \cdot u_i$.
    The closed path $u_i$ is contained in a filtration element $G_j$ with $j < i$
    that is its own core.
\end{lem}

Let us remark that unlike in \cite{FeighnHandel},
the path $u_i$ need not form a circuit.
One way this can fail is the case of a dihedral linear edge,
where $u_i = (\sigma\tau)^{d_i}\sigma$ is its own homotopy inverse.

\begin{proof}
    The proof is identical to \cite[Lemma 4.21]{FeighnHandel}.
    If $H_i$ consists of almost periodic edges, then the lemma follows from
    \hyperlink{AlmostPeriodicEdges}{(Almost Periodic Edges)}.
    Otherwise \hyperlink{Rotationless}{(Rotationless)},
    \hyperlink{CompletelySplit}{(Completely Split)} and \hyperlink{Vertices}{(Vertices)}
    imply that $H_i$ is a single edge $E_i$
    and that there is a nontrivial closed path $u_i$ such that $f(E_i) = E_iu_i$ is completely split.
    To show that $f(E_i) = E_i\cdot u_i$ is a splitting,
    we need to show that the first term in the complete splitting of $E_iu_i$
    is the single edge $E_i$.
    Clearly the first term is not contained in a zero stratum.
    It cannot be an indivisible almost Nielsen path by 
    \hyperlink{NEGAlmostNielsenPaths}{(NEG Almost Nielsen Paths)}.
    The final possibility we must eliminate is that the first term is an exceptional path.
    It cannot be an exceptional path if $E_i$ is not linear,
    so suppose $E_i$ is linear or dihedral linear.
    By \hyperlink{LinearEdges}{(Linear Edges)}, we have $f(E_i) = E_i w_i^{d_i}$
    in the linear case, and $f(E_i) = E_i(\sigma\tau)^{d_i}\sigma$ in the latter.
    If $E_j$ is another linear edge with the same axis,
    then assuming notation as in \hyperlink{LinearEdges}{(Linear Edges)}
    A path of the form $E_i w^p \bar E_j$ cannot be a subpath of $f(E_i)$,
    so we conclude that the first term in the complete splitting of $f(E_i)$ is not an exceptional path.
    Therefore $f(E_i) = E_i \cdot u_i$.

    By \hyperlink{Filtration}{(Filtration)},
    to prove the final claim, it suffices to show that if $u_i$ is contained in $G_j$,
    it is contained in the core of $G_j$.
    If the terminal vertex of $E_i$ is the center vertex of a dihedral pair,
    then by \hyperlink{Vertices}{(Vertices)}, 
    $u_i$ is contained in the dihedral pair and the claim follows.
    So suppose the terminal vertex of $E_i$ is principal and has finite vertex group.
    The turn $(Df^{k-1}(\bar u_i),Df^k(u_i))$ is the image under $Df^k$
    of the legal turn $(\bar E_i,u_i)$ and is therefore legal for $k \ge 1$.
    Since $f$ is rotationless, $v$ is principal and there are finitely many directions based at $v$,
    we have that $Df^k(d)$ is independent of $k$ for all directions $d$ based at $v$
    for sufficiently large $k$.
    Therefore for sufficiently large $k$, we have that $(Df^k(\bar u_i),Df^k(u_i))$ is legal,
    and we conclude that $(u_i,\bar u_i)$ is legal and hence nondegenerate, so $u_i$ forms a circuit.
    It follows that if the terminal vertex of $E_i$ has finite vertex group
    and $u_i$ is contained in $G_j$, then it is contained in the core of $G_j$.
    If the turn $(u_i,\bar u_i)$ is degenerate, then the argument above shows that
    the terminal vertex of $E_i$ has infinite vertex group,
    and in particular there is a circuit of the form $u_ig_i$ for vertex group element $g_i$,
    and we conclude that if $u_i$ is contained in $G_j$, then it is contained in the core of $G_j$.
\end{proof}

If $f\colon \mathcal{G} \to \mathcal{G}$ is a CT,
the lemma implies that $f\colon \mathcal{G} \to \mathcal{G}$ satisfies the conclusions
of \Cref{improvedrelativetraintrack}.

\begin{lem}[cf.~Lemma 4.22 of \cite{FeighnHandel}]
    \label{negctlinearcondition}
    Suppose that $f\colon \mathcal{G} \to \mathcal{G}$ is a rotationless relative train track map
    and that the stratum $H_i$ is a single edge $E_i$
    such that $f(E_i) = E_i \cdot u_i$ for some nontrivial closed path $u_i$ in $G_{i-1}$.
    Suppose that $f|_{G_{i-1}}$ satisfies \hyperlink{Vertices}{(Vertices)}
    and \hyperlink{AlmostPeriodicEdges}{(Almost Periodic Edges)}.
    Suppose either that there are no almost Nielsen paths of height $i$
    or that $E_i$ is a linear or dihedral linear edge
    and all almost Nielsen paths of height $i$ have the form $E_iw_i^k\bar E_i$
    for $k \ne 0$ and where $w_i$ is the axis of $E_i$.
    If $h\colon \Gamma_{i-1} \to \Gamma_{i-1}$ is the lift of $f|_{G_{i-1}}$
    as in the paragraph ``Restricting to $G_{i-1}$ for non-exponentially growing strata,''
    then we have the following.
    \begin{enumerate}
        \item $\fix(h) = \varnothing$.
        \item The edge $E_i$ is linear or dihedral linear if and only if 
            there is some nonperipheral $c\in F$ such that $T_c$ preserves $\Gamma_{i-1}$ 
            and the restriction $T_c|_{\Gamma_{i-1}}\colon \Gamma_{i-1} \to \Gamma_{i-1}$
            commutes with $h$ and whose axis covers $u_i$ if $E_i$ is linear
            and $u_if_\sharp(u_i)$ if $E_i$ is dihedral linear.
    \end{enumerate}
\end{lem}

\begin{proof}
    We follow the proof of \cite[Lemma 4.22]{FeighnHandel}.
    Write $\tilde f\colon \Gamma \to \Gamma$ and $\tilde E_i$ as in the paragraph
    ``Restricting to $G_{i-1}$ for non-exponentially growing strata.''
    We have $\tilde f(\tilde E_i) = \tilde E_i \cdot \tilde u_i$,
    where $\tilde u_i \subset \Gamma_{i-1}$ is a lift of $u_i$
    and $h$ maps the terminal vertex $\tilde x_1$ of $\tilde E_i$ 
    to the terminal endpoint $\tilde x_2$ of $\tilde u_i$.
    If $\tilde v$ belongs to $\fix(h)$ and $\tilde \gamma$ is the unique tight path
    from $\tilde x_1$ to $\tilde v$,
    then $\tilde E_i\tilde\gamma$ projects to an almost Nielsen path for $f$
    that is not of the form $E_iw_i^k\bar E_i$,
    so we conclude $\fix(h) = \varnothing$.

    If $E_i$ is a (non-dihedral) linear edge, the assumption that $f(E_i) = E_i \cdot u_i$ is a splitting 
    implies that the closed path $u_i$ represents a nonperipheral conjugacy class in $F$.
    There is a choice of representative $c$ such that $T_c$ preserves $\Gamma_{i-1}$
    and $T_c(\tilde x_1) = \tilde x_2$.
    We have $T_ch(\tilde x_1) = hT_c(\tilde x_1)$ is the terminal endpoint of the lift of $u_i$
    that begins at $\tilde x_2$, and both maps agree on the initial direction of $\tilde u_i$.
    Thus $T_c$ commutes with $h$.
    If $E_i$ is dihedral linear, the closed path $u_if_\sharp(u_i)$ 
    represents a nonperipheral conjugacy class in $F$.
    There is a choice of representative $c$ such that $T_c$ preserves $\Gamma_{i-1}$
    and $T_c(\tilde x_1)$ is the terminal endpoint $\tilde x_3$ of $\tilde u_i h(\tilde u_i)$.
    We have that $T_ch(\tilde x_1) = hT_c(\tilde x_1)$ is the terminal endpoint of the lift of $u_i$
    that begins at $\tilde x_3$.
    Since the central vertex of a dihedral pair has trivial vertex group,
    $T_c$ commutes with $h$.

    For the converse, suppose that there is some nonperipheral element $c \in F$
    such that $T_c$ preserves $\Gamma_{i-1}$ and commutes with $h$.
    \Cref{principalnonemptyfixed} implies that since $\fix(h) = \varnothing$,
    the map $h$ is not a principal lift of $f|_{G_{i-1}}$,
    so the endpoints of the axis of $T_c$ are the only fixed points in $\partial\Gamma_{i-1}$.
    But the ray $\tilde u_i\cdot h_\sharp(\tilde u_i)\cdot h^2_\sharp(\tilde u_i)\cdots $
    converges to a fixed point in $\partial\Gamma_{i-1}$,
    so the endpoint of this ray is contained in the axis of $T_c$.
    Therefore $u_i$ is a periodic Nielsen path.
    Write $u_i = \sigma_1\cdots\sigma_m$, where each $\sigma_j$
    is either an almost periodic edge or an indivisible periodic almost Nielsen path.
    If some endpoint of some $\sigma_j$ is principal, then all endpoints of all $\sigma_j$
    are principal and each $\sigma_j$ and hence $u_i$ has period one as a periodic almost Nielsen path.
    The argument after the definition of almost linear edges implies that $u_i$
    has period one as a periodic Nielsen path, and thus $E_i$ is a linear edge.
    If no endpoint is principal, then by \hyperlink{Vertices}{(Vertices)}
    and \hyperlink{AlmostPeriodicEdges}{(Almost Periodic Edges)},
    each $\sigma_j$ is an edge in a dihedral pair,
    and in fact it follows that $u_i$ is contained in a single dihedral pair.
    The assumption on almost Nielsen paths of height $i$
    implies that $E_i$ is a linear edge if the dihedral pair is fixed
    and a dihedral linear edge otherwise.
    It also follows that the axis of $T_c$ covers $u_i$ or $u_if_\sharp(u_i)$ as appropriate.
\end{proof}

The next several lemmas will be used in verifying 
\hyperlink{EGAlmostNielsenPaths}{(EG Almost Nielsen Paths)}.

\begin{lem}[cf.~Lemma 4.16 of \cite{FeighnHandel}]
    \label{egalmostnielsenpathsproperfold}
    Suppose that $H_r$ is an aperiodic exponentially growing stratum
    of a relative train track map $f\colon \mathcal{G} \to \mathcal{G}$,
    that $\rho$ is an indivisible almost Nielsen path of height $r$
    and that $\rho$ and $H_r$ satisfy the conclusions of 
    \hyperlink{EGAlmostNielsenPaths}{(EG Almost Nielsen Paths)}.
    Then the fold at the illegal turn of each indivisible almost Nielsen path
    obtained by iteratively folding $\rho$ is proper.
\end{lem}

\begin{proof}
    The proof is identical to \cite[Lemma 4.16]{FeighnHandel}.
    We may assume that $G = G_r$.

    Define the \emph{data set $S$ for $f$ and $\rho$}
    to be the ordered sequences of edges in $H_r$
    in $\rho$ and in $f(E)$ for each edge $E$ of $H_r$.
    As in the proof of \Cref{finitelymanynielsenpaths},
    we see that $S$ determines the type of fold (partial, proper or improper)
    of the fold at the illegal turn of $\rho$:
    if we decompose $\rho$ into a concatenation of $r$-legal paths $\alpha\beta$,
    then edges from $f(\bar\alpha)$ cancel with edges of $f(\beta)$
    until the first distinct $H_r$ edges are reached.
    Assuming that the fold is proper so that the extended fold is defined,
    $S$ also determines the data set for the relative train track map
    and indivisible almost Nielsen path obtained by folding $\rho$.
    So define $S_k$ to be the data set for the relative train track map
    and indivisible almost Nielsen path
    obtained by folding $\rho$ $k$ times, assuming that the folds are defined.

    We adopt the notation of \hyperlink{EGAlmostNielsenPaths}{(EG Almost Nielsen Paths)}.
    We have that $f_r\colon \mathcal{G} \to \mathcal{G}^1$
    is the composition of finitely many, say $K$, proper extended folds
    defined by iteratively folding $\rho$.
    Therefore $S_K$ is defined.
    Since $f|_{G_r} = \theta\circ f_{r-1}\circ f_2$,
    the isomorphism $\theta$ determines a bijection between the edges of the top stratum
    of $\mathcal{G}^2$ (and thus $\mathcal{G}^1$, 
    since $f_{r-1}$ does not affect edges in the top stratum)
    and the edges of $H_r$.
    Furthermore this isomorphism takes $S_K$ to $S_0$.
    Therefore the sequence of $S_k$ for $k \ge 0$ is periodic with period $K$
    so the fold at the illegal turn of each indivisible almost Nielsen path
    obtained by iteratively folding $\rho$ is always proper.
\end{proof}

\begin{prop}[cf.~Lemma 4.17 of \cite{FeighnHandel} and Theorem 5.15 of \cite{BestvinaHandel}]
    \label{egalmostnielsenpathsproperties}
    Suppose $f\colon \mathcal{G} \to \mathcal{G}$ is an eg-aperiodic relative train track map
    and that for each exponentially growing stratum $H_r$
    such that there exists an indivisible almost Nielsen path $\rho$ of height $r$,
    the fold at the illegal turn of each indivisible almost Nielsen path
    obtained by iteratively folding $\rho$ is proper.
    The following hold.
    \begin{enumerate}
        \item For each exponentially growing stratum $H_r$
            there is, up to equivalence and reversal of orientation,
            at most one indivisible almost Nielsen path $\rho$ of height $r$.
    \end{enumerate}
    Supposing $\rho$ exists, we have the following.
    \begin{enumerate}[resume]
        \item Supposing the illegal turn of $\rho$ in $H_r$ is based at the vertex $v$,
            the $\mathcal{G}_v$-orbit of this turn
            are the only nondegenerate illegal turns in $H_r$.
        \item The path $\rho$ crosses every edge in $H_r$.
        \item The length $L_r(\rho)$ (see \Cref{lengthfunction})
            satisfies $L_r(\rho) = 2\sum L_r(E)$,
            where the sum is taken over the edges of $H_r$.
    \end{enumerate}
\end{prop}

\begin{proof}
    We follow the argument in \cite[Lemma 3.9]{BestvinaHandel}.
    Fix $r$ and assume that an indivisible almost Nielsen path $\rho$ of height $r$ exists.
    By assumption,
    the fold at the illegal turn of each indivisible almost Nielsen path obtained
    by iteratively folding $\rho$ is always proper,
    so we may always continue folding.

    If $f\colon \mathcal{G}_1 \to \mathcal{G}_1$ is obtained from $f\colon \mathcal{G} \to \mathcal{G}$
    by folding $\rho$,
    the length function $L_r$ in \Cref{lengthfunction} 
    descends to a length function $L^1_r$ on $\mathcal{G}_1$.
    Since $\mathcal{G}_1$ is obtained from $\mathcal{G}$ by identifying a pair of intervals
    of some equal $L_r$-length $x$, we have
    $\sum L^1_r(E) = \sum L_r(E) - x$,
    where in each case the sum is taken over the edges of $H_r$.
    The indivisible almost Nielsen path $\rho_1$ of height $r$ in $\mathcal{G}_1$ determined by $\rho$
    satisfies
    $L^1_r(\rho_1) = L_r(\rho) - 2x$.
    If $L_r(\rho) \ne 2 \sum L_r(E)$, then
    \[  \left| \frac{L^1_r(\rho_1)}{\sum L_r^1(E)} - 2\right| 
    = \left|\frac{L_r(\rho) - 2x}{\sum L_r(E) - x} - 2\right|
    = \left|\frac{L_r(\rho) - 2\sum L_r(E)}{\sum L_r(E) - x}\right|
    > \left|\frac{L_r(\rho) - 2\sum L_r(E)}{\sum L_r(E)}\right|, \]
    and the last expression is equal to $|L_r(\rho)/\sum L_r(E) - 2|$.
    If we iteratively fold $\rho$,
    obtaining relative train track maps $f_i\colon \mathcal{G}_i \to \mathcal{G}_i$
    and indivisible almost Nielsen paths $\rho_i$,
    we see that the ratio $L^i_r(\rho_i)/\sum L^i_r(E)$ takes on infinitely many values.
    But on the other hand,
    as in the proof of \cite[Theorem 2.11, Step 3]{CollinsTurner},
    notice that the ratio $L^i_r(\rho_i)/ \sum L^i_r(E)$
    is determined by the underlying graph map of $f\colon \mathcal{G} \to \mathcal{G}$,
    since the location of indivisible almost Nielsen paths are determined by fixed points of $f$
    and the metric is determined by the transition matrix of $H_r$.
    Therefore the argument in \cite[Lemma 3.7]{BestvinaHandel}
    applies to show  that we have a contradiction;
    we recount it for the reader's convenience.
    If $\mathcal{G}_1$ and $\mathcal{G}_2$ are isomorphic as graphs of groups
    and if after identifying we find that the associated relative train track maps
    $f_1\colon \mathcal{G}_1 \to \mathcal{G}_1$ and $f_2\colon \mathcal{G}_2 \to \mathcal{G}_2$
    have the property that for each edge $E$ of the identified graph,
    the underlying paths of $f_1(E)$ and $f_2(E)$ are the same,
    then it follows that the associated ratios above are equal.
    Since the number of edges in $\mathcal{G}$ remains constant under proper folds
    and there are only finitely many irreducible matrices of a given Perron--Frobenius eigenvalue,
    it follows that there are only finitely many possibilities for the maps
    $f_i\colon \mathcal{G}_i \to \mathcal{G}_i$,
    at least up to this coarse notion of equivalence.
    This contradiction implies that $L_r(\rho) = 2\sum L_r(E)$.

    Now suppose that $T$ is an illegal turn in $\mathcal{G}$ 
    that is distinct from the illegal turn in $\rho$.
    Take $\ell$ to be the smallest positive integer such that $Df^\ell(T)$ is degenerate.
    Suppose first that $\ell = 1$.
    We may fold the edges determining $T$ (after possibly subdividing)
    to obtain a topological representative 
    which may not be a relative train track map because \hyperlink{EG-i}{(EG-i)}
    and \hyperlink{EG-ii}{(EG-ii)} may fail.
    Restore these properties by the moves
    ``(invariant) core subdivision'' and ``collapsing inessential connecting paths''
    as in \cite[Lemmas 3.4 and 3.5]{Myself} or \cite[Lemmas 5.13 and 5.14]{BestvinaHandel}.
    By \Cref{pfcorollary}, \hyperlink{EG-iii}{(EG-iii)} is satisfied,
    resulting in a relative train track map
    $\hat f\colon \hat{\mathcal{G}} \to \hat{\mathcal{G}}$
    and indivisible almost Nielsen path $\hat\rho$ determined by $\rho$. 
    The map $\hat f$ is still eg-aperiodic and 
    the fold at $\hat\rho$ is still proper.
    Invariant core subdivision and collapsing inessential connecting paths
    do not affect the resulting length function $\hat L_r$,
    so we have that $\sum \hat L_r(E) < \sum L_r(E)$ and $\hat L_r(\hat\rho) = L_r(\rho)$,
    so that $\hat L_r(\hat\rho)/(\sum \hat L_r(E)) > 2$.
    This contradicts our previous arguments.
    Therefore we must have that $\ell \ne 1$
    and $Df^{\ell - 1}(T)$ is in the $\mathcal{G}_v$-orbit of the illegal turn in $\rho$.
    Replacing the turn $T$ with $Df^{\ell - 2}(T)$,
    we may assume $\ell = 2$.
    Fold $\rho$ to obtain a relative train track map $f_1\colon \mathcal{G}_1 \to \mathcal{G}_1$,
    let $\rho_1$ be the indivisible almost Nielsen path in $\mathcal{G}_1$ determined by $\rho$
    and let $T_1$ be the turn in $\mathcal{G}_1$ determined by $T$.
    Then $Df_1(T_1)$ is degenerate (since the previously illegal turn at $\rho$ is now degenerate),
    and we reach a contradiction as in the previous case.

    Inductively define $f_k\colon \mathcal{G}_k \to \mathcal{G}_k$ and $\rho$
    by folding $\rho_{k-1}$ in $\mathcal{G}_{k-1}$,
    starting with $f_0 = f\colon \mathcal{G} \to \mathcal{G}$ and $\rho_0 = \rho$.
    For $k \ge 0$, we have $\sum L^{k+1}_r(E) = \sum L^k_r(E) = x_k$.
    By the existence of $\theta$ conjugating $L^K_r$ to a multiple of $L^0_r$,
    we see that $x_k/(\sum L^k_r(E))$ takes on finitely many values, so it is uniformly bounded below.
    It follows that $\sum L^k_r(E)$ goes to zero with $k$,
    so each $L^k_r(E)$ does as well, proving that $\rho$ crosses every edge of $H_r$.

    Now suppose $\rho' = \alpha'\beta'$ is another indivisible almost Nielsen path of height $r$.
    We must show that $\rho$ is equivalent to $\rho'$ (after possibly reversing the orientation).
    Since there is one $\mathcal{G}_v$-orbit of illegal turns in $H_r$,
    $\rho$ and $\rho'$ have the same illegal turn.
    (Recall from \Cref{basicssection} our convention on which turns a path crosses.)
    After reorienting $\rho'$ if necessary,
    we may assume that the initial edge of $\bar \alpha$ equals the initial edge of $\bar\alpha'$
    and the initial edge of $\beta$ equals the initial edge of $\beta'$.
    We factor the vertex group element at the illegal turn so that
    it is equal for $\alpha$ and $\alpha'$ and thus for $\beta$ and $\beta'$.
    Let $\rho'_k$ be the indivisible almost Nielsen path in $\mathcal{G}_k$ determined by $\rho'$.
    Write $\rho_k = \alpha_k\beta_k$ and $\rho'_k = \alpha'_k\beta'_k$ 
    as in \Cref{finitelymanynielsenpaths}.

    Suppose that $\alpha \ne g\alpha'$ for any vertex group element $g$.
    Then if the initial endpoints of $\alpha$ and $\alpha'$ are distinct,
    the same is true of $\alpha_k$ and $\alpha'_k$.
    If the endpoints are equal, then $\alpha\bar\alpha'$ is a homotopically nontrivial loop,
    and the same is true of $\alpha_k\bar\alpha_k'$.
    We see that $\alpha_k \ne \alpha'_k$ for all $k$.
    For each $k \ge 0$, write $\nu_k$ for the maximum common terminal subinterval 
    of $\alpha_k$ and $\alpha'_k$
    and let $\mu_k$ and $\mu'_k$ be the complementary initial segments of $\alpha_k = \mu_k\nu_k$
    and $\alpha'_k = \mu'_k\nu_k$.
    Suppose that the turn $\{\bar\mu_k,\mu'_k\}$ is not in the same $\mathcal{G}_v$-orbit
    as $\{\bar\alpha_k,\beta_k\}$.
    Then we have $L^{k+1}_r(\mu_{k+1}) = L^k_r(\mu_k)$.
    But the fact that the (orbits of) turns above are distinct
    implies that $\{\bar\mu_{k+1},\mu'_{k+1}\}$ is not in the $\mathcal{G}_v$-orbit
    of the turn $\{\bar\alpha_{k+1},\beta_{k+1}\}$,
    and we conclude that $L^{k+i}_r(\mu_{k+i}) = L^k_r(\mu_k)$ for $i > 0$,
    contradicting the fact that $L^k_r(\rho_k) = 2\sum L^k_r(E)$ goes to zero as $k$ increases.

    Therefore for $k \ge 0$ we may assume that the turns $\{\bar\alpha_k,\beta_k\}$
    and $\{\bar\mu_k,\mu'_k\}$ are in the same $\mathcal{G}_v$-orbit.
    Then $\nu_{k+1}$ is obtained from $\nu_k$ by deleting a terminal subinterval
    of length $x_k$ and adding an initial interval of length $x_k$,
    so we have $L^{k+1}_r(\nu_{k+1}) = L^k_r(\nu_k)$, 
    again contradicting the fact that $L^k_r(\rho_k)$ goes to zero.
    This proves that $\rho$ is equivalent to $\rho'$.
\end{proof}

\begin{lem}[cf.~Lemma 4.24 of \cite{FeighnHandel}]
    \label{egalmostnielsenpathsfurtherproperties}
    Suppose that $f\colon\mathcal{G} \to \mathcal{G}$ is a rotationless relative train track map
    satisfying the conclusions of \Cref{improvedrelativetraintrack},
    that $H_r$ is an exponentially growing stratum 
    satisfying \hyperlink{EGAlmostNielsenPaths}{(EG Almost Nielsen Paths)},
    and that $\rho$ is an indivisible almost Nielsen path of height $r$.
    Then the following hold.
    \begin{enumerate}
        \item $H^z_r = H_r$.
        \item If $\rho = a_1b_1\ldots b_{\ell}a_{\ell+1}$ is the decomposition of $\rho$
            into subpaths $a_i$ of height $r$ and maximal subpaths $b_i$
            in $G_{r-1}$, then each $b_i$ is an almost Nielsen path.
        \item If $E$ is an edge of $H_r$, then each maximal subpath of $f(E)$ in $G_{r-1}$
            is a path $b_i$ from item 2. In particular $f(E)$ splits into edges in $H_r$
            (possibly with vertex group elements at either end)
            and almost Nielsen paths in $G_{r-1}$.
    \end{enumerate}
\end{lem}

\begin{proof}
    We adopt the notation of \hyperlink{EGAlmostNielsenPaths}{(EG Almost Nielsen Paths)};
    the proof is identical to \cite[Lemma 4.24]{FeighnHandel}.
    The maps $f_r$, $f_{r-1}$ and $\theta$ induce bijections on the set of components
    in the filtration element of height $r-1$.
    It follows that $f|_{G_r} = \theta f_{r-1}f_r$ induces a bijection
    on the set of components of $G_{r-1}$ and hence that each such component is non-wandering.
    By \hyperlink{Z}{(Z)} we conclude that item 1 holds.

    For item 2, let $(f_r)_\sharp(\rho) = a'_1b'_1\ldots a'_mb'_ma'_{m+1}$
    be the decomposition into subpaths $a'_j$ of height $r$ and maximal subpaths $b'_j$
    in $G_{r-1}$.
    Since proper folds do not create new paths $b'_j$,
    we see that the set of $b_j$ is contained in the set of $b_i$.
    Now let $(\theta f_{r-1})_\sharp(a'_1b'_1\ldots a'_mb'_{m+1}) = c_1d_1\ldots d_pc_{p+1}$
    be the decomposition into subpaths $c_k$ of height $r$
    and maximal subpaths $d_k$ in $G_{r-1}$.
    For each $k$ there exists $j$ such that $d_k$ differs from $(\theta f_{r-1})_\sharp(b'_j)$
    by multiplying at the ends by vertex group elements.
    Since $a_1b_1\ldots b_\ell a_{\ell+1} =c_1d_1\ldots d_pc_{p+1}$,
    we conclude that up to multiplying at the ends by vertex group elements,
    $f_\sharp$ permutes the paths $b_i$.
    Since $f$ is rotationless and the $b_i$ cannot be contained in a dihedral pair,
    each $b_i$ is an almost Nielsen path.

    If $E$ is an edge of $H_r$ then by construction,
    each maximal subpath of $f_r(E)$ in $G_{r-1}$ is one of the paths $b_i$.
    By item 2, each $b_i$ is an almost Nielsen path for $f$ and hence  for $\theta f_{r-1}$,
    proving item 3.
\end{proof}

\paragraph{Changing the marking on $G_j$.}
To construct CTs, we require a move that plays the role of sliding
for exponentially growing and zero strata.

Suppose that $f\colon \mathcal{G} \to \mathcal{G}$ is a rotationless relative train track map
satisfying the conclusions of \Cref{improvedrelativetraintrack}
with respect to the filtration $\varnothing = G_0 \subset G_1 \subset \cdots \subset G_m = G$,
that $j$ is such that $1 \le j \le m$,
every component of $G_j$ is non-contractible
and that $f$ fixes every vertex in $G_j$ whose link is not contained in $G_j$.
Define a homotopy equivalence $g\colon \mathcal{G} \to \mathcal{G}$
by setting $g|_{G_j} = f|_{G_j}$ and setting $g|_{(G\setminus G_j)}$ to be the identity.
Suppose $\tau\colon \mathbb{G} \to \mathcal{G}$ is the original marked graph of groups.
Let $\mathcal{G}'$ be the marked graph of groups $g\tau\colon \mathbb{G} \to \mathcal{G}$.
Therefore $\mathcal{G}$ and $\mathcal{G}'$ have the same underlying graph of groups,
and there is a natural identification of $\mathcal{G}$ with $\mathcal{G}'$;
we will use this identification when discussing edges and strata.

Define $f' \colon \mathcal{G}' \to \mathcal{G}'$ by setting
$f'|_{G'_j} = f|_{G_j}$
and $f'(E) = (gf)_\sharp(E)$ for all edges $E$ in $H_i$ with $i > j$.
We say that $f'\colon \mathcal{G}' \to \mathcal{G}'$ 
is obtained from $f\colon \mathcal{G} \to \mathcal{G}$
by \emph{changing the marking on $G_j$ via $f$.}
We have the following lemma.

\begin{lem}[cf.~Lemma 4.27 of \cite{FeighnHandel}]
    \label{changingthemarking}
    Suppose that $f'\colon \mathcal{G}' \to \mathcal{G}'$ is obtained from
    $f\colon \mathcal{G} \to \mathcal{G}$
    by changing the marking on $G_j$ via $f$. Then the following hold.
    \begin{enumerate}
        \item We have $f'|_{G_j} = f|_{G_j}$.
        \item For every path $\sigma$ with endpoints at vertices and $k > 0$,
            we have $g_\sharp f_\sharp^k(\sigma) = (f')^k_\sharp g_\sharp(\sigma)$.
        \item If $f\colon \mathcal{G} \to \mathcal{G}$ represents $\varphi \in \out(F,\mathscr{A})$,
            then $f'\colon \mathcal{G}' \to \mathcal{G}'$
            is a homotopy equivalence representing $\varphi \in \out(F,\mathscr{A})$.
        \item There is a one-to-one correspondence between almost Nielsen paths for $f$
            and almost Nielsen paths for $f'$.
        \item The map $f'\colon \mathcal{G}' \to \mathcal{G}'$
            is a rotationless relative train track map satisfying the conclusions
            of \Cref{improvedrelativetraintrack}
            with respect to the original filtration
            $\varnothing = G_0 \subset G_1 \subset \cdots \subset G_m = G$.
    \end{enumerate}
\end{lem}

\begin{proof}
    The proof is identical to \cite[Lemma 4.27]{FeighnHandel}.
    Item 1 is immediate from the definitions,
    as is the fact that $f'$ preserves the original filtration.
    Observe that the following statements also hold.
    \begin{enumerate}
        \item[6.] If a point $x \in G$ satisfies $f(x) \ne f'(x)$,
            then $x \notin G_j$ and $f(x)$ and $f'(x)$ belong to $G_j$.
            In particular, $\fix(f) = \fix(f') \subset \fix(g)$;
            furthermore $\per(f) = \per(f') \subset \per(g)$.
            Under the identification of $\mathcal{G}$ with $\mathcal{G}'$,
            we see that $Df$ and $Df'$ have the same fixed and periodic directions.
        \item[7.] Suppose that $E$ is an edge of $H_i$ for $i > j$
            and that $f(E) = \mu_1\nu_1\mu_2\ldots \nu_{k-1}\mu_k$,
            where the $\nu_\ell$ are maximal subpaths in $G_j$.
            Allow the case where $\mu_1$ and $\mu_k$ are trivial.
            Then we have $f'(E) = \mu_1 f_\sharp(\nu_1)\mu_2 \ldots f_\sharp(\nu_{k-1})\mu_k$.
            Since $f$ fixes the endpoints of each $\nu_\ell$,
            it follows that the $f\sharp(\nu_\ell)$ are nontrivial.
        \item[8.] Therefore each stratum $H_i$ has the same type for $f$ as for $f'$.
    \end{enumerate}

    To prove item 2, it suffices to consider the case where $k = 1$ and $\sigma$ is a single edge $E$.
    If $E \subset G_j$,
    then $g_\sharp f_\sharp(E) = f_\sharp(f_\sharp(E)) = f'_\sharp g_\sharp(E)$.
    If $E \subset G_i$ for $i > j$,
    then $g_\sharp f_\sharp(E) = f'_\sharp(E) = f'_\sharp g_\sharp(E)$.
    Item 2 implies item 3.

    If $\rho'$ is a tight path in $\mathcal{G}$ with endpoints in $\fix(f') = \fix(f)$,
    then there is a unique tight path $\rho$ with the same endpoints such that $g_\sharp(\rho) = \rho'$.
    Condition 2 implies that $\rho'$ is an almost Nielsen path for $f'$
    if and only if it is an almost Nielsen path for $f$, proving item 4.

    By items 1, 6 and 7, to show that $f'\colon \mathcal{G}' \to \mathcal{G}'$
    is a relative train track map, it suffices to show that each exponentially growing stratum
    $H_i$ for $i > j$ satisfies \hyperlink{EG-ii}{(EG-ii)}.
    Suppose first that $\sigma$ is a connecting path contained in a non-contractible component
    $C$ of $G_{i-1}$.
    By \Cref{trivialedgegroupsEG2}, each vertex of $H_i \cap C$ is periodic;
    has valence at least two and is principal for $f$ and hence fixed by $f'$,
    so we conclude by \Cref{trivialedgegroupsEG2} that $f'_\sharp(\sigma)$ is nontrivial.
    If $\sigma$ is instead contained in a contractible component of $G_{i-1}$,
    then it is contained in a zero stratum that has height greater than $j$
    because it is contained in $H^z_i$.
    We have $(f')_\sharp(\sigma) = g_\sharp f_\sharp(\sigma)$.
    If $f_\sharp(\sigma)$ is contained in a non-contractible component of $G_{i-1}$,
    then $g_\sharp f_\sharp(\sigma)$ is nontrivial by the previous argument.
    If not, then $g_\sharp f_\sharp(\sigma) = f_\sharp(\sigma)$ and we are done.
    The remaining properties of \Cref{improvedrelativetraintrack}, and thus item 5, follow from
    the definitions and items 4 and 6.
\end{proof}

\begin{lem}[cf.~Lemma 4.25 of \cite{FeighnHandel}]
    \label{largepowerscompletelysplit}
    If $f\colon \mathcal{G} \to \mathcal{G}$ is a CT and $\sigma$ is a path in $G_r$
    with endpoints at vertices,
    then $f^k_\sharp(\sigma)$ is completely split for all sufficiently large $k$.
\end{lem}

\begin{proof}
    The proof is identical to \cite[Lemma 4.25]{FeighnHandel};
    we proceed by induction on $r$, the height of $\sigma$.
    There are no paths of height $r = 0$,
    so suppose that $\sigma$ has height $r \ge 1$
    and that the lemma holds for paths of height less than $r$.
    By \Cref{completelysplittocompletelysplit},
    \Cref{completesplittingunique}
    and the inductive hypothesis,
    it suffices to show that some path $f^k_\sharp(\sigma)$
    has a splitting into subpaths that are either themselves completely split
    or contained in $G_{r-1}$.
    This is immediate if $H_r$ is a zero stratum or an almost periodic stratum.
    If $H_r$ is non-exponentially growing but not almost periodic,
    then by \Cref{negsplitting}, 
    $\sigma$ has a splitting into basic paths of height $r$
    and subpaths in $G_{r-1}$,
    and the desired splitting of $f^k_\sharp(\sigma)$ follows by \Cref{negctsplitting}.
    If $H_r$ is exponentially growing,
    \Cref{egsplitting} implies that some $f^k_\sharp(\sigma)$
    splits into pieces, each of which is either $r$-legal
    or part of the finite set $P_r$ of equivalence classes of paths 
    satisfying the conditions in the proof of \Cref{finitelymanynielsenpaths}.
    After iteration, the elements in $P_r$ get mapped to indivisible periodic almost Nielsen paths,
    which are almost Nielsen paths by \hyperlink{Vertices}{(Vertices)}.
    \Cref{rttlemma} implies that the $r$-legal paths in $G_r$ split into single edges in $H_r$
    and subpaths in $G_{r-1}$.
    This completes the inductive step.
\end{proof}

\paragraph{Sliding revisited.}

The last tool we need before turning to \Cref{CTtheorem} is the following proposition.

\begin{prop}[cf.~Proposition 4.35 of \cite{FeighnHandel}]
    \label{neginduction}
    Suppose that $f\colon \mathcal{G} \to \mathcal{G}$ is a relative train track map
    that satisfies \hyperlink{EGAlmostNielsenPaths}{(EG Almost Nielsen Paths)},
    that $f|_{G_{s-1}}$ is a CT,
    and that $H_s$ is a non-exponentially growing stratum with a single edge $E_s$
    and that there does not exist a path $\mu$ in $G_{s-1}$
    such that $E_s\mu$ is an almost Nielsen path.

    There exists a path $\tau$ in $G_{s-1}$ with initial endpoint
    equal to the terminal vertex of $E_s$ such that after sliding along $\tau$,
    the following conditions are satisfied.
    \begin{enumerate}
        \item $f(E_s) = g_sE_s\cdot u_s$ is a nontrivial splitting.
        \item If $\sigma$ is a path or circuit with endpoints at vertices
            and height $s$,
            then there exists $k \ge 0$
            such that $f^k_\sharp(\sigma)$ splits into subpaths of the following type.
            \begin{enumerate}
                \item $E_s$ or $\bar E_s$.
                \item An exceptional path of height $s$.
                \item A subpath of $G_{s-1}$.
            \end{enumerate}
        \item $u_s$ is completely split and its initial vertex is either
            principal or the center vertex of a dihedral pair;
            in the latter case $u_s$ is contained in the dihedral pair
            and $E_s$ is a linear or dihedral linear edge.
        \item $f|_{G_s}$ satisfies \hyperlink{LinearEdges}{(Linear Edges)}.
    \end{enumerate}
\end{prop}

\begin{proof}
    We follow \cite[Proposition 5.4.3]{BestvinaFeighnHandel}.
    We adopt the notation of ``Restricting to $G_{s-1}$ for non-exponentially growing strata''
    from \Cref{principalsection}.
    In particular we have the map $h\colon \Gamma_{s-1} \to \Gamma_{s-1}$.
    If $h$ had a fixed point, 
    the path in $\Gamma$ beginning at $\tilde E_s$ and ending at this fixed point
    would project to an almost Nielsen path of the form $gE_s\mu$.
    The assumption that there is no such path therefore implies that $h$ is fixed-point free.

    By \Cref{emptyfixedsetfixedboundarypoint},
    given any starting vertex $\tilde v$ in $\Gamma_{s-1}$,
    there is a ray $\tilde R$ beginning at $\tilde v$ 
    and converging to a fixed point $P \in \partial_\infty \Gamma_{s-1}$.
    The ray has the property that it shares a nontrivial initial segment
    with the tight path from $\tilde v$ to $h(\tilde v)$.
    There are points along this ray arbitrarily close to $P$ that move towards $P$.
    In fact, the point $P$ is independent of the choice of starting vertex $\tilde v$,
    for if the fixed point produced by choosing $\tilde v'$ were different,
    then $\tilde f$ would move a pair of points away from each other,
    contradicting \Cref{producingafixedpoint}.

    Therefore we have, for any vertex $\tilde v$ in $\Gamma_{s-1}$,
    a ray $\tilde R_{\tilde v}$ beginning at $\tilde v$ and converging to $P$
    with the property that $\tilde R_{\tilde v}$ shares a nontrivial initial segment
    with the tight path from $\tilde v$ to $h(\tilde v)$.
    Given $\tilde v$ and $\tilde w$, the rays $\tilde R_{\tilde v}$ and $\tilde R_{\tilde w}$
    share a common terminal subray.
    Fix a vertex $\tilde v_0$ and inductively define $\tilde v_{i+1}$
    to be the first vertex of the common subray of $\tilde R_{\tilde v_i}$ and $\tilde R_{h(\tilde v_i)}$.
    Given two points $\tilde v$ and $\tilde w$ in $\Gamma_{r-1}$,
    write $[\tilde v,\tilde w]$ for the unique tight path from $\tilde v$ to $\tilde w$.
    We have
    \[  \tilde R_{\tilde v_0} 
    = [\tilde v_0,\tilde v_1][\tilde v_1,\tilde v_2][\tilde v_2,\tilde v_3]\ldots \]
    and we have that $h_\sharp([\tilde v_i,\tilde v_{i+1}]) = [h(\tilde v_i),h(\tilde v_{i+1})]$
    contains the path $[\tilde v_{i+1},\tilde v_{i+2}]$.
    Let $\tilde Y_m$ be the set of points $\tilde y$ in $[\tilde v_0,\tilde v_1]$
    such that $h^i(\tilde y)$ belongs to $[\tilde v_i,\tilde v_{i+1}]$
    for $i$ satisfying $0 \le i \le m$.
    The map $h^m$ sends $\tilde Y_m$ over all of $[\tilde v_m,\tilde v_{m+1}]$,
    so $\tilde Y_m$ is nonempty, and thus the intersection
    $\bigcap_{m=0}^\infty \tilde Y_m$ is nonempty.
    Let $\tilde p$ be a point in the intersection and notice that
    $\tilde p$ has the property that $\{\tilde p, h(\tilde p), h^2(\tilde p), \ldots \}$
    is an ordered subset of a ray converging to $P$.
    Let $\tilde X$ be the set of such points $\tilde p$.
    Observe that $h(\tilde X) \subset \tilde X$.
    Observe as well that given $\tilde x \in \tilde X$
    and a point $\tilde y$ in $[\tilde x,h(\tilde x)]$ such that the decomposition
    $[\tilde x,\tilde y]\cdot [\tilde y,h(\tilde x)]$ is a splitting (for $f$ and hence $h$),
    then $\tilde y$ belongs to $\tilde X$.

    We would like to find a vertex in $\tilde X$.
    Let $\ell$ be the smallest positive integer
    such that there exists $\tilde x \in \tilde X$
    such that the projection of $[\tilde x,h(\tilde x)]$ is contained in $G_\ell$
    and choose such a point $\tilde x$.
    Notice that $H_\ell$ cannot be a zero stratum.
    Suppose first that $H_\ell$ is almost periodic,
    so $H_\ell$ consists of at most two edges, $E_\ell$ and $E_\ell'$.
    Since $h$ is a topological representative,
    we have that the path $[\tilde x,h(\tilde x)]$
    is not contained in a single edge
    and thus contains a vertex $\tilde v$ that projects to the initial or terminal vertex of some edge
    of $H_\ell$, say $E_\ell$.
    \Cref{negsplitting} implies that $[\tilde x,h(\tilde x)]$ can be split at $\tilde v$
    and thus that $\tilde v \in \tilde X$.
    (Let us remark that the hypotheses of that lemma require $H_\ell$ to be a single edge,
    but the proof also works for a dihedral pair.)

    Suppose that $H_\ell$ is non-exponentially growing but not almost periodic,
    so $H_\ell$ consists of a single edge $E_\ell$
    and $f(E_\ell) = g_\ell E_\ell u_\ell$ for some path $u_\ell$ in $G_{\ell-1}$.
    After replacing $\tilde x$ by $h^k(\tilde x)$ for some $k \ge 0$ if necessary,
    we may assume that the path $[\tilde x,h(\tilde x)]$ contains an entire edge $\tilde e$
    projecting to either $E_\ell$ or $\bar E_\ell$
    and that the projection of $h(\tilde x)$ is not contained in the interior of $E_\ell$.
    If $\tilde e$ projects to $E_\ell$, let $\tilde v$ be the initial vertex of $\tilde e$;
    otherwise let it be the terminal vertex of $\tilde e$.
    \Cref{negsplitting} implies that $[\tilde x,h(\tilde x)]$ can be split at $\tilde v$,
    so $\tilde v \in \tilde X$.

    Finally suppose $H_\ell$ is exponentially growing.
    After replacing $\tilde x$ by $h^k(\tilde x)$ for some $k \ge 0$ if necessary,
    we may assume that each path $[h^i(\tilde x),h^{i+1}(\tilde x)]$
    projects to a path with the same number of illegal turns in $H_\ell$.
    \Cref{egsplitting} produces a splitting of $[\tilde x,h(\tilde x)]$.
    If one of the resulting pieces is not $\ell$-legal,
    it is a lift $\tilde\rho$ of one of finitely many equivalence classes of paths $\rho \in P_\ell$.
    Let $\tilde v$ be the initial endpoint of $\tilde\rho$.
    Replacing $\tilde v$ by $h^k(\tilde v)$ if necessary,
    we may assume that the point $\tilde v$ projects to a periodic point in the interior of $H_\ell$.
    In fact, after iterating we may assume that $\tilde\rho$ projects to an indivisible periodic 
    (hence fixed by \hyperlink{Vertices}{(Vertices)}) almost Nielsen path of height $\ell$,
    so we may assume that $\tilde v$ is a vertex.
    If there are no $\tilde \rho$ pieces, then $[\tilde x,h(\tilde x)]$ projects to an $\ell$-legal path.
    After replacing $\tilde x$ by $h^k(\tilde x)$ we may assume that the path
    $[\tilde x,h(\tilde x)]$ contains an entire edge that projects into $H_\ell$.
    \Cref{rttlemma} implies that $[\tilde x,h(\tilde x)]$ may be split at any endpoint $\tilde v$
    of this edge.
    Replacing $\tilde v$ by $h^k(\tilde v)$ if necessary,
    we may assume that $\tilde v$ projects to a periodic vertex.

    The path in $\Gamma_{s-1}$ from the terminal vertex of $\tilde E_s$
    to $\tilde v$ projects to a path $\tau$ in $G_{s-1}$.
    It is immediate from the definition of sliding and \Cref{sliding lemma}
    that after sliding along $\tau$,
    we have $f(E_s) = g_sE_su_s$, where $u_s$ is the projected image of the path $[\tilde v, h(\tilde v)]$.
    Thus by replacing $\tilde v$ with $h^k(\tilde v)$ for some $k \ge 0$,
    we may replace $u_s$ with $f_\sharp^k(u_s)$.
    Since we assume that $f|_{G_{s-1}}$ satisfies \hyperlink{CompletelySplit}{(Completely Split)},
    we may assume by \Cref{largepowerscompletelysplit} that $u_s$ is completely split.
    By replacing $\tilde v$ with another vertex in $\tilde X$,
    we may also alter $u_s$ as follows:
    if $u_s = \alpha\cdot \beta$ is a coarsening of the complete splitting of $u_s$,
    we may alter $u_s$ so that the terminal vertex of $E_s$ is the initial vertex of $\beta$.

    Therefore to arrange item 3,
    it suffices to prove that either some term of the complete splitting of $u_s$
    has a principal endpoint or that $u_s$ is entirely contained in a dihedral pair.
    So suppose $u_s$ is not entirely contained in a dihedral pair.
    Then since all other non-principal vertices are contained in exponentially-growing strata,
    the only way the complete splitting of $u_s$ could fail to contain a principal vertex
    is if $\ell$ is exponentially growing and each height-$\ell$ term of the complete splitting of $u_s$
    is a single edge.
    But then by replacing $u_s$ by $f^k_\sharp(u_s)$ for sufficiently large $k$,
    we may assume that $u_s$ has a long $\ell$-legal segment
    so that every edge in $H_\ell$ appears as a term in the complete splitting of $u_s$.
    \Cref{principalEG} implies the existence of a principal vertex.

    If $E_s$ is a non-dihedral linear edge, 
    choose a root-free almost Nielsen path $w_s$ and $d_s \ne 0$
    so that $u_s = w_s^{d_s}$.
    If $E_t \subset G_{s-1}$ is a linear edge with the same axis,
    then after reversing the orientation on $w_s$ and multiplying $d_s$ by $-1$,
    we may assume that $w_s$ and $w_t$ agree as oriented circuits.
    After sliding to change the order of the edges in $w_s$ we may further assume $w_s = w_t$.
    Since $E_s\bar E_t$ is not an almost Nielsen path by assumption,
    we conclude that $d_s \ne d_t$, proving item 4, \hyperlink{LinearEdges}{(Linear Edges)} in this case.

    If instead $E_s$ is dihedral linear,
    write the axis for $E_s$ as $\sigma\tau$ so that 
    $u_s = (\sigma\tau)^{d_s}$ or $u_s$ is homotopic to $(\sigma\tau)^{d_s}\sigma$
    for some integer $d_s$ (with $d_s \ne 0$ in the first case).
    If $E_t$ is a linear (hence dihedral linear) edge in $G_{s-1}$ with the same axis,
    then $u_t$ also has the above form for some integer $d_t$.
    Since by assumption $E_s\bar E_t$ is not an almost Nielsen path by assumption,
    we conclude that if $u_s$ and $u_t$ are both of the former type or both of the latter type,
    then $d_s \ne d_t$, again proving item 4.

    Following \cite[Proposition 5.4.3]{BestvinaFeighnHandel},
    we will show that if $w_s$ is a nontrivial initial segment of $u_s$,
    then $E_s\cdot w_s$ is a splitting.
    This is equivalent to the claim that $[\tilde v h^i(\tilde v)]$ 
    is contained in $[\tilde v,h^i(\tilde y)]$ for all points $\tilde y$
    in the path $[\tilde v,h(\tilde v)]$.

    Since $\tilde v \in \tilde X$, we have that $f(E_s) = E_s \cdot u_s$ is a nontrivial splitting.
    In order to prove item 2,
    we will prove the stronger statement 
    \begin{enumerate}
        \item[5.] If $\tilde w_s$ is an initial segment of $u_s$,
        then $E_s\cdot w_s$ is a splitting.
    \end{enumerate}
    This is equivalent to showing that for all $\tilde y \in [\tilde v,h(\tilde v)]$,
    we have that $[\tilde v,h^i(\tilde v)]$ is contained in $[\tilde v,h^i(\tilde y)]$ for all $i \ge 0$.
    Write $\tilde u_s = [\tilde v,h(\tilde v)]$.
    There is a splitting $\tilde u_s = \tilde\sigma_1\cdot \tilde\sigma_2\cdots\tilde\sigma_n$
    provided by \Cref{negsplitting} if $H_\ell$ is non-exponentially growing
    and by \Cref{egsplitting} if $H_\ell$ is exponentially growing.
    We have $\tilde R_{\tilde v} = \tilde u_s \cdot h_\sharp(\tilde u_s)\cdot h^2_\sharp(\tilde u_s)\cdots$
    which yields a splitting
    $\tilde R_{\tilde v} = \tilde\sigma_1\cdot \tilde\sigma_2\cdots$,
    where  $\tilde\sigma_{in+j} = h^i_\sharp(\tilde\sigma_j)$.
    To show that the stated claim holds,
    it suffices to show that $h^i(\tilde u_s)$  intersects $h^{i-1}_\sharp(\tilde u_s)$ in a point.
    We will prove the stronger statement 
    that $h^i(\tilde\sigma_j)$, which tightens to $\tilde\sigma_{in + j}$,
    intersects  $\tilde\sigma_{in+j-1}$  in a point.

    The proof breaks into cases, depending on whether $H_\ell$ is exponentially growing
    or non-exponentially growing.
    Suppose first that it is non-exponentially growing.
    If the initial edge of $\tilde\sigma_j$ is a lift of $E_\ell$,
    then \Cref{negsplitting} implies that $h^i(\tilde\sigma_j)$ is a lift of $E_\ell$
    possibly followed by a sequence of edges lifting edges in $G_{\ell - 1}$
    and possibly terminating in a lift of $\bar E_\ell$.
    If $h^i(\tilde\sigma_j)$ \emph{does} terminate in a lift of $\bar E_\ell$,
    then the sequence of edges lifting edges in $G_{\ell-1}$ tightens to a nontrivial path.
    Observe then that the initial lift of $E_\ell$ is disjoint from the rest of $h^i(\tilde\sigma_j)$,
    so prevents $h^i(\tilde\sigma_j)$ from intersecting $\tilde\sigma_{in + j - 1}$ in more than one point.
    On the other hand, if $\tilde\sigma_j$ does not begin with a lift of $E_\ell$,
    then the terminal end of $\tilde\sigma_{in+j-1}$ is a lift of $\bar E_\ell$
    and $h^i(\tilde\sigma_j)$ is a sequence of edges lifting edges in $G_{\ell-1}$
    possibly followed by a lift of $\bar E_\ell$.
    Again, if the last edge of $h^i(\tilde\sigma_j)$ is a lift of $\bar E_\ell$,
    then the sequence of edges lifting edges in $G_{\ell-1}$ tightens to a nontrivial path.
    We see that edges lifting edges in $G_{\ell-1}$ cannot cross the last edge of $\tilde\sigma_{in+j-1}$
    and the final lift of $\bar E_\ell$ cannot because it is separated by the homotopically nontrivial
    path the edges lifting edges in $G_{\ell-1}$ tighten to form.
    This completes the analysis in the case that $H_\ell$ is non-exponentially growing.

    Suppose on the other hand that $H_\ell$ is exponentially growing.
    If  $\tilde\sigma_j = \tilde\rho$ for some path  $\rho \in P_\ell$,
    write $\rho = \alpha\beta$ for a decomposition of $\rho$ into $\ell$-legal subpaths.
    Although the terminal end of $h^i(\tilde\alpha)$ and the initial end of $h^i(\beta)$ 
    agree up to a point,
    there is an initial subpath of $h^i(\tilde\alpha)$ 
    that is disjoint from the rest of $h^i(\tilde\sigma_j)$;
    as above this initial subpath prevents $h^i(\tilde\sigma_j)$ from intersecting $\tilde\sigma_{in+j-1}$
    in  more than a point.
    If $\tilde\sigma_j$ is not a lift of some $\rho \in P_\ell$,
    then $\sigma_j$ is $\ell$-legal.
    If the initial edge of $\sigma_j$ is in $H_\ell$,
    then \Cref{rttlemma} implies that the initial  edge of $h^i(\tilde\sigma_j)$
    is disjoint from the rest of $h^i(\tilde\sigma_j)$, and thus prevents $h^i(\tilde\sigma_j)$
    from intersecting $\tilde\sigma_{in+j-1}$ in more than a point.
    If instead the initial edge of $\sigma_j$ is in $G_{\ell-1}$,
    then the terminal edge of $\tilde\sigma_{in+j-1}$ projects to $H_\ell$
    and prevents $h^i(\tilde\sigma_j)$ from intersecting $\tilde\sigma_{in+j-1}$ in more than a point:
    this is because edges  that project to $G_{\ell-1}$ cannot cross into $H_\ell$
    and edges that project into $H_\ell$ are part of $h^i_\sharp(\tilde\sigma_j)$.
    This completes the analysis  in the exponentially growing case and with it the proof of item 5.

    Before proving item 2, we establish the following.
    \begin{enumerate}
        \item[6.] The path $u_s$ is a periodic Nielsen path if and only if
            there is a non-peripheral element $c \in F$ such that
            $T_c$ preserves $\Gamma_{i-1}$ and $h$ and $T_c|_{\Gamma_{i-1}}$ commute.
            In this case the infinite ray 
            $\tilde R_{\tilde v} = \tilde u_s h_\sharp(\tilde u_s)h^2_\sharp(\tilde u_s)\ldots$
            is contained in the axis of $T_c$.
    \end{enumerate}

    To see this, notice that first if $f^k_\sharp(u_s) = u_s$ for some $k > 0$,
    the fact that $f(E_s) = E_s\cdot u_s$ is a splitting implies that
    the element $c$ of $\pi_1(G_{s-1},v)$ determined by the loop 
    $u_s f_\sharp(u_s)\ldots f^{k-1}_\sharp(u_s)$ is nonperipheral,
    and the ray $\tilde u_s h_\sharp(\tilde u_s)h^2_\sharp(\tilde u_s)\ldots$
    is contained in the axis of $T_c$.
    Since $h_\sharp$ preserves the axis of $T_c$,
    $h$ commutes with $T$.
    
    Conversely, suppose that there is a nonperipheral element $c \in F$
    such that $T_c$ preserves $\Gamma_{s-1}$ and $T_c|_{\Gamma_{s-1}}$ commutes with $h$.
    We have $T_c([\tilde x,h(\tilde x)]) = [T_c(\tilde x),h(T_c(\tilde x))]$ for all $\tilde x$,
    so in particular $\tilde R_{T_c(\tilde v)} = T_c(\tilde R_{\tilde v})$.
    Thus $\tilde R_{\tilde v}$ and $T_c(\tilde R_{\tilde v})$ have an infinite end in common,
    and therefore $\tilde R_{\tilde v}$ and the axis of $T_c$ have an infinite end in common.
    It follows that $h^k(\tilde v)$ is contained in the axis of $T_c$ for all sufficiently large $k$.
    This implies that there is a uniform bound to the length of $[h^k(\tilde v),h^{k+1}(\tilde v)]$
    and thus that $f^k_\sharp(u_s)$ takes on only finitely many values
    up to multiplying by vertex group elements at the ends.
    These vertex group elements themselves take on only finitely many values,
    so by replacing $\tilde v$ with some $h^k(\tilde v)$ if necessary,
    we may assume that $u_s$ is a periodic Nielsen path
    and that $\tilde R_{\tilde v}$ is contained in the axis of $T_c$.
    This verifies item 6.

    We now turn to the proof of item 2,
    for which we follow the argument in \cite[Lemma 5.5.1]{BestvinaFeighnHandel}.
    Note first that by \Cref{negsplitting},
    if $\sigma$ is a path or circuit with endpoints at vertices and height $s$,
    then $\sigma$ has a splitting into subpaths in $G_{s-1}$ and paths of the form
    $E_s\gamma$, $E_s\gamma\bar E_s$ or $\gamma \bar E_s$, where $\gamma$ is a path in $G_{s-1}$.
    Therefore to complete the proof of item 2, 
    we will show that the following hold, assuming $\gamma$ is a nontrivial tight path in $G_{s-1}$.
    \begin{enumerate}[label=(\roman*)]
        \item If $E_s\gamma$ (respectively $E_s\gamma\bar E_s$)
            can be split at a point in the interior of the copy of $E_s$,
            then $f^m_\sharp(E_s\gamma)= E_s\cdot\gamma_1$
            (respectively $f^m_\sharp(E_s\gamma\bar  E_s) = E_s\cdot \gamma_1\bar E_s$)
            for some $m \ge 0$ and tight  path  $\gamma_1$ in $G_{s-1}$.
        \item If $E_s\gamma$ has no splittings, then  $f^m_\sharp(E_s\gamma)$
            is an exceptional path of height $s$ for some $m \ge 0$.
        \item If $E_s\gamma \bar E_s$  has no splittings,
            then $E_s\gamma\bar E_s$ is an exceptional path of height  $s$.
    \end{enumerate}

    Item 5 implies that for any nontrivial initial segment $\sigma_1$ of $E_s$,
    some $f^m_\sharp(\sigma_1) = E_s\cdot\gamma'$, where $\gamma'$ is contained in $G_{s-1}$.
    Thus if a path $\sigma$ of height $s$ splits as $\sigma = \sigma_1\cdot\sigma_2$
    where $\sigma_1$ is a nontrivial initial segment of $E_i$,
    then some $f^m_\sharp(\sigma)$ has a splitting of the form  $E_i\cdot\sigma'$.
    Item (i) follows.

    Therefore write $\sigma$ for a path of the  form $E_s\gamma$ or $E_s\gamma\bar E_s$
    and suppose that $\sigma$ has no splittings.
    \paragraph{Step 1: Cancelling large middle segments.}
    Suppose that $\sigma = \sigma_1'\sigma_2'\sigma_3'$ 
    is a decomposition of $\sigma$ into nontrivial subpaths.
    We will show that there exists $M > 0$ and a decomposition $\sigma = \sigma_1\sigma_2\sigma_3$
    where $\sigma_1$ is an initial subpath of $\sigma'_1$ and
    $\sigma_3$ is a terminal subpath of $\sigma'_3$
    such that $f^M_\sharp(\sigma) = f^M_\sharp(\sigma_1)f^M_\sharp(\sigma_3)$
    and the indicated juncture is a vertex.
    By making $\sigma'_1$ and $\sigma'_3$ smaller,
    we may assume that they are contained in the initial and terminal edges of $\sigma$ respectively.

    Choose lifts $\tilde  f$ and $\tilde\sigma = \tilde\sigma'_1\tilde\sigma'_2\tilde\sigma'_3$.
    Observe that the set
    \[  \tilde S_k = \{\tilde x \in \tilde\sigma 
    : \tilde f^k(\tilde x) \in \tilde f^k_\sharp(\tilde\sigma)\} \]
    is closed and that $\tilde\sigma$ can be split at any point of $\bigcap_{k=1}^\infty \tilde S_k$.
    Since there are no splittings, this infinite intersection contains 
    only the endpoints of $\tilde\sigma$,
    and thus there exists $M > 0$ such that 
    $\bigcap_{k=1}^M  \tilde S_k \subset \tilde\sigma'_1 \cup \tilde\sigma'_3$.
    One can argue by induction that $\tilde f^N$ maps $\bigcap_{k=1}^N \tilde S_k$
    onto $\tilde f^N_\sharp(\tilde\sigma)$ for all $N \ge 1$.
    Since the lift of $\tilde E_s$ that is the initial edge of $\tilde f^k(\tilde\sigma)$
    is not canceled when $\tilde f^k(\tilde\sigma)$ is tightened,
    each $\tilde S_k$ and thus $\bigcap_{k=1}^M  \tilde S_k$ 
    contains an initial segment of a lift of $E_s$.
    Therefore choose a point $\tilde x$ 
    in the intersection of $\bigcap_{k=1}^M \tilde S_k$ with $\tilde\sigma'_1$
    with the property that $\tilde f^M(\tilde x)$ 
    is as close to the terminal end of $\tilde f^M_\sharp(\sigma)$ as possible.
    This condition guarantees that $\tilde f^M(\tilde x)$ is a vertex;
    let $\tilde\sigma_1$ be the initial segment of the lift of $E_s$ terminating at $\tilde x$.
    We have that $\tilde  f^M_\sharp(\tilde\sigma_1)$ 
    is a proper initial subinterval of $\tilde f^M_\sharp(\tilde\sigma)$,
    for if not, then $f^M_\sharp(\sigma) = f^i\sharp(E_sw_s)$ for some $i$
    and some initial segment $w_s$ of $u_s$.
    This contradicts item 5 above and the assumption that $\sigma$ has no splittings.
    Finally observe that there are points of $\bigcap_{k=1}^M\tilde S_k$ in $\tilde\sigma'_3$
    that map arbitrarily close to $\tilde f^M(\tilde x)$,
    and since $\bigcap_{k=1}^M\tilde S_k$ is closed
    there exists $\tilde y$ in the intersection of $\bigcap_{k=1}^M\tilde S_k$ and $\tilde\sigma'_3$
    such that $\tilde f^M(\tilde y) = \tilde f^M(\tilde x)$.
    The subdivision of $\tilde\sigma$ at $\tilde x$ and $\tilde y$
    provides the decomposition $\sigma = \sigma_1\sigma_2\sigma_3$ required.
    
    \paragraph{}
    Suppose for a moment that $\sigma = E_i\gamma$.
    An immediate consequence of Step 1 above is that the last edge of $\sigma$
    is not eventually mapped to an almost periodic stratum.
    By replacing $\sigma$ by some $f^i_\sharp(\sigma)$ if necessary,
    then we may assume that the last edge  of $f^k_\sharp(\sigma)$ is contained in the same stratum
    for all $k \ge 0$.
    Since $f|_{G_{s-1}}$ is a CT, we have that one of the following conditions is satisfied.
    \begin{enumerate}[resume,label=(\roman*)]
        \item The final edge of $\sigma$ is a non-exponentially growing but not almost periodic edge $E_j$.
        \item The final edge of $\sigma$ is contained in an  exponentially growing stratum $H_r$.
    \end{enumerate}
    If $\sigma = E_s\gamma\bar E_s$, then Item (iv) holds with $j  = s$ without replacing $\sigma$.
    We suppose at first that Item (iv) holds.

    \paragraph{Step 2: At least three blocks cancel.}
    Write $\sigma = E_s\gamma'\bar E_j$, where $\gamma = \gamma'$ if $j = s$.
    Define the ray $R_s$ to be the infinite path $u_s\cdot f_\sharp(u_s)\cdot f^2_\sharp(u_s)\cdots$
    and let $R^m_s$ be the initial segment $u_s\cdot f_\sharp(u_s)\cdots f_\sharp^{m-1}(u_s)$.
    We refer to the paths $f^k_\sharp(u_s)$ as \emph{blocks} of $R_s$.
    By \Cref{negctsplitting}, we have may define the ray $R_j$ and $R^m_j$ analogously
    with $u_j$ replacing $u_s$.
    We have that $f^m(\sigma) = gE_sR^m_s f^m(\gamma')\bar R^m_j\bar E_jh$,
    and we claim that if $m$ is sufficiently large,
    then a subpath of $R^m_i$ containing at least three blocks of $R_i$
    cancels with a subpath of $\bar R^m_j$ containing at least three blocks of $\bar R_j$
    when $gE_iR^m_if^m(\gamma')\bar R^m_j\bar E_jh$ is tightened to $f^m_\sharp(\sigma)$.

    By Step 1, there exist a positive integer $M$ and initial subpaths $\mu_M$ of $R^M_s$
    and $\bar\nu_M$ of $\bar R^M_j$ such that
    $f^M_\sharp(\sigma) = g'E_i\mu_M\bar\nu_M\bar E_jh'$.
    Since we have $g'E_s\mu_M = f^i_\sharp(E_sw_s)$ for some initial segment $w_s$ of $u_s$,
    Item 5 above implies that $g'E_s\mu_M = g'E_s\cdot \mu_M$
    and similarly $\bar\nu_M \bar E_jh' = \bar\nu_M \cdot \bar E_jh'$.
    Thus $f^m_\sharp(\sigma)$ is obtained from the concatenation of $g''E_iR_i^{m-M}f_\sharp^{m-M}(\mu_m)$
    and $f_\sharp^{m-M}(\bar\nu_M)\bar R^{m-M}_j\bar E_jh''$ by cancelling at the juncture.
    Step 1 implies that for $m$ sufficiently large,
    long cancellation must occur in both $R^{m-M}_s$ and $\bar R^{m-M}_j$.
    The only way this can happen is if long segments of $R^{m-M}_s$
    and $R^{m-M}_j$ cancel with each other, verifying the claim.

    \paragraph{Step 3: When $H_\ell$ is non-exponentially growing.}
    The third step is the following claim.
    \begin{claim}[cf.~Lemma 5.5.1, Sublemma 1 of \cite{BestvinaFeighnHandel}]
        \label{neginductionclaim}
        If $E_s$ and $E_j$ (with $s \ge j$) have distinct lifts $\tilde E_s$ and $\tilde E_j$
        whose corresponding rays $\tilde R_s$ and $\tilde R_j$
        have a subpath in common that contains at least three blocks in each ray
        and if $H_\ell$ is non-exponentially growing,
        then the path $\tilde \delta$ that connects the initial endpoint of $\tilde E_i$
        to the terminal endpoint of $\tilde E_j$ projects to an exceptional path of height $s$.
    \end{claim}

    \begin{proof}
        We follow the proof of \cite[Lemma 5.5.1, Sublemma 1]{BestvinaFeighnHandel}.
        The set-up is a variant of that in the paragraph
        ``Restricting to $G_{s-1}$ for non-exponentially growing strata.''
        Let $\Gamma_{s-1}$ be the component of the full preimage of $G_{s-1}$
        that contains $\tilde u_s$ and $\tilde u_j$,
        and write $h_s\colon \Gamma_{s-1} \to \Gamma_{s-1}$
        and $h_j\colon \Gamma_{s-1} \to \Gamma_{s-1}$
        for the restricted lifts of $f$ fixing the initial endpoints 
        of $\tilde E_s$ and $\tilde E_j$ respectively.

        Suppose first that $h_s = h_j$.
        By the argument at the beginning of this lemma,
        we have that $h_s$ has no fixed points,
        so in particular the initial endpoint of $\tilde E_j$ does not belong to $\Gamma_{s-1}$
        and we conclude $E_s = E_j$.
        The path $\tilde\delta$ projects to an almost Nielsen path for $f$.
        There is a loop $\gamma \in G_{s-1}$ lifting to the tight path in $\Gamma_{s-1}$
        from the terminal vertex of $\tilde E_s$ to the terminal vertex of $\tilde E_j$
        such that the resulting automorphism $T_{[\gamma]}$ 
        of the natural projection $\Gamma_{s-1} \to G_{s-1}$
        takes $\tilde u_s$ to $\tilde u_j$.
        Indeed, if $\tilde v$ is the initial endpoint of $\tilde u_s$,
        we have $T_{[\gamma]}h_s(\tilde v) = h_sT_{[\gamma]}(\tilde v)$ 
        is the terminal endpoint of $\tilde u_j$.
        On the other hand, $h_sT_{[\gamma]} = T_{(h_s)_\sharp([\gamma])}h_s$,
        so we see that the element
        $[\gamma]^{-1}(h_s)_\sharp([\gamma]) = [\bar\gamma u_s f_\sharp(\gamma)\bar u_s]$
        of $\pi_1(G_{i-1},v)$ stabilizes $h(\tilde v)$,
        and is thus homotopic to a loop of the form $u_s g\bar u_s$
        for some $g \in \mathcal{G}_v$.
        Rewriting, we see that this implies that $f_\sharp(\gamma)$ is homotopic to $\bar u_s\gamma u_s g$.
        Because $E_s\gamma\bar E_s$ is an almost Nielsen path,
        we have $u_s\bar u_s\gamma u_s g \bar u_s$ is homotopic to $\gamma$,
        so we conclude that $g$ is the identity, and thus that $T_{[\gamma]}$ and $h$ commute.
        Note that if $T_{[\gamma]}$ were peripheral,
        its fixed point would be contained in $\Gamma_{s-1}$ and preserved by $h$,
        so since $h$ is fixed-point free, $T_{[\gamma]}$ is non-peripheral.
        By item 6, we have that $u_s$ is a periodic Nielsen path.
        If the initial vertex of $u_s$ is principal, then \Cref{rotationlessnielsenpaths}
        implies that $u_s$ has period one as an almost Nielsen path,
        so $E_s$ is almost linear and the discussion after the definition of almost linear edges
        implies that $E_s$ is linear, say with axis $w_s$.
        If the initial vertex of $u_s$ is not principal,
        then $u_s$ is contained in a dihedral pair and $E_s$ is a linear or dihedral linear edge,
        say with axis $w_s$.
        In this case both $\tilde R_s$ and $\tilde R_j$ are contained in the axis of $T_[\gamma]$
        and $\gamma$ itself takes the form $w_s^p$ for some $p$.
        Thus the projection $\delta$ of $\tilde\delta$ is exceptional.

        Thus we now assume that $h_s \ne h_j$.
        Let $\tilde V$ be the set of vertices in $\tilde R_s$ that are either
        the initial endpoint of a lift of $E_\ell$
        or the terminal endpoint of a lift of $\bar E_\ell$.
        Order the elements of $\tilde V$ so that $\tilde x_p < \tilde x_{p+1}$
        in the orientation on $\tilde R_s$.
        \Cref{negsplitting} implies that $h_s(\tilde x_p) = \tilde x_{p+n_0}$
        for all $p$ and fixed $n_0$.
        What's more, $h_s$ takes the direction of the lift of $E_\ell$
        to the direction of the lift of $E_\ell$.
        Define $\tilde W$ and $m_0$ using $\tilde R_j$ and $h_j$ instead of $\tilde R_s$ and $h_s$.
        The intersection of $\tilde V$ and $\tilde W$ contains by assumption
        at least $n_0 + m_0 + 1$ consecutive elements $\tilde z_0,\ldots,\tilde z_{n_0+m_0}$.
        We have $h_sh_j(\tilde z_0) = h_s(\tilde z_{m_0}) 
        = \tilde z_{n_0+m_0} = h_j(\tilde z_{n_0}) = h_jh_s(\tilde z_0)$,
        and the two maps agree on the direction of $E_\ell$ at $\tilde z_0$,
        so we conclude that as lifts of $f^2$, we have $h_sh_j = h_jh_s$.
        There is a nontrivial element $c \in F$ such that $T_c$ preserves $\Gamma_{s-1}$
        and $T_c h_s = h_j$,
        and there is an element $d \in F$ such that $T_dh_j h_sT_c$.
        We have $h_ih_j = h_i T_c h_s = T_d h_jh_i$, so we conclude that $d = 1$
        and that $h_s$ commutes with $T_c$.
        A symmetric  argument shows that $T_c$ commutes with $h_j$.
        The argument in the case $h_s = h_j$ shows that $T_c$ must be non-peripheral.
        Item 6 above then implies that $u_s$ and $u_j$ are periodic Nielsen paths
        and that $\tilde R_s$ and $\tilde R_j$ are contained in the axis of $T_c$.
        If the initial vertex of $u_s$ is principal, then so is the initial vertex of $u_j$,
        and we conclude that $E_s$ and $E_j$ are linear edges with the same axis $w$.
        If the initial vertex of $u_s$ is not principal, then $u_s$ and hence $u_j$
        are contained in a dihedral pair and $E_s$ and $E_j$ are both linear or both dihedral linear edges
        with the same axis $w$.
        The segment of the axis of $T_c$ that separates the terminal endpoint of $\tilde E_s$
        from the terminal endpoint of $\tilde E_j$ projects to $w^p$ for some integer $p$.
        Supposing  $E_s$ and $E_j$ are linear, since $\tilde R_s$ and  $\tilde R_j$
        have a common subpath containing blocks in both rays,
        we have $u_s = w^{d_s}$ and $u_j = w^{d_j}$, where $d_s$ and $d_j$ have the same sign,
        and we conclude that $\delta = E_sw^p\bar E_j$ is exceptional.
        If $E_s$ and $E_j$ are dihedral linear and $w = (\sigma\tau)$,
        again because $\tilde R_s$ and $\tilde R_j$ share blocks in both rays,
        we have that $u_s$ is homotopic to $(\sigma\tau)^{d_s}\sigma$
        and $u_j$ is homotopic  to $(\sigma\tau)^{d_j}\sigma$, where $d_s$ and $d_j$ have the same sign,
        and we again conclude that  $\delta$ is exceptional.
    \end{proof}

    \paragraph{Step  4: When $H_\ell$ is exponentially growing.}
    Suppose that $m$ is chosen as in Step 2
    and that $H_{\ell}$ is exponentially growing.
    We argue that $f^m_\sharp(\sigma)$ is an exceptional path of height $s$.
    Write $h_s\colon \Gamma_{s-1} \to \Gamma_{s-1}$ and $h_j\colon \Gamma_{s-1}   \to \Gamma_{s-1}$
    as in  Step 3.

    If the splitting of $u_s$ given by \Cref{egsplitting}
    contains a path  $\rho_\ell$ or $\bar\rho_\ell$ in $P_\ell$,
    then the argument in Step 3 goes through with $\tilde V$
    defined to be the set of lifts of $\rho_\ell$ or $\bar \rho_\ell$.
    These are ``translated'' by $h_s$ and $h_j$
    and their initial and terminal directions are mapped to each other,
    so the argument in that step works without further change.
    
    It remains to rule out the possibility that $u_s$ is $\ell$-legal,
    i.e.~that there are no lifts of $\rho_\ell$ or $\bar\rho_\ell$.
    Since blocks of $\tilde R_s$ cancel with segments of $\tilde R_j$,
    we conclude that $u_j$ is also $\ell$-legal.
    By Step 2, we have that $f^m_\sharp(\sigma) = gE_s\mu_m\bar\nu_m\bar E_jh$,
    where $\mu_m$ and $\bar\nu_m$
    are initial subpaths of $R_s$ and $\bar R_j$
    and so are $\ell$-legal.
    \Cref{kprotectedsplitting} implies that $\mu_m$ and $\bar\nu_m$
    take on only finitely many values as $m$ varies,
    so some $f^m_\sharp(\sigma)$ is a periodic Nielsen path, 
    say $f^{m+p}_\sharp(\sigma) = f^m_\sharp(\sigma)$.
    There is a lift $\tilde f\colon\Gamma \to \Gamma$
    whose restriction to $\Gamma_{s-1}$ equals $h_s$;
    we have that $\tilde f$ fixes the initial endpoint of $\tilde E_i$
    and $\tilde f^p$ fixes the initial endpoint $\tilde w$ of $\tilde E_j$.

    If $E_s \ne E_j$, then $\tilde w \in \Gamma_{s-1}$,
    and the fact that $\fix(h_s) = \varnothing$
    implies that $p > 1$.
    Write $\tilde\gamma$ for the path that connects $\tilde w$  to $h_s(\tilde w)$.
    There is a projection $\gamma$ of $\tilde\gamma$ that is a periodic Nielsen path
    with a principal endpoint and hence a Nielsen path.
    We have that $[\gamma^p] = [\gamma f(\gamma)\ldots f^{p-1}(\gamma)]$
    lifts to the trivial path $[\tilde\gamma h_s(\tilde\gamma)\ldots h^{p-1}_s(\tilde\gamma)]$,
    so the closed path $\gamma$ represents a peripheral element of $\pi_1(G_{s-1},w)$,
    so we may write $\gamma = \gamma' g\bar\gamma'$ for some vertex group element $g$
    and path $\gamma'$ in $G_{s-1}$.
    Since $f_\sharp(\gamma) = f_\sharp(\gamma')f_\sharp(g)f_\sharp(\bar\gamma') = \gamma' g\bar\gamma$,
    we conclude that $\gamma'$ is an almost Nielsen path.
    In the notation of \Cref{basicssection},
    with respect to the basepoint $w$,
    the map $h_s$ corresponds to the path $\gamma$,
    i.e.~if $\tilde x \in \Gamma_{s-1}$ corresponds to the homotopy class of the path $\tau$
    starting at $w$, the point $h_s(\tilde x)$ corresponds to the homotopy class of the path
    $\gamma f(\tau)$.
    But if we let $\tilde x$ be the point corresponding to the path $\gamma'$,
    we see that $\tilde x$ is fixed by $h$ since $\gamma'$ is an almost Nielsen path.
    This  contradiction implies that $E_s = E_j$.
    The automorphism $T$ of the natural projection
    that carries the direction of $\tilde E_s$ to the direction of $\tilde E_j$
    commutes with $\tilde f^p$,
    so the restriction of $T$ to $\Gamma_{s-1}$ therefore commutes with $h^p_s$.
    But since $h^p_j = (T|_{\Gamma_{s-1}})h^p_s (T|_{\Gamma_{s-1}})^{-1}$,
    we conclude $h^p_s = h^p_j$.
    Therefore the ray $\tilde R_s$ and $T(\tilde R_s) = \tilde R_j$ have an infinite end in common,
    and therefore $T$ is non-peripheral 
    and $\tilde R_s$ and the axis of $T$ have an infinite end in common.
    We conclude by Item 6 that $E_s$ is a linear edge,
    in contradiction to the assumption that $u_s$ is $\ell$-legal.

    \paragraph{Step 5: Items (ii) and (iii) hold when (iv) does.}
    Suppose Item (iv) holds, so that if we choose $m$ as in Step 2,
    $f^m_\sharp(\sigma)$ begins with $E_s$ and ends with some 
    non-exponentially growing edge $\bar E_j$.
    Choose a lift of $f^m_\sharp(\sigma)$,
    and let $\tilde R_s$ and $\tilde R_j$ be the lifts of the  rays $R_s$ and $R_j$
    that begin at the terminal endpoints of $\tilde E_s$ and $\tilde E_j$ respectively.
    We have chosen $m$ such that $\tilde R_i \cap \tilde R_j$ contains at least three blocks in each ray.
    By Steps 3 and 4, we have that $f^m_\sharp(\sigma)$ is an exceptional path of height $s$.
    If $s = j$, then $f^m_\sharp(\tilde\sigma)$ is fixed by $f_\sharp$.
    Since $\sigma$ and $f^m_\sharp(\sigma)$ have the same endpoints and the same image under $f^m_\sharp$,
    we conclude that $\sigma = f^m_\sharp(\sigma)$ is an exceptional path of height $s$.

    \paragraph{Step 6: Item (v) does not occur.}
    Suppose that the final edge of $\sigma$ belongs to an exponentially growing stratum $H_r$.
    Step 1 implies that for all sufficiently large $m$,
    $f^m_\sharp(E_s\gamma) = gE_s\mu_m\nu_m$,
    where $\mu_m$ belongs to $R_s$ and $\nu_m$ is $r$-legal.
    Furthermore, Step 1 implies that if $(M-m)$ is sufficiently large,
    then edges of $H_r$ in $f^{M-m}_\sharp(\nu_m)$ cancel with edges of $f^{M-m}_\sharp(\mu_m)$,
    so $r \le \ell$,
    and in fact a symmetric argument shows that $\ell = r$ and that $\mu_m$ is $\ell$-legal.
    \Cref{kprotectedsplitting} implies that $\mu_m$ takes on only finitely many values,
    and the same is true for $\nu_m$ up to multiplication by vertex group elements at the end:
    the argument is essentially the same as that given in \Cref{finitelymanynielsenpaths}.
    Therefore we conclude that $f^m_\sharp(\sigma)$ is a periodic almost Nielsen path.
    But the argument in the last paragraph of Step 4 shows that this is impossible.
    Therefore we conclude that Item (v) does not occur, completing the proof.
\end{proof}

We are ready to state and prove \Cref{CTtheorem}.
\begin{thm}
    \label{preciseCTtheorem}
    Suppose $\varphi \in \out(F,\mathscr{A})$ is rotationless
    and $\mathscr{C}$ is a nested sequence of $\varphi$-invariant free factor systems.
    Then $\varphi$ is represented by a CT
    $f\colon \mathcal{G} \to \mathcal{G}$
    and filtration $\varnothing = G_0 \subset G_1 \subset \cdots \subset G_m = G$
    that realizes $\mathscr{C}$.
\end{thm}

\begin{proof}
    We follow the outline of \cite[Theorem 4.28]{FeighnHandel}.
    We may assume that $\mathscr{C}$ is a maximal nested sequence with respect to $\sqsubset$,
    so that any filtration that realizes $\mathscr{C}$ is reduced.
    We claim that by \Cref{improvedrelativetraintrack} we may choose a relative train track map
    $f\colon \mathcal{G} \to \mathcal{G}$ with the following properties.
    \begin{enumerate}
        \item The topological representative $f$ represents $\varphi \in \out(F,\mathscr{A})$.
        \item The filtration on $f\colon \mathcal{G} \to \mathcal{G}$
            realizes $\mathscr{C}$.
        \item Each contractible component of a filtration element is a union of zero strata.
        \item The endpoints of all indivisible almost Nielsen paths of exponentially growing height
            are vertices.
    \end{enumerate}
    Items 2 and 4 follow from \Cref{improvedrelativetraintrack}.
    To see that item 3 does as well, suppose that $C$ is a contractible component of some $G_i$.
    If $C$ contains an edge in an irreducible stratum,
    observe that it is non-wandering
    and by property \hyperlink{NEG}{(NEG)} the lowest stratum $H_j$
    that has an edge in $C$ is either almost periodic or exponentially growing.
    It cannot be almost periodic by \Cref{propertyPconsequence}
    and cannot be exponentially growing by \Cref{egvalence}.
    Therefore every edge in $C$ is contained in a zero stratum.
    There is no loss in dividing up zero strata so that item 3 holds.
    (Notice as well that item 3 implies that contractible components of filtration elements
    are wandering, and thus contain no vertices with nontrivial vertex group.)

    We will assume that all relative train track maps in this proof satisfy the above properties.

    \paragraph{Step 1: (EG Almost Nielsen Paths).}
    Let $N(f)$ be the number of indivisible almost Nielsen paths of exponentially growing height.
    We adapt Feighn--Handel's algorithm for proving 
    \hyperlink{EGAlmostNielsenPaths}{(EG Almost Nielsen Paths)}:
    suppose some exponentially growing stratum does not satisfy 
    \hyperlink{EGAlmostNielsenPaths}{(EG Almost Nielsen Paths)}
    and let $H_r$ be the highest such stratum.
    We will show in \Cref{properfoldsegalmostnielsenpaths}
    that this implies that there is a sequence of proper folds
    at indivisible Nielsen paths of height $r$ leading to a relative train track map
    with either a partial fold or an improper fold.
    If the fold is partial, then $N(f)$ may be decreased (\Cref{partialfoldfewernielsenpaths})
    Since $N(f)$ is finite,
    eventually we may assume that all folds are full.
    We will show that improper folds allow us to decrease the number of edges in $H_r$
    (\Cref{improperfoldfeweredges}),
    while proper folds preserve the number of edges of $H_r$ (\Cref{properfoldsamenumberofedges}).
    Both kinds of folds preserve the number of edges 
    of each exponentially growing stratum $H_s$ for $s > r$
    (\Cref{improperfoldfeweredges} and \Cref{properfoldsamenumberofedges}).
    Since this number is finite, eventually no improper folds occur,
    at which point $H_r$ satisfies \hyperlink{EGAlmostNielsenPaths}{(EG Almost Nielsen Paths)}
    by \Cref{properfoldsegalmostnielsenpaths}.

    \begin{lem}[cf.~Lemma 5.3.6 of \cite{BestvinaFeighnHandel}]
        \label{properfoldsegalmostnielsenpaths}
        Given a relative train track map $f\colon \mathcal{G} \to \mathcal{G}$,
        an exponentially growing stratum $H_r$,
        and an indivisible almost Nielsen path $\rho$ of height $r$,
        suppose that the fold at the illegal turn of 
        each indivisible almost Nielsen path
        obtained by iteratively folding $\rho$ is proper.
        Then $H_r$ satisfies the conclusions of 
        \hyperlink{EGAlmostNielsenPaths}{(EG Almost Nielsen Paths)}.
    \end{lem}

    In fact, by \Cref{egalmostnielsenpathsproperfold},
    \hyperlink{EGAlmostNielsenPaths}{(EG Almost Nielsen Paths)}
    is equivalent to the statement that
    the fold at the illegal turn of each indivisible almost Nielsen path
    obtained by iteratively folding $\rho$ is proper.

    \begin{proof}
        We follow the notation in the definition of folding the indivisible almost Nielsen path $\rho$.
        The proof is essentially identical to \cite[Lemma 5.3.6]{BestvinaFeighnHandel}.
        To ease notation, assume that $G = G_r$.
        Write $\rho'$ for the indivisible almost Nielsen path of height $r$
        for $f'\colon \mathcal{G}' \to \mathcal{G}'$ determined by $\rho$,
        and let $\rho' = \alpha'\beta'$ be a decomposition into $r$-legal subpaths.
        Let $g'_1E'_1$ and $g'_2E'_2$ be the initial edges and vertex group elements of
        $\bar\alpha'$ and $\beta'$ respectively
        and let $F'\colon \mathcal{G}' \to \mathcal{G}''$
        be the extended fold determined by $\rho'$.
        We assume (up to reversing the orientation of $\rho'$)
        that $f'(g'E'_1)$ is a proper subpath of $f'(g'_2E'_2)$,
        and we write $\alpha' = g'_1E'_1 b'g'_3E'_3\ldots$
        where $b'$ is a maximal (possibly trivial) subpath in $G'_{r-1}$ 
        ending with trivial vertex group element and $E'_3$ is an edge of $H'_r$.
        Recall we have a map $g\colon \mathcal{G}' \to \mathcal{G}$
        such that $f = gF$, where $F\colon \mathcal{G} \to \mathcal{G}'$ 
        is the extended fold determined by $\rho$.
        Write $E'_2$ as a concatenation of subintervals $\mu'_1\mu'_2\mu'_3$
        such that $f'(g'_2\mu'_1) = f'(g'_1E'_1)$
        and $f'(\mu'_2) = f'_\sharp(b')$.
        (If $b'$ is trivial, then $\mu'_2$ is trivial.)
        We will show that if $g(E'_1)$ and $g(E'_2)$
        have a nontrivial common initial segment,
        then $g(g'_2\mu'_1) = g(g'_1E'_1)$
        and $g'(\mu'_2) = g'_\sharp(b')$.
        In other words, either the fold $F'$ is a generalized fold factor of $g$,
        or $g$ cannot be folded at the turn $\{(g'_1,E'_1),(g'_2,E'_2)\}$.

        Suppose first that $g(g'_1E'_1)$ and $g(g'_2E'_2)$ have a common initial segment
        but that the fold for $g$ is not full,
        i.e.~that the maximal $g$-fold of $g'_1E'_1$ and $g'_2E'_2$ does not use all of $\mu'_1$.
        Write $\hat F\colon \mathcal{G}' \to \hat{\mathcal{G}}$ for the maximal $g$-fold
        of $g'_1E'_1$ and $g'_2E'_2$,
        let $\hat g\colon \hat{\mathcal{G}} \to \mathcal{G}$ be the induced map satisfying
        $\hat g\hat F = g$ and let $\hat G_i = \hat F(G'_i)$ for $1 \le i \le r$.
        Since the fold is not full, $\hat H_r = \hat F(H'_r)$
        has one more edge than $H'_r$ and $H_r$.

        Since the fold $\hat F$ is maximal, the map $\hat g$ cannot be folded
        at the newly created vertex.
        By \Cref{egalmostnielsenpathsproperties},
        $\{(g'_1,E'_1),(g'_2,E'_2)\}$ is the only orbit of illegal turns for $f'$
        that involves an edge in $H'_r$.
        Note that every fold factor of $g$ is a fold factor of $f' = Fg$,
        so we conclude that $F'$ is the only fold for $g$ that involves an edge of $H'_r$.
        Therefore all folds for $\hat g$ have both edges in $\hat G_{r-1}$.
        We claim that folding edges in $\hat G_{r-1}$ according to $\hat g|_{G_{r-1}}$
        will not identify previously distinct vertices in $\hat H_r \cap \hat G_{r-1}$.
        To see this, observe that if $\hat\gamma$ is a nontrivial tight path in $\hat G_{r-1}$
        with endpoints in $\hat H_r \cap \hat G_{r-1} = H'_r \cap G'_{r-1} = H_r \cap G_{r-1}$,
        then there exists a nontrivial tight path $\gamma$ in $G_r$
        with the same endpoints satisfying $\hat F_\sharp F_\sharp(\gamma) = \hat\gamma$.
        It follows that $\hat g_\sharp(\hat\gamma) = f_\sharp(\gamma)$ is nontrivial.
        (Compare the argument in \Cref{trivialedgegroupsEG2}.)
        It follows that no new illegal turns involving edges of $\hat H_r$
        are created by folding according to $\hat g$.
        Continue to fold edges in $\hat G_{r-1}$ according to $\hat g|_{\hat G_{r-1}}$
        until no more folds are possible,
        and call the resulting composition of folds $f_{r-1}\colon \hat{\mathcal{G}} \to \mathcal{G}^*$.
        (By \cite{Dunwoody}, only finitely many folds are possible.)
        There is an induced immersion $\theta\colon \mathcal{G}^* \to \mathcal{G}$
        satisfying $f = \theta f_{r-1}\hat F F$.
        Since $\mathcal{G}$ has no valence-one vertices with trivial vertex group,
        $\theta$ is an isomorphism of graphs of groups.
        Observe that $\theta(f_{r-1}(\hat G_{r-1})) = f(G_{r-1}) = G_{r-1}$,
        so we have $\theta(f_{r-1}(\hat H_r)) = H_r$.
        This is a contradiction, as $\theta f_{r-1}|_{\hat H_r}$ induces a bijection on the set of edges,
        but $\hat H_r$ has more edges than $H_r$.
        Therefore if $g$ can be folded at the turn $\{(g'_1,E'_1),(g'_2,E'_2)\}$,
        then all of $E'_1$ can be folded with a proper initial segment of $E'_2$.

        We turn now to completing the proof of the claim.
        Let $E$ be the first edge of $f_\sharp(\beta)$
        that is not part of the maximum common initial segment of $f_\sharp(\bar\alpha)$
        and $f_\sharp(\beta)$.
        Then the initial edge $E'$ of $F(E) \subset \mathcal{G}'$
        is the first edge of $g_\sharp(\beta')$ that is not part of the maximum common initial segment
        of $g_\sharp(\bar\alpha')$ and $g_\sharp(\beta')$.
        Since $E$ is contained in $H_r$, the edges $E'$ is contained in $H'_r$.
        In particular, $E'$ is not contained in $b'$, verifying the claim.

        If $g$ cannot be folded at the turn $\{(g'_1,E'_1),(g'_2,E'_2)\}$,
        then define $f_r = F$ and construct $f_{r-1}$ and $\theta$ exactly as above.
        Otherwise let $F'\colon \mathcal{G}' \to \mathcal{G}''$ be the extended fold of $\rho'$
        with respect to $f'\colon \mathcal{G}' \to \mathcal{G}'$
        and let $g'\colon \mathcal{G}'' \to \mathcal{G}$ be the induced map
        satisfying $g'F = g$.
        If $g'$ cannot be folded at the illegal turn of $\rho''$,
        then define $f_r = F'F$ and construct $f_{r-1}$ and $\theta$ as above.
        Otherwise, repeat the argument above to conclude that the extended fold of $\rho''$
        is a fold factor of $g'$.
        We may continue this argument.
        Since there are finitely many folds in any factorization of $f\colon \mathcal{G} \to \mathcal{G}$,
        this process terminates, proving the lemma.
    \end{proof}

    \begin{lem}[cf.~Lemma 5.1.7 of \cite{BestvinaFeighnHandel}]
        \label{egalmostnielsenpathsdistinctendpoints}
        Suppose that $f\colon \mathcal{G} \to \mathcal{G}$ is a relative train track map
        satisfying our standing assumptions.
        Suppose that $H_r$ is an exponentially growing stratum and that $\rho$ is an almost Nielsen path
        of height $r$ that crosses some edge $E$ of $H_r$ exactly once
        and that the first and last edges of $\rho$ are contained in $H_r$.
        Then the following hold.
        \begin{enumerate}
            \item The endpoints of $\rho$ are distinct.
            \item At most one endpoint of $\rho$ has nontrivial vertex group.
            \item If both endpoints are contained in $G_{r-1}$,
                then at least one endpoint is contained in a contractible component of $G_{r-1}$.
            \item If one endpoint has nontrivial vertex group
                and the other is contained in $G_{r-1}$,
                it is contained in a contractible component of $G_{r-1}$.
        \end{enumerate}
    \end{lem}

    \begin{proof}
        We follow the argument in \cite[Lemma 5.1.7]{BestvinaFeighnHandel}.
        Let $\mathcal{G}'$ be the graph of groups obtained from $\mathcal{G}$
        by removing the edge $E$ and adding a new edge $E'$ with endpoints equal
        to the initial and terminal endpoints of $\rho$.
        Write $\rho = \alpha E\beta$.
        Define a homotopy equivalence $h\colon \mathcal{G} \to \mathcal{G}'$
        that is the identity on all edges other than $E$
        and that satisfies $h(E) = \bar\alpha E'\bar\beta$.

        Consider the free factor system $\mathcal{F}(G_{r-1}\cup E')$.
        It is $\varphi$-invariant.
        Clearly we have 
        $\mathcal{F}(G_{r-1}) \sqsubset \mathcal{F}(G_{r-1} \cup E') \sqsubset \mathcal{F}(G_r)$.
        Suppose that the conclusions of the lemma fail. Then one of the following holds.
        \begin{enumerate}
            \item The endpoints of $\rho$ are equal.
            \item The endpoints of $\rho$ are distinct and both have nontrivial vertex group.
            \item The endpoints of $\rho$ are distinct, exactly one has nontrivial vertex group
                and the other endpoint is contained in a noncontractible component of $G_{r-1}$.
            \item Both endpoints are contained in noncontractible components of $G_{r-1}$.
        \end{enumerate}
        In each case, we see that $\mathcal{F}(G_{r-1} \cup E')$ 
        is strictly larger than $\mathcal{F}(G_{r-1})$.
        But since $\rho$ is not $r$-legal,
        $\mathcal{F}(G_{r-1}\cup E')$ does not carry any line that is generic
        for the attracting lamination $\Lambda^+$ associated to $H_r$,
        (recall that $f$ is eg-aperiodic by \Cref{aperiodicallaperiodic})
        so $\mathcal{F}(G_{r-1}\cup E')$ is strictly smaller than $\mathcal{F}(G_r)$.
        This contradicts the assumption that the filtration for $f\colon \mathcal{G} \to \mathcal{G}$
        is reduced.
    \end{proof}

    \begin{lem}[cf.~Lemma 4.29 of \cite{FeighnHandel}]
        \label{partialfoldfewernielsenpaths}
        Suppose that $H_r$ is an exponentially growing stratum of a relative train track map
        $f\colon \mathcal{G} \to \mathcal{G}$
        and that $\rho$ is an indivisible almost Nielsen path of height $r$.
        If the fold at the illegal turn of $\rho$ is partial,
        then there is a relative train track map $f'' \colon \mathcal{G}'' \to \mathcal{G}''$
        satisfying $N(f'') < N(f)$.
    \end{lem}

    \begin{proof}
        We follow the arguments in \cite[Lemmas 5.2.3 and 5.2.4]{BestvinaFeighnHandel}.
        We may write $\rho = \alpha\beta$ as in \Cref{finitelymanynielsenpaths}.
        There exists a tight path $\tau$ in $G_r$ such that 
        $f_\sharp(\alpha) = g\alpha\tau$ and $f_\sharp(\bar\beta) = h\bar\beta\tau$
        for vertex group elements $g$ and $h$.
        Suppose that $\alpha$ is not a single edge of $G$.
        Then the final edge of $\alpha$ is entirely mapped into $\tau$ by $f$.
        The only way this edge is subdivided when folding $\rho$
        is if the first edge of $\beta$ maps entirely into $\tau$
        and in this case we have a full fold, in contradiction to our assumption.

        Therefore suppose $\alpha$ and $\beta$ are single edges.
        Notice that by the previous lemma, 
        the endpoints of $\rho$ are distinct and
        at most one endpoint of $\rho$ has nontrivial vertex group.
        Up to reversing the orientation of $\rho$, suppose it is the initial endpoint,
        so we have $f(\alpha) = g\alpha\tau$ and $f(\bar\beta) = \bar\beta\tau$
        for some vertex group element $g$.
        Change $f$ via a homotopy so that $f(\alpha) = \alpha\tau$, then
        let $\mathcal{G}'$ be the graph of groups obtained from $\mathcal{G}$ 
        by identifying $\alpha$ and $\bar\beta$ to a single edge $E'$.
        There is an induced map $f'\colon \mathcal{G}' \to \mathcal{G}'$;
        for each edge $E$ of $G$ we have $f'(q(E)) = q_\sharp(f(E))$.
        If there are any edges with trivial $f'$-image,
        they form a contractible forest; we collapse each component to a vertex.
        After tightening and collapsing finitely many times,
        we arrive at a topological representative still called $f'\colon \mathcal{G}' \to \mathcal{G}'$
        and a quotient map still called $q\colon \mathcal{G} \to \mathcal{G}'$.
        We have that an edge of $G$ is collapsed by $q$ if and only if
        some iterate of $f_\sharp$ maps it to $\rho$ or $\bar\rho$,
        so only edges in zero strata may be collapsed.
        Therefore $q\colon \mathcal{G} \to \mathcal{G}'$ is a homotopy equivalence.

        The map $f'\colon \mathcal{G}' \to \mathcal{G}'$ is a topological representative
        preserving the filtration whose elements $G'_i$ satisfy $G'_i = q(G_i)$,
        where if each component of a zero stratum $H_j$ is collapsed to a point,
        then $q(G_j)$ is not added to the filtration.
        Notice that $\mathcal{F}(G_i) = \mathcal{F}(G'_i)$,
        so the filtration is reduced and still realizes $\mathscr{C}$.

        If $E$ is an edge of $G_r$, then $q(E)$ is an edge of $G'$.
        We have $f'(q(E)) = qf(E)$: this is clear if $E$ in $G_{r-1}$
        and follows from the fact that an $r$-legal path in $G_r$
        cannot cross the illegal turn $(\bar\alpha,\beta)$ if $E \in H_r$.
        If $E$ is an edge of $G\setminus G_r$ that is not collapsed by $q$,
        then the edge path $f'q(E)$ is obtained from $qf(E)$ by cancelling edges in $H_r$.
        Thus the stratum $H'_i = q(H_i)$ has the same type as $H_i$.
        In particular if $H_i$ is exponentially growing,
        then so is $H'_i$ and the Perron--Frobenius eigenvalues are equal.

        Given a path $\sigma$ in $\mathcal{G}$ and $k > 0$,
        we have $(f')^k_\sharp(q_\sharp(\sigma)) = q_\sharp f^k_\sharp(\sigma)$.
        So if $\sigma$ is a periodic almost Nielsen path for $f$,
        then $\sigma' = q_\sharp(\sigma)$ is a periodic almost Nielsen path for $f'$
        and the period of $\sigma'$ is at most the period of $\sigma$.
        If $\sigma \ne \rho$, then $\sigma'$ is nontrivial.

        Conversely suppose that $\sigma' \subset {G}'_i$ is a periodic Nielsen path
        for $f'\colon \mathcal{G}' \to \mathcal{G}'$,
        we will construct a path $\sigma \subset G_i$ satisfying $q_\sharp(\sigma) = \sigma'$.
        If the endpoints of $\sigma'$ do not lie in $q(\alpha) = q(\bar\beta)$,
        then there is a unique path $\sigma$ satisfying $q_\sharp(\sigma) = \sigma'$.
        If an endpoint of $\sigma'$ lies in $q(\alpha) = q(\bar\beta)$
        but is not the initial endpoint of $q(\alpha) = q(\bar\beta)$,
        then there is a unique path $\sigma$ in $\mathcal{G}$
        that has periodic endpoints such that $q_\sharp(\sigma) = \sigma'$.
        If $\sigma'$ begins or ends at the initial endpoint of $q(\alpha) = q(\bar\beta)$,
        then there is a unique path $\sigma$ that does not begin or end with $\rho$
        or $\bar\rho$ that satisfies $q_\sharp(\sigma) = \sigma'$.
        In all cases we see that $\sigma$ is a periodic Nielsen path
        and that the period of $\sigma'$ equals the period of $\sigma$.
        What's more, if $\sigma'$ is indivisible, then $\sigma$ is indivisible.

        If $f'\colon \mathcal{G}' \to \mathcal{G}'$ is not a relative train track map,
        we may modify it to produce a relative train track map
        by using the moves ``(invariant) core subdivision'' 
        and ``collapsing inessential connecting paths''
        described in \cite[Lemmas 3.4 and 3.5]{Myself}
        or \cite[Lemmas 5.13 and 5.14]{BestvinaHandel}
        to restore properties \hyperlink{EG-i}{(EG-i)} and \hyperlink{EG-ii}{(EG-ii)}
        for each exponentially growing stratum of $f' \colon \mathcal{G}' \to \mathcal{G}'$.
        Since $\pf(f') = \pfmin$, \Cref{pfcorollary} 
        implies that \hyperlink{EG-iii}{(EG-iii)} is satisfied.

        We now have a relative train track map $f''\colon \mathcal{G}'' \to \mathcal{G}''$ and
        an identifying homotopy equivalence $q\colon \mathcal{G} \to \mathcal{G}''$.
        The graph of groups $\mathcal{G}''$ is obtained from $\mathcal{G}$
        by subdivision, folding and collapsing pretrivial forests
        (which, again, are contained in zero strata).
        The resulting relative train track map still realizes $\mathscr{C}$
        and thus its filtration is reduced.
        If $C''$ is a contractible component of some filtration element 
        for $f''\colon \mathcal{G}'' \to \mathcal{G}''$,
        then $q^{-1}(C'')$ is a contractible component of some filtration element
        and is thus a union of zero strata.
        It follows that $C''$ is also a union of zero strata.

        Collapsing inessential connecting paths,
        collapsing pretrivial forests
        and invariant core subdivision
        do not change the period of any indivisible periodic almost Nielsen path,
        nor do they alter the number of indivisible periodic almost Nielsen paths
        of exponentially growing height.
        It follows that $N(f'') = N(f') < N(f)$.
    \end{proof}

    \begin{lem}[cf.~Lemma 4.30 of \cite{FeighnHandel}]
        \label{properfoldsamenumberofedges}
        Suppose that $H_r$ is an exponentially growing stratum of a relative train track map
        $f\colon \mathcal{G} \to \mathcal{G}$,
        that $\rho$ is an indivisible almost Nielsen path of height $r$
        and that the fold at the illegal turn of $\rho$ is proper.
        Let $f'\colon \mathcal{G}' \to \mathcal{G}'$ be the relative train track map
        obtained from $f\colon \mathcal{G} \to \mathcal{G}$ by folding $\rho$.
        Then $N(f') = N(f)$ and there is a bijection $H_s \to H_s'$
        between the exponentially growing strata of $f$ and the exponentially growing strata of $f'$
        such that $H'_s$ and $H_s$ have the same number of edges for all $s$.
    \end{lem}
    
    \begin{proof}
        The proof that $f'$ is a relative train track map 
        satisfying our standing assumptions
        is contained in \Cref{properfoldrelativetraintrackmap}.
        It is clear from the definition of $f'\colon \mathcal{G}' \to \mathcal{G}'$
        and the proof of \Cref{properfoldrelativetraintrackmap}
        that the bijection between exponentially growing strata exists
        and that $H'_s$ and $H_s$ have the same number of edges for all $s$.
        The fact that $N(f') = N(f)$ follows as in the previous lemma.
    \end{proof}

    \begin{lem}[cf.~Lemma 4.31 of \cite{FeighnHandel}]
        \label{improperfoldfeweredges}
        Suppose that $H_r$ is an exponentially growing stratum of a relative train track map
        $f\colon \mathcal{G} \to \mathcal{G}$
        and that $\rho$ is an indivisible almost Nielsen path of height $r$.
        If the fold at the illegal turn of $\rho$ is improper,
        then there is a relative train track map $f''\colon \mathcal{G}'' \to \mathcal{G}''$
        and a bijection $H_s \to H''_s$ between the exponentially growing strata of $f$ and $f''$
        such that the following hold.
        \begin{enumerate}
            \item $N(f'') = N(f)$.
            \item $H''_r$ has fewer edges than $H_r$.
            \item If $s > r$, then $H_s$ and $H''_s$ have the same number of edges.
        \end{enumerate}
    \end{lem}

    \begin{proof}
        Assume notation as in the definition of folding $\rho$;
        i.e.~we have edges $E_1$ and $E_2$ and in $H_r$ and vertex group elements $g_1$ and $g_2$
        such that $f(g_1E_1) = f(g_2E_2)$.
        Let $F\colon \mathcal{G} \to \mathcal{G}''$ be the fold of $E_1$ with $E_2$.
        There is an induced map $g\colon \mathcal{G}'' \to \mathcal{G}$ such that $gF = f$.
        Define $f'\colon \mathcal{G}' \to \mathcal{G}'$ by tightening
        $Fg\colon \mathcal{G}' \to \mathcal{G}'$
        and collapsing a maximal pretrivial forest.
        Observe that edges that are collapsed belong to zero strata.
        There is an identifying homotopy equivalence $q\colon \mathcal{G} \to \mathcal{G}'$.
        As in the proof of \Cref{partialfoldfewernielsenpaths},
        the map $f'\colon \mathcal{G}' \to \mathcal{G}'$ is a topological representative
        preserving the filtration whose elements $G'_i$ satisfy $G'_i = q(G_i)$,
        so the filtration is reduced and still realizes $\mathscr{C}$.
        If $H_i$ is a stratum of $f\colon \mathcal{G} \to \mathcal{G}$
        which is not entirely collapsed to a point,
        then $H'_i$ is a stratum of the same type,
        and the Perron--Frobenius eigenvalues are equal if $H_i$ was exponentially growing.

        If $f' \colon \mathcal{G}' \to \mathcal{G}'$ is not a relative train track map,
        we may again perform invariant core subdivision and collapsing of inessential connecting paths
        to restore properties \hyperlink{EG-i}{(EG-i)} and \hyperlink{EG-ii}{(EG-ii)}.
        As in the proof of \Cref{partialfoldfewernielsenpaths},
        the result is a relative train track map $f''\colon \mathcal{G}'' \to \mathcal{G}''$
        satisfying our standing assumptions.
        When performing invariant core subdivision, the number of edges in the relevant
        exponentially growing stratum is unchanged.
        Thus the existence of the bijection $H_s \to H''_s$ is clear,
        as is the fact that $H''_r$ has one fewer edge than $H_r$.
        If $s > r$, to show that $H''_s$ has the same number of edges as $H_s$,
        we must show that no folding in $H'_s$ occurs when collapsing inessential connecting paths.
        Since $H'_s$ is not a zero stratum,
        the component of any filtration element it belongs to
        is noncontractible.
        (There is only one component by \Cref{aperiodicallaperiodic}.)
        Observe that the map $q\colon \mathcal{G} \to \mathcal{G}'$
        does not identify points that are not identified by some iterate of 
        $f\colon \mathcal{G} \to \mathcal{G}$.
        Thus this component satisfies \hyperlink{EG-ii}{(EG-ii)} by \Cref{trivialedgegroupsEG2}
        for $f'$ because $f$ does,
        and we conclude that no folding in $H'_s$ occurs.
    \end{proof}

    \paragraph{Step 2: \Cref{improvedrelativetraintrack}.}
    The second step is to apply the construction in the proof of \Cref{improvedrelativetraintrack}
    to our relative train track map $f\colon \mathcal{G} \to \mathcal{G}$.
    As we remarked in the proof,
    the number of indivisible almost Nielsen paths of exponentially growing height is preserved
    and the number of edges of each exponentially growing stratum is unchanged.
    If the resulting relative train track map $f\colon \mathcal{G} \to \mathcal{G}$
    does not satisfy \hyperlink{EGAlmostNielsenPaths}{(EG Almost Nielsen Paths)},
    we can return to Step 1 and restore these properties.
    After starting over finitely many times we may assume that 
    \hyperlink{EGAlmostNielsenPaths}{(EG Almost Nielsen Paths)} is still satisfied.

    \paragraph{Step 3: (Rotationless), (Filtration) and (Zero Strata).}
    The properties \hyperlink{Rotationless}{(Rotationless)} and \hyperlink{Filtration}{(Filtration)}
    follow from \Cref{rotationlessisrotationless} 
    and property \hyperlink{F}{(F)} of \Cref{improvedrelativetraintrack}.
    By property \hyperlink{Z}{(Z)} of \Cref{improvedrelativetraintrack},
    to arrange \hyperlink{ZeroStrata}{(Zero Strata)} it suffices to show that each edge
    in a zero stratum $H_i$ is $r$-taken.
    Each edge $E$ in $H_i$ is contained in an $r$-taken path $\sigma \subset H_i$,
    since $f|_{G_r}$ is a homotopy equivalence.
    If $E$ itself is not $r$-taken,
    perform a tree replacement, removing $E$ and insert a new edge $E'$
    that has the same endpoints as $\sigma$ and such that the identifying homotopy equivalence
    $p'\colon \mathcal{G}' \to \mathcal{G}$ sends $E'$ to $\sigma$.
    After finitely many tree replacements, \hyperlink{ZeroStrata}{(Zero Strata)} is satisfied.

    \paragraph{Step 4: (Almost Periodic Edges).}
    Suppose at first that no component of $\per(f)$ is either
    a topological circle nor the quotient of $\mathbb{R}$ by the standard action of $C_2*C_2$
    with exactly two periodic directions at each point.
    Then the endpoints of any almost periodic edge are principal,
    so by \hyperlink{Rotationless}{(Rotationless)} each almost periodic edge is almost fixed
    and each almost periodic stratum $H_r$ is a single edge $E_r$.
    The remainder of \hyperlink{AlmostPeriodicEdges}{(Almost Periodic Edges)}
    is the translation of \hyperlink{P}{(P)} (or more accurately, \Cref{propertyPconsequence})
    to this situation.

    In general, we will suppose that $\per(f)$ contains a component $C$ of the form above
    and modify $f\colon \mathcal{G} \to \mathcal{G}$ to make $C$ no longer problematic,
    either by adding periodic directions in the case of $C$ a circle,
    or by arranging that $C$ is a dihedral pair in the case that $C$
    is the quotient of $\mathbb{R}$ by the standard action of $C_2*C_2$.

    Item 1 of \Cref{rotationlessproperties} implies that $C$ is $f$-invariant,
    and that if $C$ is a circle, then $g = f|_C$ is orientation preserving.
    By \hyperlink{ZeroStrata}{(Zero Strata)} and the fact that there are no periodic directions
    based in $C$ and pointing out of $C$,
    if $E_j$ is an edge not in $C$ that has an endpoint in $C$,
    then $E_j$ is non-almost periodic, non-exponentially growing,
    and meets $C$ only at its terminal vertex.
    Notice that since all non-periodic vertices are contained in exponentially growing strata
    (by \hyperlink{NEG}{(NEG)} and \hyperlink{ZeroStrata}{(Zero Strata)})
    no vertex outside of $C$ maps into $C$.
    By \hyperlink{NEG}{(NEG)} again, $C$ is a component of some filtration element.
    
    Suppose now that $C$ is a circle. First we arrange that $C$ is fixed by $f$.
    Define a map $h\colon \mathcal{G} \to \mathcal{G}$ homotopic to the identity
    by extending the rotation $g^{-1}\colon C \to C$
    such that $h$ has support on a small neighborhood of $C$
    and such that $h(E_j) \subset E_j \cup C$ 
    for each edge $E_j$ with terminal endpoint in $C$.
    Redefine $f$ on each edge $E$ of $G$ to be $h_\sharp f_\sharp(E)$ 
    and observe that edges in $C$ are now fixed.
    For each edge $E_j$ with terminal vertex in $C$,
    the edge path $f(E_j) = E_j u_j$ differs from its original definition
    only by initial and terminal segments in $C$.
    The $f$-images of all other edges is unchanged;
    what's more, if $\sigma$ is a path such that the endpoints of $f(\sigma)$
    are not in the support of $h$,
    then $f_\sharp(\sigma)$ is unchanged.
    The map $f$ remains a relative train track map,
    and all the properties we have previously established remain,
    with the possible exception of property \hyperlink{P}{(P)},
    which will fail if edges (possibly in $C$ or one of the $E_j$)
    is now a fixed edge that should be collapsed.
    If there is such an edge, collapse it;
    all previously established properties continue to hold.
    If after performing these collapses $C$ now has outward-pointing periodic directions,
    then we are finished.

    Suppose that $C$ still does not have outward-pointing periodic directions.
    Then the first non-periodic edge $E_m$ with terminal endpoint in $C$
    satisfies $f(E_m) = E_mC^d$ for some nonzero integer $d$.
    In this second step we modify $f$ so that $f(E_m) = E_m$.
    Take a map $h'\colon \mathcal{G} \to \mathcal{G}$ which is the identity on $C$
    that satisfies $h'(E_j) = E_j C^{-d}$ for all edges $E_j$ with terminal endpoint in $C$
    and that has support in a small neighborhood of $C$.
    Again, this map is homotopic to the identity.
    Redefine $f$ on each edge by tightening $h'f$
    and note that $C \cup E_m$ now belongs to $\fix(f)$,
    so the component of $\per(f)$ containing $C$ is no longer a topological circle.
    If necessary, collapse fixed edges with an endpoint in $C$
    to restore property \hyperlink{P}{(P)} and repeat this step.

    So suppose instead that $C$ is the quotient of $\mathbb{R}$ by the standard action of $C_2*C_2$.
    By \hyperlink{P}{(P)}, $C$ consists of at most two edges,
    and after subdividing at a fixed point in the center of $C$ we may assume there are exactly two.
    To arrange that $C$ is a dihedral pair,
    slide each edge $E_j$ with terminal endpoint in $C$
    so that its endpoint is now the center vertex of $C$.
    
    Let us remark that we have shown that the endpoint of an almost periodic edge
    is either principal (and hence the edge is almost fixed)
    or part of a dihedral pair (and hence \emph{not} principal).

    \paragraph{Step 5: Induction, the NEG case.}
    Let $NI$ be the number of strata $H_i$ such that each component of $G_i$ is non-contractible.
    By properties \hyperlink{P}{(P)}, \hyperlink{Z}{(Z)} and \hyperlink{NEG}{(NEG)},
    this is the count of the number of strata $H_i$ that are irreducible,
    with the possible exception of those strata that are the bottom half of a dihedral pair.
    Given $m$ satisfying $0 \le m \le NI$, let $G_{i(m)}$ be the $m$th filtration element
    satisfying the property that each component of $G_i$ is non-contractible.
    If there are no dihedral pairs, this is the smallest filtration element
    containing the first $m$ irreducible strata.
    We will prove by induction on $m$ that one can modify $f$ so that $f|_{G_{i(m)}}$,
    or more properly speaking the restriction of $f$ to each component of $G_{i(m)}$, is a CT.
    The $m = 0$ case is vacuous,
    so we turn to the inductive step;
    assuming that $f|_{G_r}$ is a CT for $r = i(m)$,
    we will alter $f$ so that $f|_{G_s}$ is a CT for $s = i(m+1)$.
    There are two cases, according to whether $H_s$ is exponentially or non-exponentially growing.
    In this step we will assume that $H_s$ is non-exponentially growing.

    Suppose as a first case that $H_s$ is almost periodic.
    Then either $H_s$ consists of a single almost fixed edge with principal endpoints,
    or it is half or all of a dihedral pair.
    In the former case, \hyperlink{CompletelySplit}{(Completely Split)}
    \hyperlink{Vertices}{(Vertices)}, \hyperlink{LinearEdges}{(Linear Edges)} and 
    \hyperlink{NEGAlmostNielsenPaths}{(NEG Almost Nielsen Paths)}
    follow from the inductive hypothesis.
    In the latter case, by \hyperlink{AlmostPeriodicEdges}{(Almost Periodic Edges)},
    the dihedral pair is a component of $G_s$, and the properties above again
    follow by the inductive hypothesis.

    So suppose that $H_s$ is not almost periodic,
    so it is a single edge $E_s$ satisfying $f(E_s) = g_sE_su_s$ for some vertex group element $g_s$
    and path $u_s$ in $G_{s-1}$.
    By \hyperlink{ZeroStrata}{(Zero Strata)}, $r = s-1$.
    We will modify $E_s$ and $u_s$ by sliding along a path $\tau$ in $G_{s-1}$
    with initial endpoint equal to the terminal vertex of $E_s$.
    As we noted in Step 2,
    we may (after possibly starting over) assume that sliding preserves 
    \hyperlink{EGAlmostNielsenPaths}{(EG Almost Nielsen Paths)}.

    First suppose that after sliding, we have that $E_s$ is almost fixed.
    By \Cref{sliding lemma} this is equivalent to $[\bar\tau u_s f_\sharp(\tau)]$
    being a trivial path and hence equivalent to
    $f_\sharp(E_s\tau) = g_sE_s[u_s f_\sharp(\tau)] = g_sE_s\tau g$ for some vertex group element $g$.
    In other words, this is equivalent to $E_s\tau$ being an almost Nielsen path.

    Continuing to suppose that $E_s$ is now almost fixed,
    observe that if both endpoints of $E_s$ are contained in $G_{s-1}$,
    then \hyperlink{AlmostPeriodicEdges}{(Almost Periodic Edges)} is satisfied,
    as are all the conclusions of \Cref{improvedrelativetraintrack},
    with the possible exception of \hyperlink{P}{(P)}
    which can fail if what was formerly a fixed dihedral pair is now no longer a dihedral pair
    (because one of the $C_2$ vertex groups now has valence two, say),
    in which case we may restore this property by collapsing an edge.
    All the properties from Step 3 are still satisfied,
    and the remaining properties of a CT follow from the inductive hypothesis.

    If some endpoint of $E_s$ is not contained in $G_{s-1}$,
    then collapse $E_s$ to a point as in Step 4.
    None of the previously achieved properties are lost
    and the remainder of the properties of a CT follow by induction.

    Assume now that there is no choice of $\tau$ such that $E_s\tau$ is an almost Nielsen path.
    By \Cref{neginduction}, after sliding, we have $f(E_s) = g_sE_s\cdot u_s$
    where $u_s$ is nontrivial,
    so $f|_{G_s}$ satisfies \hyperlink{AlmostPeriodicEdges}{(Almost Periodic Edges)}
    and all the properties from the first four steps of the proof.
    Items \hyperlink{CompletelySplit}{(Completely Split)},
    \hyperlink{Vertices}{(Vertices)},
    \hyperlink{NEGAlmostNielsenPaths}{(NEG Almost Nielsen Paths)}
    and \hyperlink{LinearEdges}{(Linear Edges)} for $f|_{G_s}$
    follow from those properties for $f|_{G_r}$ and \Cref{neginduction}.
    This completes the inductive step assuming that $H_s$ is non-exponentially growing.

    \paragraph{Step 6: Induction, the EG case.}
    Suppose now that $H_s$ is exponentially growing.
    Items \hyperlink{Vertices}{(Vertices)}, \hyperlink{LinearEdges}{(Linear Edges)}
    and \hyperlink{NEGAlmostNielsenPaths}{(NEG Almost Nielsen Paths)} follow from
    \hyperlink{EGAlmostNielsenPaths}{(EG Almost Nielsen Paths)} and the inductive hypothesis on $f|_{G_r}$.
    It remains to establish \hyperlink{CompletelySplit}{(Completely Split)} for $f|_{G_s}$.

    For each edge $E$ in $H_s$, there is a decomposition $f(E) = \mu_1\cdot\nu_1\cdot \mu_2\cdots\nu_{m-1}\mu_m$,
    where the $\mu_i$ belong to $H_s^z$ and the $\nu_i$ are maximal subpaths in $G_{r}$.
    We let $\{\nu_\ell\}$ denote the collection of all such  paths that occur 
    as $E$ varies over the edges of $H_s$.
    By \hyperlink{EG-ii}{(EG-ii)}, for all $k$ and $\ell$, the path $f^k_\sharp(\nu_\ell)$ is nontrivial.
    By \Cref{largepowerscompletelysplit}, we may choose $k$ so large that
    each  $f^k_\sharp(\nu_\ell)$ is completely split.
    The endpoints of each $\nu_\ell$ are periodic by \Cref{trivialedgegroupsEG2}
    and thus principal by \hyperlink{Vertices}{(Vertices)} and hence fixed.
    There are finitely many connecting paths $\sigma$ in the strata (if any)
    between $G_r$ and $H_s$.
    Each $f_\sharp(\sigma)$ is either again one of these connecting paths
    or a nontrivial tight path in $G_r$ whose endpoints are fixed by the observation above.
    Therefore we may (after increasing $k$)
    assume that each $f^k_\sharp(\sigma)$ is completely split for each such $\sigma$.
    After applying \Cref{changingthemarking} $k$ times with $j = r$,
    we have by item 7 of the proof of that lemma that $f|_{G_s}$ is completely split.
    This completes the inductive step and with it the proof of the theorem.
\end{proof}

\begin{cor}
    \label{powersofCTsareCTs}
    If $f\colon \mathcal{G} \to \mathcal{G}$ is a CT,
    then $f^k$ is a CT for all $k \ge 1$.
\end{cor}

\begin{proof}
    Property \hyperlink{EGAlmostNielsenPaths}{(EG Almost Nielsen Paths)}
    follows from \Cref{properfoldsegalmostnielsenpaths} and \Cref{egalmostnielsenpathsproperfold}.
    The rest of the properties are straightforward to check;
    we leave them to the reader.
\end{proof}

\section{An index inequality}
\label{indexsection}
Gaboriau, Jaeger, Levitt and Lustig define in \cite[Theorem 4]{GaboriauJaegerLevittLustig}
an \emph{index} for (outer) automorphisms of free groups.
Namely, suppose $\Phi$ is an automorphism of a free group $F_n$.
Then, as in \Cref{boundarysection}, 
$\Phi$ yields a homeomorphism $\hat\Phi$ of the Gromov boundary of $F_n$,
which is equal to its Bowditch boundary.
Let $a(\Phi)$ denote the number of $\fix(\Phi)$-orbits
of attracting fixed points for $\hat\Phi$ in $\partial_\infty(F_n)$.
Gaboriau, Jaeger, Levitt and Lustig's \emph{index} of the automorphism $\Phi$ is the quantity
\[  i(\Phi) = \max\left\{0,\ \rank(\fix(\Phi)) + \frac{1}{2}a(\Phi) - 1\right\}. \]
Observe that if $i(\Phi)$ is positive,
then $\Phi$ is a principal automorphism,
so there are only finitely many isogredience classes of automorphisms $\Phi$ such that $i(\Phi) > 0$.
If $\varphi \in \out(F_n)$ is the outer class of $\Phi$,
Gaboriau, Jaeger, Levitt and Lustig define the \emph{index} of $\varphi$ to be the sum
\[  i(\varphi) = \sum i(\Phi), \]
where the sum may be equivalently defined over representatives of \emph{all} isogredience classes
of automorphisms $\Phi$ representing $\varphi$, or merely representatives
of principal isogredience classes.
Gaboriau, Jaeger, Levitt and Lustig prove that for all $\varphi \in \out(F_n)$,
the index $i(\varphi)$ satisfies $i(\varphi) \le n-1$.
The main result of Martino's paper \cite{Martino}
is that if $F_n$ is replaced by a free product $F = A_1 * \cdots * A_n * F_k$,
then the index, defined exactly as above, of any $\varphi \in \out(F,\mathscr{A})$ satisfies
$i(\varphi) \le n + k - 1$.
Let us remark that Martino's result 
is stated for the Grushko decomposition of a finitely generated group,
but that the proof does not use this assumption in any essential way.

The purpose of this section is to improve on Martino's result in two ways.
The first is an innovation of Feighn--Handel:
let $b(\Phi)$ denote the number of $\fix(\Phi)$-orbits of attracting fixed points for $\hat\Phi$
in $\partial_\infty(F,\mathscr{A})$ associated to \emph{eigenrays,}
which are referred to as \emph{NEG rays} in \cite{FeighnHandelAlg};
we follow the naming convention in \cite{HandelMosher}.
We defer a precise definition for now,
but the idea is that these attractors 
are associated in the sense of \Cref{fixedpointsfromfixeddirections}
to fixed directions coming from non-almost periodic non-exponentially growing strata
of some relative train track map representing (a rotationless iterate of) $\hat\Phi$.
The second is to use the data of the Bowditch boundary of $(F,\mathscr{A})$.
Recall that the stabilizer of each point in $V_\infty(F,\mathscr{A})$
is an infinite group $A$ conjugate into $\mathscr{A}$,
and that a point in $V_\infty(F,\mathscr{A})$ is fixed by $\hat\Phi$
if and only if the associated group $A$ is preserved by $\Phi$.
Let $c(\Phi)$ denote the number of $\fix(\Phi)$-orbits of fixed points in $V_\infty(F,\mathscr{A})$
with the property that the intersection of the associated group $A$ with $\fix(\Phi)$ is trivial.
Define the quantity $j(\Phi)$ as
\[  j(\Phi) = \max \left\{ 0,\ \rank(\fix(\Phi)) + \frac{1}{2}a(\Phi) 
+ \frac{1}{2}b(\Phi) + c(\Phi) - 1\right\}. \]
Here and throughout the section, we use a slightly nonstandard definition of $\rank(H)$
for $H$ a subgroup of $F$:
If $(H,\mathscr{A}|_H)$ takes the form $C_2 * C_2$,
define $\rank(H) = 1$; otherwise define it to be the usual Kurosh subgroup rank of $H$.
This definition has the following advantage:
if $j(\Phi)$ is positive, then $\Phi$ is a principal automorphism,
so there are again only finitely many 
isogredience classes of automorphisms $\Phi$ for which $j(\Phi) > 0$,
and we may define
\[  j(\varphi) = \sum j(\Phi) \]
where again the sum may be equivalently defined over the isogredience classes of all automorphisms
$\Phi$ representing $\varphi$, or merely the principal ones.
The main result of this section is the following theorem.
\begin{thm}
    \label{preciseindexthm}
    Suppose $\varphi \in \out(F,\mathscr{A})$ has a rotationless iterate,
    and that each $A \in \mathscr{A}$ is finitely generated.
    Then the index $j(\varphi)$ satisfies
    \[  j(\varphi) \le n + k - 1. \]
\end{thm}

Let us remark that although the conclusions of \Cref{preciseindexthm} 
are stronger than Martino's result,
the assumption that $\varphi$ has a rotationless iterate appears to be a nontrivial restriction.
The strategy of the proof of \Cref{preciseindexthm} 
is to follow the strategy in \cite[Section 15]{FeighnHandelAlg}.
We firstly show that if $\psi = \varphi^K$ is a rotationless iterate of $\varphi$,
then $j(\varphi) \le j(\psi)$.
Then, using a CT $f\colon \mathcal{G} \to \mathcal{G}$ representing $\psi$,
we construct a graph of groups $\mathscr{S}_N(f)$,
invariants of which calculate $j(\psi)$,
and argue by induction up through the filtration that $j(\psi) \le n + k - 1$.

\paragraph{Rays and attracting fixed points.}
Let $f\colon \mathcal{G} \to \mathcal{G}$ be a CT representing $\varphi \in \out(F,\mathscr{A})$.
Write $\mathcal{E}$ for the set of oriented
non-almost periodic, non-linear edges $E$ of $G$
with the property that the initial vertex of $E$ is principal
and the property that the direction of $E$ is almost fixed by $Df$.
By \Cref{negctsplitting} if $E$ is non-exponentially growing
and by \Cref{rttlemma} if $E$ is exponentially growing,
there is a path $u$ such that $f_\sharp(E) = gE\cdot u$ is a splitting.
If the length of $f^k_\sharp(u)$ does not go to infinity with $k$,
then $E$ is a non-linear almost linear edge.
Define the ray $R_E$ in $\mathcal{G}$ as
\[  R_E = E \cdot u \cdot f_\sharp(u) \cdot f^2_\sharp(u) \cdots. \]
Each lift $\tilde R_E$ of $R_E$ to $\Gamma$
determines a point in $\partial_\infty(F,\mathscr{A})$,
so $R_E$ determines an $F$-orbit $\partial R_E$ in $\partial(F,\mathscr{A})$.

Suppose $\Phi \colon (F,\mathscr{A}) \to (F,\mathscr{A})$ is an automorphism.
Write $\fix_+(\hat \Phi)$ for the subset of $\fix_N(\hat\Phi)$
comprising the attracting fixed points for $\hat\Phi$ in $\partial_\infty(F,\mathscr{A})$.
Recall from \Cref{principalsection} that for $\varphi \in \out(F,\mathscr{A})$,
we write $P(\varphi)$ for the set of principal automorphisms representing $\varphi$.

\begin{lem}[cf.~Lemma 3.10 of \cite{FeighnHandelAlg}]
    \label{CTrays2}
    Suppose that $f\colon \mathcal{G} \to \mathcal{G}$ is a CT
    and that $E$ is an edge of $\mathcal{E}$.
    If $\tilde E$ is a lift of $E$ and $\tilde f$ is the lift of $f$
    that fixes the initial direction of $\tilde E$,
    then the lift $\tilde R_{\tilde E}$ of $R_E$ that begins with $\tilde E$
    converges to an attracting fixed point in $\fix_+(\hat f)$.
    The map $E \mapsto \partial R_E$
    defines a surjection 
    $\mathcal{E} \to \left( \bigcup_{\Phi\in P(\varphi)} \fix_+(\hat\Phi) \right)/F$.
\end{lem}

\begin{proof}
    Suppose that $\tilde E$ is a lift of $E \in \mathcal{E}$,
    and that $\tilde f \colon \Gamma \to \Gamma$ is a lift of $f$
    that fixes the direction determined by $\tilde E$.
    By \Cref{principalfixedprincipallift}, since $\fix(\tilde f)$ is principal
    and projects to a non-exceptional almost Nielsen class,
    we conclude that $\tilde f$ is principal.
    Item 1 of \Cref{CTrays} implies that $\tilde R_{\tilde E}$ 
    converges to a point $P \in \fix_N(\hat f)$.
    Either the length of $f^k_\sharp(u)$ goes to infinity with $k$,
    in which case the ``superlinear attractors'' piece of \Cref{GJLLprop} 
    implies that $P \in \fix_+(\hat f)$,
    or $E$ is a non-linear almost linear edge.
    By \Cref{negctlinearcondition}, $P$ is not a limit of points that are fixed by $\tilde f$.
    Since $P$ is not the endpoint of an axis by the proof of \Cref{CTrays},
    it follows that $P$ is not in the boundary of the fixed subgroup,
    so \Cref{GJLLprop} implies that $P \in \fix_+(\hat f)$.
    Item 2 of \Cref{CTrays} implies that the map $E \mapsto \partial R_E$ is surjective.
\end{proof}

Observe that because attracting fixed points of principal automorphisms are not the endpoints of axes
by \Cref{GJLLprop} and \Cref{boundarybasics}
and so are not fixed by $T_c$ for any $c \in F$,
it follows that the union $\bigcup_{\Phi\in P(\varphi)} \fix_+(\hat\Phi)$ is a disjoint union.

\begin{lem}
    \label{rayswithcommonendpoints}
    Suppose that $E \ne E'$ are distinct edges of $\mathcal{E}$
    and that $\partial R_E = \partial R_{E'}$.
    Then $E$ and $E'$ belong to the same exponentially growing stratum $H_r$
    and are the initial and terminal edges of the unique equivalence class of
    indivisible almost Nielsen paths of height $r$.
\end{lem}

\begin{proof}
    The argument is identical to \cite[12.1]{FeighnHandelAlg}.
    Suppose that $E$ and $E'$ are as in the statement.
    Choose a lift $\tilde E$ of $E$ to $\Gamma$ and let $\tilde R_{\tilde  E}$
    be the lift of $R_E$ that begins with $\tilde E$.
    Let $P \in \partial(F,\mathscr{A})$ be the endpoint of $\tilde R_{\tilde E}$.
    By assumption there is a lift $\tilde R_{\tilde E'}$ of $R_{E'}$
    that ends at $P$ and begins at a lift $\tilde E'$ of $E'$.
    The lift $\tilde f$ that fixes the direction of $\tilde E$ fixes $P$.
    By the remark in the previous paragraph, $\tilde f$ is the only lift of $f$
    that fixes $P$, so it follows that $\tilde f$ fixes the direction of $\tilde E'$ as well,
    and thus the unique tight path beginning with $\tilde E$ and ending with $\tilde E'$
    projects to an almost Nielsen path $\rho$ in $\mathcal{G}$.
    Since there are no almost Nielsen paths of non-linear non-exponentially growing height,
    we conclude that $E$ and $E'$ are of exponentially growing height.
    In fact they each have exponentially growing height $r$,
    since $R_E$ and $R_{E'}$ have height $r$ and $\tilde R_{\tilde E}$ and $\tilde  R_{\tilde E'}$
    have a common terminal subray $\tilde R_{\tilde E,\tilde E'}$.
    By construction, since $R_E$ and $R_{E'}$ are $r$-legal, the path $\rho$ has one illegal turn in $H_r$,
    so we conclude that $\rho$ is indivisible,
    and is thus the unique (by \Cref{egalmostnielsenpathsproperties})
    equivalence class of indivisible almost Nielsen paths of height $r$.
    If we write $\rho = \alpha\beta$ as in \Cref{finitelymanynielsenpaths},
    then there is a terminal subray  $R_{E,E'}$ of $R_E$ and $R_{E'}$
    such that we have $R_E = \alpha R_{E,E'}$ and $R_{E'} = \bar\beta R_{E,E'}$.
\end{proof}

If $P$ is an attracting fixed point for an automorphism $\Phi$
represented by the lifted ray $\tilde R_{\tilde E}$ as in \Cref{CTrays2},
we say that $P$ or $\tilde R_{\tilde E}$ is an \emph{eigenray} for $\Phi$
if the stratum of $f\colon \mathcal{G} \to \mathcal{G}$ containing $E$ is non-exponentially growing.
The following lemma says that this definition is independent of the CT 
$f\colon \mathcal{G} \to \mathcal{G}$.

\begin{lem}[cf.~Definitions 2.9 and 2.10 and Lemma 2.11 of Part II of \cite{HandelMosher}]
    Suppose that $\varphi \in \out(F,\mathscr{A})$ is rotationless,
    $\Phi \in P(\varphi)$ and that $P \in \fix_+(\hat\Phi)$.
    The following are equivalent.
    \begin{enumerate}
        \item For some CT $f\colon \mathcal{G} \to \mathcal{G}$ representing $\varphi$
            there is a non-linear non-exponentially growing edge $E$
            with a lift $\tilde E$ in $\Gamma$ such that $\tilde R_{\tilde E}$ converges to $P$.
        \item For \emph{every} CT $f\colon \mathcal{G} \to \mathcal{G}$ representing $\varphi$
            there is a non-linear non-exponentially growing edge $E$
            with a lift $\tilde E$ in $\Gamma$ such that $\tilde R_{\tilde E}$ converges to $P$.
    \end{enumerate}
\end{lem}

\begin{proof}
    It is clear that item 2 implies item 1.
    Suppose item 1 holds.
    Let $\Lambda(P)$ be the limit set of $P$ 
    (see \Cref{laminationsection} before \Cref{fixedpointstolaminations}).
    The proof of \Cref{minimalfreefactorcarries} shows that $\Lambda(P)$
    is carried by the conjugacy class of a free factor $B$ of $F$ of positive complexity.
    In fact, if $E$ is the unique edge of $H_r$,
    then we have that $[[B]] \sqsubset \mathcal{F}(G_{r-1})$, 
    so $B$ is a proper free factor of $F$.
    Choose $B$ within its conjugacy class so that $P \in \partial(B,\mathscr{A}|_B)$.
    By \Cref{negctlinearcondition} and \Cref{negfixedboundarypoint},
    the set $\fix_N(\hat\Phi) \cap \partial(A,\mathscr{A}|_A)$
    contains only $P$ (uniqueness is established in \Cref{neginduction}).

    Let $f'\colon \mathcal{G}' \to \mathcal{G}'$ be a CT.
    By \Cref{CTrays2}, there is a non-almost fixed, non-linear edge $E' \in \mathcal{E}'$
    such that $\tilde R_{\tilde E'}$ converges to $P$.
    Suppose that $E'$ belonged to an exponentially growing stratum.
    In the proof of \Cref{fixedpointstolaminations},
    we showed that there is a leaf of the attracting lamination $\Lambda^+ = \Lambda(P)$
    with $P$ as one of its endpoints.
    Both endpoints belong to $\fix_N(\hat\Phi)$,
    contradicting what we established in the previous paragraph.
    Therefore we conclude that item 2 holds.
\end{proof}

If $\varphi$ is not rotationless, and so not represented by a CT,
then we say that $P$ is an \emph{eigenray} for $\Phi$ if it is an eigenray
for $\phi^K$, where $\varphi^K$ is a rotationless iterate of $\varphi$.
Notice that if $\varphi$ is rotationless and $f\colon \mathcal{G} \to \mathcal{G}$
is a CT representing $\varphi$,
then by \Cref{powersofCTsareCTs}, $f^k$ is a CT representing $\varphi^k$ for all $k \ge 1$,
so it follows that $P$ is an eigenray for $\Phi$
if and only if it is an eigenray for each $\Phi^k$ for $k \ge 1$.
If $\varphi$ is not rotationless but has rotationless iterates $K$ and $L$,
then $P$ is an eigenray for $\Phi^K$ if and only if it is an eigenray for $\Phi^{KL}$
if and only if it is an eigenray for $\Phi^L$,
so this definition is independent of the rotationless iterate.

Let $\mathcal{R}(\varphi) = \bigcup_{\Phi\in P(\varphi)} \fix_+(\Phi)$.
Let $\rneg(\Phi)$ denote the set of eigenrays for $\Phi$,
and let $\rneg(\varphi) = \bigcup_{\Phi\in P(\varphi)}\rneg(\Phi)$.

\begin{lem}[cf.~Lemma 15.8 of \cite{FeighnHandelAlg}]
    Suppose $\Phi$ and $\Psi$ are principal automorphisms representing $\varphi$.
    The following hold.
    \begin{enumerate}
        \item If $\fix_+(\Phi) \cap \fix_+(\Psi) \ne \varnothing$,
            then $\Phi = \Psi$.
            We have
            \[  \mathcal{R}(\varphi) = \coprod_{\Phi\in P(\varphi)}\fix_+(\Phi) \]
            and
            \[  \rneg(\varphi) = \coprod_{\Phi\in P(\varphi)}\rneg(\Phi).\]
        \item The stabilizers of $\fix_+(\Phi)$ and
            $\rneg(\Phi)$ respectively
            under the action of $F$ on $\mathcal{R}(\varphi)$
            and $\rneg(\varphi)$ respectively
            are each $\fix(\Phi)$.
        \item If $\{\Phi_1,\ldots,\Phi_N\}$ is a set of representatives of isogredience classes
            in $P(\varphi)$,
            the natural maps
            \[  \coprod_{i=1}^N \fix_+(\Phi_i)/\fix(\Phi_i) \to \mathcal{R}(\varphi)/F \]
            and
            \[  \coprod_{i=1}^N \rneg(\Phi_i)/\fix(\Phi_i) \to \rneg(\varphi)/F \]
            are bijections.
    \end{enumerate}
\end{lem}

\begin{proof}
    The proof is identical to \cite[Lemma 15.8]{FeighnHandelAlg}.
    For item 1, since $\Phi$ and $\Psi$ both represent $\varphi$,
    we have that $\Phi\Psi^{-1} = i_c$ for some $c \in F$,
    where we remind the reader that $i_c$ denotes the inner automorphism $x \mapsto cxc^{-1}$.
    If $R$ belongs to $\fix_+(\Phi) \cap \fix_+(\Psi)$,
    then $\hat T_c(R) = R$.
    But since $R \in \partial_\infty(F,\mathscr{A})$ is not the endpoint of an axis,
    we conclude that $c = 1$.

    For item 2, suppose first that $c \in \fix(\Phi)$ and $R \in \rneg(\Phi)$.
    Then
    \[  \hat\Phi\hat T_c(R) = T_{\Phi(c)}\hat\Phi(R) = \hat T_c(R), \]
    so $\hat T_c(R) \in \rneg(\Phi)$.
    Conversely, if we have $c \in F$ and $R \in \rneg(\Phi)$ such that
    $\hat\Phi(\hat T_c(R)) = \hat T_c(R)$, then we have
    \[  \hat T_c(R) = \hat\Phi \hat T_c(R) = T_{\Phi(c)}\hat\Phi(R) = T_{\Phi(c)}(R). \]
    As in the proof of item 1, we conclude that $\Phi(c) = c$.
    The same argument applies with $\rneg(\Phi)$ replaced by $\fix_+(\Phi)$.

    Item 3 is a consequence of item 1 and item 2,
    combined with the observation that the action of $F$
    on $\mathcal{R}(\varphi)$ by permuting the terms of the disjoint union in item 1.
    The same holds for $\rneg(\varphi)$.
\end{proof}

\begin{lem}[cf.~Lemma 15.9 of \cite{FeighnHandelAlg}]
    \label{abinequality}
    For $k \ge 1$ and $\varphi \in \out(F,\mathscr{A})$, we have
    $b(\varphi^k) \ge b(\varphi)$ and $a(\varphi^k) \ge a(\varphi)$.
\end{lem}

\begin{proof}
    The proof is identical to \cite[Lemma 15.9]{FeighnHandelAlg}.

    By the definition of $b(\varphi)$ and item 3 of the previous lemma, we have
    \[  b(\varphi) = \left|\coprod_{i=1}^N \rneg(\Phi_i)/\fix(\Phi_i)\right|
    = \left|\rneg(\varphi)/F\right|. \]
    By definition, if $R$ is an eigenray for $\Psi \in P(\varphi)$,
    then $R$ is also an eigenray for $\Psi^k \in P(\varphi^k)$ for $k \ge 1$,
    so we have $\rneg(\varphi^k) \supset \rneg(\varphi)$. Hence we conclude that
    \[ b(\varphi^k) = \left|\rneg(\varphi^k)/F\right| \ge \left|\rneg(\varphi)/F\right| = b(\varphi).\]
    The above argument works for $a(\varphi)$
    by replacing $\rneg(\varphi)$ with $\mathcal{R}(\varphi)$
    and $\rneg(\Phi)$ with $\fix_+(\Phi)$.
\end{proof}

\paragraph{Fixed subgroups.}
Given an automorphism $\Phi \colon (F,\mathscr{A}) \to (F,\mathscr{A})$,
define 
\[  \hat r(\Phi) = \max\{0,\ \rank(\fix(\Phi)) + c(\Phi) - 1\}, \]
where we remind the reader that we use the convention that $\rank(C_2*C_2) = 1$.
Given $\varphi \in \out(F,\mathscr{A})$, suppose $\{\Phi_1,\ldots,\Phi_N\}$ 
is a set of representatives of the isogredience classes 
of principal automorphisms representing $\varphi$.
Define
\[   \hat r(\varphi) = \sum_{i=1}^N\hat r(\Phi). \]
Our goal is the prove the following lemma.
\begin{lem}[cf.~Lemma 15.11 of \cite{FeighnHandelAlg}]
    \label{hatrinequality}
    Let $\varphi \in \out(F,\mathscr{A})$ and suppose $\varphi$
    has a rotationless iterate $\psi = \varphi^k$ for some $k \ge 1$.
    If each $A \in \mathscr{A}$ is finitely generated,
    then $\hat r(\psi) \ge \hat r(\varphi)$.
\end{lem}

Before proceeding to the proof,
we need a version of \cite[Theorem 3.1]{Culler}.

\begin{prop}
    \label{Cullerprop}
    Suppose $\varphi \in \out(F,\mathscr{A})$ has finite order,
    that each $A \in \mathscr{A}$ is finitely generated
    and that $\Phi$ represents $\varphi$.
    There exists a topological representative $f\colon \mathcal{G} \to \mathcal{G}$
    representing $\varphi$ which is an automorphism of graphs of groups,
    in the sense that the underlying graph map of $f$ is an isomorphism.
    Either $\fix(\Phi)$ is cyclic and non-peripheral,
    or $\fix(\Phi)$ is conjugate to the fundamental group of
    a graph of groups that has an injective-on-edges immersion 
    (in the sense of \cite{Bass})
    into a component of $\fix(f)$.
    Moreover $\sum \hat r(\Phi) \le \rank(F) - 1$,
    where the sum is taken over isogredience classes of \emph{all} automorphisms representing $\varphi$.
\end{prop}

\begin{proof}
    Let us remark that the proposition is true if $F = C_2 * C_2$,
    so suppose that $F \ne C_2 * C_2$.

    By \cite[Theorem 4.1]{HenselKielak},
    the automorphism $\varphi$ fixes a point in the \emph{Outer Space} of $(F,\mathscr{A})$
    as defined by Guirardel--Levitt \cite{GuirardelLevitt}.
    This means there is a Grushko $(F,\mathscr{A})$-tree $T$,
    an automorphism $\Phi \colon (F,\mathscr{A}) \to (F,\mathscr{A})$ representing $\varphi$
    and a $\Phi$-twisted equivariant homeomorphism $\tilde f \colon T \to T$.
    Let $\mathcal{G}$ be the marked graph of groups whose Bass--Serre tree is $T$.
    The resulting map of graphs of groups $f\colon \mathcal{G} \to  \mathcal{G}$
    (See \cite[Proposition 1.2]{Myself} or \cite[4.1--4.5]{Bass})
    is an automorphism of graphs of groups which represents $\varphi$.

    Let us remark that \cite[Theorem 4.1]{HenselKielak}
    is stated only for the Grushko decomposition of a finitely generated group,
    essentially because Guirardel--Levitt discuss only this case.
    The proof, like the definition of the Outer Space, does not use this assumption in any essential way.

    Now fix $\Phi$ representing $\varphi$, and consider the corresponding lift $\tilde f\colon T \to T$.
    Suppose first that $\Phi$ is not principal,
    so that $\fix_N(\hat\Phi)$ contains either a single point
    or the endpoints of an axis $A_c$.
    Then $c(\Phi) \le 1$ and $c(\Phi) = 1$ implies that $\fix(\Phi)$ is trivial,
    for if it were not,
    then $\fix_N(\hat\Phi)$ would contain more than one point in $V_\infty(F,\mathscr{A})$.
    Similarly $\rank(\fix(\Phi)) \le 1$,
    (where again we have $\rank(C_2*C_2) = 1$)
    and the previous argument shows that $\rank(\fix(\Phi)) = 1$ implies that $c(\Phi) = 0$.
    To see the first assertion, note that if $\rank(\fix(\Phi)) > 1$,
    then $\fix_N(\hat\Phi)$ must contain more than the endpoints of a single axis.
    Thus we see that $\hat r(\Phi) = 0$.
    If $\fix(\Phi)$ is nontrivial and contained in a single $A \in \mathscr{A}$,
    then $\Phi$-twisted equivariance implies that $\tilde f$ fixes the unique point in $T$ fixed by $A$.
    This point projects to a fixed point for $f$ in $\mathcal{G}$, so the assertion holds in this case.

    Now suppose that $\Phi$ is principal.
    Then by \Cref{principalnonemptyfixed}, $\tilde f$ fixes a point $\tilde v$,
    and in fact $\fix(\tilde f)$ projects to a non-exceptional almost Nielsen class containing $v$.
    This also holds when $\Phi$ is not principal and $\fix(\Phi) = C_2 * C_2$.
    Since $\tilde f$ is an isomorphism, we have $\tilde f = \tilde f_\sharp$.
    The fact that $\tilde f_\sharp$ preserves any ray $\tilde R$ beginning at $\tilde v$
    and converging to some point $P \in \fix_N(\hat\Phi)$
    therefore implies that $\tilde R$ is pointwise fixed.
    The same is true of the tight path between any pair of fixed points for $\tilde f$.
    This latter fact implies that $\fix(\tilde f)$ is connected and contains at least two points,
    while the former implies that it contains the $\fix(\Phi)$-minimal subtree.
    Conversely, suppose $T_c$ preserves $\fix(\tilde f)$.
    For $\tilde x \in \fix(\tilde f)$, we have $T_c\tilde f(\tilde x) = \tilde f T_c(\tilde x)$,
    so since $\fix(\tilde f)$ contains at least two points,
    we conclude by \Cref{boundarybasics} that $c \in \fix(\Phi)$.
    It follows that the quotient of $\fix(\tilde f)$ by $\fix(\Phi)$,
    call it $\mathcal{H}$,
    thought of as a graph of groups in its own right, is compact
    and immerses (in the sense of \cite{Bass}) into $\mathcal{G}$.
    This immersion is injective on edges,
    for if there were two edges of $H$ that map to the same edge of $\mathcal{G}$,
    there would be a loop in $\mathcal{G}$ based at a point in the interior of an edge
    that lifts to a non-closed path in $\mathcal{H}$.
    This loop corresponds to an element of $\pi_1(\mathcal{G})$ that preserves $\fix(\tilde f)$
    but does not belong to $\fix(\Phi)$, a contradiction.

    Let $e(\mathcal{H})$ be the number of edges of $H$ and $v_0(\mathcal{H})$
    the number of vertices of $H$ with trivial vertex group
    and whose image in $G$ has finite vertex group.
    An easy argument shows that $\hat r(\Phi)$ is bounded above by $e(\mathcal{H}) - v_0(\mathcal{H})$.
    Indeed, the only case where this computation overestimates $\hat r(\Phi)$
    is the case where $\fix(\Phi) = C_2 * C_2$.

    Suppose that $\Phi$ and $\Phi'$ are not isogredient
    and correspond to the lifts $\tilde f$ and $\tilde f'$
    such that $\fix(\tilde f)$ and $\fix(\tilde f')$ contain at least two points.
    Then $\fix(\tilde f)$ and $\fix(\tilde f')$ project to distinct non-exceptional almost Nielsen classes
    in $\fix(f)$,
    and the images of $\mathcal{H}$ and $\mathcal{H}'$ in $\mathcal{G}$ share no edges.
    We have $\rank(F) - 1 = e(\mathcal{G}) - v_0(\mathcal{G})$.
    If a vertex $v$ contributes to $v_0(\mathcal{G})$,
    then every vertex mapping to $v$ in every $\mathcal{H}$ contributes to $v_0(\mathcal{H})$.
    The final assertion follows.
\end{proof}

\begin{proof}[Proof of \Cref{hatrinequality}]
    We follow the argument in \cite[Lemma 15.11]{FeighnHandelAlg}.
    Assume that $\hat r(\varphi) > 0$,
    let $\Phi_1,\ldots,\Phi_s,\ldots,\Phi_N$ be a set of representatives
    of isogredience classes of principal automorphisms in $P(\varphi)$,
    where $\hat r(\Phi_i) > 0$ if and only if $i \le s$.
    Let $\Psi_1,\ldots,\Psi_t,\ldots,\Psi_M$ be a set of representatives
    of isogredience classes of principal automorphisms in $P(\psi)$.
    By definition, up to reordering the representatives of $P(\psi)$
    there is a function
    \[
        p\colon \{1,\ldots,s\} \to \{1,\ldots,t\}
    \]
    such that if $j = p(i)$, then $\Phi_i^k$ is isogredient to $\Psi_j$.
    By replacing $\Phi$ within its isogredience class, we may assume that $\Phi^k_i = \Psi_j$.
    It suffices to show that
    \[  \sum_{j=1}^t\hat r(\Psi_j) \ge \sum_{i=1}^s\hat r(\Phi_i), \]
    for which it suffices to show that
    \[  \hat r(\Psi_j) \ge \sum_{i \in p^{-1}(j)} \hat r(\Phi_i) \]
    for each $j$ satisfying $1 \le j \le t$.

    So fix $j$ and write $\mathbb{F} = \fix(\Psi_j)$ 
    and $C$ for the set of fixed points for $\Psi_j$ in $V_\infty(F,\mathscr{A})$
    whose $\mathbb{F}$-orbits contribute to $c(\Psi_j)$.
    For $i \in p^{-1}(j)$, we have $\Phi^k_i = \Psi_j$,
    so $\mathbb{F} = \fix(\Phi^k_i)$.
    Thus $\Phi_i$ preserves $\mathbb{F}$,
    $\fix(\Phi_i)$ is contained in $\mathbb{F}$,
    and the restriction of $\Phi_i$ to $\mathbb{F}$ is a finite order automorphism of $\mathbb{F}$.
    Similarly $\Phi_i$ permutes the elements of the set $C$.
    We claim that if $i$ and $i' \in p^{-1}(j)$ are distinct (i.e.~not isogredient)
    then $\fix_N(\Phi_{i}) \cap C$ and $\fix_N(\Phi_{i'}) \cap C$
    project to distinct $\mathbb{F}$-orbits.
    Assuming the claim for now,
    let us use it to show that the displayed inequality holds.
    If $\rank(\mathbb{F}) = 1$ (where again $\rank(C_2*C_2)= 1$,
    then every subgroup of $\mathbb{F}$ has rank at most 1,
    and by the claim, each $\mathbb{F}$-orbit in $C$
    is counted in at most one $c(\Phi_i)$, so the displayed inequality holds.
    If instead $\rank(\mathbb{F}) > 1$,
    then $\mathbb{F}$ is its own normalizer in $F$:
    to see this, take $\xi \in \partial_\infty(\mathbb{F},\mathscr{A}|_\mathbb{F})$
    that is not the endpoint of an axis.
    Then $\xi$ is not fixed by any automorphism $\Psi' \ne \Psi_j$ representing $\psi$.
    Suppose that $c \in F$ normalizes $\mathbb{F}$.
    Then $T_c(\xi) \in \mathbb{F}$ and so $\xi$ is fixed by
    $i_{c^{-1}}\Psi_j i_c = i_{c^{-1}\Psi-j(c)}\Psi_j$,
    so we conclude $c \Psi_j(c)$ and therefore $c \in \mathbb{F}$.
    Therefore the restriction of $\varphi$ to $\mathbb{F}$
    is well-defined and has finite order.
    \Cref{Cullerprop} together with the claim proves the displayed inequality.

    We now turn to the proof of the claim.
    Suppose that $\fix_N(\Phi_i)$ and  $\fix(\Phi_{i'})$
    contain $\xi$ and $\xi'$ in the same $\mathbb{F}$-orbit of $C$,
    say that there exists $c \in \mathbb{F}$ such that $\hat T_c(\xi) = \xi'$.
    We will show that $i_c\Phi_ii_c^{-1} = \Phi_{i'}$.
    Indeed, $i_c\Phi_ii_c^{-1}$ fixes $\xi'$
    and equals $i_g\Phi_{i'}$ for some $g$ in $F$.
    By twisted equivariance, we have $g \in \stab(\xi')$.
    Since $c \in \mathbb{F}$, we have
    \[  \Psi_j = i_c\Psi_ji_c^{-1} = i_c\Phi^k_ii_c^{-1} = (i_g\Phi_{i'})^k 
    = i_{g\Phi_{i'}(g)\cdots\Phi^{k-1}_{i'}(g)}\Psi_j, \]
    from which it follows that 
    \[  \Psi_j(g) = \Phi_{i'}^k(g) = g\Phi_{i'}(g)\cdots \Phi^{k}_{i'}(g) = g.\]
    But since $\stab(\xi') \cap \fix(\Psi_j)$ is trivial,
    we conclude that $g = 1$, proving the claim.
\end{proof}

The following lemma is an immediate consequence of \Cref{abinequality} and \Cref{hatrinequality}.

\begin{lem}
    Suppose that $\varphi \in \out(F,\mathscr{A})$ has a rotationless iterate $\psi = \varphi^k$
    and that each $A \in \mathscr{A}$ is finitely generated.
    Then $j(\varphi) \le j(\psi)$.\hfill\qedsymbol
\end{lem}

\paragraph{An Euler characteristic for graphs of groups.}
If $\mathcal{G}$ is a finite graph of groups with trivial edge groups,
we define the \emph{(negative) Euler characteristic} $\chi^-(\mathcal{G})$ of $\mathcal{G}$
to be the number of edges of $G$ minus the number of vertices with trivial vertex group.
Note that $\mathcal{G}$ defines a splitting of its fundamental group
as a free product $B_1*\cdots *B_m*F_\ell$.
The ordinary negative Euler characteristic of $G$ calculates the quantity $\ell - 1$.
In our definition, the $m$ vertices which have nontrivial vertex group are not counted,
so we see that $\chi^-(\mathcal{G}) = m + \ell - 1$.

\paragraph{The core filtration.}
If $f\colon \mathcal{G} \to \mathcal{G}$ is a CT,
property \hyperlink{Filtration}{(Filtration)} says that the core of a filtration element $G_r$
is a filtration element unless $H_r$ is the bottom half of a dihedral pair.
The \emph{core filtration}
\[  \varnothing = G_0 = G_{\ell_0} \subset G_{\ell_1} \subset \ldots \subset G_{\ell_M} = G_m = G \]
is the coarsening of the filtration for $f\colon \mathcal{G} \to \mathcal{G}$
defined by only including those elements that are their own core.
By \hyperlink{AlmostPeriodicEdges}{(Almost Periodic Edges)} and \hyperlink{Filtration}{(Filtration)},
we have $\ell_1 = 1$ or $2$.
The $i$th stratum of the core filtration, $H^c_{\ell_i}$ is
\[  H^c_{\ell_i} = \bigcup_{j=\ell_{i-1}+1}^{\ell_i} H_j. \]
The change in negative Euler characteristic is 
$\Delta_i\chi^- = \chi^{-}(G_{\ell_i}) - \chi^-(G_{\ell_{i-1}})$.

\paragraph{The index of a finite-type graph of groups.}
Let $f\colon \mathcal{G} \to \mathcal{G}$ be a CT,
and suppose that $\mathcal{H}$ is a graph of groups equipped with an immersion 
$\mathcal{H} \to \mathcal{G}$ (so edge groups of $\mathcal{H}$ are trivial).
We say that $\mathcal{H}$ is of \emph{finite type}
if it has finitely many connected components,
each of which is a finite graph of groups with finitely many vertices marked
and finitely many infinite rays attached;
in particular there are only finitely many nontrivial vertex groups in each component.
Let $\mathcal{H}_1,\ldots,\mathcal{H}_\ell$ be the components of $\mathcal{H}$.
Let $\alpha(\mathcal{H}_i)$ denote the number of ends (infinite rays) of $\mathcal{H}_i$,
and let $b(\mathcal{H}_i)$ denote the number of rays mapping to eigenrays in $\mathcal{G}$.
Let $c(\mathcal{H}_i)$ be the number of marked vertices, which we assume to have trivial vertex group.
Define the \emph{index} of $\mathcal{H}_i$ to be the quantity
\[  j(\mathcal{H}_i) = \rank(\pi_1(\mathcal{H}_i)) + \frac{1}{2}a(\mathcal{H}_i) 
+ \frac{1}{2}b(\mathcal{H}_i) + c(\mathcal{H}_i) - 1. \]
Again, we use the convention that $\rank(C_2*C_2) = 1$.
Define $j(\mathcal{H}) = \sum_{i=1}^\ell j(\mathcal{H}_i)$.

\paragraph{Almost Nielsen paths in CTs.}
Let $f\colon \mathcal{G} \to \mathcal{G}$ be a CT representing $\psi \in \out(F,\mathscr{A})$.
Each almost Nielsen path with endpoints at vertices
has a complete splitting into almost fixed edges and indivisible almost Nielsen paths.
There are two kinds of indivisible almost Nielsen paths $\rho$.
The first possibility, by \hyperlink{NEGAlmostNielsenPaths}{(NEG Almost Nielsen Paths)},
is that $E$ is a linear or dihedral linear edge with axis $w_E$,
and $\rho = gEw_E^k\bar Eh$ for some $k \ne 0$ and vertex group elements $g$ and $h$.
The other possibility is that $\rho$ is an indivisible almost Nielsen path 
of exponentially growing height $r$.
By \Cref{egalmostnielsenpathsproperties}, up to equivalence $\rho$ and $\bar\rho$
are the only indivisible almost Nielsen paths of height $r$
and the initial edges of $\rho$ and $\bar\rho$ are distinct.
Observe that by exhausting the possibilities, 
every almost Nielsen path has the property that its initial and terminal directions are almost fixed.
It follows that a non-exceptional almost Nielsen class based at a vertex $v$ with
nontrivial vertex group corresponds to at least one almost fixed direction based at $v$.

\paragraph{The graph of groups $\mathcal{S}_N(f)$.}
Let $f\colon \mathcal{G} \to \mathcal{G}$ be a CT representing $\psi \in \out(F,\mathscr{A})$.
In the proof of \Cref{preciseindexthm}, we build a finite-type graph of groups $\mathcal{S}_N(f)$,
analogous to that considered by Feighn--Handel in \cite[Section 12]{FeighnHandelAlg}
equipped with an immersion $s\colon \mathcal{S}_N(f) \to \mathcal{G}$
such that $j(\mathcal{S}_N(f)) = j(\psi)$.
Each nontrivial almost Nielsen path with principal endpoints
is equivalent to an almost Nielsen path lifting to $\mathcal{S}_N(f)$,
and each ray $R_E$ for $E \in \mathcal{E}$ lifts to $\mathcal{S}_N(f)$.
In the proof we build up $\mathcal{S}_N(f)$ in stages by induction up through the core filtration.
Here we describe its construction without reference to the filtration.

Since each non-exceptional almost Nielsen class containing $v$ corresponds
to at least one oriented edge $E$ beginning at $v$ and determining an almost fixed direction at $v$,
the set of such oriented edges---modulo the equivalence relation that says $E \sim E'$
if there exists $g$ and $h \in \mathcal{G}_v$ such that $Df(1,E) = (g,E)$
and $g^{-1}Df(h,E') = (h,E')$---is in bijection
with the set of non-exceptional almost Nielsen classes containing $v$.
Let $E_v$ be the set of equivalence classes of edges $E$ as above
and begin with $\mathcal{S}_N(f)$ equal to the following set of vertices
\[  \{v_{[E]} : v \text{ is a principal vertex of $G$ and } [E] \in E_v\}. \]
Call these vertices the \emph{principal vertices} of $\mathcal{S}_N(f)$
Since $f$ is rotationless, if $\mathcal{G}_v$ is finite, then there is a single equivalence class.
For each $[E] \in E_v$, choose a preferred representative edge $E$.
Then $Df(1,E) = (g,E)$ for some $g \in \mathcal{G}_v$,
and set the vertex group of $v_{[E]}$ equal to $\fix(i_{g^{-1}} f_v)$.
If $\mathcal{G}_v$ is infinite and this subgroup is trivial, 
then mark $v_{[E]}$.
Define $Df_{[E]}$ to be the map $d \mapsto g^{-1}Df(d)$ for directions $d$ based at $v$.
By assumption, if $E' \sim E$, then there is a fixed direction for $Df_{[E]}$ 
in the $\mathcal{G}_v$-orbit $E'$.
Observe as well that if the direction $d$ is fixed for $Df_{[E]}$
and $x$ belongs to $\fix(i_{g^{-1}} f_v)$,
then
\[
    Df_{[E]}(x.d) = g^{-1}f_v(x) .Df(d) = x. g^{-1}Df(d) = x.d.
\]
Conversely if $d = (x,E')$ and $(h,E')$ are both fixed by $Df_{[E]}$
and $Df(1,E') = (g_{E'},E')$,
then we have
\[  g^{-1}f_v(xh^{-1}) g =  g^{-1}f_v(x)g_{E'}g_{E'}^{-1}f_v(h^{-1}) g = xh^{-1}. \]

If $E$ is an almost fixed edge with principal endpoints $v$ and $w$,
we attach an edge (abusing notation, call it $E$)
with initial endpoint $v_{[E]}$ and terminal endpoint $w_{[\bar E]}$.
As a map of graphs, the immersion $s\colon \mathcal{S}_N(f) \to \mathcal{G}$ sends this new edge to $E$;
as a map of graphs of groups, define the map on the directions determined by $E$
so that they are sent to fixed directions for the map $Df_{[E]}$.

Suppose now that $E$ is a linear or dihedral linear edge with axis $w$ and initial vertex $v$.
To $v_{[E]}$ we attach a \emph{lollipop:} a graph consisting of two edges sharing a vertex,
exactly one of which forms a loop.
The loop maps to $w$, and the other edge to $E$.
Map the initial direction determined by the edge mapping to $E$ to a fixed direction for the map $Df_{[E]}$.

Suppose next that $\rho$ is an indivisible almost Nielsen path of exponentially growing height
with initial edge $E$, initial vertex $v$, terminal edge $\bar E'$ and terminal vertex $w$.
We attach an edge mapping to $\mu$ with initial vertex $v_{[E]}$ and terminal vertex $w_{[E']}$.
Map the initial direction of $\mu$ to a fixed direction for the map $Df_{[E]}$
and the initial direction of $\bar\mu$ to a fixed direction for the map $Df_{[E']}$.

Finally, suppose that $E \in \mathcal{E}$ with initial vertex $v$.
If $E$ belongs to a non-exponentially growing stratum,
or $E$ belongs to an exponentially growing stratum and is not the initial edge 
of an indivisible almost Nielsen path,
attach an infinite ray mapping to $R_E$ to $v_{[E]}$
so that the initial direction of this ray is mapped to a fixed direction for the map $Df_{[E]}$.
If $E$ is the initial edge of an indivisible almost Nielsen path $\rho$
and $E'$ is the initial edge of $\bar\rho$,
then by \Cref{rayswithcommonendpoints},
$R_E$ and $R_{E'}$ have a common terminal subray $R_{E,E'}$.
If we write $\rho = \alpha\beta$ as in \Cref{finitelymanynielsenpaths},
we have $R_E = \alpha R_{E,E'}$ and $R_{E'} = \bar\beta R_{E,E'}$.
In this case subdivide the edge mapping to $\rho$ into two edges,
one mapping to $\alpha$ and the other to $\beta$,
and attach $R_{E,E'}$ at the newly added vertex.
This completes the construction of $\mathcal{S}_N(f)$;
one checks directly that the map $s\colon \mathcal{S}_N(f) \to \mathcal{G}$ 
is an immersion of graphs of groups, 
and that if a vertex of $\mathcal{S}_N(f)$
has valence one and trivial vertex group,
then that vertex is marked.
Marked vertices of $\mathcal{S}_N(f)$ map to vertices of $\mathcal{G}$ with infinite vertex group.

\begin{lem}
    Every tight path in $\mathcal{S}_N(f)$ with endpoints at principal vertices of $\mathcal{S}_N(f)$
    projects to an almost Nielsen path in $\mathcal{G}$.
    Conversely, if $\sigma$ is an almost Nielsen path in $\mathcal{G}$ 
    with endpoints at principal vertices,
    then a path equivalent to $\sigma$ lifts to $\mathcal{S}_N(f)$.
    The lift $\tilde\sigma$ is closed if and only if 
    $\sigma$ is closed and a path $\sigma'$ equivalent to $\sigma$
    satisfies that $\sigma'\sigma'$ is an almost Nielsen  path.
\end{lem}

\begin{proof}
    Let $\sigma$ be a tight path in $\mathcal{S}_N(f)$ with endpoints at principal vertices
    and let $\sigma = \sigma_1\ldots \sigma_m$ be a subdivision of $\sigma$ 
    at the principal vertices of $\mathcal{S}_N(f)$.
    Then each $\sigma_i$ projects to an almost fixed edge or an indivisible almost Nielsen path,
    so we need only check that the concatenation of two subpaths projects to an almost Nielsen path.
    By assumption, the initial directions of $\bar\sigma_i$ and $\sigma_{i+1}$,
    which are based at a principal vertex $v_{[E]}$
    are mapped to fixed directions for the map $Df_{[E]} = g^{-1}Df$.
    It follows that $f_\sharp(\sigma_1\sigma_2) = h\sigma_1 g^{-1}g \sigma_2 h'$
    for vertex group elements $h$ and $h'$,
    and we conclude that the concatenation is an almost Nielsen path.

    Now suppose that $\sigma$ is an almost Nielsen path in $\mathcal{G}$
    with endpoints at principal vertices,
    and let $\sigma = \sigma_1\ldots \sigma_m$ be the complete splitting of $\sigma$
    into almost fixed edges and indivisible almost Nielsen paths.
    It is clear that if $m = 1$, then a path equivalent to $\sigma = \sigma_1$
    lifts to $\mathcal{S}_N(f)$, and that this lift is unique up to equivalence.
    We only need check that if the concatenation $\sigma_i\sigma_{i+1}$ is an almost Nielsen path,
    and $\tilde\sigma_i$ and $\tilde\sigma_{i+1}$ are lifts of $\sigma_i$ and $\sigma_{i+1}$
    with common vertex $v_{[E]}$ and associated map $Df_{[E]} = g^{-1}Df$,
    then there is some vertex group element $h\in \fix(i_{g^{-1}}f_v)$ 
    such that $\tilde\sigma_i h\tilde\sigma_{i+1}$
    lifts $\sigma_i\sigma_{i+1}$.
    Indeed, if $\sigma'_i$  and $\sigma'_{i+1}$ are the projections
    of $\tilde\sigma_i$ and $\tilde\sigma_{i+1}$ respectively,
    then we have $\sigma_i\sigma_{i+1} = g'\sigma'_ig''\sigma'_{i+1}g'''$,
    and 
    \begin{align*}  
        h'\sigma_i\sigma_{i+1}h'' = f_\sharp(\sigma_i\sigma_{i+1}) =
        f_\sharp(g')f_\sharp(\sigma'_i)f_v(g'')f_\sharp(\sigma'_{i+1})f_\sharp(g''') \\
        = f_\sharp(g')k f_\sharp(\sigma'_i)g^{-1}f_v(g'')g f_\sharp(\sigma'_{i+1})k'f_\sharp(g''')
        = h'g'\sigma'_ig''\sigma'_{i+1}g'''h'' 
    \end{align*}
    since the directions of $\bar\sigma'_i$ and $\sigma'_{i+1}$ are fixed by $Df_{[E]}$.
    From this we conclude that $g'' \in \fix(i_{g^{-1}}f_v)$, as required.    
    It is clear that if $\sigma$ is not closed, then its lift is not closed.
    Suppose that $\sigma$ is closed and that $\sigma\sigma$ is an almost Nielsen path.
    The argument in the previous paragraph shows that a path equivalent to $\sigma\sigma$ 
    has a lift to $\mathcal{S}_N(f)$, and the uniqueness of the lift of $\sigma$
    implies that the lift of $\sigma$ must be closed.
    Conversely, if $\sigma$  has a closed lift $\tilde\sigma$,
    then $\tilde\sigma\tilde\sigma$ is a tight path with endpoints at principal vertices,
    so it projects to an almost Nielsen path of the form $\sigma'\sigma'$
    for some path $\sigma'$ equivalent to $\sigma$.
\end{proof}

\begin{lem}[cf.~Lemma 12.4 of \cite{FeighnHandelAlg}]
    Let $\tilde f\colon \Gamma \to \Gamma$ be a principal lift fixing a vertex $\tilde v$
    and the direction of an oriented edge $\tilde E$ based at $\tilde v$.
    Let $S$ be the component of $\mathcal{S}_N(f)$ containing $v_{[E]}$.
    There is a lift of the immersion $s|_S \colon S \to \mathcal{G}$
    to an embedding of Bass--Serre trees $\tilde s\colon \tilde S \to \Gamma$
    such that $\tilde  E$ is in the image of $\tilde s$.
    The image of $\tilde s$ is the smallest subtree of $\Gamma$ whose limit set contains $\fix_N(\hat f)$.
\end{lem}

\begin{proof}
    Since $s|_S$ is an immersion, any lift $\tilde s$ is an embedding by \cite[Proposition 2.7]{Bass}.
    By post-composing any lift with an element of $F$, 
    we may arrange that $\tilde E$ is in the image of $\tilde s$.
    Suppose that the lift $\tilde f$ corresponds to the automorphism $\Phi$.
    We show first that $\tilde s(\tilde S)$ contains the $\fix(\Phi)$-minimal subtree of $\Gamma$.
    Given $c \in \fix(\Phi)$, 
    we have that $\tilde w = T_c(\tilde v)$ is fixed by $\tilde f$ by \Cref{boundarybasics}.
    The image of the tight path from $\tilde v$ to $\tilde w$ in $\mathcal{G}$
    is an almost Nielsen path $\sigma$ in the almost Nielsen class determined by $v_{[E]}$.
    By the previous lemma, we can lift all such almost Nielsen paths to $S$.
    If $c$ is peripheral, it follows that the fixed point of $T_c$ is contained in the image of $\tilde s$.
    If $c$ is non-peripheral, it follows inductively 
    that the axis of $T_c$ is contained in the image of $\tilde s$.
    It follows that the $\fix(\Phi)$-minimal subtree of $\Gamma$ is contained in $\tilde s(\tilde S)$.

    If $P \in \fix_N(\hat f) \cap V_\infty(F,\mathscr{A})$,
    then the tight path $\tilde\sigma$ from $\tilde v$ to $P$ projects to an almost Nielsen path
    which we can lift to $S$ and hence $\tilde\sigma$ is contained in the image of $\tilde S$.
    If $P \in \fix_N(\hat f) \cap \partial_\infty(F,\mathscr{A})$ is not isolated,
    then it is in the boundary of the $\fix(\Phi)$-minimal subtree and thus the
    ray from $\tilde v$ to $P$ is contained in the image of $\tilde S$.
    Finally if $P \in \fix_+(\hat f)$, then by \Cref{CTrays},
    there is an edge $\tilde E'$ projecting into $\mathcal{E}$ such that the ray from $\tilde E'$ to $P$
    is fixed-point free.
    The unique tight path from $\tilde v$ to the initial vertex of $\tilde E'$ 
    projects to an almost Nielsen path
    which we can lift to $S$.
    Since the direction determined by $\tilde E'$ belongs to the almost Nielsen class 
    determined by $v_{[E]}$,
    the ray $\tilde R_{\tilde E'}$ is in the image of $\tilde s$.
    This shows that the limit set of the image of $\tilde s$ contains $\fix_N(\hat f)$.
    Conversely, suppose  $\tilde R$ is a ray in the image of $\tilde s$,
    and project it to a tight ray $R$ in $\mathcal{S}_N(f)$.
    If the ray $R$ does not have a common terminal subray with some $R_E$,
    then we may subdivide at principal vertices of $\mathcal{S}_N(f)$
    and write $R$ as an almost Nielsen concatenation of almost Nielsen paths,
    and we see that $\tilde R$ ends in the $\fix(\Phi)$-minimal subtree,
    so its endpoint is in the limit set of $\fix(\Phi)$.
    Otherwise we may write $R$ as the concatenation of an almost Nielsen path with some $R_E$,
    and we see that $\tilde R$ ends in $\fix_+(\hat f)$.
\end{proof}

\begin{cor}
    If $\psi \in \out(F,\mathscr{A})$ is rotationless and $f\colon \mathcal{G} \to \mathcal{G}$
    is a CT representing $\psi$,
    then $j(\mathcal{S}_N(f)) = j(\psi)$.\hfill\qedsymbol
\end{cor}

\begin{ex}
    \label{CT1}
    Consider the following CT $f\colon \mathcal{G} \to \mathcal{G}$,
    where $\mathcal{G}$ is the graph of groups in \Cref{CTex1}.
    The map on the $C_2$ vertex groups is the unique isomorphism,
    and we write for example $\hat A$ to represent the path $\bar AgA$,
    where $g$ is the nontrivial element of $C_2$.
    Note that none of the $C_2$ vertices are principal.
    \[  f \begin{dcases}
        A \mapsto E\hat D\hat C\hat D \hat B \hat A \\
        B \mapsto A \hat B \\
        C \mapsto B\hat A\hat B\\
        D \mapsto C \hat D \\
        E \mapsto D \hat C \hat D \hat E.
    \end{dcases} \]
    \begin{figure}[ht!]
        \centering
	    \def\svgwidth{\columnwidth}
	        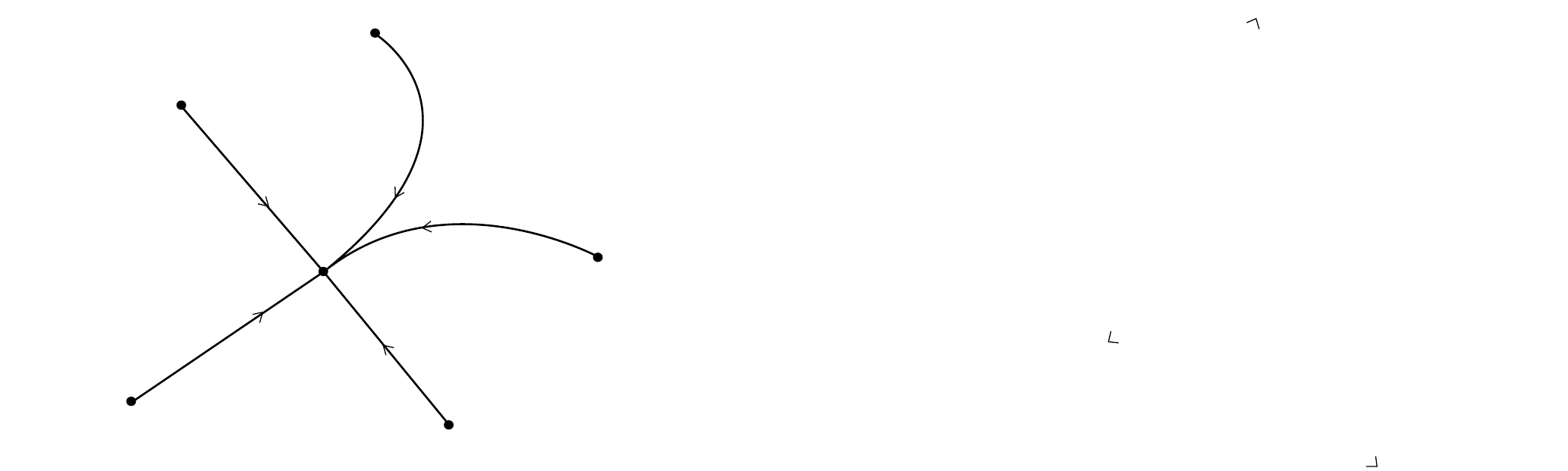
	
        \caption{The CT $f\colon \mathcal{G} \to \mathcal{G}$ and the Stallings graph $\mathcal{S}_N(f)$.}
        \label{CTex1}
    \end{figure}
    The turn $(\bar B,\bar C)$ is the unique illegal turn of $\mathcal{G}$
    and there is an indivisible (almost) Nielsen path $\rho = \hat A\hat B\hat C \hat D\hat E$
    and the vertex with trivial vertex group is principal.
    To this vertex we attach the rays 
    $R_{\bar B} = \hat B\hat A\hat B\hat A \hat B \hat D\hat C \hat D\hat E\ldots$,
    $R_{\hat D} = \hat D\hat C\hat D\hat B \hat A \hat B\hat A \hat B\ldots$
    and the indivisible almost Nielsen path $\rho$, which forms a loop.
    We write $\alpha = \hat A\hat B$ and $\beta = \hat C\hat D\hat E$,
    subdivide $\rho$ accordingly
    and we attach the ray 
    $R_{\bar A,\bar E} = \hat D\hat C\hat D\hat E\hat D\hat C\hat D\hat B\hat A\ldots$
    to the added vertex.
    See \Cref{CTex1}
    We have the index formula $j(f) = \frac{3}{2} + 1 - 1 = \frac{3}{2}$.
\end{ex}

\begin{ex}
    \label{CT2}
    Consider the following CT $f\colon \mathcal{G} \to \mathcal{G}$,
    where $\mathcal{G}$ is the graph of groups in \Cref{CTex2}.
    The map on the $C_2$ vertex groups is again the unique isomorphism,
    and we continue to write $\hat A$ to represent the path $\bar AgA$,
    where $g$ is the nontrivial element of $C_2$.
    The $C_2$ vertex group incident to the edge $C$ is principal,
    while the edges $A$ and $B$ form a dihedral pair.
    \[  f \begin{dcases}
        A \mapsto B \\
        B \mapsto A \\
        C \mapsto C\hat A.
    \end{dcases}\]
    \begin{figure}[ht!]
        \centering
	    \def\svgwidth{\columnwidth}
	        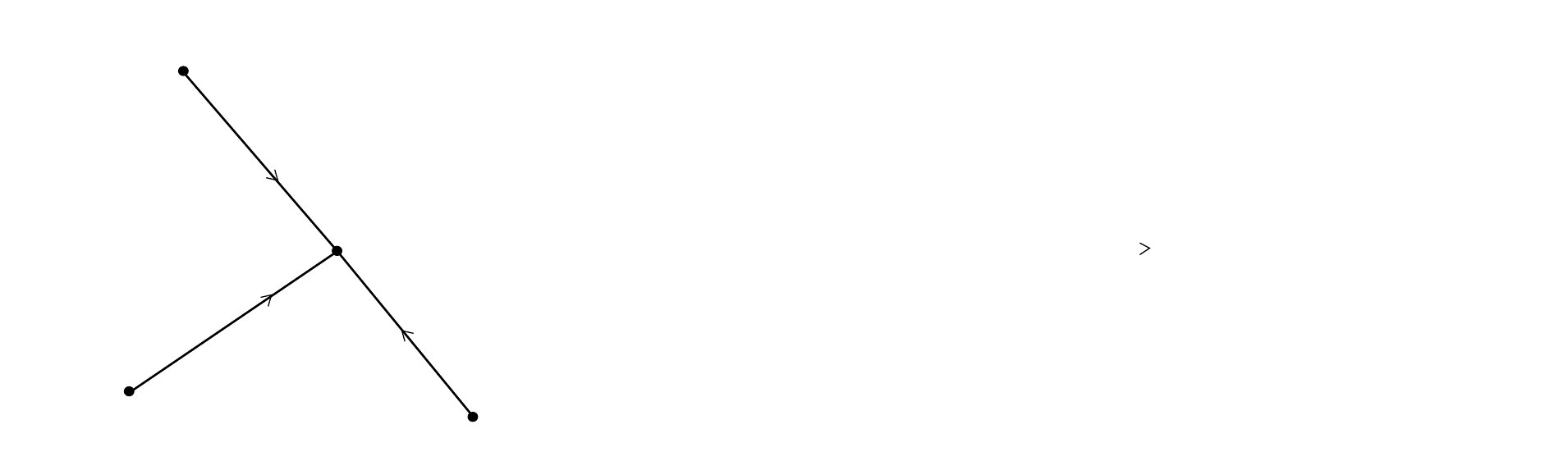
	
        \caption{The CT $f\colon \mathcal{G} \to \mathcal{G}$ and the Stallings graph $\mathcal{S}_N(f)$.}
        \label{CTex2}
    \end{figure}
    The edge $C$ is a dihedral linear edge,
    so to its initial vertex (which has $C_2$ vertex group in $\mathcal{S}_N(f)$)
    we attach a lollipop.
    We have the index formula $j(f) = 2 - 1 = 1$.
\end{ex}

\begin{ex}
    \label{CT3}
    Consider the following CT $f\colon \mathcal{G} \to \mathcal{G}$
    where $\mathcal{G}$ is the graph of groups in \Cref{CTex3}.
    The groups $A$ and $A'$ are infinite (but somewhat arbitrary).
    Choose the automorphisms $f_A$ and $f_{A'}$ to have no periodic points.
    The corresponding vertices form a single almost Nielsen class.
    \[  f \begin{dcases}
        E \mapsto E \\
        B \mapsto BEg\bar E g'
    \end{dcases}\]
    \begin{figure}[ht!]
        \centering
	    \def\svgwidth{\columnwidth}
\begingroup%
  \makeatletter%
  \providecommand\color[2][]{%
    \errmessage{(Inkscape) Color is used for the text in Inkscape, but the package 'color.sty' is not loaded}%
    \renewcommand\color[2][]{}%
  }%
  \providecommand\transparent[1]{%
    \errmessage{(Inkscape) Transparency is used (non-zero) for the text in Inkscape, but the package 'transparent.sty' is not loaded}%
    \renewcommand\transparent[1]{}%
  }%
  \providecommand\rotatebox[2]{#2}%
  \newcommand*\fsize{\dimexpr\f@size pt\relax}%
  \newcommand*\lineheight[1]{\fontsize{\fsize}{#1\fsize}\selectfont}%
  \ifx\svgwidth\undefined%
    \setlength{\unitlength}{496.06299213bp}%
    \ifx\svgscale\undefined%
      \relax%
    \else%
      \setlength{\unitlength}{\unitlength * \real{\svgscale}}%
    \fi%
  \else%
    \setlength{\unitlength}{\svgwidth}%
  \fi%
  \global\let\svgwidth\undefined%
  \global\let\svgscale\undefined%
  \makeatother%
  \begin{picture}(1,0.2)%
    \lineheight{1}%
    \setlength\tabcolsep{0pt}%
    \put(0,0){\includegraphics[width=\unitlength,page=1]{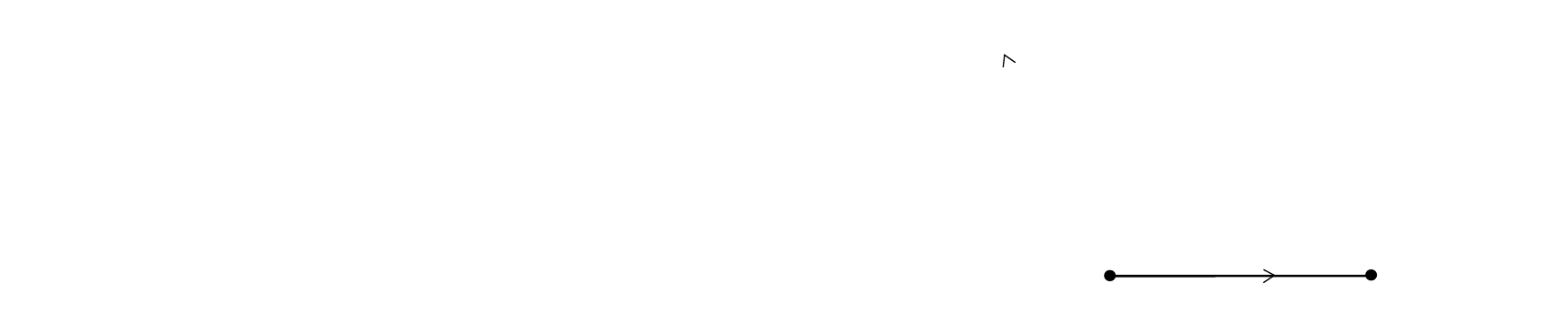}}%
    \put(0.77766262,0.03539451){\color[rgb]{0,0,0}\makebox(0,0)[lt]{\lineheight{1.25}\smash{\begin{tabular}[t]{l}$E$\end{tabular}}}}%
    \put(0.68104244,0.10103422){\color[rgb]{0,0,0}\makebox(0,0)[lt]{\lineheight{1.25}\smash{\begin{tabular}[t]{l}$R_B$\end{tabular}}}}%
    \put(0,0){\includegraphics[width=\unitlength,page=2]{CT3.pdf}}%
    \put(0.23485384,0.05348733){\color[rgb]{0,0,0}\makebox(0,0)[lt]{\lineheight{1.25}\smash{\begin{tabular}[t]{l}$E$\end{tabular}}}}%
    \put(0.09917113,0.13391533){\color[rgb]{0,0,0}\makebox(0,0)[lt]{\lineheight{1.25}\smash{\begin{tabular}[t]{l}$B$\end{tabular}}}}%
    \put(0,0){\includegraphics[width=\unitlength,page=3]{CT3.pdf}}%
    \put(0.33291624,0.08874927){\color[rgb]{0,0,0}\makebox(0,0)[lt]{\lineheight{1.25}\smash{\begin{tabular}[t]{l}$A'$\end{tabular}}}}%
    \put(0.16867254,0.08767071){\color[rgb]{0,0,0}\makebox(0,0)[lt]{\lineheight{1.25}\smash{\begin{tabular}[t]{l}$A$\end{tabular}}}}%
  \end{picture}%
\endgroup%

        \caption{The CT $f\colon \mathcal{G} \to \mathcal{G}$ and the Stallings graph $\mathcal{S}_N(f)$.}
        \label{CTex3}
    \end{figure}
    To them in $\mathcal{S}_N(f)$ we attach the fixed edge $E$.
    These vertices are marked.
    We also attach the eigenray $R_B = BEg\bar E g'E f_A(g) \bar E f_{A'}(g')\ldots$.
    We have the index formula $j(f) = 2 + 1 - 1 = 2$, which is the maximum possible.
\end{ex}

\begin{proof}[Proof of \Cref{preciseindexthm}]
    Let $f\colon \mathcal{G} \to \mathcal{G}$ be a CT representing $\psi \in \out(F,\mathscr{A})$.
    We construct $\mathcal{S}_N(f)$ by inducting up through the core filtration of $\mathcal{G}$,
    following the outline of \cite[Proposition 15.14]{FeighnHandelAlg}.
    For $i$ satisfying $0 \le i \le M$,
    let $\mathcal{S}_N(i)$ denote the subgraph of groups of $\mathcal{S}_N(f)$
    constructed by the $i$th term of the core filtration $G_{\ell_i}$.
    We shall prove that
    \[  j(\mathcal{S}_N(i)) \le \chi^-(G_{\ell_i}), \]
    from which the theorem follows.
    We prove  this inequality by induction.
    The case $i = 0$ holds trivially, since $\mathcal{S}_N(0)$ and $G_{\ell_0}$ are empty.
    Let $\Delta_k j = j(\mathcal{S}_N(k)) - j(\mathcal{S}_N(k-1))$.
    Suppose that the inequality holds for $k-1$.
    We prove that $\Delta_k j \le \Delta_k\chi^{-}$,
    so that the inequality holds for $k$ as well.
    Throughout we will in fact slightly overestimate $\Delta_kj$
    by not considering the special case of $C_2*C_2$.
    The proof has two cases.

    \paragraph{Case 1.}
    Suppose that $H^c_{\ell_k}$ contains no exponentially growing strata.
    We claim that one of the following occurs.
    \begin{enumerate}[label=(\alph*)]
        \item We have $\ell_k = \ell_{k-1} + 1$ or $\ell_{k-1} + 2$.
            The stratum $H^c_{\ell_k}$ is disjoint from $G_{\ell_{k-1}}$
            and forms a dihedral pair.
        \item We have $\ell_k = \ell_{k-1} + 1$.
            The unique edge in $H^c_{\ell_k}$ is disjoint from $G_{\ell_{k-1}}$ 
            and is fixed or almost fixed with principal endpoints.
            Either it forms a loop,
            or both incident vertices have nontrivial vertex group.
        \item We have $\ell_k = \ell_{k-1} + 1$.
            The unique edge in $H^c_{\ell_k}$ either has both endpoints contained in $G_{\ell_{k-1}}$
            or one endpoint contained in $G_{\ell_{k-1}}$ and the other has nontrivial vertex group.
        \item We have $\ell_k = \ell_{k-1}+ 2$.
            The two edges in $H^c_{\ell_k}$ share an initial endpoint,
            are not almost periodic, and have terminal endpoints in $G_{\ell_{k-1}}$.
            (Note that the initial endpoint has trivial vertex group,
            otherwise we would be in case (c).)
    \end{enumerate}
    The proof is similar to \cite[Lemma 8.3]{FeighnHandelAbelian}.
    First suppose that some edge in $H^c_{\ell_k}$ belongs to a dihedral pair.
    Then by \hyperlink{Filtration}{(Filtration)} if the dihedral pair is fixed,
    and by the definition of a dihedral pair otherwise, we see that we are in case (a).
    So suppose that no edge in $H^c_{\ell_k}$ belongs to a dihedral pair.
    Because $H^c_{\ell_k}$ contains no exponentially growing strata,
    it contains no zero strata by \hyperlink{ZeroStrata}{(Zero Strata)}.
    Since CTs satisfy the conclusions of \Cref{improvedrelativetraintrack},
    the initial endpoint of each non-almost periodic edge $E_j$ in $H^c_{\ell_k}$
    is principal, and thus $E_j$ is the only edge in $H_j$ for some $j$
    satisfying $\ell_{k-1} + 1 \le j \le \ell_k$.
    The same holds by \hyperlink{AlmostPeriodicEdges}{(Almost Periodic Edges)}
    for almost periodic (hence almost fixed) edges,
    so each  $H_j$ is a single edge $E_j$.
    If some $E_j$ is almost fixed, then either (b) or (c) holds
    by \hyperlink{AlmostPeriodicEdges}{(Almost Periodic Edges)}.
    If each $E_j$ is not fixed or almost fixed and (c) does not hold,
    then $E_j$ adds a valence-one vertex with trivial vertex group to $G_{\ell_{k-1}}$,
    and the terminal vertex of $E_j$ belongs to $G_{\ell_{k-1}}$ by \Cref{negctsplitting}.
    Each new valence-one vertex must be an endpoint of $E_{\ell_k}$, so we are in case (d).
    We analyze each subcase.

    \begin{enumerate}[label=(\alph*)]
        \item Nothing is added to $\mathcal{S}_N(k)$ because no vertex of the dihedral pair is principal,
            so $\Delta_kj = 0$ while $\Delta_k\chi^- = 1$.
        \item Suppose that $E_{\ell_k}$ forms a loop with trivial incident vertex group.
            Then $\Delta_kj = \Delta_k\chi^- = 0$.
            In all other cases $\Delta_k\xi^- = 1$.
            The nontrivial vertex group(s) contribute
            new principal vertices to $\mathcal{S}_N(k)$ that either have nontrivial vertex group
            or are marked. The unique new edge in $\mathcal{S}_N(k)$ either forms a loop or does not.
            In all cases we see that $\Delta_kj = 1$.
        \item In this case we always have $\Delta_k\chi^- = 1$ and $\Delta_kj = 1$.
            There are several possibilities for the change to $\mathcal{S}_N(k)$,
            but they may be summarized as follows:
            at most two new vertices are added to $\mathcal{S}_N(k)$,
            each of which is either marked or has nontrivial vertex group.
            If $E_{\ell_k}$ is almost fixed, one new edge is added to $\mathcal{S}_N(k)$;
            if there are any new vertices, then the new edge is incident to them.
            If $E_{\ell_k}$ is linear, then we add a lollipop;
            the valence-one vertex of the lollipop is added to the new vertex, if there is one.
            If $E_{\ell_k}$ is nonlinear,
            then we add the eigenray $E_{\ell_k}$,
            attached to the new vertex, if there is one.
            In all cases, we see that $\Delta_kj = 1$.
        \item In this case we have $\Delta_k\chi^- = 1$.
            The vertex in $\mathcal{G}_{\ell_k}$ not contained in $G_{\ell_{k-1}}$ is principal,
            so determines a new component of $\mathcal{S}_N(k)$.
            This vertex has trivial vertex group and is unmarked.
            To this new vertex we attach a lollipop for each  linear edge,
            and an eigenray for each nonlinear edge, for a total of three possible combinations.
            In all cases $\Delta_kj = 1$.
    \end{enumerate}

    \paragraph{Case  2.}
    The argument again is similar to \cite[Lemma 8.3]{FeighnHandelAlg}.
    Suppose that $H^c_{\ell_k}$ contains an exponentially growing stratum $H_r$.
    By the argument in Case 1, we see that $H^c_{\ell_k}$ contains no dihedral pairs.
    Therefore in this case the core of a filtration element is a filtration element,
    and because $G_{r-1}$ does not carry the unique attracting lamination $\Lambda^+$
    associated to $H_r$,
    we see that $G_r$ is its own core.
    Thus $H_{\ell_k}$ is the unique exponentially growing stratum in $H^c_{\ell_k}$.
    We claim that there exists some integer $u_k$ satisfying
    $\ell_{k-1} \le u_k < \ell_k$ such that the following hold.
    \begin{enumerate}[label=(\alph*)]
        \item For $j$ satisfying $\ell_{k-1} < j \le u_k$,
            the stratum $H_j$ is a single edge which is not almost fixed
            whose terminal vertex is in $G_{\ell_{k-1}}$ and whose initial vertex
            has trivial vertex group and valence one in $G_{u_k}$.
        \item For $j$ satisfying $u_k < j < \ell_k$, the stratum $H_j$ is a zero stratum.
    \end{enumerate}

    The existence $u_k$ and that (b) holds follows from \hyperlink{ZeroStrata}{(Zero Strata)}.
    That (a) holds follows from \hyperlink{AlmostPeriodicEdges}{(Almost Periodic  Edges)}
    and \Cref{negctsplitting}.

    For  $j$ as in item (a) above, the contribution to $\mathcal{S}_N(k)$
    is the addition of a new component, with new vertex unmarked and with trivial vertex group,
    to which we attach either a lollipop or an eigenray.
    Thus in each case $\Delta_kj = 0$, while $\Delta_k\chi^- = 0$.
    This completes the analysis up to $G_{u_k}$.

    To calculate $\Delta\chi^-$, it will be useful to note that $\chi^-(\mathcal{G})$ is equal to the sum
    over the vertices $v$ of $G$ of the quantity $\frac{1}{2}\operatorname{valence}(v) - 1$
    if $v$ has trivial vertex group and $\frac{1}{2}\operatorname{valence}(v)$ otherwise.
    Note that each edge contributes valence to two vertices (which may be equal).

    For each vertex  $v \in H_{\ell_k}$, let $\Delta_kj(v)$ and $\Delta_k\chi^-(v)$
    be the contributions to $\Delta_kj$ and $\Delta_k\chi$ from $v$ that are not already considered.
    If $v$ is principal, let $\kappa(v)$ denote the number of oriented edges of $H_{\ell_k}$
    incident to $v$ that do not determine almost fixed directions at $v$.
    If $v \in H_{\ell_k}$ is not principal,
    then it either has trivial vertex group or vertex group $C_2$.
    In the former case let $\kappa(v) = \operatorname{valence}(v) - 2$,
    and in the latter case let $\kappa(v) = \operatorname{valence}(v)$.

    Suppose first that there no indivisible almost Nielsen paths of height $\ell_k$.
    If the vertex $v$ is not principal, then $\Delta_kj(v) = 0$.
    Such a vertex has link in $G_{\ell_k}$ entirely contained in $H^z_{\ell_k}$ and is thus new.
    We have $\Delta_k\chi^-  = \frac{1}{2}\kappa(v)$
    and we have $\Delta_k\chi^-(v) - \Delta_kj(v) = \frac{1}{2}\kappa(v) > 0$.

    If instead the vertex $v$ is principal, let $L(v)$ denote the number of oriented edges of $H_{\ell_k}$
    based at $v$ and contained in $\mathcal{E}$.
    We add $L(v)$ rays to $\mathcal{S}_N(k)$, possibly to new vertices of $\mathcal{S}_N(k)$.
    If $v$ has nontrivial vertex group, then each $v_{[E]}$ in $\mathcal{S}_N(k)$
    has nontrivial vertex group or is marked,
    and we see that $\Delta_kj(v) = \frac{1}{2}L(v)$ and $\Delta_k\chi^- = \frac{1}{2}(L(v) + \kappa(v))$.
    The same calculation holds if $v$ has trivial vertex group 
    and has already been added to $\mathcal{S}_N(k)$.
    Finally if $v$ is new and has trivial vertex  group,
    then $\Delta_kj(v) = \frac{1}{2}L(v) - 1$ and $\Delta_k\chi^-(v) = \frac{1}{2}(L(v) + \kappa(v)) - 1$.
    In all cases we see that $\Delta_k\chi^-(v) - \Delta_kj(v) = \frac{1}{2}\kappa(v) \ge 0$.

    Suppose finally that there is an indivisible almost Nielsen path $\rho$ of height $\ell_k$.
    By \Cref{egalmostnielsenpathsfurtherproperties},
    we have $\ell_k = u_k + 1$.
    Let $w$ and $w'$ be the endpoints of $\rho$.
    By \Cref{egalmostnielsenpathsdistinctendpoints},
    if $w \ne w'$, then at least one of $w$ and $w'$ is new and has trivial vertex group.
    In both cases, the initial edges of $\rho$ and $\bar\rho$ are distinct, say $e$ and $e'$.

    Let $\mathcal{V}$ be the set of vertices of $H_{\ell_k}$ that are not endpoints of $\mu$.
    Each vertex of $\mathcal{V}$ may be handled as in the case without indivisible almost Nielsen paths,
    so we conclude
    \[  \sum_{v\in \mathcal{V}}\Delta_k\chi^-(v) - \sum_{v\in\mathcal{V}}\Delta_kj(v) =
    \sum_{v\in\mathcal{V}} \frac{1}{2}\kappa(v) \ge 0. \]

    Let $\Delta_kj(\rho)$ and $\Delta_k\chi^-(\rho)$ be the contributions 
    to $\Delta_kj$ and $\Delta_k\chi^-$ coming from the endpoints of $\rho$ and not already considered.
    There are two subcases to consider according to whether $w = w'$.
    \begin{enumerate}
        \item Suppose $\rho$ is a closed path based at the vertex $w$.
            In $\mathcal{S}_N(k)$ we possibly add a new vertex, a loop mapping to $\rho$,
            a ray for each $E \in \mathcal{E}$ incident to $w$ other than $e$ and $e'$,
            and then one ray corresponding to $e$ and $e'$.
            The real picture is in fact slightly more complicated:
            the rays contributing to $L(v)$ may lie in distinct almost Nielsen classes,
            in which case multiple vertices mapping to $w$ are added.
            In this case each such vertex either has nontrivial vertex group or is marked.
            Therefore the final calculation is that
            \[  \Delta_kj(\rho) = 1 + \frac{1}{2}(L(w) - 1) - 1 = \frac{1}{2}L(w) - \frac{1}{2} \]
            if $w$ is new and has trivial vertex group
            and $\Delta_kj(\mu) = \frac{1}{2}L(w) + \frac{1}{2}$ otherwise.
            Similarly we have
            \[  \Delta_k\chi^-(\mu) = \frac{1}{2}(L(w) + \kappa(w)) - 1 \]
            if $w$ is new and has trivial vertex group
            and $\Delta_k\chi^-(\mu) = \frac{1}{2}(L(w) + \kappa(w))$ otherwise.
            Thus
            \[  \Delta_k\chi^-(\mu) - \Delta_kj(\mu) \ge \frac{1}{2}\kappa(w) - \frac{1}{2}. \]
            Since there is always at least one illegal turn in $H_{\ell_k}$---for instance,
            the one in $\rho$---and since illegal but nondegenerate turns
            are determined by distinct edges,
            there must be at least one vertex of $H_{\ell_k}$ with $\kappa(v) \ne 0$,
            so we conclude that $\Delta_kj \le \Delta_k\chi^-$ as desired.
        \item Suppose that $w$ and $w'$ are distinct.
            By \Cref{egalmostnielsenpathsdistinctendpoints}, 
            at least one vertex, say $w$, is new and has trivial vertex group.
            In $\mathcal{S}_N(k)$ we add the new vertex $w$,
            an edge connecting $w$ to $w'$, one ray for each $E \in \mathcal{E}$
            based at $w$ or $w'$ other than $e$ and $e'$
            and one ray corresponding to $e$ and $e'$.
            We may add one or more vertices mapping to $w'$;
            if we add more than one, each said vertex has nontrivial vertex group or is marked.
            Thus we have
            \[  \Delta_kj(\rho) = \frac{1}{2}(L(w) + L(w')) - \frac{1}{2} \]
            if $w'$ is old or has nontrivial vertex group
            and $\Delta_kj(\rho) = \frac{1}{2}(L(w) + L(w')) - \frac{3}{2}$
            if $w'$ is new and has trivial vertex group.
            Similarly,
            \[  \Delta_k\chi^-(\rho) = \frac{1}{2}(L(w) + L(w') + \kappa(w) + \kappa(w')) - 1 \]
            if $w'$ is old or has  nontrivial vertex group
            and $\Delta_k\chi^-(\rho)   = \frac{1}{2}(L(w) + L(w') + \kappa(w) + \kappa(w')) - 2$
            if $w'$ is new and has trivial vertex group.
            The argument concludes as in the previous subcase.
    \end{enumerate}
\end{proof}

\begin{cor}
    \label{actuallynoneedtotreatC2differently}
    If in the definition of $j(\psi)$, we allowed $\rank(C_2 * C_2) = 2$,
    the conclusion
    \[  j(\psi) \le n + k - 1 \]
    holds, where now the sum defining $j(\psi)$ should be taken over
    the isogredience classes of all automorphisms $\Psi$ representing $\psi$.
\end{cor}

\begin{proof}
    In the proof of \Cref{preciseindexthm}, the only place where $C_2*C_2$
    is directly treated differently is the case of a dihedral pair,
    where we had $\Delta_kj = 0$ while $\Delta_k\chi^- = 1$.
    If we were to add the dihedral pair to $\mathcal{S}_N(k)$ 
    in the case where the $C_2$ factors are fixed,
    we would get $\Delta_kj = 1$.
    The rest of the proof shows that this new $j(\psi)$ satisfies $j(\psi) \le n + k -1$.
    If $\fix(\Psi) = C_2*C_2$ and $\Psi$ is not principal, but corresponds to the lift $\tilde f$
    of a CT $f\colon \mathcal{G} \to \mathcal{G}$ representing $\psi$,
    then $\fix(\tilde f)$ projects to the quotient of $\mathbb{R}$ by the standard action of $C_2*C_2$;
    this quotient must be a dihedral pair,
    so we see that there are finitely many isogredience classes of automorphisms $\Psi$
    with this new $j(\Psi) > 0$, and that each one which is not principal corresponds to a dihedral pair.
\end{proof}

To avoid confusion, we will continue to use the index $j(\psi)$ from \Cref{preciseindexthm}.
The following corollary is immediate from the proof of \Cref{preciseindexthm}.

\begin{cor}
    Suppose $f\colon \mathcal{G} \to \mathcal{G}$ is a CT representing $\psi \in \out(F,\mathscr{A})$.
    If $f$ has a dihedral pair, then $j(\psi) \le n + k - 2$.
    If $f$ has no dihedral pairs and no exponentially  growing strata,
    then $j(\psi) = n + k - 1$.
    If $f$ has an exponentially growing stratum  without an indivisible almost Nielsen path,
    then $j(\psi) < n + k - 1$.\hfill\qedsymbol
\end{cor}

\begin{cor}[cf.~Corollaries 15.17 and 15.18 of \cite{FeighnHandelAlg}]
    The following statements hold for a CT $f\colon \mathcal{G} \to \mathcal{G}$ 
    representing $\psi \in \out(F,\mathscr{A})$.
    \begin{enumerate}
        \item Each component of $\mathcal{S}_N(f)$ contributes at least $\frac{1}{2}$ to $j(\psi)$.
        \item $\mathcal{S}_N(f)$ has at most $2j(\psi)$ components,
            and thus $P(\psi)$ has at most $2j(\psi)$ isogredience classes.
        \item $\left|\bigcup_{\Psi\in P(\psi)}\fix_+(\hat\Psi)/F\right| \le 6n + 6k - 6$.
    \end{enumerate}
\end{cor}

\begin{proof}
    The proof is similar to \cite[Corollaries 15.17 and 15.18]{FeighnHandelAlg}.
    Item 2 follows from item 1.
    The only possibility where a component of $\mathcal{S}_N(f)$
    could contribute less than $\frac{1}{2}$ to $j(\psi)$
    is the case where that component is topologically a line,
    no vertices of which have nontrivial vertex group and none of which are marked,
    and each end of the line is a ray in an exponentially growing stratum.
    It follows that if $v$ is a principal vertex such that $v_{[E]}$ lies in our component,
    then $v$ has trivial vertex group and is not the endpoint of an indivisible almost Nielsen path.
    Let $w$ be the lowest principal vertex such that $w_{[E]}$ belongs to our putative component;
    we rule out each case of the proof of \Cref{preciseindexthm} to rule out the existence of $w_{[E]}$.
    Since $w$ is principal we are not in Case (1a).
    Case (1b) is ruled out because $w$ has trivial vertex group but is principal,
    so must have at least three fixed directions.
    Since $w$ has trivial vertex group but is new in $H_{\ell_k}$, we are not in case (1c).
    If we were in case (1d), then both edges incident to $w$ must be non-linear,
    in contradiction to our assumption.
    In Case 2, if $w$ belongs to $H_{u_k}$, then an edge incident to $w$ is non-linear,
    in contradiction to our assumption.
    Therefore $w$ is a new vertex in the exponentially growing stratum $H_{\ell_k}$,
    and by \Cref{egvalence}, there are at least two fixed directions in $H_{\ell_k}$ based at $w$.
    Since $w$ is principal but is not the endpoint of an indivisible almost Nielsen path,
    there must be at least \emph{three} fixed directions in $H_{\ell_k}$ based at $w$;
    this contradiction completes the proof of item 1.

    For item 3, we have
    \[  j(\psi) \ge \frac{\left|\text{ends of $\mathcal{S}_N(f)$}\right|}{2} 
    - \left|\text{components of $\mathcal{S}_N(f)$}\right| \ge 
    \frac{1}{2}\left|\bigcup_{\Psi\in P(\psi)}\fix_+(\hat\Psi)/F\right| - 2j(\psi). \]
    So we see that $\left|\bigcup_{\Psi\in P(\psi)}\fix_+(\hat\Psi)/F\right| \le 6j(\psi)
    \le 6n + 6k - 6$.
\end{proof}

\bibliographystyle{alpha}
\bibliography{bib.bib}
\end{document}